\documentclass{amsart}

\usepackage{hyperref} 
\hypersetup{
	colorlinks=true, 
	     urlcolor= blue,  
	     linkcolor= blue, 
     citecolor=blue,	
	bookmarksopen=true,           
}
\usepackage{amsmath,amsfonts,amssymb,amsthm,enumitem,tikz-cd,graphicx,color,verbatim,xspace,bbm}

	\newcommand\Graph{\Gamma} 
	\newcommand\rGraph{\widetilde\Graph} 
	\newcommand\Cpt{C} 
	\newcommand\rCpt{\widetilde C}
	
	\newcommand\gGamma{{g(\Gamma )}}
	\renewcommand\Vert[1]{{V(#1)}} 
	\newcommand\Edge[1]{{E(#1)}} 
	\newcommand\plEdge[1]{{E(#1)^+}} 
	\newcommand\mnEdge[1]{{E(#1)^-}} 
	\newcommand\HEdge[1]{{D(#1)}} 
	\newcommand\plMaxCut[1]{{\mathfrak E(#1)^+}} 
	\newcommand\QVert[1]{{#1_\Q}} 
	\newcommand\source{\partial_0}
	\newcommand\target{\partial_1}
	\newcommand\length{l}
	\newcommand{\val}{d} 
	\newcommand\Meas{\mathrm{Meas}} 
	\newcommand\bMeas{\mathrm{Meas}_{\log}}
	\newcommand\Div{\mathrm{Div}}
	
	\newcommand\Omlog{\Omega^0_{\log}}
	\newcommand\Lapl{{\nabla^2}}
	\newcommand\dLapl{L} 
	\newcommand\invLapl{{\nabla^{-2}}}

	\newcommand\can{\mathrm{can}}
	\newcommand\tree{\mathrm{tree}}
	\newcommand\classify{j} 
	\newcommand\classifyleq[1]{\classify_{\leq #1}} 
	\newcommand\classifyeq[1]{\classify_{#1}}
	\newcommand\classifyeqe[2]{\classify_{#1,#2}}
	\newcommand\e{\mathfrak e} 
	\newcommand\bc{\mathrm{bc}}

	\newcommand\1{\mathbbm{1}}
		\input cyracc.def
		\font\sixcyr=wncyr6
		\font\eightcyr=wncyr8
	\def\littlesha{\text{\sixcyr\cyracc{Sh}}}
	\def\shuf{\mathbin{\text{\eightcyr\cyracc{Sh}}}}
	\DeclareMathOperator*\CFree{\widehat{\scalebox{1.8}{\raisebox{-0.2ex}{$\ast$}}}}
	\newcommand\Sym{\mathrm{Sym}}
	\newcommand\CSym{\mathrm{S\widehat{y}m}}
	\newcommand\HAlg{\mathcal H}
	\newcommand\oHAlg{\mathring{\mathcal H}}
	\newcommand\LAlg{\mathcal L}
	\newcommand\aHAlg{H}
	\newcommand\aHIdeal{J}
	\newcommand\aLAlg{L}
	\newcommand\acoAlg{C}
	\newcommand\n{\mathfrak n}
	\newcommand\gplike{\mathrm{gplike}}
	\newcommand\prim{\mathrm{prim}}
	\newcommand\CUEnv{\widehat\UEnv}
	\newcommand\UEnv{\mathcal{U}}
	\newcommand\Lie{\mathrm{Lie}}
	\newcommand\FLie{\mathrm{FrLie}}
	\newcommand\CFLie{\widehat\FLie}
	\newcommand\Fhat{\widehat\F}
	\newcommand\Shuf{T^\otimes_{\littlesha}} 
	\newcommand\Shufleq[1]{T^{\otimes\leq#1}_{\littlesha}}
	
	\newcommand\CTensor{\widehat T^\otimes}
	\newcommand\CCTensor{\widehat T^{\hatotimes}}
	\newcommand\CTensorleq[1]{\widehat T^{\otimes\leq#1}}

	\newcommand\ad{\mathrm{ad}}
	
	\newcommand\ab{\mathrm{ab}}
	\DeclareMathOperator\alog{\Hiso\!\log}
	\newcommand\dual{*}
	\newcommand\reddual{\circ}
	\newcommand\hatotimes{\mathbin{\widehat\otimes}}

	\def\Spec{\mathrm{Spec}}
	\def\Spf{\mathrm{Spf}}
	\renewcommand\O{\mathcal O}
	\def\G{\mathbb G}
	\def\A{\mathbb{A}}
	\renewcommand\P{\mathbb{P}}
	
	\def\ns{\mathrm{ns}}
	\def\Zar{\mathrm{Zar}}
	\def\Sel{\mathrm{Sel}}
	\def\sh{\mathrm{sh}}

\renewcommand\d{\mathrm d}
\newcommand\dby[1]{\frac{\d}{\d #1}}

\newcommand\V{\mathcal V}
\newcommand\s{\mathfrak s}
\newcommand\bigwedgesquare{\mathchoice{\bigwedge\nolimits^{\!\!2}}{\bigwedge\nolimits^{\!2}}{\bigwedge\nolimits^2}{\bigwedge\nolimits^2}}
\newcommand{\Gal}{\mathrm{Gal}}
\newcommand{\et}{\mathrm{\acute{e}t}}
\newcommand{\Cov}{\mathrm{Cov}}
\newcommand{\fin}{\mathrm{fin}}
\newcommand{\Fet}{\mathrm{F\acute{e}t}}
\newcommand{\red}{\mathrm{red}}
\newcommand{\sgn}{\mathrm{sgn}}
\newcommand{\Tr}{\mathrm{Tr}}

\DeclareMathOperator{\rk}{\mathrm{rk}}
\newcommand{\ob}{\mathrm{ob}}
\newcommand{\Jac}{\mathrm{Jac}}
\newcommand{\NS}{\mathrm{NS}}
\newcommand\Hiso{\alpha}

	\newcommand\Disc{\mathbb D}
	\newcommand\Disce[1]{\Disc(#1)}
	\newcommand\Discx{\Disc^\times}
	\newcommand\Discex[1]{\Disc(#1)^\times}
	\newcommand\mM{\mathbb M}

	\renewcommand{\H}{\mathrm H}
	\def\loc{\mathrm{loc}}

           \newcommand{\Fil}{\mathrm{Fil}}
	\def\gr{\mathrm{gr}}
	\newcommand\M{M}
	\newcommand\W{W}
	\newcommand\coM{M^c}
	\newcommand\Cent{C} 
	\newcommand\WM{\ensuremath{\mathrm{(WM_{<0})}}\xspace}
	
	\newcommand\condn{\normalfont{(\ref{condn:2-conn})}}
	\newcommand\transp[1]{{}^t#1}
          
          \newcommand{\an}{\mathrm{an}}
          \newcommand{\Frac}{\mathrm{Frac}}
          \newcommand{\dR}{\mathrm{dR}}
        \newcommand{\sep}{\mathrm{sep}}

\def\N{\mathbb N}
\def\Z{\mathbb Z}
\def\Zhat{\widehat\Z}
\def\Q{\mathbb Q}
\def\Qhat{\widehat\Q}
\def\R{\mathbb R}
\def\C{\mathbb C}
\def\F{\mathbb F}
\def\Adele{\mathbb A}
\def\nr{\mathrm{nr}}

\newcommand\SET{\mathbf{Set}}
\newcommand\AFF{\mathbf{Aff}}

\newcommand\VEC{\mathbf{Vec}}
\newcommand\fVEC{\f\VEC}
\newcommand\grVEC{\gr\VEC}
\newcommand\CVEC{\widehat{\mathbf{Vec}}}
\newcommand\fleqCVEC{\fleq\CVEC}

\newcommand\fCVEC{\f\CVEC}
\newcommand\fsCVEC{\fs\CVEC}
\newcommand\grCVEC{\gr\CVEC}

\newcommand\GPD{\mathbf{Gpd}}
\newcommand\fleqGPD{\fleq\GPD}
\newcommand\coALG{\mathbf{coAlg}}
\newcommand\HALG{\mathbf{HAlg}}
\newcommand\fleqHALG{\fleq\HALG}
\newcommand\HGPD{\mathbf{HGpd}}
\newcommand\fleqHGPD{\fleq\HGPD}
\newcommand\fHGPD{\f\HGPD}
\newcommand\grHGPD{\gr\HGPD}
\newcommand\CHGPD{\widehat{\mathbf{HGpd}}}

\newcommand\LIE{\mathbf{Lie}}
\newcommand\CLIE{\widehat\LIE}

\newcommand\f{\mathrm{f}}
\newcommand\fs{\f^{\mathrm{s}}}
\newcommand\fgeq{\f_{\geq0}}
\newcommand\fleq{\f_{\leq0}}

\newcommand\Cat{\mathcal C}
\newcommand\Dat{\mathcal D}
\newcommand\TCat{\Cat^{\otimes}}
\newcommand\TDat{\Dat^{\otimes}}
\newcommand\fleqTCat{\fleq\TCat}
\newcommand\carTCat{\Cat^{\times}}
\newcommand\CTCat{\Cat^{\hatotimes}}
\newcommand\CTDat{\Dat^{\hatotimes}}
\newcommand\fleqCTCat{\fleq\CTCat}

\newcommand\nil{\mathrm{nil}}
\newcommand\uni{\mathrm{uni}}
\newcommand\op{\mathrm{op}}

\def\Hom{\mathrm{Hom}}
\def\Aut{\mathrm{Aut}}
\def\Out{\mathrm{Out}} 

\def\End{\mathrm{End}}
\def\im{\mathrm{im}}

\def\GL{\mathrm{GL}}

\def\GSp{\mathrm{GSp}}
\def\Sp{\mathrm{Sp}}
\def\MCG{\mathrm{MCG}}

\def\coker{\mathrm{coker}}

\DeclareMathOperator\liminv{\underset\longleftarrow\lim}

\newcommand\iso{\cong} 
\newcommand\nciso{\simeq} 
\newcommand\isoarrow{\xrightarrow{\scalebox{0.6}{$\sim$}}}
\newcommand\longisoarrow{\overset\sim\longrightarrow}

\newcommand\xisoarrow[1]{\xrightarrow[{\raisebox{.6ex}{$\scriptstyle\sim$}}]{#1}}
\newcommand\acts{\curvearrowright}
\newcommand\actsarrow{\overset\acts\rightarrow}

\newcommand\hiso{\sim} 
\newcommand\viso{\wr} 

\numberwithin{equation}{subsection}
\theoremstyle{plain}
\newcounter{theoremcount}[subsection]
\makeatletter
\@addtoreset{theoremcount}{section}
\@addtoreset{equation}{section}
\@addtoreset{equation}{subsection}
\makeatother

\newtheorem{theorem}[theoremcount]{Theorem}
\newtheorem*{theorem*}{Theorem}

\newtheorem{lemma}[theoremcount]{Lemma}
\newtheorem{proposition}[theoremcount]{Proposition}
\newtheorem*{proposition*}{Proposition}
\newtheorem{corollary}[theoremcount]{Corollary}

\theoremstyle{definition}
\newtheorem{definition}[theoremcount]{Definition}
\newtheorem{definition-lemma}[theoremcount]{Definition--Lemma}
\newtheorem{example}[theoremcount]{Example}
\newtheorem{remark}[theoremcount]{Remark}
\newtheorem{notation}[theoremcount]{Notation}

\newtheorem{construction}[theoremcount]{Construction}

\usepackage{morefloats}
\usepackage{imakeidx}

\title[The local theory of unipotent Kummer maps]{The local theory of unipotent Kummer maps and refined Selmer schemes}
\author{L.\ Alexander Betts and Netan Dogra}
\address{Max-Planck Institut f\"ur Mathematik. Vivatsgasse 7, 53111 Bonn, Germany}
\email{betts@mpim-bonn.mpg.de}
\address{Mathematical Institute, University of Oxford. Andrew Wiles Building, Woodstock Rd, Oxford OX2 6GG, United Kingdom}
\email{netan.dogra@maths.ox.ac.uk}
\date\today

\makeindex[title=Index of notation]

\begin{document}
\maketitle

\begin{abstract}
We study the Galois action on paths in the $\Q_\ell$-pro-unipotent \'etale fundamental groupoid of a hyperbolic curve $X$ over a $p$-adic field with $\ell\neq p$. We prove an Oda--Tamagawa-type criterion for the existence of a Galois-invariant path in terms of the reduction of $X$, as well as an anabelian reconstruction result determining the stable reduction type of $X$ in terms of its fundamental groupoid. We give an explicit combinatorial description of the non-abelian Kummer map of $X$ in arbitrary depth, and deduce consequences for the non-abelian Chabauty method for affine hyperbolic curves and for explicit quadratic Chabauty.
\end{abstract}

\tableofcontents

\section{Introduction}
\label{s:intro}

\subsection{Non-abelian Kummer maps}
For $X$ a geometrically connected variety over a field $K$, and $b$ an $K^{\sep }$-point of $X$, the \'etale homotopy exact sequence
\begin{equation}\label{vapid}
1\to \pi _1 ^{\et }(X_{K^{\sep }},b)\to \pi _1 ^{\et }(X,b)\to \Gal (K^{\sep }|K)\to 1,
\end{equation}
allows one to study the $K$-points of $X$ by group-theoretic or cohomological means. More precisely, to each $K$-point one may associate a conjugacy class of sections of \eqref{vapid}. When $X$ is a smooth projective curve of genus $g>1$ and $K$ is a number field, Grothendieck's Section Conjecture posits that this association is bijective.

Whilst the section conjecture is in general completely open, in recent years significant progress has been by replacing $\pi _1 ^{\et }(X_{K^{\sep }},b)$ with various coarser quotients or completions. In this paper, we focus on the exact sequence
\begin{equation}\label{vapid2}
1\to \pi _1 ^{\Q _{\ell } }(X_{\overline K},b)(\Q _{\ell })\to \pi _1 ^{\Q_{\ell } }(X,b)\to \Gal (\overline{K}|K)\to 1,
\end{equation}
where $\pi _1 ^{\Q _{\ell } }(X_{\overline K},b)$ is the continuous $\Q _{\ell }$-Mal\u cev completion of $\pi_1^\et(X_{\overline K},b)$, and $\pi _1 ^{\Q _{\ell }}(X_K ,b)$ is defined by pushing out \eqref{vapid} along $\pi _1 ^{\et }(X_{\overline{K}},b)\to \pi _1 ^{\Q _{\ell } }(X_{\overline K},b)$. In this case it is not expected that the map from sections will be bijective in general. Instead, the section construction is used as an \textit{obstruction} to an adelic point being global. Kim's conjecture \cite[Conjecture 3.1]{BDCKW} states that when $\ell$ is a prime of good reduction (which we henceforth assume) this obstruction should also recover the set $X(\Q )$, or more precisely if $(x_v )\in X(\Adele_\Q^f)$ is an adelic point whose associated tuple of local sections lifts to a global section, then $x_{\ell }\in X(\Q _{\ell })$ should in fact lie in $X(\Q )$.
In many cases, the Chabauty--Kim method gives a practical method for algorithmically determining the set $X(\Q )$, or for bounding its size.

The Chabauty--Kim method hence depends on two things: an understanding of the image of \textit{global} sections in \textit{local} sections, and an understanding of the local maps to sections. This paper focuses on the latter, taking $K$ to be a finite extension of $\Q _p $, which we fix from now on. To study the local theory, we can reduce to the case where $b$ is a $K$-point of $X$, in which case we identify conjugacy classes of sections with equivalence classes of torsors in $\H ^1 (G_K ,\pi _1 ^{\Q _{\ell }}(X_{\overline K} ,b))$. When $\ell =p$, the \textit{unipotent Kummer map}
\[
X(\Q _p )\to \H ^1 (G_K ,\pi _1 ^{\Q _{\ell }}(X_{\overline{K}},b))
\]
may be described using non-abelian $p$-adic Hodge theory and iterated Coleman integration (recall $\ell$ is a prime of good reduction for $X$).

A striking feature of the unipotent Kummer map is that, unlike the $\Q _{\ell }$-linear Kummer map for abelian varieties, its local description when $p\neq \ell $ is non-trivial. If $X$ is proper and $p$ is a prime of potential good reduction, then it is straightforward to show that the map is trivial. In general, however, this map has finite image when $X$ is proper \cite{kimtamagawa}, and can have infinite image when $X$ is affine. In \cite{BalakrishnanDogra1}, certain quotients of this map were computed in special cases, utilising a description in terms of local height pairings.

The main result of this paper is an essentially complete description of the unipotent Kummer map\[\classify\colon X(K)\rightarrow\H^1(G_K,\pi_1^{\Q_\ell}(X_{\overline K},b))\]when $K$ is a finite extension of $\Q_p$, $X/K$ is a smooth geometrically connected hyperbolic curve (see \S\ref{ss:nonab_Kummer} for the definition) and $p\neq \ell$. 
In fact, we will work geometrically: for a finite extension $L/K$, the restriction map\[\H^1(G_K,\pi_1^{\Q_\ell}(X_{\overline K},b))\hookrightarrow\H^1(G_L,\pi_1^{\Q_\ell}(X_{\overline K},b))\]is injective (Lemma \ref{lem:finite_index}), and the non-abelian Kummer map for $X_L$ is compatible for that for $X$, so that the non-abelian Kummer maps for varying $L$ induce a natural map\[\classify\colon X(\overline{K})\rightarrow\varinjlim_L\H^1(G_L,\pi_1^{\Q_\ell}(X_{\overline K},b)).\]We will describe this map explicitly in terms of terms of harmonic analysis on the stable reduction graph $\rGraph$ of $X$ (see Definition \ref{def:reduction_graphs_of_curves}); these are concretely related by a certain \emph{reduction map} whose definition we now recall.

\begin{definition}[Reduction map]\label{def:reduction_map}\index{reduction\dots!\dots map $\red$}
Let $X/K$ be a smooth geometrically connected hyperbolic curve and $x\in X(\overline K)$. We may choose a regular semistable model $\mathcal X/\O_L$ over a finite extension $L/K$ such that $x\in\mathcal X(\O_L)$, and we denote by $\red(x)$ the irreducible component of the geometric special fibre $\mathcal X_{\overline k}$ containing the reduction of $x$. The point $\red(x)\in\rGraph$ corresponding to this component is independent of the choice of model, and we thus obtain a map\[\red\colon X(\overline K)\rightarrow\rGraph\]called the \emph{reduction map}. The image of this map is exactly the set $\QVert{\Graph}$ of points a rational distance from vertices.
\end{definition}

Using this map, we can now state our explicit calculation of the non-abelian Kummer map, at least in outline.

\begin{theorem}\label{thm:description_of_kummer}
Let $X/K$ be a smooth geometrically connected hyperbolic curve with a geometric basepoint $b\in X(\overline K)$ and write $\rGraph$ for the stable reduction graph of $X$. We will define a certain $\Q$-vector space\[\V=\prod_{n>0}\gr^\W_{-n}\V,\]where each of the graded pieces $\gr^\W_{-n}\V$ is finite-dimensional, along with a series $(\mu_n)_{n>0}$ of $\gr^\W_{-n}\V$-valued piecewise polynomial measures on $\rGraph$ with log poles on half-edges, each of total mass $0$. Both $\V$ and the measures $\mu_n$ depend only on the reduction graph $\rGraph$ (see Definition~\ref{def:V} and Construction~\ref{cons:measures} for the constructions).

Now for any prime $\ell\neq p$, there is an injection\[\iota\colon\varinjlim_L\H^1(G_L,\pi_1^{\Q_\ell}(X_{\overline K},b))\hookrightarrow\V_{\Q_\ell}:=\prod_{n>0}\left(\Q_\ell\otimes\gr^\W_{-n}\V\right)\]identifying $\varinjlim_L\H^1(G_L,\pi_1^{\Q_\ell}(X_{\overline K},b))$ as a graded pro-finite-dimensional $\Q_\ell$-subspace of $\V_{\Q_\ell}$. The composite map\[X(\overline K)\overset\classify\rightarrow\varinjlim_L\H^1(G_L,\pi_1^{\Q_\ell}(X_{\overline K},b))\overset\iota\hookrightarrow\V_{\Q_\ell}=\prod_{n>0}(\Q_\ell\otimes\gr^\W_{-n}\V)\]is given by\[x\mapsto(\langle\red(x)-\red(b),\mu_n\rangle)_{n>0},\]where $\langle\cdot,\cdot\rangle$ denotes the height pairing on measures of total mass $0$.
\end{theorem}

\begin{remark}
A consequence of Theorem \ref{thm:description_of_kummer} is that the non-abelian Kummer map $\classify$ factors through a map $\QVert{\Graph}\rightarrow\V$ determined solely by $\rGraph$. Thus, this theorem provides an independence of $\ell$ result for the non-abelian Kummer map.
\end{remark}

\begin{example}
We defer the general definitions of $\V$ and of the $\mu_n$, which although conceptually rather straightforward are somewhat elaborate, to the main body of the paper (Definition~\ref{def:V} and Construction~\ref{cons:measures}). However, we illustrate these constructions by stating a simple special case.

Suppose that $X$ is projective and that its Jacobian has potentially totally toric reduction -- in other words, assume that $\rGraph$ has no half-edges and has genus $0$ at every vertex. We write $\H_1(\Graph)=\H_1(\Graph,\Q)$ for the $\Q$-linear homology of the underlying graph of $\rGraph$, and $\beta\colon\Sym^2\H_1(\Graph)\rightarrow\Q$ for the cycle pairing given by the length of the intersection of cycles. There is then an isomorphism\[\gr^\W_{-2}\V\iso\coker\left(\beta^*\colon\Q\rightarrow\Sym^2\H^1(\Graph)\right)\]and the $\gr^\W_{-2}\V$-valued measure $\mu_2$ is given by\[\mu_2=\sum_{e\in\plEdge{\rGraph}}(e^*)^2\cdot\d s_e,\]where $\d s_e$ denotes arc-length along an edge $e$ and $e^*\in\H^1(\Graph)=\Hom(\H_1(\Graph),\Q)$ denotes the functional sending a cycle $\gamma$ to the multiplicity of $e$ in $\gamma$.
\end{example}

\subsection{Applications to the Chabauty--Kim method}\label{ss:application_CK}

Our main Theorem \ref{thm:groupoid_oda} has consequences for the Chabauty--Kim method. This method, developed in \cite{siegel,selmer_varieties,kim-coates,BDCKW}, seeks to determine the integral points of hyperbolic curves $X/\Q$ by means of the \emph{non-abelian Kummer map}\[\classify\colon X(\Q)\rightarrow\H^1(G_\Q,\pi_1^{\Q_\ell}(X_{\overline\Q},b))\]which sends a rational point $x\in X(\Q)$ to the class of the $G_\Q$-equivariant right $\pi_1^{\Q_\ell}(X_{\overline\Q},b)$-torsor $\pi_1^{\Q_\ell}(X_{\overline\Q};b,x)$, where $b\in X(\Q)$ is a fixed rational basepoint. More precisely, if $S$ is a finite set of rational primes and $\mathcal X/\Z_S$ is a model of $X$, not necessarily minimal, such that $b$ extends to a smooth $\Z_S$-point of $\mathcal X$, the Chabauty--Kim method aims to locate $\mathcal X(\Z_S)$ inside $\prod_{p\notin S}\mathcal X(\Z_p)\times\prod_{p\in S}X(\Q_p)$ via the commuting square
\begin{center}
\begin{tikzcd}\label{the_commutative_diagram}
\mathcal X(\Z_S) \arrow{r}{\classify}\arrow{d} & \H^1(G_\Q,U) \arrow{d} \\
\prod_{p\notin S}\mathcal X(\Z_p)\times\prod_{p\in S}X(\Q_p) \arrow{r}{\prod_p\classify_p} & \prod_p\H^1(G_{\Q_p},U)
\end{tikzcd}
\end{center}
where $U$ is a finite dimensional $G_{\Q }$-stable quotient of $\pi _1 ^{\Q _\ell }(X_{\overline \Q },b)$ and $\classify_p\colon X(\Q_p)\rightarrow\H^1(G_{\Q_p},U)$ is the local non-abelian Kummer map, defined in exactly the same way as $\classify$. One typically assumes in the above that $\ell\notin S$ is a prime of good reduction for the model $\mathcal X$.

Under reasonable assumptions\footnote{Precisely, under the assumption that $U$ satisfies condition~\WM from \S\ref{ss:gal_coh_field}.} on $U$ (for instance if $U=U_n$ is the largest $n$-step unipotent quotient of $\pi_1^{\Q_\ell}(X_{\overline\Q},b)$), all of the non-abelian cohomology sets involved are the $\Q_\ell$-points of pointed affine schemes over $\Q_\ell$ and the localisation map is a scheme morphism\footnote{Technically, one should replace $\H^1(G_\Q,U)$ with $\H^1(G_{\Q,T},U)$, where $T$ is a finite set of primes containing $S,p$ and all primes of bad reduction.}. One defines the \emph{Selmer scheme} $\Sel_{S,U}(\mathcal X)\subseteq\H^1(G_\Q,U)$ to be the subscheme of cohomology classes which are \emph{locally geometric} outside $S$, i.e.\ classes $\xi$ such that $\xi|_{G_{\Q_p}}\in\classify_p(\mathcal X(\Z_p))^\Zar$ for all $p\notin S$ \cite[\S8]{BDCKW}, where $(-)^\Zar$ denotes the Zariski closure. The nature of these local geometricity conditions is clarified by the following qualitative description of the images of the local non-abelian Kummer maps.

\begin{proposition*}[{\cite[Corollary 0.2]{kimtamagawa},\cite[Theorem 1]{selmer_varieties},\cite[\S4]{faltings_around}}]
\leavevmode
\begin{itemize}
	\item If $p\notin S\cup\{\ell\}$, then $\classify_p(\mathcal X(\Z_p))^\Zar=\classify_p(\mathcal X(\Z_p))$ is finite. If moreover $p$ is of potentially good reduction for $\mathcal X$, i.e.\ $\mathcal X_{\O_L}$ is dominated by a good model of $X_L$ for a finite extension $L/\Q_p$, then $\classify_p(\mathcal X(\Z_p))=\{\ast\}$ is the basepoint of $\H^1(G_{\Q_p},U)$.
	\item If $p=\ell$ then our assumptions that $\ell\notin S$ is a prime of good reduction ensure that $\classify_\ell(\mathcal X(\Z_\ell))^\Zar=\H^1_f(G_{\Q_\ell},U)$ is the non-abelian Bloch--Kato Selmer scheme defined in \cite[\S2]{selmer_varieties}.
\end{itemize}
\end{proposition*}

The output of the Chabauty--Kim method is a subset $\mathcal X(\Z_\ell)_{S,U}\subseteq\mathcal X(\Z_\ell)$ containing $\mathcal X(\Z_S)$, namely the set of all $x\in\mathcal X(\Z_\ell)$ such that $\classify_\ell(x)$ lies in the scheme-theoretic image of the localisation map $\Sel_{S,U}(\mathcal X)\rightarrow\H^1_f(G_{\Q_\ell},U)$ \cite[\S8]{BDCKW}. The set $\mathcal X(\Z_\ell)_{S,U}$ is in theory (and often in practice) computable and conjecturally is equal to $\mathcal X(\Z_S)$ for $U$ sufficiently large \cite[Segment~3.1]{BDCKW}. The set $\mathcal X(\Z_\ell)_{S,U}$, and hence $\mathcal X(\Z_S)$, is finite as soon as the localisation map is not scheme-theoretically dominant, for example if the inequality
\begin{equation}\label{eq:selmer_ineq}
\sum_{i>0}\dim_{\Q_\ell}\H^1_{f,S}(G_\Q,\gr_\Cent^iU) < \sum_{i>0}\dim_{\Q_\ell}\H^1_f(G_{\Q_\ell},\gr_\Cent^iU)
\end{equation}
holds. Here, $\H^1_{f,S}(G_\Q,-)$ denotes the subspace of global continuous Galois cohomology classes which are crystalline at $\ell$ and unramified at all $p\notin S\cup\{\ell\}$, and $\gr_\Cent^i U$ denotes the graded pieces of the descending central series filtration\footnote{The particular choice of filtration here is not especially important; one could equally use the analogue of~\eqref{eq:selmer_ineq} with $\Cent^\bullet U$ replaced by, for example, the weight filtration $\W_\bullet U$ defined in Definition~\ref{def:wt1}.} on $U$ (so in particular are abelian).

\smallskip

The techniques we develop in this paper will allow us to precisely determine the images of the local non-abelian Kummer maps at all primes $p\neq\ell$, even primes in $S$ where the image need not be finite. For example, we can refine and extend the above descriptions of the local images.

\begin{proposition}\label{prop:kummer_bounds}
\leavevmode
\begin{enumerate}
	\item\label{proppart:size_bound} Let $p\notin S\cup\{\ell\}$ be a prime and let $\mathcal X'/\O_L$ be a regular semistable model of $X$ over the ring of integers of a finite extension $L/\Q_p$ which dominates $\mathcal X_{\O_L}$, with geometric special fibre $\mathcal X'_{\overline{\F}_p}$. Then $\#\classify_p(\mathcal X(\Z_p))$ is at most the number of irreducible components of $\mathcal X'_{\overline{\F}_p}$ containing the reduction of a point in $\mathcal X(\Z_p)\subseteq\mathcal X'(\O_L)$. If $U$ dominates $U_3$, then we have equality.
	\item\label{proppart:dimension_bound} Let $p\in S\setminus\{\ell\}$. Then the $\Q_\ell$-dimension of $\classify_p(X(\Q_p))^\Zar$ is at most $1$. If $U$ dominates $U_2$, then the dimension is equal to $1$ if and only if $X$ has a $\Q_p$-rational cusp.
\end{enumerate}
In both cases, the image $\classify_p(\mathcal X(\Z_p))$ or $\classify_p(X(\Q_p))$ is explicitly computable, in a sense we shall make precise later.
\end{proposition}


In fact, we will prove the evident analogue of Proposition~\ref{prop:kummer_bounds} over arbitrary number fields in place of $\Q$. Part~\eqref{proppart:size_bound} is used in work of the second author \cite[Lemma~2.2]{BDeff}.



\subsubsection{A refined Selmer scheme}

The fact that Proposition~\ref{prop:kummer_bounds} also constrains the images of local non-abelian Kummer maps at places in $S$ enables us to refine the Chabauty--Kim method in the case that $S\neq\emptyset$ (so this is only of interest for affine curves $X$). We do this via the following refinement of the Selmer scheme from \cite{BDCKW}.

\begin{definition}[Refined Selmer scheme]
Keep notation as above, and assume that $\ell\notin S$ is a prime of good reduction. We define a refined Selmer scheme $\Sel_{S,U}^{\min}(\mathcal X)\subseteq\H^1(G_\Q,U)$ to be the subscheme consisting of cohomology classes that are everywhere locally geometric in the sense that $\xi|_{G_{\Q_p}}\in\classify_p(\mathcal X(\Z_p))^\Zar$ for all $p\notin S$ and $\xi|_{G_{\Q_p}}\in\classify_p(X(\Q_p))^\Zar$ for all $p\in S$. This is a subscheme of the usual Selmer scheme $\Sel_{S,U}(\mathcal X)$ containing the image of the global non-abelian Kummer map $\classify$.
\end{definition}

Using this \textit{refined} Selmer scheme, one can define refined subsets $\mathcal X(\Z_\ell)_{S,U}^{\min}\subseteq\mathcal X(\Z_\ell)$ in exactly the same way as for $\mathcal X(\Z_\ell)_{S,U}$, which by definition sit in a sequence of containments\[\mathcal X(\Z_S)\subseteq\mathcal X(\Z_\ell)_{S,U}^{\min}\subseteq\mathcal X(\Z_\ell)_{S,U}\subseteq\mathcal X(\Z_\ell).\]The set $\mathcal X(\Z_\ell)_{S,U}^{\min}$ is finite as soon as the localisation map $\Sel_{S,U}^{\min}(\mathcal X)\rightarrow\H^1_f(G_{\Q_p},U)$ is not scheme-theoretically dominant, for example if the inequality
\begin{equation}\label{eq:modified_selmer_ineq}
\#S+\sum_{i>0}\dim_{\Q_\ell}\H^1_f(G_\Q,\gr_\Cent^iU) < \sum_{i>0}\dim_{\Q_\ell}\H^1_f(G_{\Q_\ell},\gr_\Cent^iU)
\end{equation}
holds. Here, $\H^1_f(G_\Q,-)$ denotes the subspace of global continuous Galois cohomology classes which are crystalline at $\ell$ and unramified at all $p\neq\ell$.

\smallskip

In general, we expect~\eqref{eq:modified_selmer_ineq} to be satisfied more often than~\eqref{eq:selmer_ineq}; we give simple examples in \S\ref{sss:P1_finite} and \S\ref{sss:E_finite} where~\eqref{eq:modified_selmer_ineq} holds but~\eqref{eq:selmer_ineq} does not.

\subsubsection{Applications to quadratic Chabauty}

A particularly tractable case for making the non-abelian Chabauty method explicit and computable is when $X$ is projective and satisfies the inequality
\[
\rk J(\Q ) < \rho (J) +g-1,
\]
where $J$ is the Jacobian of $X$ and $\rho $ is the Picard number.
In this case $X(\Q_\ell)_2=\mathcal X(\Z_\ell)_{\emptyset,U_2}$ is finite, and in \cite{Balakrishnanetc} an algorithm is given for computing a finite set containing it in certain cases. A major restriction of this algorithm is the need for $X$ to have potential good reduction at all primes. In this paper we explain how to remove this restriction. As explained in \S \ref{s:examples}, this amounts to giving an explicit formula for certain functions $j_F :X(\Q _p )\to \Q _{\ell }$ arising from a choice of base-point $b\in X(\Q _p )$ and an endomorphism $F$ of the Jacobian of $X$ satisfying certain conditions. In Corollary \ref{cor:explicit_formula}, we give an explicit formula for $j_F $ in terms of height pairings $\langle\cdot,\cdot\rangle $ on the dual graph of $X$.

\subsection{An analogue of Oda's theorem for fundamental groupoids}

Let $K$ be a finite extension of $\Q_p$, and $A$ an abelian variety over $K$. The Serre--Tate criterion \cite[Theorem 1]{serre-tate} says that we can detect good reduction of $A$ from the Galois action on the $\ell$-adic Tate module of $A$, or equivalently from the action on $\H^1_\et(A_{\overline{K}},\Q_\ell)$, for $\ell$ a prime different from $p$. Specifically, $A$ has good reduction if and only if this Galois action is unramified.

If we replace $A$ with a smooth geometrically connected curve $X$ over $K$, then the Galois action on cohomology no longer detects good reduction of $X$. However, by replacing the cohomology of $X$ by a finer arithmetic-homotopical invariant -- its \'etale fundamental group -- one can indeed detect good reduction of $X$. When $X$ is projective this is due to Oda \cite{oda}, and for $X$ affine this is due to Tamagawa \cite{tamagawa_grothendieck_conjecture}.

\begin{theorem*}[{\cite[Theorem 3.2]{oda}}, {\cite[Theorem 0.8]{tamagawa_grothendieck_conjecture}}]
A smooth geometrically connected hyperbolic curve $X/K$ has good reduction if and only if the outer action of $G_K=\Gal(\overline{K}|K)$ on the pro-$\ell$ \'etale fundamental group $\pi_1^{\et,\ell}(X_{\overline{K}},b)$ is unramified (for any or every geometric basepoint $b$).
\end{theorem*}

In fact, Oda proves a more refined version of this theorem. If we let $U_n$ denote the maximal $n$-step unipotent quotient of the $\Q_\ell$-pro-unipotent \'etale fundamental group of $X_{\overline K}$ -- i.e.\ the continuous $\Q_\ell$-Mal\u cev completion of $\pi_1^{\et,\ell}(X_{\overline K},b)$ -- then $X$ has good reduction if and only if the induced outer Galois action on $U_n$ is unramified, for any $n\geq3$.

\smallskip

The techniques in this paper also prove a relative analogue of the Oda--Tamagawa criterion for $\Q_\ell$-pro-unipotent \'etale path-torsors. Recall that to any two $K$-rational points $x,y\in X(K)$, one can associate a \emph{$\Q_\ell$-pro-unipotent torsor of \'etale paths} $\pi_1^{\Q_\ell}(X_{\overline K};x,y)$, which is a certain affine $\Q_\ell$-scheme (which is isomorphic to an infinite-dimensional affine space) endowed with an action of the Galois group $G_K$. There are two possibilities for the scheme $\pi_1^{\Q_\ell}(X_{\overline K};x,y)^{I_K}$ of inertia-invariant paths: it is either empty or a torsor under $\pi_1^{\Q_\ell}(X_{\overline K},x)^{I_K}$. Our main result is an Oda--Tamagawa-type criterion describing which of these occurs in terms of the reduction of $X$.

\begin{theorem}[Relative Oda--Tamagawa criterion for path-torsors]\label{thm:groupoid_oda}
Let $X$ be a smooth geometrically connected hyperbolic curve over $K$ and $x,y\in X(K)$ rational points. Let $\mathcal X/\O_L$ be a regular semistable model of $X$ over the ring of integers of a finite extension $L/K$ such that $x$ and $y$ extend to $\O_L$-integral points of $\mathcal X$. Then the following are equivalent:
\begin{itemize}
	\item the reductions of $x$ and $y$ lie on the same irreducible component of the special fibre of $\mathcal X$;
	\item there is an inertia-invariant $\Q_\ell$-pro-unipotent \'etale path from $x$ to $y$, i.e.\ $\pi_1^{\Q_\ell}(X_{\overline K};x,y)^{I_K}\neq\emptyset$;
	\item\label{thmpart:gal_invt_path} there is a Galois-invariant $\Q_\ell$-pro-unipotent \'etale path from $x$ to $y$, i.e.\ $\pi_1^{\Q_\ell}(X_{\overline K};x,y)^{G_K}\neq\emptyset$.
\end{itemize}

In the above, by a \emph{model}\index{model of a curve $\mathcal X$} $\mathcal X$ over $\O_L$, we mean the complement in a proper flat generically smooth $\O_L$-scheme $\overline{\mathcal X}$ of a divisor $\mathcal D$ which is \'etale over $\O_L$, along with an isomorphism $X_L\isoarrow(\overline{\mathcal X}\setminus\mathcal D)_L$. We say that $\mathcal X$ is \emph{regular} just when $\overline{\mathcal X}$ is a regular scheme, and we say that $\mathcal X$ is \emph{semistable} just when the special fibre of $\overline{\mathcal X}$ is a reduced normal crossings divisor, each of whose smooth geometric components meet the remaining geometric components and $\mathcal D$ in $\geq2$ points. Every hyperbolic curve $X/K$ has a regular semistable model over $\O_L$ for a finite extension $L/K$.
\end{theorem}

\begin{remark}
Akin to the more refined statement of Oda's criterion, we will also prove a finite level version of our relative Oda--Tamagawa criterion. If we let $P_n(x,y)$ denote the $\Q_\ell$-pro-unipotent \'etale torsor of paths in depth $n$ -- i.e.\ the pushout of $\pi_1^{\Q_\ell}(X_{\overline K};x,y)$ along the quotient map $\pi_1^{\Q_\ell}(X_{\overline K},x)\twoheadrightarrow U_n$ -- then the equivalent conditions of Theorem \ref{thm:groupoid_oda} are also equivalent to the condition $P_n(x,y)^{I_K}\neq\emptyset$, or that $P_n(x,y)^{G_K}\neq\emptyset$, for any $n\geq3$.

In terms of the \textit{non-abelian Kummer map} (whose definition is recalled in Section \ref{ss:application_CK}), condition~\eqref{thmpart:gal_invt_path} of Theorem \ref{thm:groupoid_oda} is equivalent to the class of $y$ in $\H^1(G_K,\pi_1^{\Q_\ell}(X_{\overline K},x))$ being trivial. The finite-level version of this stated in the previous paragraph is equivalent to the class of $y$ in $\H^1 (G_K ,U_3 )$ being trivial.

The equivalence of the second and third conditions in Theorem \ref{thm:groupoid_oda} is a special case of a Galois cohomological result (the general result is Corollary \ref{cor:G-invt_is_I-invt}).
\end{remark}

The Oda--Tamagawa criterion is best understood in the wider context of anabelian reconstruction theorems: results which aim to reconstruct certain properties of suitable varieties from the Galois action on their \'etale fundamental groups. Most famously, a theorem of Mochizuki assures us that a hyperbolic curve is already determined by its pro-$p$ fundamental group.

\begin{theorem*}[$\subseteq${\cite[Theorem A]{mochizuki}}]
Let $X$ and $X'$ be hyperbolic curves over $K$. Then the pro-$p$ \'etale fundamental groups $\pi_1^{\et,p}(X_{\overline K},x)$ and $\pi_1^{\et,p}(X'_{\overline K},x')$ are $G_K$-equivariantly isomorphic (for the outer Galois action) if and only if $X$ and $X'$ are isomorphic over $K$.
\end{theorem*}

\begin{remark}
Mochizuki's Theorem is much more general than the version we present above. Most notably, $K$ can be taken to be any field contained in a finitely generated extension of $\Q_p$. In particular, taking $K$ to be a finitely generated extension of $\Q$, Mochizuki's Theorem resolves the anabelian conjecture of Grothendieck on determination of hyperbolic curves by their fundamental groups.
\end{remark}

In contrast to Mochizuki's Theorem, there is no hope of recovering a hyperbolic curve $X/K$ from the Galois action on its pro-$\ell$ or $\Q_\ell$-pro-unipotent fundamental group, since these vary locally constantly in algebraic families (or see Remark \ref{rmk:no_pro-l_reconstruction} for a specific example). Nonetheless, the Oda--Tamagawa criterion suggests that the pro-$\ell$ fundamental group should at least recover some information about the reduction type of $X$.

This is exactly the content of our second main theorem -- a consequence of the first -- which asserts that one can recover the stable reduction type of a hyperbolic curve from its $\Q_\ell$-pro-unipotent fundamental groupoid. In order to make the notion of stable reduction types precise, we record the following standard definition, which will be used throughout this paper.

\begin{definition}[Stable reduction graphs]\label{def:reduction_graphs_of_curves}
Let $X/K$ be a smooth geometrically connected hyperbolic curve, choose a regular semistable model $\mathcal X=\overline{\mathcal X}\setminus\mathcal D$ of $X$ over the ring of integers $\O_L$ of a finite extension $L/K$, and write $\mathcal X_{\overline k}=\overline{\mathcal X}_{\overline k}\setminus\mathcal D_{\overline k}$ for the geometric special fibre of $\mathcal X$. The \emph{stable reduction graph} $\rGraph=\rGraph(X)$ of $X$ is defined to be the dual graph of $\mathcal X_{\overline k}$, i.e.\ the graph consisting of:
\begin{itemize}
	\item a vertex for each irreducible component of $\overline{\mathcal X}_{\overline k}$;
	\item an (unoriented) edge for each singular point of $\overline{\mathcal X}_{\overline k}$, whose endpoints are the two irreducible components containing it (these two components may coincide, in which case the edge is a loop); and
	\item a `half-edge' for each point of $\mathcal D_{\overline k}$, whose endpoint is the irreducible component containing it.
\end{itemize}
We regard $\rGraph$ as a metric graph by assigning each finite edge a length $\frac1{e(L/K)}$ and each half-edge an infinite length, and endow it with a \emph{genus function} $g\colon\Vert{\rGraph}\rightarrow\N_0$ assigning to each vertex the geometric genus of the corresponding irreducible component. Up to subdivision of edges and half-edges, the reduction graph $\rGraph$ is canonically independent of the choice of model $\mathcal X/\O_L$.
\end{definition}

We will recover the stable reduction graph of our hyperbolic curve $X$ not from its $\Q_\ell$-pro-unipotent \'etale fundamental group $\pi_1^{\Q_\ell}(X_{\overline K},b)$ but its $\Q_\ell$-pro-unipotent \'etale fundamental group\emph{oid}\index{fundamental groupoid!\dots of a curve $\pi_1^{\Q_\ell}(X_{\overline K})$} $\pi_1^{\Q_\ell}(X_{\overline K})$; that is, the structure comprised of the path-spaces $\pi_1^{\Q_\ell}(X_{\overline K};x,y)$ for points $x,y\in X(\overline K)$ together with the path-composition
\[\pi_1^{\Q_\ell}(X_{\overline K};y,z)\times\pi_1^{\Q_\ell}(X_{\overline K};x,y)\rightarrow\pi_1^{\Q_\ell}(X_{\overline K};x,z)\]
and path-reversal maps
\[\pi_1^{\Q_\ell}(X_{\overline K};x,y)\longisoarrow\pi_1^{\Q_\ell}(X_{\overline K};y,x)\]
for $x,y,z\in X(\overline K)$. All of these path-spaces are affine $\Q_\ell$-schemes (and the structure maps are morphisms thereof), and there is an action of $G_K$ on $\pi_1^{\Q_\ell}(X_{\overline K})$ extending the action on $X(\overline K)$ which is continuous in the sense that an open subgroup of $G_K$ acts continuously on the $\Q_\ell$-points of $\pi_1^{\Q_\ell}(X_{\overline K};x,y)$ for every $x,y\in X(\overline K)$. Our second theorem asserts that the Galois action on this fundamental groupoid is enough to recover its stable reduction type.

\begin{theorem}[Anabelian reconstruction of the reduction graph]\label{thm:anabelian_reconstruction}
Let $X$ and $X'$ be two smooth geometrically connected hyperbolic curves over $K$. If the $\Q_\ell$-pro-unipotent \'etale fundamental groupoids $\pi_1^{\Q_\ell}(X_{\overline K})$ and $\pi_1^{\Q_\ell}(X'_{\overline K})$ are $G_K$-equivariantly isomorphic, then the stable reduction graphs $\rGraph(X)$ and $\rGraph(X')$ are isometric compatibly with the genus function.
\end{theorem}

\begin{remark}
An alternative version of this reconstruction theorem is proved by Mochizuki in \cite[Corollary 3.11]{anabelioids}, which implies that the stable reduction graph of $X$ (without its metric) is determined by its geometric tempered fundamental group. Our reconstruction result is somewhat orthogonal to Mochizuki's in that while we work with a coarser version of the fundamental group, we are very much concerned with its Galois action and groupoid structure.

Another partial reconstruction result along these lines appears in \cite[Theorem 1.1]{AMO}, which shows that many graph-theoretic invariants of $\rGraph$ are determined from the outer mapping class group action on the pro-$\ell$ fundamental group of a totally degenerate curve.
\end{remark}

\begin{remark}\label{rmk:no_pro-l_reconstruction}
Somewhat surprisingly, Theorem \ref{thm:anabelian_reconstruction} fails if we replace the $\Q_\ell$-pro-unipotent fundamental groupoid with the $\Q_\ell$-pro-unipotent or even pro-$\ell$ fundamental group, say based at rational points $b\in X(K)$ and $b'\in X'(K)$. To sketch the argument, suppose that $X$ is a once-punctured Tate elliptic curve $(\G_m/q^\Z)\setminus\{1\}$ and that the rational basepoint $b$ corresponds to a point in $\O_K^\times$ not in the residue disc of $1$. It then follows from the calculations in \S \ref{s:oda_reduction} that the pro-$\ell$ \'etale fundamental group of $(X_{\overline K},b)$ is free on two generators $x,y$, with inertia acting via $\sigma(x)=x$ and $\sigma(y)=yx^{v_K(q)\cdot t_\ell(\sigma)}$ where $t_\ell\colon I_K\rightarrow\Z_\ell^\times$ is the tame character.

In particular, the pro-$\ell$ fundamental groups of two such curves with parameters $q$ and $q^p$ are $I_K$-equivariantly isomorphic, via the isomorphism $(x,y)\mapsto(x^p,y)$. However, their reduction graphs are easily seen to be non-isometric.
\end{remark}

\begin{remark}\label{rmk:mono-anabelian}
In the proof of Theorem \ref{thm:anabelian_reconstruction} in \S\ref{s:reconstruction}, we will in fact describe an explicit recipe to read off the stable reduction graph of a hyperbolic curve $X$ from the $G_K$-action on its fundamental groupoid. In the language of \cite{mono-anabelian}, Theorem \ref{thm:anabelian_reconstruction} is a mono-anabelian reconstruction result.
\end{remark}

\subsection{Overview of proofs and sections}

The bulk of this paper will be devoted to proving Theorem \ref{thm:description_of_kummer}. We begin in Section \ref{s:galois_cohomology} by recalling some facts about continuous Galois cohomology of $\Q_\ell$-pro-unipotent groups $U$. Although much of this is standard, we already begin to see some new structure in Lemma \ref{lem:description_of_cohomology_basic}, where we show that $\H^1(G_K,U)$ admits a description as a graded $\Q_\ell$-vector space rather than just an affine space.

Section \ref{s:oda_reduction} then applies deformation techniques of \cite{oda,AMO} to explicitly describe the $\Q_\ell$-pro-unipotent fundamental groupoid of a hyperbolic curve $X$ as the $\Q_\ell$-Mal\u cev completion of a certain (discrete) fundamental groupoid $\pi_1(\rGraph)$ associated to its reduction graph. This section contains the majority of the number-theoretic content of the paper, and effectively reduces our main problems to ones in combinatorics.

The proofs of these combinatorial analogues are then accomplished in the following sections. In Section \ref{s:M-trivialisation}, we set up our general framework for tackling such problems, by giving an explicit and basepoint-independent description of the $\M$-graded fundamental groupoid of $\rGraph$, for a certain filtration $\M_\bullet$. The key ingredient in the proof is a certain \emph{canonical} choice of $\Q$-pro-unipotent path $\gamma_{x,y}^\can$ between any two points of a rationally metrised graph $\Graph$, use of which ensures that our constructions are basepoint-independent. The canonical paths $\gamma_{x,y}^\can$ owe their origin to the theory of iterated integration on graphs due to Cheng and Katz; this theory is reviewed in Section \ref{s:cheng-katz}.

The proof of Theorem \ref{thm:description_of_kummer} is then accomplished in Sections \ref{s:n-a_kummer_for_graphs}, \ref{s:harmonic_analysis} and \ref{s:computation}. The short Section \ref{s:n-a_kummer_for_graphs} uses the description of the $\M$-graded fundamental groupoid from Section \ref{s:M-trivialisation} to give a reinterpretation of the non-abelian Kummer map of a curve in terms of the action of a certain monodromy operator on the canonical paths $\gamma_{x,y}^\can$. This purely combinatorially defined map is then studied in Section \ref{s:computation}, where it is shown to satisfy a certain second-order ordinary differential equation. This allows us to compute the graph-theoretic Laplacian of the non-abelian Kummer map, and hence write it as a height pairing against a certain measure. The relevant aspects of harmonic analysis on graphs are recalled in Section \ref{s:harmonic_analysis}.

Sections \ref{s:graph_ops} and \ref{s:injectivity} are concerned with the proof of our relative Oda--Tamagawa criterion (Theorem \ref{thm:groupoid_oda}), by now rephrased as an injectivity result for the graph-theoretic non-abelian Kummer map. The proof proceeds in two steps. Firstly, one uses the explicit form of the measure $\mu_2$ to show that the map given by the height pairing against the measure $\mu_2$ is almost injective, with the only failures of injectivity arising from certain symmetries of the $2$-connected components of the reduction graph. Exploiting these symmetries, one then is able to show that the maps given by height pairing against the measures $\mu_2$ and $\mu_3$ are jointly injective. Many of the proofs in Section \ref{s:injectivity} utilise certain simplification operations on reduction graphs inspired by the theory of electrical circuits, and these operations are studied in Section \ref{s:graph_ops}.

The proof of the anabelian reconstruction theorem (Theorem \ref{thm:anabelian_reconstruction}) is undertaken in Section \ref{s:reconstruction}. The key observation here is that our relative Oda--Tamagawa criterion ensures that the non-abelian Kummer map of a hyperbolic curve $X$ embeds its reduction graph inside an affine space, and hence we can recover the graph as the image of this map. The recovery of the metric and genus function are then straightforward from this.

In the final Section \ref{s:examples}, we survey several applications of these techniques to explicit quadratic Chabauty. Some other applications to explicit quadratic Chabauty are given elsewhere in the paper: in particular the description of the non-abelian Kummer map for certain genus two curves given in Lemma \ref{lemma:used_in_BD} is used in \cite{BD} to implement the quadratic Chabauty method in explicit examples, and the bound on the size of the image of the non-abelian Kummer map given in the first part of Proposition \ref{prop:kummer_bounds} is used in \cite{BDeff} to bound the size of $X(\Q )$ for certain hyperbolic curves. As explained in Section \ref{s:examples}, Theorem \ref{thm:description_of_graph_kummer} can be used to improve on these bounds, giving results which seem to be more amenable to proving \textit{uniform} effective Chabauty--Kim results analogous to \cite{KRZB}.

We also include a lengthy Appendix \ref{appx:hopf_gpds} of background material on affine groupoid-schemes, pro-nilpotent Lie algebras and completed Hopf groupoids, as well as their derivations and filtrations. Despite being well-known to experts, the precise results we need are not available in the literature to the best of our knowledge, so we include them in the appendix for completeness.

\subsection*{Notation and conventions}\label{ss:conventions}

Throughout the paper, $K$ will denote a finite extension of $\Q_p$, with ring of integers and residue field $\O_K$ and $k$ respectively. We write $q=\# k$ for the residue cardinality, and fix a choice of uniformiser $\varpi$\index{uniformiser $\varpi$} of $K$. We also fix a prime $\ell$ different from $p$.

We fix an algebraic closure $\overline K/K$ and write $G_K=\Gal(\overline K|K)$\index{Galois group $G_K$} for the absolute Galois group of $K$ and $I_K\unlhd G_K$\index{inertia\dots!\dots group $I_K$} for its inertia subgroup. We fix a choice of geometric Frobenius $\varphi\in G_K$\index{Frobenius $\varphi$} and an inertia element $\sigma$\index{inertia\dots!\dots generator $\sigma$} such that $t_\ell(\sigma)$ is a generator of $\Z_\ell(1)$, where $t_\ell\colon I_K\twoheadrightarrow\Z_\ell(1)$\index{tame character $t_\ell$} denotes the tame character. Thus $\sigma$ is a generator of the maximal pro-$\ell$ quotient of $I_K$.

Curves, usually denoted $X$\index{curve $X$}, will always be assumed smooth and geometrically connected, but not necessarily projective; families of curves will be generically smooth with connected geometric fibres. The smooth completion of a curve $X$ will be denoted $\overline X$, and $D=\overline X\setminus X$ will denote the complementary divisor. A curve is said to be hyperbolic if its Euler characteristic is less than 0.

We strongly advise readers on a first reading to assume that every hyperbolic curve $X/K$ is projective and has semistable reduction, and that its reduction graph $\rGraph$ has no half-edges, has genus $0$ at every vertex, and is $2$-connected. This simplifies many of the constructions developed throughout the paper.

\smallskip

An index of notation is included for ease of reference.

\subsection*{Acknowledgements}

We are extremely grateful to Minhyong Kim for initiating this collaboration, and for several helpful suggestions. We would also like to thank David Holmes, Steffen M\"uller and Matthew Baker for helpful discussions during the preparation of this paper. Part of this research was carried out at the workshop ``Arithmetic of Hyperelliptic Curves" in Baskerville Hall.
\section{Galois cohomology of local fields}\label{s:galois_cohomology}

In this section we recall the results about non-abelian Galois cohomology of local fields which will be needed later. 
In \S \ref{ss:gal_coh_field}, we recall foundational results about the Galois cohomology of $\ell $-adic representations of $p$-adic fields (when $p\neq \ell $) focusing on the case of representations satisfying a suitable weight--monodromy property. In \S\ref{ss:can_rep}, we describe how, if $U$ is a $\Q_\ell$-unipotent group with a continuous action of the Galois group of a local field $K$ and $U$ satisfies a weight--monodromy condition (and some other assumptions) we have an explicit identification of $\H^1(G_K,U)$ with a certain subquotient of $\Lie (U)$. In \S\ref{ss:nonab_Kummer}, we recall the definition of the non-abelian Kummer map from $X(K)$ to a non-abelian cohomology variety and use the previous section to identify it with a certain sub-quotient of the Lie algebra of $U$. Roughly speaking, we show that associated to each point in $X(K)$ there is a `canonical cocycle' representing its image in cohomology. 

We adopt the following group cohomological conventions. We always work with continuous group cohomology. The topology can always be understood from context: the only groups which arise are finitely presented (discrete) groups, profinite groups, or the $R$-points of unipotent groups for $R$ a $\Q_\ell$-algebra. Unless otherwise stated, when $U$ is a unipotent group over $\Q_\ell$ with a continuous action of a profinite group, the non-abelian cohomology set $\H ^1(G,U)$ will refer to $\H^1(G,U(\Q_\ell))$. 
When we say that $G$ acts continuously on a unipotent group over $\Q_\ell$, we mean the representation on $\Lie(U)$ is a continuous $\ell$-adic representation.

\subsection{Weight filtrations, monodromy filtrations and Galois cohomology of local fields}\label{ss:gal_coh_field}

Let $V$ be a $\Q_\ell$-vector space with a continuous action of $G_K$. Grothendieck's Monodromy Theorem \cite[Appendix]{serre-tate} ensures that a positive power $\sigma^e$ of $\sigma$ acts unipotently on $V$, and the \emph{monodromy operator} $N$\index{monodromy operator $N$} is defined to be the nilpotent endomorphism of $V$ given by $N:=\frac1e\log(\sigma^e)$. The \index{filtration!monodromy $\M_\bullet$}\emph{monodromy filtration} $\M_\bullet V$ is defined by setting $\M_iV_{\overline\Q_\ell}$ to be the span of the generalised Frobenius eigenspaces whose corresponding eigenvalues are $q$-Weil numbers of weight $\leq i$. This filtration is separated, increasing, $G_K$-stable, and satisfies $N(\M_iV)\leq\M_{i-2}V$ for all $i$.

We say that a $G_K$-representation $V$ with an exhaustive separated $G_K$-stable increasing filtration \index{filtration!weight $\W_\bullet$}$\W_\bullet V$ satisfies property \index{weight--monodromy \WM}\WM if:
\begin{enumerate}
\item $\W _{-1}V=V$;
\item $\M _0 V=V$; and
\item for all $k\in\Z$ and $i>0$, the maps
\[
N^i\colon\gr^\M_{k+i}\gr^\W_k V \longisoarrow \gr^\M_{k-i}\gr^\W_kV
\]are isomorphisms.
\end{enumerate} 
We sometimes simply say that $V$ is a $G_K$-representation satisfying \WM, leaving the $\W$-filtration implicit.

If $U$ is a unipotent group over $\Q_\ell$ with a continuous action of $G_K$, we define subgroup-schemes $\M_{-i}U\subseteq U$ for $i\geq0$ by
\[
\M_{-i}U:=\exp\left(\M_{-i}\Lie(U)\right).
\]
The fact that these are subgroups follows from the Baker--Campbell--Hausdorff formula, and each $\M_{-i}U$ is normal in $\M_0U$. In particular, each $\M_{-i}U$ is normal in $\M_0U$. If $U$ carries an exhaustive separated $G_K$-stable increasing filtration\footnote{Here a \emph{filtration} on a group consists of an increasing sequence of subgroups satisfying $[\W_iU,\W_jU]\leq\W_{i+j}U$ for all $i,j$. The collection of subgroups $\M_{-i}U$ also forms a filtration in this sense, indexed by non-positive integers.} $\W_\bullet U$, we say that it satisfies \WM if and only if $\Lie(U)$ does.

The class of filtered unipotent groups with continuous $G_K$-actions satisfying \WM is closed under products, kernels and cokernels (of filtered $G_K$-equivariant homomorphisms), $\W$-strict extensions, and twists, as well as tensor products of abelian unipotent groups (i.e.\ vector spaces). If $U$ satisfies \WM then so does each $\W_{-i}U$, as well as each term of the descending and ascending central series.

\smallskip

The following lemma is proved in \cite{serre} in the case where $U$ has a discrete $G$-action, but the proof also applies in the continuous case.
\begin{lemma}[\cite{serre}, I.5.7 Proposition 43, 1.5.4 Corollary 2]\label{lem:7term}
\leavevmode
\begin{enumerate}
\item\label{lempart:6term}
If $G$ is a topological group and 
\[
1\to Z \to U \to U/Z\to 1
\]
is a extension of unipotent groups with continuous $G$-action, then there is an exact sequence of pointed sets
\[
1 \!\rightarrow\! \H^0(G,Z) \!\rightarrow\! \H^0(G,U) \!\rightarrow\! \H^0(G,U/Z) \!\actsarrow\! \H^1(G,Z) \!\rightarrow\! \H^1(G,U) \!\rightarrow\! \H^1(G,U/Z).
\]
\item\label{lempart:7term}
If in the above $Z$ is central in $U$ then this extends to an exact sequence
\[
\scalebox{0.94}{$\displaystyle{
1 \!\rightarrow\! \H^0\!(G,Z) \!\rightarrow\! \H^0\!(G,U) \!\rightarrow\! \H^0\!(G,U/Z) \!\rightarrow\! \H^1\!(G,Z) \!\actsarrow\! \H^1\!(G,U) \!\rightarrow\! \H^1\!(G,U/Z) \!\rightarrow\! \H^2\!(G,Z).
}$}
\]
\end{enumerate}
(Here, the decorated arrow $\actsarrow$ serves as a mnemonic for the following extra properties of these exact sequences: the terms to the left of this arrow form an exact sequence of groups, and there is a right action of the group on the left of the $\actsarrow$ on the set on the right whose orbits are the the fibres of the next map on the right and whose point-stabiliser is the image of the previous map on the left.)
\end{lemma}

\begin{remark}
If $c$ is a continuous cocycle for a unipotent group $U$ with a continuous $G$-action, one can form the \emph{twist} $U^c$ of $U$ by $c$, and there is a bijection
\[
\H^1(G,U^c) \longisoarrow \H^1(G,U)
\]
taking the basepoint of $\H^1(G,U^c)$ to the class $[c]\in\H^1(G,U)$. These twists are compatible with all the standard constructions. For instance the exact sequences from Lemma \ref{lem:7term}\eqref{lempart:6term} fit into a commuting diagram
\begin{center}
\begin{tikzcd}
\H^0(G,U^c/Z^c) \arrow{r}{\acts} & \H^1(G,Z^c) \arrow{r} & \H^1(G,U^c) \arrow{r}\arrow{d}{\viso} & \H^1(G,U^c/Z^c) \arrow{d}{\viso} \\
\H^0(G,U/Z) \arrow{r}{\acts} & \H^1(G,Z) \arrow{r} & \H^1(G,U) \arrow{r} & \H^1(G,U/Z)
\end{tikzcd}
\end{center}
where $Z^c$ is $Z$ with the $c$-twisted Galois action, so that the fibre of $\H^1(G,U)\rightarrow\H^1(G,U/Z)$ containing $[c]$ is identified as $\H^1(G,Z^c)/\H^0(G,U^c/Z^c)$. We will use compatibilities such as this without explicit comment.
\end{remark}

Recall that, given a short exact sequence 
\[
1\to H \to G \to G/H \to 1
\]
of profinite groups and a $\Q_\ell$-unipotent group $U$ with a continuous action of $G$, 
the action of $g\in G$ on the set of $U$-valued cocycles $H\to U$ given by
\[
c\mapsto (h\mapsto g\cdot c(g^{-1}hg))
\]
induces an action of $G/H$ on $\H^1 (H,U)$. We obtain an exact sequence of pointed sets (non-abelian inflation--restriction) \cite[I.5.8]{serre}
\begin{equation}\label{eq:nonab_HS}
1\to \H ^1 (G/H ,U^H )\to \H ^1 (G,U)\to \H ^1 (H,U)^{G/H}.
\end{equation}

Using this exact sequence, we can pass up and down along finite index subgroups.
\begin{lemma}\label{lem:finite_index}
Let $G$ be a profinite group acting continuously on a $\Q_\ell$-unipotent group $U$, and $H\leq G$ an open subgroup. Then the restriction map $\H^1(G,U)\hookrightarrow\H^1(H,U)$ is injective. If additionally $H$ is normal in $G$, then this restriction map sets up a bijection
\[
\H^1(G,U)\longisoarrow\H^1(H,U)^{G/H}.
\]
\end{lemma}
\begin{proof}
It suffices to prove the latter statement. Suppose first that $G$ is finite and $H=1$. If $V$ is any $\Q_\ell$-linear representation of $G$, then we have $\H^i(G,V)=1$ for all $i>0$ since $\Q_\ell$ has characteristic $0$. It then follows from Lemma \ref{lem:7term} that $\H^1(G,U)=1$ for all representations of $G$ on $\Q_\ell$-unipotent groups.

Now in general, the inflation--restriction sequence \eqref{eq:nonab_HS} implies that the fibres of the restriction map $\H^1(G,U)\rightarrow\H^1(H,U)^{G/H}$ are given by $\H^1(G/H,(U^c)^H)$ for twists $U^c$ of $U$. These vanish by the above and so the restriction map is injective.

For surjectivity, when $U$ is abelian the cokernel of $\H^1(G,U)\rightarrow\H^1(H,U)^{G/H}$ is contained in $\H^2(G/H,U)=1$ and we are done in this case. In general, we proceed by induction, writing $U$ as a central extension
\[
1 \rightarrow Z \rightarrow U \rightarrow U/Z \rightarrow 1
\]
where we know the desired result for $U/Z$ and $Z$. From Lemma \ref{lem:7term}\ref{lempart:7term} the restriction maps fit into a morphism of exact sequences
\begin{center}
\begin{tikzcd}[column sep=small]
\H^0(G,U/Z) \arrow{r}\arrow[hook]{d} & \H^1(G,Z) \arrow{r}{\acts}\arrow[hook]{d} & \H^1(G,U) \arrow{r}\arrow[hook]{d} & \H^1(G,U/Z) \arrow{r}\arrow[hook]{d} & \H^2(G,Z) \arrow[hook]{d} \\
\H^0(H,U/Z) \arrow{r} & \H^1(H,Z) \arrow{r}{\acts} & \H^1(H,U) \arrow{r} & \H^1(H,U/Z) \arrow{r} & \H^2(H,Z),
\end{tikzcd}
\end{center}
the lower of which is $G/H$-equivariant. The right-hand vertical map is injective by the Hochschild--Serre spectral sequence.

Consider any element $\eta\in\H^1(H,U)^{G/H}$. An easy diagram-chase shows that there is a $\xi\in\H^1(G,U)$ having the same image in $\H^1(H,U/Z)$ as $\eta$. Replacing $U$ with a twist if necessary, we may assume that $\xi=1$, so that $\eta$ lies in $\H^1(H,Z)/\H^0(H,U/Z)$. But the map $\H^0(H,U/Z)\rightarrow\H^1(H,Z)$ is a $G/H$-equivariant homomorphism from a unipotent group to a vector space, and therefore (e.g.\ by considering the maps on Lie algebras) we have
\[
\left(\H^1(H,Z)/\H^0(H,U/Z)\right)^{G/H} = \H^1(H,Z)^{G/H}/\H^0(H,U/Z)^{G/H}.
\]
Thus $\eta$ lifts to $\H^1(G,Z)=\H^1(H,Z)^{G/H}$, and hence to $\H^1(G,U)$.
\end{proof}

We will also need the following similar lemma, which is proved via an entirely different technique. When $U$ is abelian, this is a consequence of the Hochschild--Serre spectral sequence and the fact that $\widehat{\mathbb{Z}}$ has cohomological dimension one.

\begin{lemma}\label{lem:nonab_H2=1}
Let $G$ be a profinite group acting continuously on a $\Q_\ell$-unipotent group $U$, and $H\unlhd G$ a closed normal subgroup with $G/H\nciso\Zhat$. Then the morphism
\[
\H^1 (G,U)\to \H ^1 (H,U)^{G/H}
\]
is surjective.
\end{lemma}
\begin{proof}
Let $\alpha $ be a lift of a topological generator of $G/H\nciso\Zhat$ to $G$. Let $c$ be a continuous cocycle $H\to U$ whose class in $\H^1 (H,U)$ is $G/H$-invariant.
Then a lift of $c$ to a $U$-valued cocycle $\widetilde{c}\colon G\to U$ is uniquely determined by where it sends $\alpha $, and $\widetilde{c}(\alpha ) \in U$ defines a cocycle if and only if it satisfies
\begin{equation}\label{eq:cocyc}
c(h)(h\cdot \widetilde{c}(\alpha ))=\widetilde{c}(\alpha )(\alpha \cdot c(\alpha ^{-1}h\alpha )).
\end{equation}
for $h$ in $H$.
Since the class of $c$ in $\H^1 (H,U)$ is $G/H$-invariant, there exists $\gamma \in U$ such that, for all $h$ in $H$,
\[
\alpha \cdot c(\alpha ^{-1}h\alpha )=\gamma^{-1} c(h)(h\cdot \gamma ).
\]
Hence taking $\widetilde c(\alpha )=\gamma $ satisfies \eqref{eq:cocyc}, and hence defines a lift of the cocycle to~$U$.
\end{proof}

\begin{lemma}[{\cite[\S3.2.4 and \S3.3.11]{autour}}]\label{lem:euler-char}
Let $V$ be a $\Q_\ell$-vector space with a continuous action of $G_K$. Then we have
\begin{align*}
\dim_{\Q_\ell}\H^1(G_K,V) &= \dim_{\Q_\ell}\H^0(G_K,V) + \dim_{\Q_\ell}\H^0(G_K,V^*(1)) \\
\dim_{\Q_\ell}\H^2(G_K,V) &= \dim_{\Q_\ell}\H^0(G_K,V^*(1)) \\
\dim_{\Q_\ell}\H^1(G_K/I_K,V^{I_K}) &= \dim_{\Q_\ell}\H^0(G_K,V).
\end{align*}
\end{lemma}

\begin{corollary}\label{cor:mod_M-3}
Let $U$ be a $\Q_\ell$-unipotent group with a continuous action of $G_K$. Then the natural map
\[
\H^1(G_K,U) \longisoarrow \H^1(G_K,U/\M_{-3})
\]
is bijective.
\end{corollary}
\begin{proof}
It suffices to prove that if $Z$ is a normal subgroup of $U$ contained in $\M_{-3}$, then the natural map $\H^1(G_K,U) \longisoarrow \H^1(G_K,U/Z)$ is bijective. Proceeding inductively, we may assume that $Z$ is central in $U$, so that we have an exact sequence
\[
\H^1(G_K,Z) \actsarrow \H^1(G_K,U) \rightarrow \H^1(G_K,U/Z) \rightarrow \H^2(G_K,Z)
\]
in which the outer terms vanish by Lemma \ref{lem:euler-char} and the fact that all Frobenius eigenvalues of $Z$ are Weil numbers of weight $<-2$. It follows that $\H^1(G_K,U) \rightarrow \H^1(G_K,U/Z)$ is bijective, as claimed. 
\end{proof}

\begin{lemma}\label{lem:description_of_cohomology_basic}
Let $U$ be a $\Q_\ell$-unipotent group with a continuous action of $G_K$ satisfying \WM and such that the action of $I_K$ on $\Lie(U)$ is unipotent. Then the twisted conjugation action of $U$ on $\gr^\M_{-2}U$ given by
\[
u\colon v\mapsto u^{-1}v\sigma(u)
\]
is well-defined and factors through an action of $\gr^\M_0U$. If $\gr^\M_{-2}U(-1)$ denotes the twist of $\gr^\M_{-2}U$ by the inverse of the cyclotomic character, then this action of $\gr^\M_0U$ on $\gr^\M_{-2}U(-1)$ is $G_K$-equivariant, and there is a bijection
\[
\H^1(G_K,U) \longisoarrow \left(\gr^\M_{-2}U(-1)/\gr^\M_0U\right)^{G_K}
\]
given by evaluating a cocycle at $\sigma\in I_K$.
\end{lemma}
\begin{proof}
The claims about the twisted conjugation action are easy to verify from the fact that $v^{-1}u^{-1}vu,u^{-1}\sigma(u)\in\M_{-i-2}U$ for $v\in\M_{-2}U$ and $u\in\M_{-i}U$.

For the description of the cohomology, we first claim that the restriction map
\[
\H^1(G_K,U/\M_{-3}) \longisoarrow \H^1(I_K,U/\M_{-3})^{G_K/I_K}
\]
is bijective. Surjectivity is given by Lemma \ref{lem:nonab_H2=1}. For injectivity, it follows from the inflation--restriction sequence \eqref{eq:nonab_HS} and Corollary \ref{cor:mod_M-3} that the fibres of the restriction map are given by $\H^1(G_K/I_K,(U^c/\M_{-3})^{I_K})$ for twists $U^c$ of $U$. Since $U^c$ also satisfies \WM, it follows that all the Frobenius weights of $(U^c/\M_{-3})^{I_K}\leq\M_{-1}U^c/\M_{-3}$ are in $\{-1,-2\}$. But any representation of $G_K/I_K$ on a $\Q_\ell$-unipotent group $\overline U$ with Frobenius weights in $\{-1,-2\}$ automatically has $\H^1(G_K/I_K,\overline U)=1$ -- for $\overline U$ abelian this follows from Lemma \ref{lem:euler-char} and for $\overline U$ general one inducts using Lemma \ref{lem:7term}. This completes the proof of bijectivity of the restriction map.

Combined with Corollary \ref{cor:mod_M-3}, we see that the natural map $\H^1(G_K,U)\longisoarrow\H^1(I_K,U/\M_{-3})^{G_K/I_K}$ is bijective. To compute this latter set, we use the $G_K/I_K$-equivariant exact sequence
\[
\H^0(I_K,U/\M_{-2}) \actsarrow \H^1(I_K,\gr^\M_{-2}U) \rightarrow \H^1(I_K,U/\M_{-3}) \rightarrow \H^1(I_K,U/\M_{-2}).
\]
Since $\H^1(I_K,U/\M_{-2})^{G_K/I_K}=\Hom^\Out(\Q_\ell(1),U/\M_{-2})^{G_K/I_K}=1$ by Frobenius weight considerations, we have that
\[
\H^1(I_K,U/\M_{-3})^{G_K/I_K}=\left(\Hom(\Q_\ell(1),\gr^\M_{-2}U)/(U/\M_{-2})\right)^{G_K/I_K}.
\]
Identifying $\Hom(\Q_\ell(1),\gr^\M_{-2}U)$ with the twist of $\gr^\M_{-2}U$ by the inverse of the cyclotomic character, we obtain the desired description.
\end{proof}

\begin{remark}\label{rmk:presheaf_1}
In \cite{siegel}, the continuous Galois cohomology set $\H^1(G_K,U)$ is extended to a presheaf
\[
\Spec(\Lambda) \mapsto \H^1(G_K,U(\Lambda))
\]
on affine $\Q_\ell$-schemes, and (under certain assumptions weaker than \WM) it is shown that this presheaf is represented by an affine $\Q_\ell$-scheme of finite type, also denoted $\H^1(G_K,U)$. All of our arguments also apply on the level of presheaves, so that for example we have an isomorphism
\[
\H^1(G_K,U) \longisoarrow\left(\gr^\M_{-2}U(-1)/\gr^\M_0U\right)^{G_K}
\]
of presheaves.
\end{remark}

\begin{corollary}[to the proof of Lemma \ref{lem:description_of_cohomology_basic}]\label{cor:G-invt_is_I-invt}
Let $U$ be a $\Q_\ell$-unipotent group with a continuous action of $G_K$ satisfying \WM, and let $P$ be a $U$-torsor with a compatible continuous action of $G_K$. Then $P(\Q_\ell)^{G_K}\neq\emptyset$ if and only if $P(\Q_\ell)^{I_K}\neq\nobreak\emptyset$.
\begin{proof}
If the action of $I_K$ on $\Lie(U)$ is unipotent, the proof of Lemma \ref{lem:description_of_cohomology_basic} shows that $\H^1(G_K,U)\hookrightarrow\H^1(I_K,U)$ is injective, which proves the result in this case. In general we reduce to the case of unipotent inertia action via Lemma \ref{lem:finite_index}.
\end{proof}
\end{corollary}

\subsection{Canonical representatives}\label{ss:can_rep}

The following lemma describes the key calculation which enables us to describe the non-abelian cohomology varieties arising from fundamental groups. The calculation will show up again in a slightly different guise in the context of Cheng--Katz integration in \S \ref{s:cheng-katz}.

\begin{lemma}\label{lemma:sigma_map}
Let $V$ be a finite dimensional $\Q_\ell$-vector space with a continuous action of $G_K$ satisfying \WM. Let $\pi_k$ denote the projection $V\to V/W_{-k-1}V$. Then
\[
\prod_{k>0}N^k \circ \pi _k :V\to \prod _{k>0}\gr_{-2k}^\M\gr_{-k}^\W V
\]
is surjective, with kernel $\M_{-1}$, inducing an isomorphism
\begin{equation}\label{eq:induce_iso}
\gr _0 ^{\M }V \longisoarrow\prod_{k>0}\gr_{-2k}^\M\gr_{-k}^\W V.
\end{equation}
\end{lemma}
\begin{proof}
It is clear that the kernel contains $\M_{-1}$, so the map factors through $\gr^\M_0V$. The weight--monodromy condition ensures that each
\[
N^k\colon \gr_0^\M\gr_{-k}^\W V \longisoarrow \gr_{-2k}^\M\gr_{-k}^\W V
\]
is an isomorphism; proceeding inductively, we find that the common kernel of the maps $N^k\circ\pi_k$ for $k\leq m$ is $\W_{-m-1}\gr^\M_0V$. It follows that \eqref{eq:induce_iso} is injective, and hence surjective for dimension reasons.
\end{proof}

This motivates the following definition.

\begin{definition}\label{def:canonical}\index{canonical\dots!\dots subquotient $\Lie(U)^\can$}
Let $U$ be a unipotent group over $\Q_\ell$ with a continuous action of $G_K$ satisfying \WM.
Define $\Lie(U)^\can\leq\gr^\M_{-2}\Lie(U)$ to be the intersection of the kernels of the maps
\begin{equation}\label{eqn:canonical_sigma_map}
N^{k-1}\circ\pi_k\colon \gr^\M_{-2}\Lie(U) \to \gr^\M_{-2k}\gr^\W_{-k}\Lie(U).
\end{equation}
for $k>0$.
\end{definition}

\begin{lemma}\label{lem:canonical_cocycle}
Let $U$ be a unipotent group over $\Q_\ell$ with a continuous action of $G_K$ satisfying \WM, such that the action of $I_K$ on $\Lie(U)$ is unipotent. Then the map
\[
\Lie(U)^\can\times\gr^\M_0U\rightarrow\gr^\M_{-2}U
\]
given by $(v,u)\mapsto u^{-1}\exp(v)\sigma(u)$ is bijective. In particular, the composite map
\[
\Lie(U)^\can \hookrightarrow \gr^\M_{-2}\Lie(U) \overset\exp\longrightarrow\gr^\M_{-2}U \twoheadrightarrow \gr^\M_{-2}U/\gr^\M_0U
\]
is bijective, where the quotient is as in Lemma \ref{lem:description_of_cohomology_basic}.
\end{lemma}
\begin{proof}
Write $U^\W_n=U/\W_{-n-1}$. To show surjectivity, we inductively show that every element of $\gr^\M_{-2}U^\W_n$ is equivalent, modulo the action of $\gr^\M_0U^\W_n$, to an element of $\Lie(U^\W_n)^\can$. Suppose we have an element $v_{n-1}\in\gr^\M_{-2}U^\W_n$, whose image in $\gr^\M_{-2}U^\W_{n-1}$ lies in $\exp(\Lie(U^\W_{n-1})^\can)$. From condition \WM we find some $\log(u_n)\in\gr^\M_0\gr^\W_{-n}U^\W_n$ such that $N^n(\log(u_n))=N^{n-1}(\log(v_{n-1}))$. It then follows by Baker--Campbell--Hausdorff that
\[
v_n:=u_nv_{n-1}\sigma(u_n^{-1})=\exp\bigl(\log(v_{n-1})-N(\log(u_n))\bigr)
\]
in $U^\W_n/\M_{-3}$, so that $\log(v_n)\in\Lie(U^\W_n)^\can$ as desired. Injectivity can be proved in a similar manner.
\end{proof}

Combining this with Lemma \ref{lem:description_of_cohomology_basic}, we arrive at the main result of this section.

\begin{theorem}\label{thm:cohomology_is_vector_space}
Let $U$ be a unipotent group over $\Q_\ell$ with a continuous action of $G_K$ satisfying \WM. Then there is a bijection
\[
\H^1(G_K,U) \longisoarrow \Hom\left(\Q_\ell(1),\Lie(U)^\can\right)^{G_K}
\]
functorial in $U$, compatible with restriction to open subgroups $G_L\leq G_K$, and independent of the choice of $\varphi$ and $\sigma$.
\end{theorem}
\begin{proof}
Let us suppose firstly that $I_K$ acts unipotently on $\Lie(U)$. The bijection $\Lie(U)^\can(-1) \longisoarrow \gr^\M_{-2}U(-1)/\gr^\M_0U$ from Lemma \ref{lem:canonical_cocycle} is $G_K$-equivariant, so we define the desired map to be the composite
\[
\H^1(G_K,U) \longisoarrow \left(\gr^\M_{-2}U(-1)/\gr^\M_0U\right)^{G_K} \longisoarrow \left(\Lie(U)^\can(-1)\right)^{G_K}
\]
with the first map that from Lemma \ref{lem:description_of_cohomology_basic}. This map is easily checked to be independent of the choice of $\varphi$ and $\sigma$.

In general, we choose a finite Galois extension $L/K$ such that $I_L$ acts unipotently on $\Lie(U)$. The bijection
\[
\H^1(G_L,U) \longisoarrow \left(\Lie(U)^\can(-1)\right)^{G_L}
\]
constructed above can be checked to be equivariant for the action of $G_{L/K}$, and we define the desired bijection to be the composite
\[
\H^1(G_L,U)^{G_{L/K}} \longisoarrow \left(\Lie(U)^\can(-1)\right)^{G_K} \overset{1/e}\longrightarrow \left(\Lie(U)^\can(-1)\right)^{G_K}
\]
using Lemma \ref{lem:finite_index}. This is independent of the choice of $L$, hence compatible with restriction to open subgroups, and is easily seen to be functorial in $U$.
\end{proof}

\begin{remark}\label{rmk:canonical_elt}\index{canonical\dots!\dots element $\gamma^\can$}
Theorem \ref{thm:cohomology_is_vector_space} says concretely that any class in $\H^1(G_K,U)$ can be represented by a cocycle $c$ such that $\frac1e\log(c(\sigma^e))\in\Lie(U)^\can$ for every power of $\sigma$ which acts unipotently on $\Lie(U)$, and this element of $\Lie(U)^\can$ is the one corresponding to the class of the cocycle $c$. It follows from Lemma \ref{lem:canonical_cocycle} that such a representing cocycle is unique up to the action of $\M_{-1}U$.

Put another way, this says that every $U$-torsor $P$ with compatible continuous Galois action has, up to the action of $\M_{-1}U$, a distinguished element $\gamma^\can\in P(\Q_\ell)$ characterised by the property that $\frac1e\log\left((\gamma^\can)^{-1}\sigma^e(\gamma^\can)\right)\in\Lie(U)^\can$ for $\sigma^e$ as above. When $U$ is the $\Q_\ell$-pro-unipotent \'etale fundamental group of a hyperbolic curve, we will describe these elements explicitly in Proposition~\ref{prop:canonical_path_is_canonical_path} in terms of iterated integration on reduction graphs.

These canonical elements are $\ell$-adic analogues of the elements of crystalline path-torsors from \cite[Theorem~30]{vologodsky}, and will be used in upcoming work of the first author and Daniel Litt.
\end{remark}

\begin{remark}\label{rmk:cohomology_grading}
One surprising aspect of Theorem \ref{thm:cohomology_is_vector_space} is that it provides the set $\H^1(G_K,U)$ with the structure of a $\Q_\ell$-vector space, whose addition law remains a mystery to the authors. In fact, more is true: the vector space $\H^1(G_K,U)=\left(\Lie(U)^\can(-1)\right)^{G_K}$ is canonically graded, with its grading coming from the isomorphism
\[
\prod_{k\geq2}N^{k-2}\circ\pi_k\colon\Lie(U)^\can \longisoarrow \prod_{k\geq2}\ker\left(\gr^\M_{2-2k}\gr^\W_{-k}\Lie(U)\overset N\longrightarrow\gr^\M_{-2k}\gr^\W_{-k}\Lie(U)\right).
\]
Note that the graded pieces of this filtration are $\H^1(G_K,\gr^\W_{-k}U)$.
\end{remark}

\begin{remark}\label{rmk:presheaf_2}
Just as in Remark \ref{rmk:presheaf_1}, Theorem \ref{thm:cohomology_is_vector_space} extends to a corresponding statement about presheaves, where the vector space $\Hom\left(\Q_\ell(1),\Lie(U)^\can\right)^{G_K}$ is viewed as an affine space in the usual way. This gives an alternative, much more explicit proof of representability of the presheaf $\H^1(G_K,U)$.
\end{remark}

\begin{remark}
Although we have chosen to work only with finite dimensional objects, all of the discussion in this section remains true if we permit $U$ to be a finitely generated pro-unipotent group satisfying \WM (appropriately interpreted). For example, using \cite[Lemma 4.0.5]{betts2017motivic} we can take an inverse limit to obtain an isomorphism
\[
\H^1(G_K,U) \longisoarrow \Hom\left(\Q_\ell(1),\Lie(U)^\can\right)^{G_K}
\]
also in this case, along with a product decomposition as in Remark \ref{rmk:cohomology_grading}.
\end{remark}

\begin{corollary}[to Theorem \ref{thm:cohomology_is_vector_space}]
Let $U$ be a finitely generated pro-unipotent group over $\Q_\ell$ with a continuous action of $G_K$ satisfying \WM. Then there is a finite extension $K'/K$ such that for any finite extension $L/K'$, the restriction map
\[
\H^1(G_{K'},U) \rightarrow \H^1(G_L,U)
\]
is an isomorphism.
\begin{proof}
Passing to a finite extension if necessary, we may assume that $I_K$ acts unipotently on $U^\ab$, and that the subgroup of $\overline\Q_\ell^\times$ generated by $q$ and the eigenvalues of $\varphi$ on $U^\ab$ contains no roots of unity other than $1$. We will show that $K'=K$ suffices in this case. Indeed, the assumptions guarantee that the action of $I_K$ on $\Lie(U)$ is pro-unipotent, and hence that the Galois action on $\Lie(U)^\can\leq\gr^\M_{-2}\Lie(U)$ is unramified. Our assumption on Frobenius eigenvalues ensures that $\Hom(\Q_\ell(1),\Lie(U)^\can)$ has no Frobenius eigenvalues which are roots of unity other than $1$, and hence $\Hom(\Q_\ell(1),\Lie(U)^\can)^{G_L}=\Hom(\Q_\ell(1),\Lie(U)^\can)^{G_K}$ for any finite extension, as desired.
\end{proof}
\end{corollary}

\subsection{Fundamental groups of curves}\label{ss:pi1_curves}

Let $X/K$ be a curve, and $U=\pi_1^{\Q_\ell}(X_{\overline K},b)$ the $\Q_\ell$-pro-unipotent \'etale fundamental group of $X$ based at a point $b\in X(K)$. In order to apply the results of the preceding section, we want to show that $U$ satisfies condition \WM. The \emph{weight filtration} $\W_\bullet$ on $U$ is described in \cite[Definition 1.5]{AMO}.

\index{filtration!weight $\W_\bullet$}
\begin{definition}\label{def:wt1}
Let $\overline X$ be the smooth compactification of $X$. We define $\W_{-1}U=U$ and define $\W_{-2}U$ to be the kernel of
\[
\pi_1^{\Q_\ell}(X_{\overline K},b) \twoheadrightarrow \pi_1^{\Q_\ell}(\overline X_{\overline K},b)^\ab.
\]
For $k>2$, we set $\W_{-k}U$ to be the (normal) subgroup subscheme generated by the commutator subgroups $[\W_{-i}U,\W_{-j}U]$ for $i+j=k$ ($i,j>0$).

In particular, if $X=\overline X$ or $X=\overline X\setminus\{\text{$K$-pt}\}$, then we have $\W_{-k}U=\Cent^kU$ where $\Cent^\bullet U$ denotes the descending central series of $U$.
\end{definition}

The weight filtration $\W_\bullet$ admits an alternative (equivalent) definition of a more arithmetic nature.

\begin{definition}\label{def:wt2}
Let $S_0$ be a integral scheme of finite type over $\Z\left[\frac1\ell\right]$ with a geometric basepoint $\bar s$ and $V$ a continuous representation of $\pi_1^\et(S_0,\bar s)$ on a finite dimensional (or pro-finite dimensional) $\Q_\ell$-vector space. Suppose that $w$ is a closed point of $S_0$ and that $\varphi_w\in\pi_1^\et(S_0,\bar s)$ is a geometric Frobenius at $w$. The \emph{weight filtration $\W_\bullet V$ relative to $w$} is defined by setting $\W_iV_{\overline\Q_\ell}$ to be the span of the generalised $\varphi_w$-eigenspaces whose corresponding eigenvalues are $q_w$-Weil numbers of weight~$\leq i$.

If the weight filtrations relative to $w$ are independent of the choice of $\varphi_w$ and agree for all $w$ in a non-empty Zariski-open subset of $S_0$, we refer to this filtration unambiguously as \emph{the} weight filtration on~$V$.
\end{definition}

\begin{proposition}\label{prop:wt_is_wt}
The $G_K$-action on $U$ factors through an action of $\pi_1^\et(S_0,\bar s)$ for $S_0$ an integral $\Z\left[\frac1\ell\right]$-scheme of finite type such that the $\W$-filtration on $U$ defined in Definition \ref{def:wt1} agrees with the $\W$-filtration on $\Lie(U)$ defined in Definition \ref{def:wt2}.
\begin{proof}
We may choose an integral $\Z\left[\frac1\ell\right]$-scheme $S_0$ of finite type and a pointed relative curve $(\mathcal X,b)$ over $S_0$ which is a complement of an \'etale divisor $\mathcal D$ in a smooth projective $S_0$-curve $\overline{\mathcal X}$, such that $X$ is the pullback of $\mathcal X$ along a morphism $\Spec(K)\rightarrow S_0$. It follows from Grothendieck's Specialisation Theorem \cite[Expos\'e XIII, Proposition 4.3 \& Exemples 4.4]{SGA1} that the $G_K$-action on $U$ factors through $\pi_1^\et(S_0,\bar s)$, and that the restriction of this action to a decomposition group $G_{\kappa(w)}$ at some closed point $w$ can be identified with the natural Galois action on $\pi_1^{\Q_\ell}(\mathcal X_{\bar w},b(\bar w))$ for any geometric point $\bar w$ lying over $w$. The $G_{\kappa(w)}$-equivariant identification
\[
\pi_1^{\Q_\ell}(X_{\overline K},b) \longisoarrow \pi_1^{\Q_\ell}(\mathcal X_{\bar w},b(\bar w))
\]
above is strictly compatible with the $\W$-filtrations of either side, defined as in Definition \ref{def:wt1}. Abusing notation slightly, we also denote the right-hand side by~$U$.

It now suffices to verify that the $\varphi_w$-eigenvalues of each $\gr^\W_{-k}U$ are $q_w$-Weil numbers of weight $-k$. When $k=1$, we have
\[
\gr^\W_{-1}U = \pi_1^{\Q_\ell}(\overline{\mathcal X}_{\bar w},b(\bar w))^\ab = \H^1_\et(\overline{\mathcal X}_{\bar w},\Q_\ell)^\dual,
\]
all of whose $\varphi_w$-eigenvalues are weight $-1$ by the Weil Conjectures.

When $k=2$, we have a $G_{\kappa(w)}$-equivariant exact sequence
\[
\gr^\W_{-1}U\otimes\gr^\W_{-1}U \rightarrow \gr^\W_{-2}U \rightarrow \Q_\ell(1)\cdot\mathcal D_{\bar w}
\]
where the left-hand map is the commutator map. It follows that the $\varphi_w$-eigenvalues of $\gr^\W_{-2}U$ are weight $-2$.

When $k>2$, we have a $G_{\kappa(w)}$-equivariant surjection
\[
\bigoplus_{\substack{i+j=k \\ i,j>0}}\gr^\W_{-i}U\otimes\gr^\W_{-j}U \twoheadrightarrow \gr^\W_{-k}U
\]
given by commutator maps. It then follows by induction that the $\varphi_w$-eigenvalues of $\gr^\W_{-k}U$ are all weight $-k$, as desired.
\end{proof}
\end{proposition}

To show that $U$ satisfies condition \WM, we use the following preparatory lemma.

\begin{lemma}\label{lem:summand}
Let $\Cent^\bullet U$ denote the central series of $U$. Then $\gr^k_\Cent U$ is a direct summand of $(U^\ab)^{\otimes k}$ as a $\W$-filtered representation of $G_K$, for all $k>0$.
\end{lemma}
\begin{proof}
Pick a $\Z\left[\frac1\ell\right]$-scheme $S_0$ as in Proposition \ref{prop:wt_is_wt}, so that the $G_K$-action on $U$ factors through $\pi_1^\et(S_0,\bar v)$. We write $L=\Lie(U)$ for the Lie algebra of $U$ and $\gr_\Cent^\bullet L$ for its associated graded with respect to the descending central series filtration -- we will show that $\gr_\Cent^kL$ is a $\W$-filtered and $G_K$-equivariant direct summand of $(L^\ab)^{\otimes k}$ for all $k>0$.

Suppose first that $X\neq\overline X$, so that $\gr_\Cent^\bullet L$ is the free pro-nilpotent Lie algebra on $L^\ab$. It follows that there is an algebraic action of $\GL(L^\ab)$ on $\gr_\Cent^\bullet L$ compatible with the Lie bracket, and the action of $\pi_1^\et(S_0,\bar s)$ factors through this $\GL(L^\ab)$-action via the homomorphism $\rho\colon\pi_1^\et(S_0,\bar s)\rightarrow\GL(L^\ab)$ given by the action on $L^\ab$.

Since $\GL(L^\ab)$ is reductive, the $k$-fold Lie bracket map $(L^\ab)^{\otimes k} \twoheadrightarrow \gr_\Cent^kL$ induces a $\GL(L^\ab)$-equivariant splitting
\[
(L^\ab)^{\otimes k} = \gr_\Cent^kL \oplus V_k.
\]
This splitting is automatically $\pi_1^\et(S_0,\bar s)$-equivariant, and hence $G_K$-equivariant, and $\W$-filtered by Proposition \ref{prop:wt_is_wt}.

In the case that $X=\overline X$, by \cite[Proposition 9.2]{hain:universal} $\gr_\Cent^\bullet L$ is the free pro-nilpotent Lie algebra on $L^\ab$ modulo the Lie ideal generated by a non-degenerate symplectic form on $L^\ab$. Hence there is an algebraic action of $\GSp(L^\ab)$ on $\gr_\Cent^\bullet L$ compatible with the Lie bracket, and the action of $\pi_1^\et(S_0,\bar s)$ factors through this $\GSp(L^\ab)$-action via the homomorphism $\rho\colon\pi_1^\et(S_0,\bar s)\rightarrow\GSp(L^\ab)$ given by the action on $L^\ab$. Since $\GSp(L^\ab)$ is also reductive, we conclude as in the previous case.
\end{proof}

\begin{lemma}\label{lem:weight-monodromy}
$U$, equipped with the $\W$-filtration from Definition \ref{def:wt1}, satisfies \WM.
\end{lemma}
\begin{proof}
Since the class of $\W$-filtered $G_K$-representations satisfying \WM is closed under $\W$-strict extensions, it suffices to prove that each $\gr_\Cent^k\Lie(U)$ satisfies \WM. Since the class of $\W$-filtered $G_K$-representations satisfying \WM is closed under tensor products, it suffices by Lemma \ref{lem:summand} to prove that $\Lie(U)^\ab$ satisfies \WM. Since $\Lie(U)^\ab\iso\H^1_\et(X_{\overline K},\Q_\ell)^\dual$ as $\W$-filtered representations of $G_K$, we are done by weight--monodromy for abelian varieties \cite[Expos\'e IX, Th\'eor\`eme~4.3(b) and Corollaire~4.4]{grothendieck:1972} (technically this only proves that the weight $-1$ part of $U^{\ab }$ satisfies \WM, but the \WM property for the weight $-2$ part is a simple duality argument).
\end{proof}

\subsection{The non-abelian Kummer map}\label{ss:nonab_Kummer}

Continuing from the preceding section, let $X/K$ be a curve with basepoint $b\in X(K)$, and let $U$ be a $G_K$-equivariant quotient of $\pi_1^{\Q_\ell}(X_{\overline K},b)$ satisfying \WM (with respect to the induced $\W$-filtration). For instance, one could take $U$ to be $\pi_1^{\Q_\ell}(X_{\overline K},b)$ itself, its partial quotients with respect to the descending central series or the $\W$-filtration, or the quotients considered in \cite[\S3]{BalakrishnanDogra1}.

There is a (profinite) non-abelian Kummer map
\[
\classify\colon X(K) \rightarrow \H^1(G_K,\pi_1^\et(X_{\overline K},b))
\]
sending a point $x\!\in\! X(K)$ to the class of the profinite \'etale path-torsor $\pi_1^\et(X_{\overline K};b,x)$. We define the \emph{non-abelian Kummer map} for $U$ to be the map
\[
\classify\colon X(K) \rightarrow \H^1(G_K,U)
\]
given by composing the above map with the map $\H^1(G_K,\pi_1^\et(X_{\overline K},b))\rightarrow\H^1(G_K,U)$ induced from the group homomorphism $\pi_1^\et(X_{\overline K},b)\rightarrow U(\Q_\ell)$.

By Theorem~\ref{thm:cohomology_is_vector_space}, we can view $\classify$ as a map
\[
\classify\colon X(K) \rightarrow \Hom(\Q_\ell(1),\Lie(U)^\can)^{G_K}.
\]
According to Remark~\ref{rmk:cohomology_grading}, the codomain of $\classify$ canonically has the structure of a (pro-finite dimensional) graded $\Q_\ell$-vector space, and we write
\[
\classify_n\colon X(K)\rightarrow\Hom(\Q_\ell(1),\gr^\W_{-n}\Lie(U)^\can)^{G_K}
\]
for the $-n$th graded piece of the non-abelian Kummer map $\classify$. This notation is non-standard; in \cite{siegel,selmer_varieties,kim-coates,BDCKW}, $\classify_n$ instead denotes the non-abelian Kummer map associated to the particular quotient $U/\Cent^{n+1}U$, where $\Cent^\bullet U$ denotes the descending central series filtration.

\smallskip

As we enlarge the base field to a finite extension $L/K$, the non-abelian Kummer maps $\classify\colon X(L) \rightarrow \H^1(G_L,U) = \Hom(\Q_\ell(1),\Lie(U)^\can)^{G_L}$ are all compatible, so can be packaged together into a single map
\[
\classify\colon X(\overline K) \rightarrow \varinjlim\H^1(G_L,U) \leq \Hom(\Q_\ell(1),\Lie(U)^\can).
\]

\begin{remark}\label{rmk:abstract_kummer}
The construction of the non-abelian Kummer map is a special case of a more general construction. Given a connected topological groupoid $\pi$ with a continuous action of a topological group $G$, one obtains for every object $b\in\ob(\pi)^G$ a non-abelian Kummer map
\[
\classify\colon\ob(\pi)^G \rightarrow \H^1(G,\pi(b))
\]
sending an object $x\in\ob(\pi)^G$ to the class of the $G$-equivariant torsor $\pi(b,x)$. We will see several instances of this construction in \S\ref{s:oda_reduction} (with $\ob(\pi)$ a discrete set on which $G$ acts trivially).
\end{remark} 
\section{A combinatorial description of the fundamental groupoid}\label{s:oda_reduction}
In this section we give a purely graph-theoretic description of the action of tame inertia on the fundamental groupoid of a curve $X$, in terms of Dehn twists in the fundamental groupoid of an associated graph of groups. As well as capturing the tame inertia action, we recover the $\M $ and $\W $ filtrations from this description. The idea behind the construction is not new, and very similar ideas appear in \cite{oda} and \cite{AMO}. However, as these papers are usually interested in the outer action of the Galois group on the fundamental group, the statements of their results tend not to be enough to describe the Galois cohomology classes of path torsors.

\subsection{Reduction graphs}

\begin{definition}[Reduction graphs]\label{def:reduction_graphs}\index{graph}
For us, a \emph{graph} $\Graph$ consists of three finite sets $\Vert{\Graph}$, $\Edge{\Graph}$, $\HEdge{\Graph}$ (of \emph{vertices}, \emph{(oriented) edges} and \emph{half-edges} respectively), along with a \emph{source function} $\source\colon \Edge{\Graph}\sqcup\HEdge{\Graph}\rightarrow\Vert{\Graph}$ and an edge-inversion function
$(\cdot)^{-1}\colon\Edge{\Graph}\rightarrow\Edge{\Graph}$ which is an involution without fixed points. We write $\target(e):=\source(e^{-1})$ for the \emph{target} of an edge $e\in\Edge{\Graph}$. We will always demand that our graphs be non-empty and connected, that is, for any $u,v\in\Vert{\Graph}$, there is a sequence $e_1,\dots,e_n$ of edges such that $\source(e_1)=u$, $\target(e_n)=v$, and $\source(e_{i+1})=\target(e_i)$ for all $1\leq i<n$.

A \emph{rationally metrised graph} is a pair $(\Graph ,\length )$ consisting of a graph $\Graph$ together with an edge-length function $\length\colon\Edge{\Graph}\rightarrow\Q^{>0}$ satisfying $\length(e)=\length(e^{-1})$ for all edges $e$. We sometimes write the rationally metrised graph simply as $\Graph $.

Finally, a \emph{reduction graph}\index{reduction\dots!\dots graph $\rGraph$} is a tuple $\rGraph=(\Graph,\length,g)$ consisting of a rationally metrised graph $(\Graph,l)$ together with a \emph{genus function} $g\colon\Vert{\Graph}\rightarrow\N_0$.

We define the \textit{degree} of a vertex $v$, denoted $\deg (v)$, by
\[
\deg (v):= \# \{ e\in E(\Graph ):\source (e)=v \} +\# \{ e\in D(\Graph ):\source (e)=v \}.
\]
We say a reduction graph $\rGraph$ is \textit{stable} if every vertex satisfies $2g(v)+\deg (v) >2$, and that $\rGraph $ is \textit{semistable} if every vertex satisfies $2g(v)+\deg (v) \geq 2$.

We will write $\plEdge{\Graph}$ for the set of \emph{unoriented edges}, i.e.\ the quotient of $\Edge{\Graph}$ by identifying $e\sim e^{-1}$.
\end{definition}

\begin{remark}
The reader should think of half-edges as intervals $[0,\infty)$, attached to the rest of the graph at their single boundary points. Thus, the underlying metric space of a graph in the sense of Definition~\ref{def:reduction_graphs} is a length space which is a finite union of closed intervals $[0,\length(e)]$ (one for each unoriented edge) and intervals $[0,\infty)$ (one for each half-edge). Many of our constructions will depend only on this underlying metric space.
\end{remark}

\smallskip

For us, reduction graphs provide a convenient way of encoding groupoids. This is a special case of the more general definition of the fundamental groupoid of a graph of groups.

\begin{definition}[Fundamental groupoid, {\cite[\S2]{higgins}}]\label{def:pi1}
A \emph{graph of groups} is a tuple $\rGraph=(\Graph,(G_v)_{v\in\Vert{\Graph}},\delta_e)$ consisting of a graph $\Graph $, a group $G_v$ for each $v\in \Vert\Gamma $, and for each $e\in \Edge{\Graph}$, and an element $\delta_e\in G_{\source(e)}$.

The \emph{fundamental groupoid} of $\rGraph$, which we denote $\pi _1 (\rGraph)$, is the groupoid with objects $\Vert{\Graph }$ and with hom-sets generated by $e\in\Hom(\source(e),\target(e))$, $G_v\leq\Hom (v,v)$, modulo the \emph{edge relations}\index{relation!edge}
\[
\begin{array}{ccc}
ee^{-1} = 1, & & \delta_{e^{-1}}e=e\delta_e^{-1}, \\
\end{array}
\]
(here $e^{-1}$ denotes the inverse edge in the sense of Definition~\ref{def:reduction_graphs}, so the first equation is saying that this really is the inverse of $e$ in the fundamental groupoid of the graph of groups).

We denote the set of homomorphisms from $u$ to $v$ by $\pi _1 (\rGraph;u,v)$. When $u=v$ we denote this by $\pi _1 (\rGraph,u)$ and refer to it as the \emph{fundamental group} of $\rGraph$ at $u$.

One can equivalently view $\pi _1 (\widetilde{\Gamma };u,v)$ as the set of strings
\[
g_0 e_1 g_1 \ldots g_{n-1}e_n g_n ,
\]
where $e_i \in \Edge{\Graph }$, $g_i \in G_{v_i }$, where $v_i =\source (e_i )=\target (e_{i+1})$, and $v_1 =v,v_n =u$, modulo the relations above.

For a characteristic $0$ field $F$ (for us either $\Q$ or $\Q_\ell$) we denote by $\pi_1^F(\rGraph)$ the $F$-Mal\u cev completion of the fundamental groupoid $\pi_1(\rGraph)$ above. We similarly denote by $\widehat\pi_1(\rGraph)$ the profinite completion of the fundamental groupoid.
\end{definition}

\begin{remark}
The definition above is somewhat less general than the definition given in~\cite[\S2]{higgins} and~\cite[I.5.1]{serre:trees}, in which one also assigns a group $G_e$ to each edge $e$, such that $G_{e^{-1}}=G_e$, along with homomorphisms (usually embeddings) $\mu_e\colon G_e\rightarrow G_{\target(e)}$. The above definition corresponds to the special case where each $G_e$ is infinite cyclic, with the maps $\mu_e\colon G_e\rightarrow G_{\target(e)}$ and $\mu_{e^{-1}}\colon G_e\rightarrow G_{\source(e)}$ sending one generator to $\delta_e$ and the inverse generator to $\delta_{e^{-1}}$, cf.~\cite[Definition~2.9]{AMO}.
\end{remark}

\begin{definition}\label{def:pi1_reduction_graph}\index{fundamental groupoid!\dots of a reduction graph $\pi_1(\rGraph)$}
Given a reduction graph $\rGraph$, we associate a graph of groups  as follows:
\begin{itemize}
	\item the underlying graph of the graph of groups is the underlying graph of $\rGraph$;
	\item each vertex group $G_v$ is the group generated by elements $\beta_{v,i}$ and $\beta_{v,i}'$ for $1\leq i\leq g(v)$ along with elements $\delta_e$ for each edge or half-edge $e\in\source^{-1}(\{v\})$, subject to the single \emph{vertex relation}\footnote{Throughout this paper, sums or products indexed over ``$\source(e)=v$'' indicate sums over both edges and half-edges, unless otherwise indicated.}\index{relation!vertex}
	\[
	\prod_{i=1}^g[\beta_{v,i}',\beta_{v,i}]\cdot\prod_{\source(e)=v}\delta_e=1;
	\]
	\item the edge elements are the elements $\delta_e$ in the above presentation.
\end{itemize}
We define $\pi_1(\rGraph)$ to be the fundamental groupoid of this graph of groups. There is an implicit choice of ordering of the set $\source^{-1}(\{v\})$ in the definition of the groups $G_v$, but $\pi_1(\rGraph)$ is independent of this choice up to non-canonical isomorphism.

There is an automorphism $\sigma$ of the $\Q$-Mal\u cev completion $\pi_1^\Q(\rGraph)$, acting trivially on vertex-groups, and acting on edges via \emph{edge twists} $e\mapsto e\cdot\delta_e^{\length(e)}$ \cite[\S2.4]{AMO}\index{edge twist $\sigma$}. If $n$ is divisible by all denominators of edge-lengths in $\rGraph$, then $\sigma^n$ even acts on $\pi_1(\rGraph)$ before Mal\u cev completion.
\end{definition}

In \cite[(2.7.1) and subsequent paragraphs]{oda}, Oda shows that the outer action of $I_K$ on $\pi _1 ^{\et ,(p')} (X_{\overline{K}})$ can be recovered from the outer action of the edge-twist $\sigma$ on $\pi _1 (\rGraph)$ (a geometric analogue appears in \cite[Theorems 2.1 and 2.2]{AMO}). However, as we shall explain, the proofs of these results imply a stronger result which remembers base-points.

\begin{theorem}[Non-abelian Picard--Lefschetz]\label{thm:graph_comparison}
For any regular semistable model $\mathcal X/\O_K$ of a curve $X/K$ and any $b\in\mathcal X(\O_K)$, there is an isomorphism
\[
\vartheta\colon\pi_1^{\Q_\ell}(X_{\overline K},b) \longisoarrow \pi_1^{\Q_\ell}(\rGraph,\red(b))
\]
equivariant for the action of $\sigma$ (on the left-hand side as an element of inertia, on the right-hand side via edge twists as in Definition~\ref{def:pi1_reduction_graph}), such that the square
\begin{equation}\label{eq:pic-lef_square}
\begin{tikzcd}
\mathcal X(\O_K) \arrow{r}{\classify}\arrow{d}{\red} & \H^1(I_K,\pi_1^{\Q_\ell}(X_{\overline K},b)) \arrow{d}{\vartheta_*} \\
\Vert{\Graph} \arrow{r}{\classify} & \H^1(\Z,\pi_1^{\Q_\ell}(\rGraph,\red(b)))
\end{tikzcd}
\end{equation}
commutes. Here the right-hand vertical map is given by restriction along the homomorphism $\Z\rightarrow I_K$ sending $1\mapsto\sigma$, followed by pushout along $\vartheta$. The lower horizontal map is produced from the $\Z$-equivariant groupoid $\pi_1^{\Q_\ell}(\rGraph)$ as in Remark~\ref{rmk:abstract_kummer}.
\end{theorem}

\begin{remark}\label{rmk:pic-lef_for_groupoids}
The non-abelian Picard--Lefschetz Theorem above implies that the reduction map $\red\colon\mathcal X(\O_K)\rightarrow\Vert{\Graph}$ extends to a $\sigma$-equivariant equivalence
\[
\vartheta\colon\pi_1^{\Q_\ell}(X_{\overline K})|_{\mathcal X(\O_K)} \longisoarrow \pi_1^{\Q_\ell}(\rGraph)
\]
of $\Q_\ell$-pro-unipotent groupoids. Indeed, at the basepoint $b$ this equivalence can be taken to be given by the $\sigma$-equivariant isomorphism $\vartheta\colon\pi_1^{\Q_\ell}(X_{\overline K},b) \isoarrow \pi_1^{\Q_\ell}(\rGraph,\red(b))$, and then~\eqref{eq:pic-lef_square} ensures that for every $x$ there is a $\sigma$-equivariant isomorphism $\pi_1^{\Q_\ell}(X_{\overline K};b,x) \longisoarrow \pi_1^{\Q_\ell}(\rGraph;\red(b),\red(x))$ compatible with the torsor structures. It is easy to see that these choices determine a $\sigma$-equivariant equivalence on the whole fundamental groupoid.

We will see in \S\ref{ss:pic-lef_filtrations} that the isomorphism $\vartheta\colon\pi_1^{\Q_\ell}(X_{\overline K},b) \longisoarrow \pi_1^{\Q_\ell}(\rGraph,\red(b))$ is a $\W$- and $\M$-filtered isomorphism for certain filtrations on the right-hand side, and in \S\ref{ss:pic-lef_plus} that it can be chosen compatibly as we enlarge both the base field $K$ and the model $\mathcal X$. It follows that the equivalence of groupoids above can be extended to a $\sigma$-equivariant, $\W$- and $\M$-filtered equivalence
\begin{equation}\label{eq:pic-lef_for_groupoids}
\vartheta\colon\pi_1^{\Q_\ell}(X_{\overline K}) \longisoarrow \pi_1^{\Q_\ell}(\rGraph)
\end{equation}
of $\Q_\ell$-pro-unipotent groupoids lying over the reduction map $\red\colon X(\overline K)\rightarrow\QVert{\Graph}$.

Such an equivalence of groupoids even exists without assuming $X$ has semistable reduction, if we make the following minor modification. For every $x,y\in X(\overline K)$, there is a power $\sigma^e$ of $\sigma$ acting ind-unipotently on the algebra $\O(\pi_1^{\Q_\ell}(X_{\overline K};x,y))$, so that the ind-nilpotent derivation $N:=\frac1e\log(\sigma^e)$ is well-defined. This induces an automorphism $\exp(N)$ of the groupoid $\pi_1^{\Q_\ell}(X_{\overline K})$, and there is a $\W$- and $\M$-filtered equivalence~\eqref{eq:pic-lef_for_groupoids} which is equivariant for the action of $\exp(N)$ on the left-hand side and the action of $\sigma$ on the right.
\end{remark}

The proof of Theorem~\ref{thm:graph_comparison} will be completed in \S\ref{ss:arithmetic}.

\subsection{Graphs of groups in topology}
\label{ss:graphs_of_groups_topology}

The graphs of groups we consider arise naturally from topology in the following manner. Let $X$ be the complement in a compact oriented surface $\overline X$ of a finite set $D$, and let $\tau_1,\dots,\tau_n\colon S^1\times[0,1]\hookrightarrow X$ be oriented\footnote{For our convention for orienting $S^1\times[0,1]$, see \cite[\S3.1.1]{primer}.} closed immersions with pairwise disjoint images. For an embedding $\tau\colon S^1\times[0,1]\rightarrow X$, we write $\bar\tau\colon S^1\times[0,1]\hookrightarrow X$ for the embedding given by $\bar\tau(\theta,s)=\tau(-\theta,1-s)$.

There is a reduction graph $\rGraph$ naturally associated with this setup, namely the graph with vertex-set $\pi_0\left(X\setminus\bigcup_i\im(\tau_i)\right)$, edge-set $\{\tau_1,\dots,\tau_n\}\cup\{\bar\tau_1,\dots,\bar\tau_n\}$, and half-edge-set $D$. The inverse of an edge $e$ corresponding to an embedding $\tau_e$ is the edge corresponding to the embedding $\bar\tau_e$, and $\source(e)$ is the unique component whose closure contains $\tau_e(S^1\times\{0\})$. The genus of a vertex $v$ is the genus of the corresponding surface $X_v$, and the source of a half-edge $e$ is the unique component containing the corresponding puncture.

The significance of this reduction graph is that it computes the fundamental groupoid of $X$ \cite{Althoen}. Specifically, let us fix a basepoint $x_v$ in each connected component $X_v$, and for each edge $e$ with $\source(e)=v$ let us choose a path $p_e$ from $x_v$ to $\tau_e(0,0)$ inside the closure $\overline{X_v}$ of the component $X_v$. We write $\delta_e\in\pi_1(\overline{X_v},x_v)$ for the composite $p_e^{-1}\tau_e(\delta)p_e$ with $\delta=\mathbf1^{-1}_{S^1}\times\{0\}$ the inverse of the standard loop inside $S^1\times[0,1]$. Thus $\delta_e$ is conjugate to a loop running anticlockwise around one of the boundary components of $\overline{X_v}$. For each half-edge $e$ with $\source(e)=v$, we may also choose an element $\delta_e\in\pi_1(X_v,x_v)$ conjugate to a small loop anticlockwise around the corresponding puncture. The paths $p_e$ and loops $\delta_e$ may be chosen so that the group $\pi_1=(X_v,x_v)=\pi_1(\overline{X_v},x_v)$ is generated by the above elements $\delta_e$ along with $2g(v)$ auxiliary elements\footnote{One may also assume that these elements form a symplectic basis of the homology of the compactification of $X_v$.} $\beta_{v,i}$ and $\beta_{v,i}'$ ($1\leq i\leq g(v)$) subject to the single relation
\[
\prod_{i=1}^{g(v)}[\beta_{v,i}',\beta_{v,i}]\cdot\prod_{\source(e)=v}\delta_e=1.
\]

\begin{lemma}[\cite{Althoen}]\label{lem:decompose_surface}
There is an equivalence of groupoids
\[
\Theta\colon \pi_1(\rGraph) \longisoarrow \pi_1(X)
\]
sending $v\in\Vert{\Graph}$ to $x_v\in X$, acting in the obvious way on vertex groups $G_v\nciso\pi_1(X_v,x_v)$, and sending an edge $e$ to $p_{e^{-1}}^{-1}\tau_e(\rho)p_e$, where $\rho=\{0\}\times\1_{[0,1]}$ is the standard path inside $S^1\times[0,1]$.
\end{lemma}

\subsection{The graph of groups description of monodromy: analytic setting}\label{ss:pic-lef_analytic}

Let $\Disc$ denote the disc\index{disc $\Disc$} $\{t\in \mathbb{C}:|t|<1\}$, and $\Discx:=\Disc-\{0\}$. By a \textit{proper regular semistable curve} $\overline{\mathcal X}$ over $\Disc$, we will mean a proper, flat holomorphic map
\[
\pi\colon \overline{\mathcal X}\to \Disc
\]
of complex manifolds of relative dimension one which is a submersion over $\Discx$ and the fibre $\mathcal X_0$ at $0$ is a semistable complex analytic curve. By a \emph{regular semistable curve} $\pi\colon\mathcal X\rightarrow\Disc$ we shall mean the complement, in such a $\overline{\mathcal X}$, of a finite number of holomorphic sections of $\pi\colon\overline{\mathcal X}\to\Disc$.

For a regular semistable curve $\pi\colon\mathcal X\rightarrow\Disc$, we define the \emph{fundamental groupoid} $\pi_1(\mathcal X)|_{\mathcal X(\Disc)}$ to be the groupoid whose objects are the holomorphic sections of $\pi$, and whose morphism-sets are given by $\pi_1(\mathcal X;x,y):=\pi_1(\mathcal X_t;x(t),y(t))$, where $\mathcal X_t$ denotes the fibre of $\mathcal X$ over some $t\neq0$. Since $\mathcal X\rightarrow\Disc$ is a locally trivial fibration over $\Discx$, any path $\gamma\in\pi_1(\Discx;t,t')$ induces a bijection $\pi_1(\mathcal X_t;x(t),y(t))\longisoarrow\pi_1(\mathcal X_{t'};x(t'),y(t'))$ compatible with path-composition. In particular, there is an action of $\Z=\pi_1(\Discx,t)$ on the groupoid $\pi_1(\mathcal X)|_{\mathcal X(\Disc)}$, and the groupoid $\pi_1(\mathcal X)|_{\mathcal X(\Disc)}$ does not depend on the choice of $t$ up to $\Z$-equivariant isomorphism\footnote{Technically, we need to fix identifications of the fundamental groupoids defined with different values of $t$. We do this by choosing a trivialisation of the fundamental groupoid of $\Discx$, which we fix once and for all.}.

Slightly more generally, for any $\epsilon>0$, we let $\mathcal X(\Disce{\epsilon})$ denote the set of holomorphic sections of $\pi$ over the disc of radius $\epsilon$, and let $\pi_1(\mathcal X)$ denote the $\Z$-equivariant groupoid with object-set $\varinjlim_\epsilon\mathcal X(\Disce{\epsilon})$ and morphism-sets given by $\pi_1(\mathcal X;x,y):=\pi_1(\mathcal X_t;x(t),y(t))$ for some $t\in\Discx$ sufficiently close to $0$.

\smallskip

There is a Picard--Lefschetz Theorem also for regular semistable curves over $\Disc$. Specifically, one can associate to $\mathcal X$ a reduction graph $\rGraph$, namely the dual graph of its special fibre $\mathcal X_0$ as in Definition~\ref{def:reduction_graphs_of_curves}, and there is a reduction map $\red\colon\varinjlim_\epsilon\mathcal X(\Disce{\epsilon})\rightarrow\Vert{\Graph}$ sending a section $x$ to the component of $\mathcal X_0$ containing $x(0)$. The Picard--Lefschetz Theorem says that this reduction graph determines the fundamental groupoid $\pi_1(\mathcal X)$.

\begin{lemma}[Non-abelian Picard--Lefschetz for analytic families]\label{lem:pic-lef_analytic}
For any regular semistable curve $\pi\colon\mathcal X\rightarrow\Disce{\epsilon}$ and any holomorphic section $b$ of $\pi$, there is a $\Z$-equivariant isomorphism
\[
\vartheta\colon\pi_1(\mathcal X,b) \longisoarrow \pi_1(\rGraph,\red(b))
\]
such that the square
\begin{equation}
\begin{tikzcd}
\varinjlim_{\epsilon'}\mathcal X(\Disce{\epsilon'}) \arrow{r}{\classify}\arrow{d}{\red} & \H^1(\Z,\pi_1(\mathcal X,b)) \arrow{d}{\vartheta_*} \\
\Vert{\Graph} \arrow{r}{\classify} & \H^1(\Z,\pi_1(\rGraph,\red(b)))
\end{tikzcd}
\end{equation}
commutes. The horizontal maps in this square are the ones constructed from the $\Z$-equivariant groupoids $\pi_1(\mathcal X)$ and $\pi_1(\rGraph)$, as in Remark~\ref{rmk:abstract_kummer}.
\end{lemma}
\begin{proof}[Proof, following the proof of {\cite[Theorem 2.2]{AMO}}]
We will shrink the disc $\Disce{\epsilon}$ as needed throughout the proof, often without comment. We write
\[
\mM:=\{(t,z,w)\in\Disce{\epsilon}\times\Disce{\epsilon^{1/3}}^2 : wz=t\},
\]
which we view as a complex manifold with a map to $\Disce{\epsilon}$ given by projection onto the $t$-coordinate. For every singular point of $\mathcal X_0$ we may choose a holomorphic open immersion $\tau_i\colon\mM\hookrightarrow\mathcal X$ compatible with the projections to $\Disc$ taking $(0,0,0)$ to the singular point. We may assume that the images of the $\tau_i$ are pairwise disjoint.

Let $\mM_0=\{(t,z,w)\in\mM : |w|,|z|\leq|t|^{1/3}\}$, so that $\mM_0$ is a closed subspace of $\mM$ whose fibre over any $t\neq0$ is a closed annulus. In particular, the maps $\tau_i\colon\mM_0\hookrightarrow\mathcal X$ endow each non-special fibre of $\mathcal X$ with the structure considered in \S\ref{ss:graphs_of_groups_topology}.

The complement $\mM\setminus\mM_0$ consists of two components $\mM_-:=\{|w|>|t|^{1/3}\}$ and $\mM_+:=\{|z|>|t|^{1/3}\}$, each diffeomorphic over $\Disce{\epsilon}$ to $\Disce{\epsilon}\times\Discx$. It follows from Ehresmann's Theorem \cite[Theorem~4.1.2]{carlson-mueller-stach-peters} that the map $\pi\colon\mathcal X\setminus\bigcup_i\tau_i(\mM_0)\rightarrow\Disce{\epsilon}$ is a locally trivial fibration, so we may choose a diffeomorphism
\begin{equation}\label{eq:fibration_surgery}\tag{$\ast$}
\mathcal X\setminus\bigcup_i\tau_i(\mM_0) \nciso \coprod_{v\in\Vert{\Graph}}\Disce{\epsilon}\times X_v
\end{equation}
over $\Disce{\epsilon}$, for $X_v$ the connected components of $\mathcal X_0^\ns$. One sees from this description that the reduction graph of $\mathcal X$ is isomorphic to the reduction graph assigned to any non-special fibre $\mathcal X_t$ in \S\ref{ss:graphs_of_groups_topology}.

We now construct the desired isomorphism $\vartheta\colon\pi_1(\mathcal X,b)\longisoarrow\pi_1(\rGraph,\red(b))$. In doing so, it is convenient to enlarge the groupoid $\pi_1(\mathcal X)$ to allow as basepoints not just holomorphic sections of $\mathcal X$ (over small discs), but also continuous sections of $\mathcal X\setminus\bigcup_i\tau_i(\mM_0)$. If we fix basepoints $x\in X_v$ and paths $p_e$ inside some fibre $\mathcal X_t$ as in \S\ref{ss:graphs_of_groups_topology}, then from Lemma~\ref{lem:decompose_surface} we have an equivalence of categories $\Theta\colon\pi_1(\rGraph)\rightarrow\pi_1(\mathcal X)$, sending a vertex $v$ to the constant section with value $x_v\in X_v$ (with respect to~\eqref{eq:fibration_surgery}).

Over a sufficiently small disc $\Disce{\epsilon'}$, the holomorphic section $b$ factors through the one component $\Disce{\epsilon}\times X_{\red(b)}$ of the decomposition~\eqref{eq:fibration_surgery}, so we may choose a path $\gamma_b\in\pi_1(\mathcal X;b,x_{\red(b)})$ from $b$ to the constant section with value $x_{\red(b)}$ by choosing a path from $b(t')$ to $x_{\red(b)}$ inside $X_v\subseteq\mathcal X_{t'}$ for some non-zero $t'\in\Disce{\epsilon'}$. We thus obtain the desired isomorphism $\vartheta$ as the composite
\[
\pi_1(\mathcal X,b) \longisoarrow \pi_1(\mathcal X, x_{\red(b)}) \longisoarrow \pi_1(\rGraph,x_{\red(b)})
\]
where the first arrow is conjugation by $\gamma_b$ and the second is the inverse of $\Theta$.

To complete the proof, it remains to show that $\vartheta$ is $\sigma$-equivariant and that for every holomorphic section $y$ (maybe over a smaller disc), there is a $\sigma$-equivariant isomorphism $\vartheta_y\colon\pi_1(\mathcal X;b,y)\longisoarrow\pi_1(\rGraph;b,y)$ compatible with the torsor structures -- we will prove the latter assertion, the former being a special case. The isomorphism $\vartheta_y$ is constructed in the same way as $\vartheta$, choosing a path $\gamma_y\in\pi_1(\mathcal X;y,x_{\red(y)})$ in the same manner as $\gamma_b$, and defining $\vartheta_y(\gamma):=\Theta^{-1}(\gamma_y\gamma\gamma_b^{-1})$. This is compatible with the torsor structures, so it remains to check $\sigma$-equivariance.

It follows by construction that the action of the monodromy of the family $\mathcal X$ on the paths $\gamma_b,\gamma_y$ factors through the monodromy of the trivial family $\mathcal X\setminus\bigcup_i\tau_i(\mM_0)$, and hence they are $\sigma$-fixed. It thus suffices to show that $\Theta$ is $\sigma$-equivariant, which we check on generators of $\pi_1(\rGraph)$. For any vertex $v$, the monodromy action on $\Theta(G_v)$ is trivial by the same argument as above. For any edge $e$ corresponding one of the $\tau_i$, we may write $\Theta(e)=q_{e^{-1}}^{-1}\tau_i(\rho)q_e$ where $\rho\in\pi_1(\mM_t;s_-(t),s_+(t))$ with $s_\pm$ continuous sections of $\mM_\pm$ respectively. The monodromy action on $\rho$ is given by right-handed Dehn twist \cite[\S3.1.1]{primer}, i.e.\ $\sigma(\rho)=\rho\delta$ with $\delta$ a loop in the open annulus $\mM_t$ which winds once clockwise around the origin with respect to the $z$-coordinate. This shows that $\sigma\Theta(e)=\Theta(e)\cdot\Theta(\delta_e)=\Theta(\sigma(e))$, as desired.
\end{proof}

\subsection{The graph of groups description of monodromy: geometric setting}\label{ss:pic-lef_equichar}

Let $R=\mathbb{C}[t]^{\sh }_t$ denote the (strict) Henselisation of $\mathbb{C}[t]$ at $t$, and write $F=\Frac(R)=R\left[\frac1t\right]$ for the fraction field. An algebraic closure $\overline F$ is given by adjoining all roots of $t$, so that the absolute Galois group $G_F:=\Gal(\overline F/F)$ is canonically isomorphic to $\Zhat(1)$. We let $\sigma\in G_F$ denote the generator corresponding to the standard compatible system of roots of unity in $\C$, i.e.\ $\left(\exp\left(\frac{2\pi i}n\right)\right)_{n>0}$.

Let $\mathcal X/R$ be a regular semistable curve such that $X=\mathcal X_F$ is smooth. As in Definition~\ref{def:reduction_graphs_of_curves}, we can associate to $\mathcal X$ a reduction graph $\rGraph$, namely the dual graph of its special fibre $\mathcal X_0$. As usual, there is a reduction map $\red\colon \mathcal X(R) \to \Vert{\Graph}$ sending an $R$-point $x$ to the component of $\mathcal X_0$ containing $x(0)$.

\begin{lemma}[Non-abelian Picard--Lefschetz in equicharacteristic $0$]\label{lem:pic-lef_equichar}
For any regular semistable curve $\mathcal X/R$ and any $b\in\mathcal X(R)$, there is a $\sigma$-equivariant isomorphism
\[
\vartheta\colon\pi_1^\et(X_{\overline F},b) \longisoarrow \widehat\pi_1(\rGraph,\red(b))
\]
such that the square
\begin{equation*}
\begin{tikzcd}
\mathcal X(R) \arrow{r}{\classify}\arrow{d}{\red} & \H^1(G_F,\pi_1^\et(X_{\overline F},b)) \arrow{d}{\vartheta_*} \\
\Vert{\Graph} \arrow{r}{\classify} & \H^1(\Z,\widehat\pi_1(\rGraph,\red(b)))
\end{tikzcd}
\end{equation*}
commutes. The horizontal maps in this square are the ones constructed from the $\Zhat$-equivariant profinite groupoids $\pi_1^\et(X_{\overline F})$ and $\widehat\pi_1(\rGraph)$, as in Remark~\ref{rmk:abstract_kummer}; the vertical map is restriction along the homomorphism $\Z\rightarrow G_F$ taking $1$ to $\sigma$, followed by pushout along $\vartheta$.
\end{lemma}
\begin{proof}
There is a connected \'etale neighbourhood $U$ of $0\in\A^1_\C$ and a curve $\pi\colon\widetilde{\mathcal X}\rightarrow U$, smooth away from $0$, such that $\mathcal X$ is the pullback of $\widetilde{\mathcal X}$ along the canonical map $\Spec(R)\rightarrow U$. Shrinking $U$ if necessary, we may assume that $b\in\mathcal X(R)$ extends to a section of $\pi$. For $\epsilon$ sufficiently small, the disc $\Disce{\epsilon}\subseteq\A^{1,\an}_\C$ lifts to an open neighbourhood $\Disce{\epsilon}\subseteq U^\an$ of $0$, and we write $\pi\colon\mathcal X^\an\rightarrow \Disce{\epsilon}$ for the restriction of the analytification of $\pi\colon\widetilde{\mathcal X}\rightarrow U$.

Now by Grothendieck's Specialisation Theorem \cite[Expos\'e~XIII, Proposition~4.3 \& Exemples~4.4]{SGA1} we have that for any geometric point $s$ of $U\setminus\{0\}$ the natural sequence
\[
1 \rightarrow \pi_1^\et(\widetilde{\mathcal X}_s,b(s)) \rightarrow \pi_1^\et(\widetilde{\mathcal X}_{U\setminus\{0\}},b(s)) \rightarrow \pi_1^\et(U\setminus\{0\},s) \rightarrow 1
\]
is exact, and split by the section $b$. We apply this to two geometric points: the point $\overrightarrow0$ given by $\Spec(\overline F)\rightarrow\Spec(R)\setminus\{0\}\rightarrow U\setminus\{0\}$; and the point $t\in U(\C)\setminus\{0\}$ corresponding to some non-zero $t\in\Disce{\epsilon}$. Choosing a path $\gamma_{\overrightarrow0,t}\in\pi_1^\et(U\setminus\{0\};\overrightarrow0,t)$, we obtain an isomorphism
\begin{equation}\label{eq:pic-lef_equichar}\tag{$\ast$}
\pi_1^\et(X_{\overline F},b) \longisoarrow \pi_1^\et(\widetilde{\mathcal X}_t,b(t)) = \widehat\pi_1(\mathcal X^\an_t,b(t))
\end{equation}
which is equivariant for the conjugation action of $\pi_1^\et(U\setminus\{0\},\overrightarrow0)\simeq\pi_1^\et(U\setminus\{0\},t)=\widehat\pi_1(U^\an\setminus\{0\},t)$ (where the first identification is conjugation by $\gamma_{\overrightarrow0,t}$). Composing with the isomorphism from Lemma~\ref{lem:pic-lef_analytic}, we obtain the desired isomorphism $\vartheta\colon\pi_1^\et(X_{\overline F},b)\longisoarrow\widehat\pi_1(\rGraph,\red(b))$.

Using Lemma~\ref{lem:pic-lef_analytic}, it suffices to prove that~\eqref{eq:pic-lef_equichar} is $\Zhat$-equivariant (for an appropriate choice of $\gamma_{\overrightarrow0,t}$), and that for any $x\in\mathcal X(R)$ there is a $\Zhat$-equivariant isomorphism $\pi_1^\et(X_{\overline F};b,x)\isoarrow\widehat\pi_1(\mathcal X^\an;b,x)$ compatible with the torsor structures via~\eqref{eq:pic-lef_equichar}. As usual, we will just prove the latter, deducing the former as a special case. The desired isomorphism is constructed as above, choosing if necessary a smaller \'etale neighbourhood $U'$ of $0$ such that $x$ defines a section of $\widetilde{\mathcal X}_{U'}$, and taking the isomorphism to be
\begin{equation}\label{eq:pic-lef_equichar_torsor}\tag{$\ast\ast$}
\pi_1^\et(X_{\overline F};b,x) \longisoarrow \pi_1^\et(\widetilde X_{t'};b(t'),x(t'))=\widehat\pi_1(\mathcal X^\an_{t'};b(t'),x(t'))
\end{equation}
for some sufficiently small $t'$, where the first arrow denotes pre- and postmultiplication by $b(\gamma'_{\overrightarrow0,t'})^{-1}$ and $x(\gamma'_{\overrightarrow0,t'})$ for some $\gamma'_{\overrightarrow0,t'}\in\pi_1^\et(U'\setminus\{0\};\overrightarrow0,t')$. We want to prove that~\eqref{eq:pic-lef_equichar_torsor} is $\Zhat$-equivariant and compatible with the torsor structures for some choice of $\gamma_{\overrightarrow0,t}$ and $\gamma'_{\overrightarrow0,t'}$. This is provided by the following proposition.
\end{proof}

\begin{proposition}
For any \'etale neighbourhood $U$ of $0\in\A^1_\C$, any $\epsilon>0$ such that $\Disce{\epsilon}\subseteq\A^{1,\an}_\C$ lifts to a neighbourhood of $0\in U^\an$ and any non-zero $t\in\Disce{\epsilon}$, there is a $\gamma_{\overrightarrow0,t}\in\pi_1^\et(U\setminus\{0\};\overrightarrow0,t)$ such that the isomorphism
\[
\pi_1^\et(U\setminus\{0\},\overrightarrow0) \longisoarrow \pi_1^\et(U\setminus\{0\},t) = \widehat\pi_1(U^\an\setminus\{0\},t)
\]
identifies the images of the maps
\begin{align*}
G_F &\rightarrow \pi_1^\et(U\setminus\{0\},\overrightarrow0) \\
\widehat\pi_1(\Discex{\epsilon},t) &\rightarrow \widehat\pi_1(U^\an\setminus\{0\},t)
\end{align*}
along the isomorphism $G_F\simeq\Zhat(1)\simeq\Zhat$ sending $\sigma$ to $1\in\Zhat$. Moreover, these paths may be chosen compatibly as we shrink $U$, decrease $\epsilon$ and vary $t$ (the latter meaning compatibly with the trivialisation of the fundamental groupoid of $\Discx$ chosen in \S\ref{ss:pic-lef_analytic}).
\end{proposition}
\begin{proof}[Proof (sketch)]
The first part amounts to the assertion that the diagram
\begin{center}
\begin{tikzcd}
\Fet(U\setminus\{0\}) \arrow{r}\arrow{d} & \Cov_\fin(\Discex{\epsilon}) \arrow{d}{t^*}[swap]{\viso} \\
\Fet(\Spec(F)) \arrow{r}{\overrightarrow0^*}[swap]{\sim} & \SET_\fin(\Zhat)
\end{tikzcd}
\end{center}
commutes up to natural isomorphism, where the lower and right-hand maps are the fibre functors associated to $\overline F/F$ and $t\in\Discex{\epsilon}$ respectively. To show this, we introduce another fibre functor $0^*\colon\Fet(U\setminus\{0\})\rightarrow\SET_\fin(\Zhat)$, which is defined on objects as follows. Given a finite \'etale cover $V\rightarrow U\setminus\{0\}$, we let $\overline V\rightarrow U$ denote the normalisation, and $\overline V_0$ the fibre over $0$. We define
\[
0^*(V) := \coprod_{v\in\overline V_0}\Z/e_v\Z,
\]
where $e_v$ denotes the ramification degree of $\overline V\rightarrow U$ at $v$, and the right-hand side is given the obvious structure as a finite $\Zhat$-set. This functor factors through analogously-defined functors on $\Cov_\fin(\Discex{\epsilon})$ and $\Fet(\Spec(F))$, also denoted $0^*$, so it suffices to show that the functors $t^*,0^*\colon\Cov_\fin(\Discex{\epsilon})\rightarrow\SET_\fin(\Zhat)$ are naturally isomorphic, and similarly for $\overrightarrow0^*,0^*\colon\Fet(\Spec(F))\rightarrow\SET_\fin(\Zhat)$. But it follows from the general mantra of Galois categories that the fibre functors $t^*$ and $0^*$ are naturally isomorphic as $\SET_\fin$-valued functors \cite[Corollaire~V.5.7]{SGA1}, and such an isomorphism is automatically $\Zhat$-equivariant since $\Zhat$ is abelian. The same applies between $\overrightarrow0^*$ and $0^*$.

It is easy to see that this construction of a natural isomorphism is compatible with shrinking $U$ and $\epsilon$, and can be chosen compatibly as we vary $t$.
\end{proof}

\subsection{The graph of groups description of monodromy: arithmetic setting}\label{ss:arithmetic}

We now finally come to the proof of the Picard--Lefschetz Theorem~\ref{thm:graph_comparison} for fundamental groupoids of curves in mixed characteristic. We fix a regular semistable curve $\mathcal X/\O_K$ with a basepoint $b\in\mathcal X(R)$ and write $X=\mathcal X_K$ for its generic fibre. As in \cite[\S2.3]{oda}, the approach is to deform the problem to equicharacteristic $0$, where we have already proved the analogous result (Lemma~\ref{lem:pic-lef_equichar}).

\begin{lemma}\label{lem:deformation_exists}
Let $R=\Z_p^\nr[t]^\sh_t$ denote the (strict) Henselisation of $\Z_p^\nr[t]$ at $(p,t)$. There exists a regular semistable curve $\pi\colon\widetilde{\mathcal X}\to \Spec(R)$
such that the formal neighbourhoods of all singular points of the special fibre are isomorphic, over $\widehat{\Z_p^\nr}[\![t]\!]$, to $\Spf\left(\widehat{\Z_p^\nr}[\![x,y,t]\!]/(xy-t)\right)$, and 
such that the pullback of $\widetilde{\mathcal X}$ to $\O_{K^\nr}$ via the map $t\mapsto\varpi$ is isomorphic to $\mathcal X_{\O_{K^\nr}}$.
\end{lemma}
\begin{proof}
We follow the construction in \cite[\S 2.3]{oda}, using the description of the universal curve in a formal neighbourhood of a stable curve. Removing an \'etale divisor from $\mathcal X$ if necessary, we may assume that $\mathcal X/\O_K$ is in fact regular stable, so that $\mathcal X_{\O_{K^\nr}}$ classified by a map $x\colon\Spec(\O_{K^\nr})\rightarrow M_{g,n}$. We write $\overline x$ for the restriction of $x$ to $\Spec(\overline k)$.

By \cite[proof of Theorem 2.7]{Knudsen83}, the formal completion of the local ring at $\overline x$ is isomorphic to $\widehat{\Z_p^\nr}[\![t_1,\ldots,t_k]\!]$ in such a way that the formal completion of the universal curve $\pi\colon M_{g,n+1}\to M_{g,n}$ is isomorphic, in a formal neighbourhood of each singular point of $\pi^{-1}(\overline x)$, to the map $\Spf(\widehat{\Z_p^\nr}[\![t_1,\ldots,\widehat{t_i},\ldots,t_k,x,y]\!]) \to \Spf(\widehat{\Z_p^\nr}[\![t_1,\ldots,t_k]\!])$ sending $t_i$ to $xy$. Hence the lemma follows from taking a suitable lift $\O_{M_{g,n},\overline x}^\sh \to R$ of the morphism $\O_{M_{g,n},\overline x}^\sh \to \O_{K^\nr} $ on strictly Henselian local rings induced by $x$, as in \cite[\S 2.3]{oda}. Concretely, we choose a regular system of parameters $\varpi,t_1,\ldots,t_{3g-3+n}$ for $\O^\sh_{M_{g,n},\overline x}$ such that the kernel of the map to $\O_{K^\nr}$ defined by $x$ is generated by $t_i-u_i\varpi$ for units $u_i$, and then define 
\[
R=\O^\sh_{M_{g,n},\overline x}/(u_1 t_2 -u_2 t_1 ,u_1 t_3 -u_3 t_1 ,\ldots ,u_1 t_{3g-3+n}-u_{3g-3+n}t_1)
\]
and $t=t_1/u_1 $. Note that $R$ is equal to the strict Henselisation of $\Z_p [t]$, since the map $\Z_p[t]^\sh_t\to R$ is an isomorphism on formal completions and $R$ is strictly Henselian.
\end{proof}

Using this deformation, we can pass from the mixed characteristic setting to the equicharacteristic $0$ setting via Abhyankar's Lemma \cite[Expos\'e VIII, \S 5]{grothendieck:1972}.

\begin{proposition}[Abhyankar's Lemma, {\cite[Lemma 2.6]{oda}}]\label{prop:abhyankar}
The morphisms $\Spec(\O_{K^\nr}) \to \Spec(R)$ and $\Spec(\C[t]^\sh_t) \to \Spec(R)$, defined by $t\mapsto\varpi$ and a choice of embedding $\Z_p^\nr\hookrightarrow\C$ respectively, induce isomorphisms
\begin{align*}
G_{K^\nr}^{(p')} &\longisoarrow\pi_1^{\et,(p')}(\Spec(R[1/t]),\eta_K) \\
G_F^{(p')} &\longisoarrow\pi_1^{\et,(p')}(\Spec(R[1/t]),\eta_F),
\end{align*}
where $F=\Frac(\C[t]^\sh_t)$ as in \S\ref{ss:pic-lef_equichar} and $\eta_K$ and $\eta_F$ denote the geometric points of $\Spec(K)$ and $\Spec(F)$ determined by their respective algebraic closures.

The identification $G_{K^\nr}^{(p')}\iso G_F^{(p')}$ given by conjugation by a choice of path $\gamma_{K,F}\in\pi_1^{\et,(p')}(\Spec(R[1/t]);\eta_K,\eta_F)$ doesn't depend on $\gamma_{K,F}$. Under the canonical isomorphisms $G_{K^\nr}^{(p')}\simeq\Zhat^{(p')}(1)(K^\nr)=\Zhat^{(p')}(1)(\Q_p^\nr)$ and $G_F^{(p')}\simeq\Zhat^{(p')}(1)(\C)$ given by the tame character, this identification is the one induced by the inclusion $\Q_p^\nr\hookrightarrow\C$. In particular, by choosing the embedding $\Z_p^\nr\hookrightarrow\C$ appropriately, we may assume that our chosen generator of tame inertia $\sigma\in I_K=G_{K^\nr}$ maps under this identification to the preferred generator $\sigma\in G_F$ from \S\ref{ss:pic-lef_equichar}.
\end{proposition}

Let now $\pi\colon\widetilde{\mathcal X}\to \Spec(R)$ be a deformation of $\mathcal X$ as in Lemma~\ref{lem:deformation_exists} and let $\widetilde x,\widetilde y\in\widetilde{\mathcal X}(R)$ be two sections of $\pi$ deforming sections $x,y\in\mathcal X(\O_K)$. The map $\pi$ is smooth over $\Spec(R[1/t])$, and hence by applying Grothendieck's Specialisation Theorem as in the proof of Lemma~\ref{lem:pic-lef_equichar} we obtain an isomorphism
\begin{equation}\label{eq:groth_spec}
\pi_1^{\et,(p')}(X_{\overline K};x,y) \longisoarrow \pi_1^{\et,(p')}(\widetilde{\mathcal X}_{\overline F};\widetilde x(\eta_F),\widetilde y(\eta_F))
\end{equation}
given by precomposition by $\widetilde x(\gamma_{K,F})^{-1}$ and postcomposition by $\widetilde y(\gamma_{K,F})$. Using Abhyankar's Lemma, we see that the isomorphism~\eqref{eq:groth_spec} is equivariant for the action of $G_K^{(p')}\cong G_F^{(p')}$ on either side.

\begin{proof}[Proof of Theorem~\ref{thm:graph_comparison}]
We will in fact prove the corresponding result for the pro-prime-to-$p$ \'etale fundamental group, obtaining the $\Q_\ell$-pro-unipotent version by taking continuous $\Q_\ell$-Mal\u cev completions.

We fix a deformation $\widetilde{\mathcal X}/R$ as in Lemma~\ref{lem:deformation_exists}. The restriction $\widetilde{\mathcal X}_{\Z_p^\nr}$ to the $t=0$ locus is a curve over $\Z_p^\nr$ such that the local rings of all singularities in the special fibre are \'etale-locally given by an equation $xy=0$. It follows that the normalisation of $\widetilde{\mathcal X}_{\Z_p^\nr}$ is smooth over $\Spec(\Z_p^\nr)$, and so the dual graphs of the geometric generic and special fibres of $\widetilde{\mathcal X}_{\Z_p^\nr}$ are canonically identified. In particular, the reduction graphs of $\mathcal X$ and $\widetilde{\mathcal X}_{\C[t]^\sh_t}$ are the same and their reduction maps agree on $\widetilde{\mathcal X}(R)$.

Now choose a deformation $\widetilde b\in\widetilde{\mathcal X}(R)$ of the basepoint $b\in\mathcal X(\O_K)$. We define the isomorphism $\vartheta\colon\pi_1^{\et,(p')}(X_{\overline K},b)\longisoarrow\pi_1^{(p')}(\rGraph,\red(b))$ by composing the isomorphism~\eqref{eq:groth_spec} (for $\widetilde x=\widetilde y=\widetilde b$) with the isomorphism $\pi_1^{\et,(p')}(\widetilde{\mathcal X}_{\overline F},\widetilde b(\eta_F))$ provided by Lemma~\ref{lem:pic-lef_equichar}. This isomorphism $\vartheta$ is then equivariant for the action of $\sigma$ on either side, and it is easy to see, again by~\eqref{eq:groth_spec} and Lemma~\ref{lem:pic-lef_equichar}, that the square~\ref{eq:pic-lef_square} commutes.
\end{proof}

\begin{remark}
An alternative (related) approach to proving Theorem~\ref{thm:graph_comparison} is via Sa\"idi's Theorem \cite[Theorem 2.8]{saidi}. Sa\"idi's approach may be characterised as proving an equivalence of Galois categories between tame \'etale covers of $X_{\overline{K}}$ and `\'etale covers of the graph of groups $\rGraph $', for a suitable interpretation of the latter.
\end{remark}

\begin{remark}\label{rmk:too_many_choices}
Note that the isomorphism in Theorem~\ref{thm:graph_comparison} depend on several choices: the construction depends on a choice of uniformiser, a choice of deformation, a choice of orderings of the edges meeting each vertex, a choice of embedding of $\Q_p^\nr$ into $\mathbb{C}$, a choice of local coordinates at singular points of the special fibre of the associated complex analytic fibration, etc.
\end{remark}

\subsection{Compatibility with the $\W$- and $\M$-filtrations}\label{ss:pic-lef_filtrations}

In \S\ref{ss:pi1_curves}, we saw that the $\Q_\ell$-pro-unipotent \'etale fundamental group $\pi_1^{\Q_\ell}(X_{\overline K},b)$ of a hyperbolic curve carries two filtrations $\M_\bullet$ and $\W_\bullet$, the former determined by the weights of Frobenius and the latter being a variant of the descending central series which takes account of the failure of $X$ to be projective. We will want to describe these filtrations purely combinatorially via the description of $\pi_1^{\Q_\ell}(X_{\overline K},b)$ from the non-abelian Picard--Lefschetz Theorem~\ref{thm:graph_comparison}. The relevant definition on the side of reduction graphs is as follows.

\begin{definition}\label{def:filtrations_on_graph}
Let $\rGraph$ be a reduction graph and $b\in\Vert{\Graph}$ a basepoint. We endow $\pi_1(\rGraph,b)$ with increasing filtrations $\W_\bullet$ and $\M_\bullet$ as follows.

We set $\W_{-1}\pi_1(\rGraph,b)=\pi_1(\rGraph,b)$, and define $\W_{-2}\pi_1(\rGraph,b)$ to be the subgroup spanned by the commutator subgroup $[\W_{-1},\W_{-1}]$ and conjugates of the elements $\delta_e$ for $e$ a half-edge of $\Graph$. For $k\geq3$, we define $\W_{-k}\pi_1(\rGraph,b)$ to be the subgroup spanned by $[\W_{-i},\W_{-j}]$ for positive integers $i,j$ with $i+j=k$.

We set $\M_0\pi_1(\rGraph,b)=\pi_1(\rGraph,b)$, and define $\M_{-1}\pi_1(\rGraph,b)$ to be the subgroup spanned by conjugates of the elements $\beta_{v,i}$, $\beta_{v,i}'$ and $\delta_e$ for $v$ a vertex, $1\leq i\leq g(v)$ an integer and $e$ an edge or half-edge. We define $\M_{-2}\pi_1(\rGraph,b)$ to be the subgroup spanned by $[\M_{-1},\M_{-1}]$ and conjugates of the elements $\delta_e$ for $e$ an edge or half-edge of $\Graph$. For $k\geq3$, we define $\M_{-k}\pi_1(\rGraph,b)$ to be the subgroup spanned by $[\M_{-i},\M_{-j}]$ for positive integers $i,j$ with $i+j=k$.

Throughout the above, we adopt the notation from Definition~\ref{def:pi1_reduction_graph}. By a conjugate of an element of $\pi_1(\rGraph,v)$, we mean a conjugate by an element of $\pi_1(\rGraph;b,v)$.
\end{definition}

\begin{remark}\label{rmk:interpretation_of_quotients}
The filtrations defined above are independent of $b$ in the sense that the isomorphisms $\pi_1(\rGraph,b)\longisoarrow\pi_1(\rGraph,b')$ induced by conjugation by any element of $\pi_1(\rGraph;b,b')$ are $\W$- and $\M$-filtered isomorphisms. In particular, for any $k$ one can make sense of the quotient groupoids $\pi_1(\rGraph)/\W_{-k}$ and $\pi_1(\rGraph)/\M_{-k}$. The first few quotients can be described explicitly:
\begin{itemize}
	\item $\pi_1(\rGraph)/\M_{-1}=\pi_1(\Graph)$ is the fundamental groupoid of the underlying graph of~$\rGraph$;
	\item $\pi_1(\rGraph)/\W_{-2}=\pi_1(\rGraph')^\ab$, where $\rGraph'$ denotes the reduction graph formed by eliminating all half-edges from $\rGraph$.
\end{itemize}
\end{remark}

\begin{lemma}\label{lem:pic-lef_filtered}
The isomorphism $\vartheta\colon\pi_1^{\Q_\ell}(X_{\overline K},b)\longisoarrow\pi_1^{\Q_\ell}(\rGraph,\red(b))$ from Theorem~\ref{thm:graph_comparison} is a $\W$- and $\M$-filtered isomorphism, for the filtrations defined in \S\ref{ss:gal_coh_field}, Definition~\ref{def:wt1} and Definition~\ref{def:filtrations_on_graph}.
\end{lemma}
\begin{proof}
It follows from the proof of Theorem~\ref{thm:graph_comparison} that the isomorphism $\vartheta$ fits into a commuting square
\begin{center}
\begin{tikzcd}
\pi_1^{\Q_\ell}(X_{\overline K},b) \arrow[two heads]{r}\arrow{d}{\vartheta_*}[swap]{\viso} & \pi_1^{\Q_\ell}(\overline X_{\overline K},b) \arrow{d}{\vartheta_*}[swap]{\viso} \\
\pi_1^{\Q_\ell}(\rGraph,\red(b)) \arrow[two heads]{r} & \pi_1^{\Q_\ell}(\rGraph',\red(b))
\end{tikzcd}
\end{center}
where $\overline X$ is the smooth compactification of $X$ and $\rGraph'$ its reduction graph. Now $\W_{-2}\pi_1^{\Q_\ell}(X_{\overline K},b)$ is the kernel of the map $\pi_1^{\Q_\ell}(X_{\overline K},b)\twoheadrightarrow\pi_1^{\Q_\ell}(\overline X_{\overline K},b)^\ab$; since the same is true of $\W_{-2}\pi_1^{\Q_\ell}(\rGraph,\red(b))$ (see Remark~\ref{rmk:interpretation_of_quotients}), it follows that $\vartheta_*(\W_{-2})=\W_{-2}$. It then follows from the definition that $\vartheta_*(\W_{-k})=\W_{-k}$ for all $k$.

Now the $\M$-filtration on $\Lie(\pi_1^{\Q_\ell}(X_{\overline K},b))$ is uniquely characterised, in terms of the $\W$-filtration and monodromy operator $N=\log(\sigma)$, by the fact that it satisfies the weight--monodromy condition~\WM \cite[Proposition~1.6.13]{weilii}. We will see in Corollary~\ref{cor:graph_w-m} that the $\M$-filtration on $\Lie(\pi_1^{\Q_\ell}(\rGraph,\red(b)))$ satisfies this same condition; it follows that the $N$-equivariant and $\W$-filtered isomorphism $\vartheta_*\colon\Lie(\pi_1^{\Q_\ell}(X_{\overline K},b))\longisoarrow\Lie(\pi_1^{\Q_\ell}(\rGraph,\red(b)))$ is automatically $\M$-filtered, as desired.
\end{proof}

\begin{remark}
If $X/K$ is a semistable curve, then $\gr^\M_0\pi_1^{\Q_\ell}(X_{\overline K},b)$ is the $\Q_\ell$-Mal\u cev completion of a discrete group $\pi_1(\Graph,\red(b))$. This discrete group has a geometric interpretation as the rigid fundamental group in the sense of \cite[\S5.7]{FvdP}, i.e.\ the group of deck transformations of a rigid analytic universal cover.
\end{remark}

\subsection{Change of model}\label{ss:pic-lef_plus}

As stated, our non-abelian Picard--Lefschetz Theorem only describes the non-abelian Kummer map on $\O_K$-integral points of a single model of the curve $X/K$. In order to describe the non-abelian Kummer map on $K$-rational points, we will want to vary the model $\mathcal X=\overline{\mathcal X}\setminus\mathcal D$ in Theorem~\ref{thm:graph_comparison}, blowing up points in $\mathcal D_k$. In parallel, we will want to apply Theorem~\ref{thm:graph_comparison} simultaneously over finite extensions of $K$, by base-changing the model and performing an appropriate blow-up. Both of these constructions change the reduction graph $\rGraph$ of the model in the following very mild way.

\begin{definition}[Edge subdivision and rational vertices]\label{def:rational_vertices}
Let $\Graph$ be a rationally metrised graph, $e$ an edge of $\Graph$, and $0<s<\length(e)$ a rational number. The \emph{rational subdivision of $\Graph$ at $v_{e,s}$} is the rationally metrised graph $\Graph_{\{v_{e,s}\}}$ with vertex-set $\Vert{\Graph}\sqcup\{v_{e,s}\}$, edge-set $\Edge{\Graph}\sqcup\{e_1^{\pm1},e_2^{\pm1}\}\setminus\{e^{\pm1}\}$ and half-edge set $\HEdge{\Graph}$, where $e_1$ has source $\source(e)$, target $v_{e,s}$ and length $s$, and $e_2$ has source $v_{e,s}$, target $\target(e)$ and length $\length(e)-s$. Similarly, if $e$ is a half-edge of $\Graph$ and $0<s$ is a rational number, we define the \emph{rational subdivision of $\Graph$ at $v_{e,s}$} to be the rationally metrised graph $\Graph_{\{v_{e,s}\}}$ with vertex-set $\Vert{\Graph}\sqcup\{v_{e,s}\}$, edge-set $\Edge{\Graph}\sqcup\{e_1^{\pm1}\}$ and half-edge set $\HEdge{\Graph}\sqcup\{e_2\}\setminus\{e\}$, where $e_1$ has source $\source(e)$, target $v_{e,s}$ and length $s$, and $e_2$ has source $v_{e,s}$.

In general, a \emph{rational subdivision} of $\Graph$ is a rationally metrised graph $\Graph'$ obtained from $\Graph$ by a finite sequence of such operations. The vertices of rational subdivisions of $\Graph$ can be identified with finite sets $\Vert{\Graph}\subseteq\Vert{\Graph'}\subseteq\QVert{\Graph}$, where $\QVert{\Graph}$ denotes the set of \emph{rational vertices}\index{rational vertex $v_{e,s}$, $\QVert{\Graph}$}, i.e.\ points of the underlying metric space of $\Graph$ lying a rational distance from vertices.

\smallskip

If $\rGraph$ is a reduction graph, then we make any subdivision $\Graph'$ of $\Graph$ into a reduction graph $\rGraph'$ by declaring that $g(v)=0$ for all $v\in\Vert{\Graph'}\setminus\Vert{\Graph}$. The inclusion $\Vert{\Graph}\hookrightarrow\Vert{\Graph'}$ extends to an injective equivalence
\[
\pi_1(\rGraph) \hookrightarrow \pi_1(\rGraph')
\]
by making the obvious definitions. This equivalence is equivariant for the edge-twisting automorphism from Definition~\ref{def:pi1_reduction_graph}. We will sometimes also write $\pi_1^\Q(\rGraph)$ for the $\Z$-equivariant groupoid $\varinjlim\pi_1^\Q(\rGraph')$, where the directed colimit is taken over rational subdivisions -- in effect, we enlarge the object-set of $\pi_1^\Q(\rGraph)$ from $\Vert{\Graph}$ to~$\QVert{\Graph}$.
\end{definition}

Suppose now that $\mathcal X/\O_K$ is a regular semistable model of a curve $X/K$. Blowing up points in $\mathcal D_k$ changes the reduction graph $\rGraph$ of $\mathcal X$ by subdividing half-edges, while base-changing from $\O_K$ to $\O_L$ and performing appropriate blow-ups changes the reduction graph by subdividing each edge into $e(L/K)$ equal segments (where we follow the conventions on the metric from Definition~\ref{def:reduction_graphs_of_curves}). These operations are compatible with the identification from the Picard--Lefschetz Theorem~\ref{thm:graph_comparison}.

\begin{lemma}\label{lem:graph_comparison_plus}
Let $\vartheta\colon\pi_1^{\Q_\ell}(X_{\overline K},b) \longisoarrow \pi_1^{\Q_\ell}(\rGraph,\red(b))$ be the isomorphism from Theorem~\ref{thm:graph_comparison}. Then the square
\begin{equation}\label{eq:pic-lef_square_plus}
\begin{tikzcd}
X(\overline K) \arrow{r}{\classify}\arrow{d}{\red} & \varinjlim\H^1(I_L,\pi_1^{\Q_\ell}(X_{\overline K},b)) \arrow{d}{\vartheta_*} \\
\QVert{\Graph} \arrow{r}{\classify} & \varinjlim\H^1(e\Z,\pi_1^{\Q_\ell}(\rGraph,\red(b)))
\end{tikzcd}
\end{equation}
also commutes, where the direct limits\footnote{In fact, these direct limits are not particularly complicated objects: all the restriction maps involved are isomorphisms.} are taken over finite extensions $L/K$ contained in $\overline K$ and over $e\in\N$ respectively. Here the right-hand vertical map is induced by restriction along the maps $e\Z\rightarrow I_L$ sending $e$ to $\sigma^e$ with $e=e(L/K)$, followed by pushout along $\vartheta$.
\end{lemma}
\begin{proof}[Proof (sketch)]
It suffices to show that for any finite Galois extension $L/K$ contained in $\overline K$ and any regular semistable model $\mathcal X'/\O_L$ of $X_L$ dominating $\mathcal X_{\O_L}$, the square
\begin{equation}\label{eq:pic-lef_square_prime}
\begin{tikzcd}
\mathcal X'(\O_L) \arrow{r}{\classify}\arrow{d}{\red} & \H^1(I_L,\pi_1^{\Q_\ell}(X_{\overline K},b)) \arrow{d}{\vartheta_*} \\
\Vert{\Graph'} \arrow{r}{\classify} & \H^1(e\Z,\pi_1^{\Q_\ell}(\rGraph',\red(b)))
\end{tikzcd}
\end{equation}
commutes, where $e=e(L/K)$ and $\rGraph'$ denotes the reduction graph of $\mathcal X'$ endowed with the metric whereby each edge has length $\frac1e$. There certainly is a $\sigma^e$-equivariant isomorphism
\[
\vartheta'\colon \pi_1^{\Q_\ell}(X_{\overline K},b) \longisoarrow \pi_1^{\Q_\ell}(\rGraph',\red(b))
\]
making~\eqref{eq:pic-lef_square_prime} commute, namely the isomorphism provided by Theorem~\ref{thm:graph_comparison} for the model $\mathcal X'$, so it suffices to prove that $\vartheta'=\vartheta$. More precisely, Theorem~\ref{thm:graph_comparison} provides a method for constructing such an isomorphism depending on a number of choices (see Remark~\ref{rmk:too_many_choices}), and it suffices to show that having fixed a sequence of choices defining $\vartheta$, there is a sequence of choices defining $\vartheta'$ such that $\vartheta'=\vartheta$. In fact, since the induced action of $I_K$ on $\H^1(I_L,\pi_1^{\Q_\ell}(X_{\overline K},b))$ is trivial, it suffices to prove that these choices can be made in such a way that $\vartheta'$ is equal to $\vartheta$ up to the action of $I_K$.

Verifying that these choices can be made in a suitably compatible manner is mostly a routine exercise; in the interests of brevity, we will merely outline how one should make a single choice compatibly, namely the deformation of the model $\mathcal X'/\O_L$ over the ring $R':=\Z_p^\nr[t']^\sh_{t'}$ from Lemma~\ref{lem:deformation_exists}. To do this, choose a uniformiser $\varpi_L$ of $L$ and write $\varpi=u\varpi_L^e$ for $u\in\O_L^\times$. We may lift $u$ to a unit $\widetilde u\in R'^\times$ along the map $R'\twoheadrightarrow\O_{L^\nr}$ sending $t'$ to $\varpi_L$. We thus have a commuting diagram
\begin{center}
\begin{tikzcd}
\Spec(\O_{L^\nr}) \arrow{r}\arrow{d} & \Spec(R') \arrow{d} & \Spec(\C[t']^\sh_{t'}) \arrow{l}\arrow{d} \\
\Spec(\O_{K^\nr}) \arrow{r} & \Spec(R) & \Spec(\C[t]^\sh_t) \arrow{l},
\end{tikzcd}
\end{center}
where the central and right-hand vertical arrows are given by $t\mapsto\widetilde u t'^e$. We will show that for any deformation $\widetilde{\mathcal X}$ of $\mathcal X$ over $\Spec(R)$, there is a deformation $\widetilde{\mathcal X}'$ of $\mathcal X'$ over $\Spec(R')$ dominating $\widetilde{\mathcal X}_{R'}$.

To do this, we note that the model $\mathcal X'$ is produced from $\mathcal X_{\O_L}$ by blowing up each singular point on the special fibre $\lfloor e/2\rfloor$ times (e.g.\ by~\cite[Example~8.3.53 and Proposition~8.1.12(c)]{liu}), and then possibly blowing up points in the reduction of the divisor $\mathcal D_{\O_L}$. The singular locus of $\widetilde{\mathcal X}_{R'}$ consists of a finite number of sections over the $t'=0$ locus $\Spec(\Z_p^\nr)$, with the formal neighbourhood of any such section being isomorphic to $\Spf\left(\Z_p^\nr[\![x,y,t']\!]/(xy-t'^e)\right)$. Blowing up along each such section $\lfloor e/2\rfloor$ times, and then possibly blowing up sections of $\widetilde{\mathcal D}_{\Z_p^\nr}$, we arrive at a semistable curve $\widetilde{\mathcal X}'$ over $R'$ which dominates $\mathcal X_{R'}$ and satisfies the conditions of Lemma~\ref{lem:deformation_exists}.
\end{proof}

\subsection{Application to the Chabauty--Kim method}

As a result of Theorem~\ref{thm:graph_comparison}, we are now able to prove the inequality from Proposition~\ref{prop:kummer_bounds}\eqref{proppart:size_bound} (without the equality condition), which is used in \cite{BDeff}. This amounts to showing that the fibres of the non-abelian Kummer map from \S\ref{ss:nonab_Kummer} are unions of fibres of the reduction map from Definition~\ref{def:reduction_map}.

\begin{proposition}
If $x,y\in X(K)$ satisfy $\red(x)=\red(y)$, then they also satisfy $\classify(x)=\classify(y)$ in $\H^1(G_K,\pi_1^{\Q_\ell}(X_{\overline K},b))$.
\begin{proof}
Replacing $K$ with a finite extension if necessary and using Lemma~\ref{lem:finite_index}, we may assume that $X$ has semistable reduction. By Theorem~\ref{thm:graph_comparison}, there is a $\sigma$-invariant element of $\pi_1^{\Q_\ell}(X_{\overline K};x,y)$. Since the action of $I_K$ on $\pi_1^{\Q_\ell}(X_{\overline K};x,y)$ factors through a continuous action of $\Z_\ell$ with generator $\sigma$, this path is even $I_K$-invariant. By Corollary~\ref{cor:G-invt_is_I-invt}, there is even a $G_K$-invariant element of $\pi_1^{\Q_\ell}(X_{\overline K};x,y)$, which ensures that $\classify(x)=\classify(y)$.
\end{proof}
\end{proposition}

\subsection{The Hopf groupoid of a reduction graph}

In subsequent sections, we will perform detailed computations with the $\Q$-pro-unipotent fundamental groupoid $\pi_1^\Q(\rGraph)$, with its $\W$- and $\M$-filtrations from \S\ref{ss:pic-lef_filtrations} and its edge-twist automorphism $\sigma$ from Definition~\ref{def:pi1_reduction_graph}. In order to facilitate these calculations, it will be useful to replace pro-unipotent groupoids with a linear algebraic analogue, similar to replacing a group $G$ with its group algebra $\Q G$. Thus, our calculations will mainly take place in the \emph{complete Hopf groupoid} $\O(\pi_1^\Q(\rGraph))^\dual$, i.e.\ the structure comprised of the complete cocommutative coalgebras $\O(\pi_1^\Q(\rGraph;u,v))^\dual=\Hom(\O(\pi_1^\Q(\rGraph;u,v),\Q)$ for (rational) vertices $u,v$, together with the natural path-composition maps
\[
\O(\pi_1^\Q(\rGraph;v,w))^\dual\hatotimes\O(\pi_1^\Q(\rGraph;u,v))^\dual\rightarrow\O(\pi_1^\Q(\rGraph;u,w))^\dual
\]
and path-reversal maps
\[
\O(\pi_1^\Q(\rGraph;u,ve^{-1}))^\dual \rightarrow \O(\pi_1^\Q(\rGraph;v,u))^\dual.
\]
Thus, for a vertex $u$, $\O(\pi_1^\Q(\rGraph,u))^\dual$ is the complete cocommutative Hopf algebra dual to the commutative Hopf algebra $\O(\pi_1^\Q(\rGraph,u))$. (For more background on complete Hopf groupoids, see Appendix~\ref{appx:hopf_gpds}.)

The filtrations $\W_\bullet$ and $\M_\bullet$ on $\pi_1^\Q(\rGraph)$ induce corresponding filtrations on each complete coalgebra $\O(\pi_1^\Q(\rGraph;u,v))^\dual$. The edge-twist automorphism $\sigma$ acts pro-unipotently on each $\O(\pi_1^\Q(\rGraph;u,v))^\dual$, and hence has a well-defined logarithm $N=\log(\sigma)$, which is a $\W$- and $\M$-filtered endomorphism of degree $0$ and $-2$ respectively. All of these structures are compatible with the coalgebra structure on $\O(\pi_1^\Q(\rGraph;u,v))^\dual$, as well as the path-composition and path-reversal maps (e.g.\ $N$ is a coderivation and satisfies the Leibniz rule with respect to composition). In particular, for each $k\geq0$, the collection of (pro-finite dimensional) subspaces $\W_{-k}\O(\pi_1^\Q(\rGraph;u,v))^\dual\leq\O(\pi_1^\Q(\rGraph;u,v))^\dual$ forms an \emph{ideal} of $\O(\pi_1^\Q(\rGraph))^\dual$, i.e.\ is closed under path-composition with arbitrary elements of $\O(\pi_1^\Q(\rGraph))^\dual$.

For the benefit of the reader, we will recall here the explicit definition of these structures, setting up the notation for later sections.

\begin{proposition}\label{prop:explicit_hopf_gpd}
Let $\rGraph$ be a reduction graph. Its pro-nilpotent fundamental Hopf groupoid $\O(\pi_1^\Q(\rGraph))^\dual$ is the pro-nilpotent Hopf groupoid on object-set $\Vert{\Graph}$ (or $\QVert{\Graph}$, cf.\ Definition~\ref{def:rational_vertices}) generated by:
\begin{itemize}
	\item grouplike elements $e\in\O(\pi_1^\Q(\rGraph;\source(e),\target(e)))^\dual$ for every edge $e$;
	\item primitive elements $\log(\beta_{v,i}),\log(\beta_{v,i}')\in\O(\pi_1^\Q(\rGraph,v))^\dual$ for every vertex $v$ and integer $1\leq i\leq g(v)$; and
	\item primitive elements $\log(\delta_e)\in\O(\pi_1^\Q(\rGraph,\source(e)))^\dual$ for every edge or half-edge $e$;
\end{itemize}
subject to the following relations:
\begin{itemize}
	\item (edge relations) $ee^{-1}=1$ and $\log(\delta_{e^{-1}})\cdot e+e\cdot\log(\delta_e)=0$ for every edge $e$; and
	\item (vertex relations) $\prod_{i=1}^{g(v)}[\beta_{v,i}',\beta_{v,i}]\cdot\prod_{\source(e)=v}\delta_e=1$ for every vertex $v$, where $\delta_e$ denotes the exponential of $\log(\delta_e)$.
\end{itemize}
The coalgebra structure on each $\O(\pi_1^\Q(\rGraph;u,v))^\dual$ is uniquely determined by specifying that each element $e$ is grouplike, and the other generators are primitive.

\smallskip

The pro-nilpotent Hopf derivation $N$ is uniquely specified by the fact that it sends all generators $\log(\beta_{v,i})$, $\log(\beta_{v,i}')$ and $\log(\delta_e)$ to $0$, and acts on edges $e$ via
\[
N(e) = \length(e)\cdot e\log(\delta_e).
\]

\smallskip

The $\W$-filtration on $\O(\pi_1^\Q(\rGraph))^\dual$ is defined as follows. We have $\W_0\O(\pi_1^\Q(\rGraph))^\dual=\O(\pi_1^\Q(\rGraph))^\dual$, and $\W_{-1}\O(\pi_1^\Q(\rGraph))^\dual$ is the augmentation ideal, i.e.\ the kernel of the evident map $\O(\pi_1^\Q(\rGraph))^\dual\rightarrow\Q$. The ideal $\W_{-2}\O(\pi_1^\Q(\rGraph))^\dual$ is generated by $\W_{-1}^2$ (i.e.\ by pairwise composites of elements of $\W_{-1}$) and the elements $\log(\delta_e)$ for half-edges $e$. For $k\geq3$, $\W_{-k}\O(\pi_1^\Q(\rGraph))^\dual$ is the ideal generated by $\W_{-i}\W_{-j}$ for positive integers $i,j$ with $i+j=k$.

The $\M$-filtration on $\O(\pi_1^\Q(\rGraph))^\dual$ is defined as follows. We have $\M_0\O(\pi_1^\Q(\rGraph))^\dual=\O(\pi_1^\Q(\rGraph))^\dual$, and $\M_{-1}\O(\pi_1^\Q(\rGraph))^\dual$ is the ideal generated by the elements $\log(\beta_{v,i})$, $\log(\beta_{v,i}')$ and $\log(\delta_e)$ for vertices $v$, integers $1\leq i\leq g(v)$ and edges or half-edges $e$. The ideal $\M_{-2}\O(\pi_1^\Q(\rGraph))^\dual$ is generated by $\M_{-1}^2$ and the elements $\log(\delta_e)$ for edges or half-edges $e$. For $k\geq3$, $\M_{-k}\O(\pi_1^\Q(\rGraph))^\dual$ is the ideal generated by $\M_{-i}\M_{-j}$ for positive integers $i,j$ with $i+j=k$.
\end{proposition}

\begin{remark}\label{rmk:easier_vertex_relation}
Our calculations will predominantly take place in the associated $\M$-graded $\gr^\M_\bullet\O(\pi_1^\Q(\rGraph))^\dual$. The vertex relation in Proposition~\ref{prop:explicit_hopf_gpd} implies that the simpler relation
\[
\sum_{i=1}^{g(v)}[\log(\beta_{v,i}'),\log(\beta_{v,i})] + \sum_{\source(e)=v}\log(\delta_e) = 0
\]
holds in $\gr^\M_{-2}\O(\pi_1^\Q(\rGraph,v))^\dual$ for every vertex $v$, where $[x,y]=xy-yx$ denotes the algebra commutator. We will often adopt the shorthand notation that $\log(\delta_v):=\sum_{i=1}^{g(v)}[\log(\beta_{v,i}'),\log(\beta_{v,i})]$.
\end{remark}
\section{Cheng--Katz integration on graphs}
\label{s:cheng-katz}

In this section, we introduce what will be one of the main combinatorial tools utilised in this paper: canonical choices of $\Q$-pro-unipotent paths between any two rational vertices of a rationally metrised graph $\Graph$. We will see in Proposition~\ref{prop:canonical_path_is_canonical_path} that these \emph{canonical paths} correspond to the canonical elements of torsors from Remark~\ref{rmk:cohomology_grading}, and so are closely tied to non-abelian Kummer maps. On the other hand, the explicit combinatorial description of these canonical paths renders them inherently computable, and will ultimately lead to our computation of the non-abelian Kummer map in \S\ref{s:computation}.

\smallskip

These canonical paths are defined via the \emph{higher cycle pairing} of Cheng and Katz \cite{cheng2017higher}, whose definition we now briefly recall. Fix a rationally metrised graph\footnote{According to our definition, graphs are permitted to have half-edges -- these will make no difference during this section.} $\Graph$. The $\Q$-rational homology $\H_1(\Graph):=\H_1(\Graph,\Q)$ carries a positive-definite \emph{cycle pairing}\footnote{The symbol $\int$ is used for several different notions in this paper: here for the (higher) cycle pairing; in Definition \ref{def:height_pairing} for the integration of functions over measures; and in Proposition \ref{prop:loop_integral_along_edge} for a genuine line integral.} or \emph{intersection-length pairing}
\begin{align*}
\int\colon\H_1(\Graph)\otimes\H_1(\Graph)&\rightarrow\Q \\
\gamma\otimes\omega &\mapsto\int_\gamma\omega,
\end{align*}
namely the restriction of the positive-definite pairing\[\int\colon\Q\cdot\mnEdge{\Graph}\otimes\Q\cdot\mnEdge{\Graph}\rightarrow\Q\]defined by\[\int_ee':=\begin{cases}\pm\length(e)&\text{if $e'=e^{\pm1}$,}\\0&\text{else.}\end{cases}\](Here $\Q\cdot\mnEdge{\Graph}$\index{graph!edges $\Q\cdot\mnEdge{\Graph}$} denotes the $\Q$-vector space generated by edges of $\Graph$ modulo the relation $e^{-1}=-e$ for every edge $e$ of $\Graph$.) We will write\[N\colon\H^1(\Graph)\isoarrow\H_1(\Graph)\]for the isomorphism induced by the cycle pairing.

In \cite[Segment 3.3]{cheng2017higher}, the cycle pairing is extended to a \emph{higher cycle pairing}, which takes the form of $\Q$-linear pairings
\begin{align*}
\int\colon\Q\pi_1(\Graph;u,v)\otimes\Shuf\H_1(\Graph) &\rightarrow\Q \\
\gamma\otimes\underline\omega\mapsto\int_\gamma\underline\omega
\end{align*}
for each pair of vertices $u,v\in\Vert{\Graph}$, where $\Q\pi_1(\Graph;u,v)$ denotes the free $\Q$-vector space on the set $\pi_1(\Graph;u,v)$ of homotopy classes of paths from $u$ to $v$ in $\Graph$ and $\Shuf\H_1(\Graph):=\bigoplus_{n\geq0}\H_1(\Graph)^{\otimes n}$ denotes the shuffle algebra on $\H_1(\Graph)$. When $\underline\omega=\omega_1\dots\omega_n\in\Shuf\H_1(\Graph)$ is a fixed pure tensor of length $n$ and $\gamma$ is a path in $\Graph$, the higher cycle pairing $\int_\gamma\underline\omega$ is defined recursively by specifying that\[\int_e\underline\omega:=\frac1{n!}\prod_{i=1}^n\int_e\omega_i\]for all edges $e$, and\[\int_{e\gamma}\underline\omega=\sum_{i=0}^n\int_e\omega_{i+1}\dots\omega_n\cdot\int_\gamma\omega_1\dots\omega_i\]for every path $\gamma$ and edge $e$ with $\source(e)=\target(\gamma)$. The higher cycle pairing between general elements of $\Shuf\H_1(\Graph)$ and $\Q\pi_1(\Graph;u,v)$ is then produced by extending bilinearly.

\begin{remark}
Subdividing edges and half-edges of a rationally metrised graph $\Graph$ as in Definition \ref{def:rational_vertices} doesn't affect the higher cycle pairings $\int\colon\Q\pi_1(\Graph;u,v)\otimes\Shuf\H_1(\Graph)\rightarrow\Q$, and hence these pairings are naturally extended to the case when $u$ and $v$ are rational vertices of $\Graph$. When proving results for this more general case, we will often tacitly make an appropriate subdivision of $\Graph$ so that all the rational vertices appearing are in fact vertices.
\end{remark}

The higher cycle pairings enjoy a number of useful properties \cite[Proposition 2.4 \& Segment 2.8]{cheng2017higher}.

\smallskip

\noindent\textbf{(Composition.)} Given composable paths $\gamma',\gamma$ in $\Graph$, we have\[\int_{\gamma'\gamma}\omega_1\dots\omega_n=\sum_{i=1}^n\int_{\gamma'}\omega_{i+1}\dots\omega_n\cdot\int_\gamma\omega_1\dots\omega_i.\]

\smallskip

\noindent\textbf{(Identity.)} If $1$ denotes the constant path at any rational vertex of $\Graph$, then\[\int_1\omega_1\dots\omega_n=\begin{cases}1&\text{if $n=0$,}\\0&\text{else.}\end{cases}\]

\smallskip

\noindent\textbf{(Comultiplication.)} For any path $\gamma$ we have\[\int_\gamma\underline\omega_1\shuf\underline\omega_2=\int_\gamma\underline\omega_1\cdot\int_\gamma\underline\omega_2,\]where $\shuf$ denotes the shuffle product on $\Shuf\H_1(\Graph)$.

\smallskip

\noindent\textbf{(Counit.)} For any path $\gamma$ we have\[\int_\gamma1=1.\]

\smallskip

\noindent\textbf{(Antipode.)} For any path $\gamma$ we have\[\int_{\gamma^{-1}}\omega_1\dots\omega_n=(-1)^n\int_\gamma\omega_1\dots\omega_n.\]

\smallskip

\noindent\textbf{(Nilpotence.)} Suppose that we have a sequence $u_0,\dots,u_r$ of vertices of $\Graph$ and elements $\gamma_i\in\Q\pi_1(\Graph;u_{i-1},u_i)$ for $1\leq i\leq r$ lying in the kernel of the augmentation map $\Q\pi_1(\Graph;u_{i-1},u_i)\rightarrow\Q$ (for instance, each $\gamma_i$ could be the difference of parallel paths). Then we have\[\int_{\gamma_r\dots\gamma_1}\omega_1\dots\omega_n=\begin{cases}\frac1{n!}\prod_{i=1}\int_{\gamma_i}\omega_i&\text{if $n=r$,}\\0&\text{if $n<r$.}\end{cases}\]

\smallskip

The nilpotence property above allows one to make precise the assertion that the higher cycle class pairing is perfect. For us, the importance of this result is that the spaces $\Q\pi_1(\Graph;u,v)$ are canonically independent of $u$ and $v$ (up to a certain completion operation).

\begin{theorem}[Duality Theorem]\label{thm:c-k_duality}\index{trivialisation $\Hiso$}
Let $u,v$ be two vertices of $\Graph$ and let $\W_\bullet$ denote the filtration of the vector space $\Q\pi_1(\Graph;u,v)$ defined by
\[
\W_{-k}\Q\pi_1(\Graph;u,v) = \Q\pi_1(\Graph;u,v)\cdot J^k
\]
where $J^k$ denotes the $k$th power of the augmentation ideal of $\Q\pi_1(\Graph,u)$. Then the higher cycle class pairing induces a perfect pairing\[\int\colon\Q\pi_1(\Graph;u,v)/\W_{-n-1}\otimes\Shufleq n\H_1(\Graph)\rightarrow\Q\]for all $n$. The corresponding isomorphisms $\Q\pi_1(\Graph;u,v)/\W_{-n-1}\isoarrow\CTensorleq n\H^1(\Graph)$ induce an isomorphism\[\Hiso\colon\Qhat\pi_1(\Graph;u,v):=\liminv\Q\pi_1(\Graph;u,v)/\W_{-n-1}\isoarrow\CTensor\H^1(\Graph)\]where $\CTensor\H^1(\Graph):=\prod_{n\geq0}\H^1(\Graph)^{\otimes n}$ denotes the complete tensor algebra on $\H^1(\Graph)$.
\end{theorem}

\begin{remark}\label{rmk:c-k_duality_posh}
It follows from the properties of the higher cycle pairing that the isomorphisms $\Hiso\colon\Qhat\pi_1(\Graph;u,v)\isoarrow\CTensor\H^1(\Graph)$ are compatible with all the natural structures on the sets $\Qhat\pi_1(\Graph;u,v)$. To make this precise, we observe that the pro-finite dimensional vector spaces $\Qhat\pi_1(\Graph;u,v)$ naturally assemble into a \emph{complete Hopf groupoid} $\Qhat\pi_1(\Graph)$ on vertex-set $\Vert{\Graph}$, whose composition and antipode are the maps
\begin{align*}
\Qhat\pi_1(\Graph;v,w)\hatotimes\Qhat\pi_1(\Graph;u,v) &\rightarrow \Qhat\pi_1(\Graph;u,w) \\
\Qhat\pi_1(\Graph;u,v) &\rightarrow \Qhat\pi_1(\Graph;v,u)
\end{align*}
induced by path-composition and path-reversal, and whose completed comultiplication 
\begin{align*}
\Qhat\pi_1(\Graph;u,v) &\rightarrow \Qhat\pi_1(\Graph;u,v)\hatotimes\Qhat\pi_1(\Graph;u,v)
\end{align*}
is the map induced by the diagonal. The affine groupoid-scheme corresponding to $\Qhat\pi_1(\Graph)$ is the $\Q$-Mal\u cev completion $\pi_1^\Q(\Graph)$ of the fundamental groupoid $\pi_1(\Graph)$ of $\Graph$.

The compatibility of the isomorphisms $\Hiso$ can be expressed as saying that they constitute an isomorphism
\[
\Qhat\pi_1(\Graph) \isoarrow \CTensor\H^1(\Graph)
\]
of complete Hopf groupoids, where the complete Hopf algebra $\CTensor\H^1(\Graph)$ is regarded as a constant complete Hopf groupoid in the usual way.

An outline of the basic properties of complete Hopf groupoids, which will be used throughout this paper, can be found in Appendix \ref{appx:hopf_gpds}.
\end{remark}

\begin{remark}
Under the isomorphism $\Hiso$ from Theorem \ref{thm:c-k_duality}, the $\W$-filtration on $\Qhat\pi_1(\Graph;u,v)$ corresponds to the natural filtration on $\CTensor\H^1(\Graph)$ by tensor-length. In particular, the $\W$-filtration on $\Qhat\pi_1(\Graph;u,v)$ has a canonical splitting as a product $\prod_{k>0}\gr^\W_{-k}\Qhat\pi_1(\Graph;u,v)$.
\end{remark}

\begin{proof}[Proof of Theorem \ref{thm:c-k_duality}]
\cite[Theorem 2.13]{cheng2017higher} proves this assertion with $\R$ coefficients in the special case that $u=v$ and all edges of $\Graph$ have length $1$. One can deduce Theorem \ref{thm:c-k_duality} from this case, or can mimic the proof in our setting, observing that the nilpotence property ensures that the higher cycle pairing induces a map
\[
\Q\pi_1(\Graph;u,v)/\W_{-n-1} \rightarrow \CTensorleq n\H^1(\Graph)
\]
which becomes an isomorphism when passing to graded pieces on either side. This ensures that the map was already an isomorphism, and hence that the pairing was perfect.
\end{proof}

As a consequence of the Duality Theorem, we find the promised preferred choice of $\Q$-pro-unipotent paths in the graph $\Graph$.

\begin{definition}[Canonical paths]\label{def:can_paths}\index{canonical\dots!\dots path $\gamma_{x,y}^\can$}
Let $u,v$ be two rational vertices in the rationally metrised graph $\Graph$. We define the \emph{canonical path} $\gamma_{u,v}^\can\in\Qhat\pi_1(\Graph;u,v)$ from $u$ to $v$ to be the preimage of $1\in\CTensor\H^1(\Graph)$ under the duality isomorphism\[\Hiso\colon\Qhat\pi_1(\Graph;u,v)\isoarrow\CTensor\H^1(\Graph)\]from Theorem \ref{thm:c-k_duality}. In other words, $\gamma_{u,v}^\can$ is the unique element such that\[\int_{\gamma_{u,v}^\can}\omega_1\dots\omega_n=\begin{cases}1&\text{if $n=0$,}\\0&\text{if $n>0$,}\end{cases}\]for all $\omega_1,\dots,\omega_n\in\H_1(\Graph)$.
\end{definition}

\begin{example}\label{ex:bridges_are_canonical}
If $e$ is a bridge in the rationally metrised graph $\Graph$ (an edge whose removal disconnects $\Graph$), then $e$ is the canonical path from $\source(e)$ to $\target(e)$.

\end{example}

\begin{lemma}[Basic properties of canonical paths]\label{lem:basic_can_path}
For any two rational vertices $u,v$ in a rationally metrised graph $\Graph$, the canonical path $\gamma_{u,v}^\can\in\Qhat\pi_1(\Graph;u,v)$ is grouplike, so constitutes a $\Q$-point of $\pi_1^\Q(\Graph;u,v)$. Moreover:
\begin{itemize}
	\item for any rational vertex $u$, we have $\gamma_{u,u}^\can=1$ is the identity loop at $u$;
	\item for any rational vertices $u,v$, we have $\gamma_{u,v}^\can=(\gamma_{v,u}^\can)^{-1}$; and
	\item for any rational vertices $u,v,w$, we have $\gamma_{u,w}^\can=\gamma_{v,w}^\can\gamma_{u,v}^\can$.
\end{itemize}
\begin{proof}
These all follow immediately from the observation in Remark \ref{rmk:c-k_duality_posh} that the duality isomorphism is an isomorphism of complete Hopf groupoids.
\end{proof}
\end{lemma}

\begin{remark}
The canonical paths $\gamma_{u,v}^\can$ constitute a non-abelian generalisation of current flows on circuits. Specifically, the pushout of $\pi_1^\Q(\Graph;u,v)$ along the abelianisation map $\pi_1^\Q(\Graph,u)\twoheadrightarrow\pi_1^\Q(\Graph,u)^\ab$ can be identified as the set of $\Q$-linear unit flows on $\Graph$ from $u$ to $v$, i.e.\ the set of formal $\Q$-linear sums of edges $\gamma$ such that $\partial(\gamma)=v-u$ (modulo edge-reversal). If we view $\Graph$ as a circuit where each edge $e$ is a resistor of resistance $\length(e)$, then it follows from Kirchoff's Laws that the image of $\gamma_{u,v}^\can$ is the current flow induced by putting a potential difference of $R$ across $\Graph$, where $R$ is the total resistance between $u$ and $v$ in $\Graph$.
\end{remark}

We conclude this section with a few more properties of the duality isomorphism which will be necessary in later calculations. In order to express these properties, we will adopt the following notation.

\begin{notation}\label{notn:e*}
Let $e$ be an edge of a rationally metrised graph $\Graph$. We denote by $e^*$ the element of $\H^1(\Graph)=\Hom(\H_1(\Graph),\Q)$ sending a homology class $\gamma$ to the multiplicity of $e$ in $\gamma$. If $e$ is instead a half-edge of $\gamma$, we adopt the convention that $e^*=0$.
\end{notation}

\begin{remark}\label{rmk:e*e_is_identity}
The elements $e^*\in\H^1(\Graph)$ defined above satisfy $(e^{-1})^*=-e^*$ for every edge $e$ and $\sum_{\source(e)=v}e^*=0$ for every vertex $v$. In particular, we have\[\sum_{e\in\plEdge{\Graph}}e^*\otimes e\in\H^1(\Graph)\otimes\H_1(\Graph).\]Moreover, under the usual isomorphism $\H^1(\Graph)\otimes\H_1(\Graph)\iso\End(\H_1(\Graph))$, this element is easily seen to correspond to the identity map on $\H_1(\Graph)$.
\end{remark}

\begin{lemma}\label{lem:C-K_on_edges}
Let $e$ be an edge of a rationally metrised graph $\Graph$. Then the image of $e$ under the duality isomorphism $\Hiso\colon\Qhat\pi_1(\Graph;\source(e),\target(e))\isoarrow\CTensor\H^1(\Graph)$ is\[\Hiso(e)=\exp\left(\length(e)e^*\right):=\sum_{r\geq0}\frac1{r!}(\length(e)e^*)^r.\]
\begin{proof}
It follows from the definition of the higher cycle pairing that\[\int_e\omega_1\dots\omega_n=\frac{\length(e)^n}{n!}\prod_{i=1}^ne^*(\omega_i)=\exp(\length(e)e^*)(\omega_1\dots\omega_n)\]for all $\omega_1,\dots,\omega_n\in\H_1(\Graph)$, where we identify $\CTensor\H^1(\Graph)=\Hom(\Shuf\H_1(\Graph),\Q)$ as usual. This is what we wanted to prove.
\end{proof}
\end{lemma}

\begin{corollary}[Variation of canonical paths along edges]\label{cor:can_path_along_edge}
Let $e$ be an edge (resp.\ half-edge) of a rationally metrised graph $\Graph$, let $v_{e,s}$ denote the rational vertex a distance $s\in[0,\length(e)]\cap\Q$ (resp.\ $s\in[0,\infty)\cap\Q$) along $e$, and let $e_s$ denote the rational edge\footnote{By a \emph{rational edge}, we just mean an edge of some subdivision of $\Graph$, most often viewed as an element of the fundamental groupoid.}\index{rational edge, $e_s$} from $\source(e)$ to $v_{e,s}$ inside $e$.

Then for any rational vertex $b$, the canonical path from $b$ to $v_{e,s}$ is given by\[\gamma_{b,v_{e,s}}^\can=e_s\gamma_{b,\source(e)}^\can\Hiso_b^{-1}\left(\exp(-se^*)\right),\]where $\Hiso_b\colon\Qhat\pi_1(\Graph,b)\isoarrow\CTensor\H^1(\Graph)$ denotes the duality isomorphism at basepoint~$b$.
\begin{proof}
By Lemma \ref{lem:C-K_on_edges}, we know that $\Hiso(e_s)=\exp(se^*)$ for all $s$. Hence
\[\Hiso\left(e_s\gamma_{b,\source(e)}^\can\Hiso_b^{-1}\left(\exp(-se^*)\right)\right)=\exp(se^*)\cdot1\cdot\exp(-se^*)=1,\]so $e_s\gamma_{b,\source(e)}^\can\Hiso_b^{-1}\left(\exp(-se^*)\right)$ is the canonical path, as desired.
\end{proof}
\end{corollary}

\section{The $\M$-graded fundamental groupoid of a reduction graph}
\label{s:M-trivialisation}

We saw in the previous section that the $\Q$-pro-unipotent fundamental groupoid of a rationally metrised graph $\Graph$ admits a canonical and basepoint-independent description in the form of the isomorphism
\[
\Hiso\colon\O(\pi_1^\Q(\Graph))^\dual \isoarrow \CTensor\H^1(\Graph)
\]
of complete Hopf groupoids induced from Theorem \ref{thm:c-k_duality} (from here we write the completion $\Qhat\pi_1(\Graph)$ of the groupoid-algebra of $\pi_1(\Graph)$ as the dual affine ring $\O(\pi_1^\Q(\Graph))^\dual$ of the $\Q$-Mal\u cev completion, see Lemma~\ref{lem:J-adic_and_malcev_completion}). Our aim in this section is to carry out the analogous calculation for reduction graphs $\rGraph$, producing a canonical isomorphism
\[
\Hiso\colon\gr^\M_\bullet\O(\pi_1^\Q(\rGraph))^\dual \isoarrow \HAlg
\]
for an explicit $\M$-graded complete Hopf algebra $\HAlg$. As well as giving us a canonical setting for performing computations in $\pi_1^\Q(\rGraph)$, this complete Hopf algebra $\HAlg$ will be used in \S\ref{s:n-a_kummer_for_graphs} to define the pro-finite dimensional vector space $\V$ from Theorem~\ref{thm:description_of_kummer}.

\smallskip

To begin this section, let us give the definition of the complete Hopf algebra $\HAlg$. As part of the definition, we will endow $\HAlg$ with two gradings $\gr^\W_\bullet$ and $\gr^\M_\bullet$. The latter is supposed to correspond to the tautological $\M$-grading on $\gr^\M_\bullet\O(\pi_1^\Q(\rGraph))^\dual$, while the former is supposed to correspond to the $\W$-filtration, in that its underlying filtration recovers the $\W$-filtration.

\begin{construction}\label{cons:HAlg}\index{Hopf algebras:!$\HAlg$, $\oHAlg$}\index{relation!surface $\sigma$}
Let $\rGraph$ be a reduction graph. Let $\oHAlg=\oHAlg(\rGraph)$ be the $(\W,\M)$-bigraded complete tensor algebra generated by\footnote{When we refer to the tensor algebra generated by a collection of vector spaces and elements, we of course mean the tensor algebra on the direct sum of the given vector spaces with the vector space with basis the given elements.}:
\begin{itemize}
	\item the $\Q$-linear cohomology $\H^1(\Graph)$ of (the underlying graph of) $\rGraph$ in $(\W,\M)$-bidegree $(-1,0)$;
	\item formal symbols\footnote{Despite the long name, $\alog(\beta_{v,i})$ is intended to be read as a single formal symbol. Of course, the choice of symbol is deliberately evocative -- $\alog(\beta_{v,i})$ will indeed be the image of $\log(\beta_{v,i})\in\gr^\M_{-1}\O(\pi_1^\Q(\rGraph))^\dual$ under the map $\Hiso$.} $\alog(\beta_{v,i})$ and $\alog(\beta_{v,i}')$ in $(\W,\M)$-bidegree $(-1,-1)$ for each vertex $v\in\Vert{\Graph}$ and each integer $1\leq i\leq g(v)$;
	\item the $\Q$-linear homology\footnote{Of course, the homology and cohomology of $\Graph$ are canonically identified with one another via the cycle pairing, so our choice to identify $\gr^\M_0\gr^\W_{-1}\oHAlg$ with the cohomology and $\gr^\M_{-2}\gr^\W_{-1}\oHAlg$ with the homology is largely a mnemonic device to make a distinction between how we write elements of these two subspaces. Note that our conventions are somewhat non-standard, in that it would be more usual to interchange the roles of $\H_1$ and $\H^1$.} $\H_1(\Graph)$ of (the underlying graph of) $\rGraph$ in $(\W,\M)$-bidegree $(-1,-2)$; and
	\item a formal symbol $\alog(\delta_e)$ in $(\W,\M)$-bidegree $(-2,-2)$ for each half-edge $e\in\HEdge{\Graph}$.
\end{itemize}
The Hopf algebra structure is determined by specifying that each of these generators is primitive (i.e.\ the comultiplication is defined on these generators by $\Delta(x)=1\otimes x+x\otimes1$).

\smallskip

There is a primitive element $\sigma$, bihomogenous of $(\W,\M)$-bidegree $(-2,-2)$, called the \emph{surface relation} and defined to be equal to $\sigma_0+\sigma_1+\sigma_2$ where
\begin{align*}
\sigma_0 &= \sum_{i=1}^\gGamma [\xi_i',\xi_i], \\
\sigma_1 &= \sum_v\sum_{i=1}^{g(x)}[\alog(\beta_{v,i}'),\alog(\beta_{v,i})], \\
\sigma_2 &= \sum_{e\in\HEdge{\Graph}}\alog(\delta_e),
\end{align*}
where $(\xi_i)_{i=1}^\gGamma$ is a basis of $\H^1(\Graph)$ with dual basis $(\xi_i')$ ($\sigma_0$ is independent of this choice). We then define the complete bigraded Hopf algebra $\HAlg=\HAlg(\rGraph)$ to be the quotient\[\HAlg:=\oHAlg/(\sigma).\]As usual, we denote the first and second gradings on $\oHAlg$ and $\HAlg$ by $\gr^\W_\bullet$ and $\gr^\M_\bullet$ respectively, and their associated filtrations by $\W_\bullet$ and $\M_\bullet$ respectively.
\end{construction}

The complete Hopf algebra $\HAlg$ will provide the desired trivialisation of the $\M$-graded fundamental groupoid of the reduction graph $\rGraph$.

\begin{theorem}\label{thm:description_of_M-graded}\index{trivialisation $\Hiso$}
Let $\rGraph$ be a reduction graph, and let $\pi_1^\Q(\rGraph)$ be the $\Q$-Mal\u cev completion of its fundamental groupoid (Definition~\ref{def:pi1_reduction_graph}). Then the complete Hopf groupoid $\gr^\M_\bullet\O(\pi_1^\Q(\rGraph))^\dual$ is constant, and there is a canonical $\W$-filtered and $\M$-graded isomorphism
\[\Hiso\colon\gr^\M_\bullet\O(\pi_1^\Q(\rGraph))^\dual\isoarrow\HAlg.\]

In particular, the $\W$-filtration on $\gr^\M_\bullet\O(\pi_1^\Q(\rGraph))^\dual$ canonically splits.
\end{theorem}

\begin{remark}\label{rmk:N_on_HAlg}
In addition to the $\W$- and $\M$-gradings, the complete Hopf algebra $\HAlg$ from Construction \ref{cons:HAlg} carries an additional piece of structure, namely a \emph{monodromy operator} $N$, which is a Hopf derivation of $\HAlg$, bigraded of $(\W,\M)$-bidegree $(0,-2)$. However, the interaction between the monodromy operators on $\HAlg$ and $\gr^\M_\bullet\O(\pi_1^\Q(\rGraph))^\dual$ is somewhat complicated: the isomorphism $\Hiso$ from Theorem~\ref{thm:description_of_M-graded} only becomes $N$-equivariant after passing to the associated $\W$-graded on either side. We will thus omit discussion of monodromy operators for the majority of this section, returning in \S\ref{ss:equivariance} to prove some weak $N$-equivariance statements that will be used in the sequel.
\end{remark}

The proof of Theorem~\ref{thm:description_of_M-graded} will occupy most of the remainder of this section.

\subsection{Trivialisation of the $\M$-graded fundamental groupoid}
\label{ss:trivialisation}

Let us fix for the rest of the section a reduction graph $\rGraph$. The first step in our proof of Theorem~\ref{thm:description_of_M-graded} is to construct a trivialisation of the complete Hopf groupoid $\gr^\M_\bullet\O(\pi_1^\Q(\rGraph))^\dual$. This fundamentally arises from the trivialisation of the $\Q$-pro-unipotent fundamental groups of graphs arising from Cheng--Katz integration. More precisely, we know that $\gr^\M_0\O(\pi_1^\Q(\rGraph))^\dual$ is canonically identified with $\O(\pi_1^\Q(\Graph))^\dual$ with $\Graph$ the underlying graph of $\rGraph$ (Remark~\ref{rmk:interpretation_of_quotients}), and hence we have canonical paths\[\gamma_{u,v}^\can\in\gr^\M_0\O(\pi_1^\Q(\rGraph;u,v))^\dual\]for all vertices $u,v$ of $\rGraph$. Conjugation by these elements provides the desired trivialisation of $\gr^\M_\bullet\O(\pi_1^\Q(\rGraph))^\dual$.

\begin{proposition}\label{prop:trivialisation}
The maps
\begin{align*}
\psi_{u,v,u',v'}\colon\gr^\M_\bullet\O(\pi_1^\Q(\rGraph;u,v))^\dual &\isoarrow \gr^\M_\bullet\O(\pi_1^\Q(\rGraph;u',v'))^\dual \\
\gamma &\mapsto \gamma^\can_{v,v'}\cdot\gamma\cdot\gamma^\can_{u',u}
\end{align*}
for every four vertices $u,v,u',v'\in\Vert{\Graph}$ constitute a $\W$-filtered and $\M$-graded trivialisation of $\gr^\M_\bullet\O(\pi_1^\Q(\rGraph))^\dual$. That is, these maps are $\W$-filtered and $\M$-graded isomorphisms of complete coalgebras, are compatible with the path-composition, path-reversal and identity paths, and satisfy $\psi_{u,v,u'',v''}=\psi_{u',v',u'',v''}\psi_{u,v,u',v'}$ and $\psi_{u,v,u,v}=1$ for all vertices $u,v,u',v',u'',v''\in\Vert{\Graph}$
\begin{proof}
This is a straightforward consequence of Lemma~\ref{lem:basic_can_path}.
\end{proof}
\end{proposition}

\subsection{Construction of $\Hiso^{-1}$}

We will construct the isomorphism $\Hiso$ from Theorem~\ref{thm:description_of_M-graded} by first constructing its inverse, a $\W$-filtered and $\M$-graded map\[\Hiso^{-1}\colon\oHAlg\rightarrow\gr^\M_\bullet\O(\pi_1^\Q(\rGraph))^\dual\]of complete Hopf algebras, where we use the trivialisation from \S\ref{ss:trivialisation} to view the codomain of $\Hiso^{-1}$ as a constant complete Hopf groupoid, i.e.\ as a Hopf algebra. We will then show that $\Hiso^{-1}$ factors through $\HAlg$, and that the factored map is a $\W$-filtered and $\M$-graded isomorphism.

Since $\oHAlg$ is free as a complete Hopf algebra, to define this $\W$-filtered and $\M$-graded map $\Hiso^{-1}$ we need only specify linear maps and elements
\begin{align*}
\Hiso^{-1}\colon\H^1(\Graph) &\rightarrow \W_{-1}\gr^\M_0\O(\pi_1^\Q(\rGraph))^{\dual,\prim}, \\
\Hiso^{-1}(\alog(\beta_{v,i})) &\in \W_{-1}\gr^\M_{-1}\O(\pi_1^\Q(\rGraph))^{\dual,\prim}, \\
\Hiso^{-1}(\alog(\beta_{v,i}')) &\in \W_{-1}\gr^\M_{-1}\O(\pi_1^\Q(\rGraph))^{\dual,\prim}, \\
\Hiso^{-1}(\alog(\delta_e)) &\in\W_{-2}\gr^\M_{-2}\O(\pi_1^\Q(\rGraph))^{\dual,\prim}, \\
\Hiso^{-1}\colon\H_1(\Graph) &\rightarrow \W_{-1}\gr^\M_{-2}\O(\pi_1^\Q(\rGraph))^{\dual,\prim},
\end{align*}
landing in the primitive elements of $\gr^\M_\bullet\O(\pi_1^\Q(\rGraph))^\dual$. We now construct these maps and elements in turn.

\subsubsection{Generators in $\H^1(\Graph)$}
\label{sss:action_on_H^1}

We define $\Hiso^{-1}\colon\H^1(\Graph)\rightarrow\W_{-1}\gr^\M_0\O(\pi_1^\Q(\rGraph))^\dual$ to be the composite\[\H^1(\Graph)\hookrightarrow\CTensor\H^1(\Graph)\isoarrow\gr^\M_0\O(\pi_1^\Q(\rGraph))^\dual,\]where the latter map is the inverse of the duality isomorphism from Theorem~\ref{thm:c-k_duality} (cf.\ Remark~\ref{rmk:interpretation_of_quotients}). It is clear that this map lands in the primitive elements of $\gr^\M_\bullet\O(\pi_1^\Q(\rGraph))^\dual$.

\subsubsection{Generators arising from genus and punctures}
\label{sss:action_on_punctures}

For each vertex $v$ and integer $1\leq i\leq g(v)$, and for each half-edge $e$, we define
\begin{align*}
\Hiso^{-1}(\alog(\beta_{v,i})) &= \log(\beta_{v,i})\in\W_{-1}\gr^\M_{-1}\O(\pi_1^\Q(\rGraph,v))^\dual, \\
\Hiso^{-1}(\alog(\beta_{v,i}')) &= \log(\beta_{v,i}')\in\W_{-1}\gr^\M_{-1}\O(\pi_1^\Q(\rGraph,v))^\dual, \\
\Hiso^{-1}(\alog(\delta_e)) &= \log(\delta_e)\in\W_{-2}\gr^\M_{-2}\O(\pi_1^\Q(\rGraph,\source(e)))^\dual,
\end{align*}
where $\log(\beta_{v,i})$, $\log(\beta_{v,i}')$ and $\log(\delta_e)$ denote the image of the logarithm of the corresponding generator of $\pi_1^\Q(\rGraph)$ in the appropriate $\M$-graded piece of $\O(\pi_1^\Q(\rGraph))^\dual$ (see Definition~\ref{def:filtrations_on_graph}). It is clear that these elements of $\gr^\M_\bullet\O(\pi_1^\Q(\rGraph))^\dual$ are all primitive.

\subsubsection{Generators in $\H_1(\Graph)$}
\label{sss:action_on_H_1}

The most technical part of the definition of the map $\Hiso^{-1}\colon\oHAlg\rightarrow\gr^\M_\bullet\O(\pi_1^\Q(\rGraph))^\dual$ is specifying its action on $\H_1(\Graph)$. The reader interested purely in the technical definition may skip ahead to Definition~\ref{def:inverse_iso_on_H_1}, but we will for now give a preparatory calculation that will also feature in our computations of non-abelian Kummer maps.

\begin{proposition}\label{prop:loop_integral_along_edge}
Let $e$ be an edge of the reduction graph $\rGraph$, and by abuse of notation write $e^*=\Hiso^{-1}e^*\in\gr^\W_{-1}\gr^\M_0\O(\pi_1^\Q(\rGraph,\source(e)))^\dual$ for the element corresponding to $e^*\in\H^1(\Graph)$ from Notation~\ref{notn:e*} under the duality isomorphism from Theorem~\ref{thm:c-k_duality}. Writing $\log(\delta_e)\in\gr^\M_{-2}\O(\pi_1^\Q(\rGraph,\source(e)))^\dual$ for the image in $\gr^\M_{-2}$ of the logarithm of the generator $\delta_e$ from Definition~\ref{def:pi1_reduction_graph}, the components of the map\[s\mapsto\exp(se^*)\log(\delta_e)\exp(-se^*)\]are polynomials in $s$, and we have
\[
\scalebox{0.96}{$\displaystyle{
\int_0^{\length(e)}\!\!\!\!\!\exp(se^*)\log(\delta_e)\exp(-se^*)\d s=\frac{\exp(\length(e)\ad_{e^*})-1}{\ad_{e^*}}\log(\delta_e):=\sum_{r\geq0}\frac{(\length(e)e^*)^r}{(r+1)!}(\log(\delta_e)).
}$}
\]
\begin{proof}
The fact that the coefficients are polynomials in $s$ follows from the corresponding fact for $s\mapsto\exp(se^*)=\sum_{r\geq0}\frac{s^r(e^*)^r}{r!}$ (note that $(e^*)^r$ is eventually zero in any finite-dimensional quotient of $\gr^\M_0\O(\pi_1^\Q(\rGraph,\source(e)))^\dual$). The evaluation of the integral is then purely formal.
\end{proof}
\end{proposition}

\begin{corollary}
If $e$ is an edge as above, then\[\frac{\exp(\length(e)e^*)-1}{\ad_{e^*}}\log(\delta_e)+\frac{\exp(\length(e^{-1})(e^{-1})^*)-1}{\ad_{(e^{-1})^*}}\log(\delta_{e^{-1}})=0.\](Here we identify the two summands as elements of a common vector space using the trivialisation from Proposition~\ref{prop:trivialisation}.)
\begin{proof}
Under the trivialisation from Proposition~\ref{prop:trivialisation}, we can use Lemma~\ref{lem:C-K_on_edges} to rewrite the edge relation from Definition~\ref{def:pi1} as
\[
\log(\delta_{e^{-1}})=-\exp(\length(e)e^*)\log(\delta_e)\exp(-\length(e)e^*).
\]
Plugging this into the integral relation from Proposition~\ref{prop:loop_integral_along_edge} and using that $(e^{-1})^*=-e^*$ gives the desired equality.
\end{proof}
\end{corollary}

This corollary ensures that the following definition of $\Hiso^{-1}$ on $\H_1(\Graph)$ is well-defined.

\begin{definition}\label{def:inverse_iso_on_H_1}
We define the map
\begin{align*}
\Hiso^{-1}\colon\H_1(\Graph) &\rightarrow\W_{-1}\gr^\M_{-2}\O(\pi_1^\Q(\rGraph))^\dual \\
\Hiso^{-1}\colon\sum_{e\in\plEdge{\Graph}}\lambda_ee &\mapsto \sum_{e\in\plEdge{\Graph}}\lambda_e\frac{\exp(\length(e)\ad_{e^*})-1}{\ad_{e^*}}(\log(\delta_e)),
\end{align*}
where we view the summands as elements of a common vector space using the trivialisation from Proposition~\ref{prop:trivialisation}. It is clear that this map lands in the primitive elements of $\gr^\M_\bullet\O(\pi_1^\Q(\rGraph))^\dual$.
\end{definition}

\subsection{The surface relation}\index{relation!surface $\sigma$}

We have now constructed a $\W$-filtered and $\M$-graded morphism $\Hiso^{-1}\colon\oHAlg\rightarrow\gr^\M_\bullet\O(\pi_1^\Q(\rGraph))^\dual$ of constant complete Hopf groupoids, by specifying where it sends the generators of the complete tensor algebra $\oHAlg$. To show that this map factors through $\HAlg$, we will now show that $\Hiso^{-1}(\sigma)=0$, where $\sigma\in\gr^\M_{-2}\oHAlg$ is the surface relation from Construction~\ref{cons:HAlg}. For this, we use the following preliminary calculation.

\begin{proposition}\label{prop:sigma_0}
Let $(\xi_i)$ be a basis of $\H^1(\Graph)$ with dual basis $(\xi_i')$ of $\H_1(\Graph)$. Then\[\Hiso^{-1}\left(\sum_i[\xi_i',\xi_i]\right)=\sum_{e\in\Edge{\Graph}}\log(\delta_e),\]where we view the summands as lying in a common vector space via the trivialisation from Proposition~\ref{prop:trivialisation} as usual.
\begin{proof}
The left-hand side is the image of $1\in\End(\H_1(\Graph))\iso\H^1(\Graph)\otimes\H_1(\Graph)$ under the map $\xi\otimes\xi'\mapsto-\ad_{\Hiso^{-1}\xi}(\Hiso^{-1}\xi')$. Hence by Remark~\ref{rmk:e*e_is_identity}, we know that it is equal to\[-\sum_{e\in\plEdge{\Graph}}\ad_{e^*}\left(\frac{\exp(\length(e)\ad_{e^*})-1}{\ad_{e^*}}\log(\delta_e)\right)=\sum_{e\in\plEdge{\Graph}}(1-\exp(\length(e)\ad_{e^*}))\log(\delta_e).\]But the edge relations from Definition~\ref{def:pi1} ensure that $\exp(\length(e)\ad_{e^*})\log(\delta_e)=\exp(\length(e)e^*)\log(\delta_e)\exp(-\length(e)e^*)=-\log(\delta_{e^{-1}})$ for every edge $e$, and hence this quantity is equal to $\sum_{e\in\Edge{\Graph}}\log(\delta_e)$, as desired.
\end{proof}
\end{proposition}

\begin{corollary}
Let $\rGraph$ be a reduction graph. Then $\Hiso^{-1}\colon\oHAlg\rightarrow\gr^\M_\bullet\O(\pi_1^\Q(\rGraph))^\dual$ sends the surface relation $\sigma$ from Construction~\ref{cons:HAlg} to $0$; this map factors uniquely through a map\[\Hiso^{-1}\colon\HAlg\rightarrow\gr^\M_\bullet\O(\pi_1^\Q(\rGraph))^\dual.\]
\begin{proof}
Writing $\sigma=\sigma_0+\sigma_1+\sigma_2$ as in Construction~\ref{cons:HAlg}, we have that $\Hiso^{-1}(\sigma_0)$ is given by Proposition~\ref{prop:sigma_0}, and hence
\begin{align*}
\Hiso^{-1}(\sigma) &= \sum_{e\in\Edge{\Graph}}\log(\delta_e)+\sum_{v\in\Vert{\Graph}}\sum_{i=1}^{g(v)}[\log(\beta_{v,i}'),\log(\beta_{v,i})]+\sum_{e\in\HEdge{\Graph}}\log(\delta_e) \\
 &= \sum_{v\in\Vert{\Graph}}\left(\sum_{i=1}^{g(v)}[\log(\beta_{v,i}'),\log(\beta_{v,i})]+\sum_{\source(e)=v}\log(\delta_e)\right).
\end{align*}
But each of the summands vanishes in $\gr^\M_{-2}\O(\pi_1^\Q(\rGraph))^\dual$ by the vertex relations from Definition~\ref{def:pi1_reduction_graph}.
\end{proof}
\end{corollary}

\subsection{Proof of isomorphy}

The final and most technical step of our proof of Theorem~\ref{thm:description_of_M-graded} is to show that the map $\Hiso^{-1}\colon\HAlg\rightarrow\gr^\M_\bullet\O(\pi_1^\Q(\rGraph))^\dual$ constructed above is a $\W$-filtered and $\M$-graded isomorphism. To accomplish this, we will use the description in Definition~\ref{def:pi1_reduction_graph} to exhibit a non-canonical one-relator presentation of $\O(\pi_1^\Q(\rGraph,b))^\dual$ for some basepoint $b$, which will take the form of a $\W$- and $\M$-filtered isomorphism $\check\Hiso^{-1}\colon\HAlg\isoarrow\O(\pi_1^\Q(\rGraph,b))^\dual$. This isomorphism is not quite a lift of the $\M$-graded map $\Hiso^{-1}$; in fact we will see in Proposition~\ref{prop:lift_of_alpha} that $\gr^\W_\bullet\Hiso^{-1}=\gr^\M_\bullet\gr^\W_\bullet\check\Hiso^{-1}$, which is enough to show that $\Hiso^{-1}$ is an isomorphism.

To define the map $\check\Hiso^{-1}$, we fix some notation. Fix a basepoint $b\in\Vert{\Graph}$, and choose distinct\footnote{When we refer to edges being distinct, we mean that no two edges in the list are equal to or inverses of one another.} edges $e_1,\dots,e_\gGamma$ such that their complement $\Graph\setminus\{e_1^{\pm1},\dots,e_\gGamma ^{\pm1}\}$ in the underlying graph of $\rGraph$ is a tree. For any element $\gamma\in\O(\pi_1^\Q(\rGraph;u,v))^\dual$, we will write $\check\gamma=\gamma_{v,b}^\tree\gamma\gamma_{b,u}^\tree\in\O(\pi_1^\Q(\rGraph,b))^\dual$, where $\gamma_{b,u}^\tree$ (resp.\ $\gamma_{v,b}^\tree$) denotes the unique path from $b$ to $u$ (resp.\ $v$ to $b$) in the underlying graph of $\rGraph$.

We write $\log(\gamma_1),\dots,\log(\gamma_\gGamma)$ for the basis of $\H^1(\Graph)$ corresponding to the loops $\check e_1,\dots,\check e_\gGamma \in\H_1(\Graph)$ under the isomorphism $\H^1(\Graph)\isoarrow\H_1(\Graph)$ induced from the cycle pairing on $\H_1(\Graph)$. We also write $\log(\delta_1),\dots,\log(\delta_\gGamma )$ for the basis of $\H_1(\Graph)$ corresponding to the elements $e_1^*,\dots,e_\gGamma ^*\in\H^1(\Graph)$ under this same isomorphism. The elements $\log(\gamma_i)$, $\alog(\beta_{v,j})$, $\alog(\beta_{v,j}')$, $\log(\delta_i)$ and $\alog(\delta_e)$ (for $1\leq i\leq \gGamma$, $v\in\Vert{\Graph}$, $1\leq j\leq g(v)$ and $e\in\HEdge{\Graph}$) are all primitive and bihomogenous in $\oHAlg$, and generate it freely as a complete algebra.

Finally, we fix integers $a_{ij}$ for $1\leq i,j\leq \gGamma$ to be determined later, and define the morphism\[\check\Hiso^{-1}\colon\oHAlg\rightarrow\O(\pi_1^\Q(\rGraph,b))^\dual\]of complete Hopf algebras by
\begin{align*}
\check\Hiso^{-1}(\gamma_i) &= \check e_i\cdot\prod_{j=1}^\gGamma\check\delta_{e_j}^{a_{ij}}, \\
\check\Hiso^{-1}(\Hiso\beta_{v,i}) &= \check\beta_{v,i}, \\
\check\Hiso^{-1}(\Hiso\beta_{v,i}') &= \check\beta_{v,i}', \\
\check\Hiso^{-1}(\delta_i) &= \check\delta_{e_i}, \\
\check\Hiso^{-1}(\Hiso\delta_e) &= \check\delta_e,
\end{align*}
where $\Hiso\beta_{v,i}$ denotes the exponential of the formal symbol $\alog(\beta_{v,i})$ (and similarly for $\Hiso\beta_{v,i}'$ and $\Hiso\delta_e$).

\begin{proposition}\label{prop:one-relator_description}
The morphism $\check\Hiso^{-1}\colon\oHAlg\rightarrow\O(\pi_1^\Q(\rGraph,b))^\dual$ is surjective, simultaneously strict\footnote{By this, we mean that every element of $\W_{-i}\M_{-j}\O(\pi_1^\Q(\rGraph,b))^\dual$ in the image of $\check\Hiso^{-1}$ is the image of an element of $\W_{-i}\M_{-j}\oHAlg$. This is a stronger property than being strict for the $\W$- and $\M$-filtrations simultaneously.} for the $\W$- and $\M$-filtrations, and its kernel is the ideal generated by $\log(\rho)$, where $\rho\in\oHAlg$ is a grouplike element given up to reordering of terms by the product\[\rho=\prod_{i=1}^\gGamma \delta_i\cdot\prod_{i=1}^\gGamma \left(\gamma_i\left(\prod_{j=1}^\gGamma \delta_j^{a_{ij}}\right)^{-1}\!\!\!\!\!\!\delta_i^{-1}\left(\prod_{j=1}^\gGamma \delta_j^{a_{ij}}\right)\!\gamma_i^{-1}\!\right)\cdot\!\!\prod_{v\in\Vert{\Graph}}\prod_{i=1}^{g(v)}[\Hiso\beta_{v,i}',\Hiso\beta_{v,i}]\cdot\!\!\!\prod_{e\in\HEdge{\Graph}}\Hiso\delta_e.\]
In particular, $\check\Hiso^{-1}$ induces a $\W$- and $\M$-filtered isomorphism\[\check\Hiso^{-1}\colon\oHAlg/(\log(\rho))\isoarrow\O(\pi_1^\Q(\rGraph,b))^\dual\]of complete Hopf algebras.
\begin{proof}[Proof (sketch)]
Definition~\ref{def:pi1_reduction_graph} tells us that $\pi_1^\Q(\rGraph,b)$ is the $\Q$-pro-unipotent group generated by $\check e_i$, $\check\beta_{v,j}$, $\check\beta_{v,j}$ and $\check\delta_e$ for $1\leq i\leq \gGamma$, $v\in\Vert{\Graph}$, $1\leq j\leq g(v)$, and $e\in\Edge{\Graph}\cup\HEdge{\Graph}$, subject to the relations:
\begin{itemize}
	\item (edge relations)
	\begin{itemize}
		\item for every edge $e\in\Edge{\Graph}\setminus\{e_1^{\pm1},\dots,e_\gGamma ^{\pm1}\}$ we have $\check\delta_{e^{-1}}=\check\delta_e^{-1}$;
		\item for every $1\leq i\leq \gGamma$ we have $\check\delta_{e_i^{-1}}=\check e_i\check\delta_{e_i}^{-1}\check e_i^{-1}$; and
	\end{itemize}
	\item (vertex relations) for every vertex $v\in\Vert{\Graph}$ we have\[\prod_{i=1}^{g(v)}[\check\beta_{v,i}',\check\beta_{v,i}]\cdot\prod_{\source(e)=v}\check\delta_e=1.\]
\end{itemize}

We may now formally manipulate this presentation to eliminate the generators $\check\delta_e$ for $e\in\Edge{\Graph}\setminus\{e_1,\dots,e_\gGamma \}$. The end result of this elimination process is that $\pi_1^\Q(\rGraph,b)$ is generated by $\check e_i$, $\check\beta_{v,j}$, $\check\beta_{v,j}$, $\check\delta_{e_i}$ and $\check\delta_e$ for $1\leq i\leq \gGamma$, $v\in\Vert{\Graph}$, $1\leq j\leq g(v)$, and $e\in\HEdge{\Graph}$, subject to the single relation $\check\rho=1$, where\[\check\rho=\prod_{i=1}^\gGamma \check\delta_{e_i}\cdot\prod_{i=1}^\gGamma \left(\check e_i\check\delta_{e_i}^{-1}\check e_i^{-1}\right)\cdot\prod_{v\in\Vert{\Graph}}\prod_{i=1}^{g(v)}[\check\beta_{v,i}',\check\beta_{v,i}]\cdot\prod_{e\in\HEdge{\Graph}}\check\delta_e,\]up to reordering of terms. Translated into the language of complete Hopf algebras and adjusting the generators according to the integers $a_{ij}$, this says that $\check\Hiso^{-1}$ is surjective, with kernel the ideal generated by $\log(\rho)$.

It remains to check simultaneous strictness of $\check\Hiso^{-1}$. Recall from Proposition~\ref{prop:explicit_hopf_gpd} that the filtrations on $\O(\pi_1^\Q(\rGraph,b))^\dual$ are the ones induced from declaring each $\log(\check e_i)$ to lie in $\W_{-1}\M_0$, each $\log(\check\beta_{v,i})$ and $\log(\check\beta_{v,i}')$ to lie in $\W_{-1}\M_{-1}$, and each $\log(\check\delta_e)$ to lie in $\W_{-1}\M_{-2}$ or $\W_{-2}\M_{-2}$ according as $e\in\Edge{\Graph}$ or $e\in\HEdge{\Graph}$. Since $\check\Hiso^{-1}$ is a surjective algebra homomorphism, we need only check that each generator of $\O(\pi_1^\Q(\rGraph,b))^\dual$ lying in $\W_{-i}\M_{-j}$ has a lift to $\W_{-i}\M_{-j}\oHAlg$; this is immediate for all generators save the $\log(\check\delta_e)$ for $e\in\Edge{\Graph}\setminus\{e_1,\dots,e_\gGamma \}$. But lifts of these generators can be found by using the edge and vertex relations to recursively express each such $\check\delta_e$ as a product of terms already in the image of $1+\W_{-1}\M_{-2}\oHAlg$, completing the proof.
\end{proof}
\end{proposition}

The connection between this presentation of $\O(\pi_1^\Q(\rGraph,b))^\dual$ and the desired canonical presentation of $\gr^\M_\bullet\O(\pi_1^\Q(\rGraph))^\dual$ is provided by the following proposition.

\begin{proposition}\label{prop:lift_of_alpha}
We have $\gr^\M_\bullet\gr^\W_\bullet\check\Hiso^{-1}=\gr^\W_\bullet\Hiso^{-1}$.
\begin{proof}
We may check this on the standard generators of $\oHAlg$. For the generators $\log(\beta_{v,i})$ for $v\in\Vert{\Graph}$ and $1\leq i\leq g(v)$, we have\[\gr^\M_\bullet\gr^\W_\bullet\check\Hiso^{-1}(\log(\beta_{v,i}))\!=\!\gamma_{v,b}^\tree\log(\beta_{v,i})\gamma_{b,v}^\tree\!=\!\gamma_{v,b}^\can\log(\beta_{v,i})\gamma_{b,v}^\can\!=\!\gr^\W_\bullet\Hiso^{-1}(\log(\beta_{v,i}))\]in $\gr^\M_{-1}\gr^\W_{-1}\O(\pi_1^\Q(\rGraph,b))^\dual$, since $\gamma_{b,v}^\tree=\gamma_{b,v}^\can=1$ in $\gr^\M_0\gr^\W_0\O(\pi_1^\Q(\rGraph;b,v))^\dual=\Q$. The same argument applies to the generators $\log(\beta_{v,i}')$, and to the generators $\log(\delta_e)$ for $e\in\HEdge{\Graph}$. Equality on the generators $\log(\gamma_i)$ follows essentially by definition.

Finally, to check equality on the generators $\log(\delta_i)$, we note that we may write $\log(\delta_i)=\frac1{\length(e_i)}e_i+\sum_{v\in\Vert{\Graph}}\lambda_v\cdot\left(\sum_{\source(e)=v}\frac1{\length(e)}e\right)$ for some $\lambda_v\in\Q$, and hence we have\[\gr^\W_\bullet\Hiso^{-1}(\log(\delta_i))=\log(\delta_{e_i})+\sum_{v\in\Vert{\Graph}}\lambda_v\cdot\sum_{\source(e)=v}\log(\delta_e)=\log(\delta_{e_i})\]in $\gr^\M_{-2}\gr^\W_{-1}\O(\pi_1^\Q(\rGraph))^\dual$, since by the vertex relations (Definition~\ref{def:pi1_reduction_graph}) we have $\sum_{\source(e)=v}\log(\delta_e)=0$ in $\gr^\M_{-2}\gr^\W_{-1}\O(\pi_1^\Q(\rGraph))^\dual$ for every $v$. But the right-hand side is equal to $\gamma_{\source(e_i),b}^\tree\log(\delta_{e_i})\gamma_{b,\source(e_i)}^\tree=\gr^\M_\bullet\gr^\W_\bullet\check\Hiso^{-1}(\log(\delta_i))$ in $\gr^\M_{-2}\gr^\W_{-1}\O(\pi_1^\Q(\rGraph,b))^\dual$, and we are done.
\end{proof}
\end{proposition}

In light of this proposition, in order to prove that $\Hiso^{-1}\colon\HAlg\rightarrow\gr^\M_\bullet\O(\pi_1^\Q(\rGraph))^\dual$ is a $\W$-filtered isomorphism, it will suffice to prove that $\gr^\M_\bullet\gr^\W_\bullet\check\Hiso^{-1}$ sets up an isomorphism $\HAlg=\oHAlg/(\sigma)\isoarrow\gr^\M_\bullet\gr^\W_\bullet\O(\pi_1^\Q(\rGraph))^\dual$. This is fundamentally a question of interchanging the operator $\gr^\W_\bullet$ with quotients by single relations, for which we follow \cite{labute}.

\begin{proposition}\label{prop:lift_of_alpha_is_iso}
For a suitable choice of integers $a_{ij}$ in Proposition~\ref{prop:one-relator_description}, the element $\log(\rho)\in\oHAlg$ is congruent to the surface relation $\sigma$ mod $\W_{-3}$. The induced morphism\[\gr^\W_\bullet\check\Hiso^{-1}\colon\HAlg=\oHAlg/(\sigma)\rightarrow\gr^\W_\bullet\O(\pi_1^\Q(\rGraph,b))^\dual\]is a $\W$-graded and $\M$-filtered isomorphism.
\begin{proof}
One verifies by direct calculation that
\begin{align*}
\log(\rho) &\equiv \sum_{i=1}^\gGamma [\log(\delta_i),\log(\gamma_i)] + \sum_{v\in\Vert{\Graph}}\sum_{i=1}^{g(v)}[\log(\beta_{v,i}'),\log(\beta_{v,i})] \\
 & \:\:\:+ \sum_{e\in\HEdge{\Graph}}\log(\delta_e) + \sum_{1\leq i,j\leq \gGamma}(b_{ij}-a_{ij})[\log(\delta_i),\log(\delta_j)]
\end{align*}
modulo $\W_{-3}$ for some integers $b_{ij}$ depending only on the ordering of the terms in the expression for $\rho$ in Proposition~\ref{prop:one-relator_description} (in particular, independent of the $a_{ij}$). Since the $\log(\gamma_i)$ and $\log(\delta_i)$ comprise dual bases of $\H^1(\Graph)$ and $\H_1(\Graph)$, setting $a_{ij}=b_{ij}$ ensures that $\log(\rho)=\sigma$ in $\gr^\W_{-2}\oHAlg$.

It thus follows from Proposition~\ref{prop:one-relator_description} that\[\gr^\W_\bullet\check\Hiso^{-1}\colon\oHAlg\rightarrow\gr^\W_\bullet\O(\pi_1^\Q(\rGraph,b))^\dual\]is surjective, $\M$-strict and contains the ideal generated by $\sigma$ in its kernel. It remains to show that the kernel is exactly this ideal, for which we will use the main result of \cite{labute}.

For this, let $\pi_1(\rGraph)$ denote the (discrete) fundamental groupoid of $\rGraph$ (Definition~\ref{def:pi1_reduction_graph}). The \emph{graded $\Q$-Lie algebra $\gr^\W_\bullet\Lie(\pi_1)$ of $\pi_1=\pi_1(\rGraph,b)$} (with respect to the $\W$-filtration) is, by definition, the pro-nilpotent Lie algebra\[\gr^\W_\bullet\Lie_\Q(\pi_1)=\prod_{n>0}\left(\Q\otimes\frac{\W_{-n}\pi_1}{\W_{-n-1}\pi_1}\right),\]with Lie bracket induced by the commutator bracket on $\pi_1$. The completed universal enveloping algebra of $\gr^\W_\bullet\Lie_\Q(\pi_1(\rGraph,b))$ is isomorphic to $\gr^\W_\bullet\O(\pi_1^\Q(\rGraph,b))^\dual$ as a $\W$-graded complete Hopf algebra (see Proposition~\ref{prop:graded_malcev_liegebra} and Example~\ref{ex:graded_CUEnv}).

If we let $F$ denote the free discrete group on formal symbols $\gamma_i$, $\beta_{v,j}$, $\beta_{v,j}'$, $\delta_i$ and $\delta_e$ (for $1\leq i\leq \gGamma$, $v\in\Vert{\Graph}$, $1\leq j\leq g(v)$ and $e\in\HEdge{\Graph}$), then $\oHAlg$ is similarly the completed universal enveloping algebra of the $\W$-graded $\Q$-Lie algebra $\gr^\W_\bullet\Lie_\Q(F)$. The map $\check\Hiso^{-1}\colon\oHAlg\rightarrow\O(\pi_1^\Q(\rGraph,b))^\dual$ is induced by a surjection $F\twoheadrightarrow\pi_1(\rGraph,b)$, and the same proof as in Proposition~\ref{prop:one-relator_description} establishes that the kernel of this map is the normal subgroup generated by the element $\rho\in F$. It follows from \cite[Theorem on p17]{labute} (see \cite[Remarks on p22]{labute}) that the induced map $\gr^\W_\bullet\Lie_\Q(F)\rightarrow\gr^\W_\bullet\Lie_\Q(\pi_1)$ is surjective, with kernel the Lie ideal generated by the image of $\rho$ in $\gr^\W_{-2}\Lie(F)=\W_{-2}F/\W_{-3}F$. Hence the map $\check\Hiso^{-1}\colon\oHAlg\rightarrow\O(\pi_1^\Q(\rGraph,b))^\dual$ is also surjective, with kernel the ideal generated by $\log(\rho)=\sigma$. This completes the proof.
\end{proof}
\end{proposition}

We now prove Theorem~\ref{thm:description_of_M-graded}, with $\Hiso $ the inverse of the map $\Hiso ^{-1}$ above.
\begin{proof}[Proof of Theorem~\ref{thm:description_of_M-graded}]
Fix integers $a_{ij}$ as in Proposition~\ref{prop:lift_of_alpha_is_iso}. We know that $\Hiso^{-1}$ induces a $\W$-filtered and $\M$-graded morphism $\HAlg\rightarrow\gr^\M_\bullet\O(\pi_1^\Q(\rGraph))^\dual$ of complete Hopf algebras. To show that this is an $\W$-filtered isomorphism, it suffices to show that $\gr^\W_\bullet\Hiso^{-1}\colon\HAlg\rightarrow\gr^\M_\bullet\gr^\W_\bullet\O(\pi_1^\Q(\rGraph))^\dual$ is an isomorphism. But this map is equal to $\gr^\M_\bullet\gr^\W_\bullet\check\Hiso^{-1}\colon\HAlg\rightarrow\gr^\M_\bullet\gr^\W_\bullet\O(\pi_1^\Q(\rGraph,b))^\dual$ by Proposition~\ref{prop:lift_of_alpha}, and this map is an isomorphism by Proposition~\ref{prop:lift_of_alpha_is_iso}.
\end{proof}

\subsection{Weak $N$-equivariance}
\label{ss:equivariance}

To conclude this section, let us discuss to what extent the monodromy operator $N$ on the complete Hopf groupoid $\gr^\M_\bullet\O(\pi_1^\Q(\rGraph))^\dual$ can be controlled using the isomorphism $\Hiso\colon\gr^\M_\bullet\O(\pi_1^\Q(\rGraph))^\dual\isoarrow\HAlg$ from Theorem~\ref{thm:description_of_M-graded}. As mentioned in Remark~\ref{rmk:N_on_HAlg}, this isomorphism is not in general $N$-equivariant (for any Hopf algebra derivation $N$ on $\HAlg$), but we will need certain partial $N$-equivariance results at several points in the coming arguments. To this end, we endow $\HAlg$ with the following monodromy operator.

\begin{definition}\label{def:N_on_HAlg}
Let $\oHAlg$ and $\HAlg$ denote the complete Hopf algebras from Construction~\ref{cons:HAlg}. We define a bigraded derivation $N$ of $(\W,\M)$-bidegree $(0,-2)$ on $\oHAlg$ by specifying that $N$ vanishes on the generators $\alog(\beta_{v,i})$, $\alog(\beta_{v,i}')$ and $\alog(\delta_e)$ for all vertices $v$, integers $1\leq i\leq g(v)$ and half-edges $e$, and specifying that the maps
\[
\gr^\M_0\gr^\W_{-1}\oHAlg \overset N\longrightarrow \gr^\M_{-2}\gr^\W_{-1}\oHAlg \overset N\longrightarrow \gr^\M_{-4}\gr^\W_{-1}\oHAlg
\]
are given by
\[
\H^1(\Graph) \longisoarrow \H_1(\Graph) \longrightarrow 0
\]
with the left-hand arrow the isomorphism induced by the cycle pairing. One checks straightforwardly, e.g.\ by choosing a basis of $\H^1(\Graph)$ which is orthogonal with respect to the cycle pairing, that $N$ vanishes on the surface relation $\sigma$ from Construction~\ref{cons:HAlg}, and hence induces a derivation $N$ on $\HAlg$. Both of these derivations are Hopf derivations (i.e.\ are also compatible with the comultiplication).
\end{definition}

The basic significance of this monodromy operator $N$ on $\HAlg$ is that it captures the monodromy operator $N$ on the fundamental groupoid after passing to bigraded pieces: the isomorphism
\[
\gr^\W_\bullet\Hiso\colon \gr^\W_\bullet\gr^\M_\bullet\O(\pi_1^\Q(\rGraph))^\dual \isoarrow \HAlg
\]
arising from Theorem~\ref{thm:description_of_M-graded} is $N$-equivariant. Put another way, the map $N\Hiso-\Hiso N\colon \gr^\M_\bullet\O(\pi_1^\Q(\rGraph))^\dual \rightarrow \HAlg$ is $\W$-filtered of degree $-1$ (i.e.\ takes each $\W_{-n}$ into $\W_{-n-1}$). In fact, we will need the following stronger result, to be proved at the end of this section.

\begin{lemma}\label{lem:equivariance}
Let $\Hiso\colon\gr^\M_\bullet\O(\pi_1^\Q(\rGraph))^\dual\!\isoarrow\!\HAlg$ be the isomorphism from Theorem~\ref{thm:description_of_M-graded}. Then for all integers $k>0$, the map
\[
\ad_N^k\Hiso:=\sum_{i+j=k}{k\choose i}(-1)^jN^i\circ\Hiso\circ N^j\colon\gr^\M_\bullet\O(\pi_1^\Q(\rGraph))^\dual\rightarrow\HAlg
\]
is $\W$-filtered of degree $-k-1$.
\end{lemma}

\begin{corollary}\label{cor:graph_w-m}
For all non-negative integers $n$ and $i$, the map
\begin{equation}\label{eq:weight-monodromy_maps}
N^i\colon\gr^\M_{-n+i}\gr^\W_{-n}\O(\pi_1^\Q(\rGraph))^\dual \isoarrow \gr^\M_{-n-i}\gr^\W_{-n}\O(\pi_1^\Q(\rGraph))^\dual
\end{equation}
is an isomorphism. The same is true with $\O(\pi_1^\Q(\rGraph))^\dual$ replaced with $\Lie(\pi_1^\Q(\rGraph,b))$.
\begin{proof}
Since $\gr^\W_\bullet\Hiso\colon\gr^\W_\bullet\gr^\M_\bullet\O(\pi_1^\Q(\rGraph))^\dual\isoarrow\HAlg$ is $N$-equivariant by Lemma~\ref{lem:equivariance}, this condition may be checked just for $\HAlg$ and its Lie algebra $\LAlg=\HAlg^\prim$ of primitive elements. That the corresponding maps~\eqref{eq:weight-monodromy_maps} for $\oHAlg$ are all isomorphisms is immediate, since $\oHAlg$ is the complete tensor algebra on a vector space for which the maps~\eqref{eq:weight-monodromy_maps} are isomorphisms by construction. This already implies that the corresponding maps~\eqref{eq:weight-monodromy_maps} for $\HAlg$ are surjective. That they are bijective follows from a dimension-count, e.g.\ by writing down a bigraded automorphism of $\HAlg$ interchanging $\H^1(\Graph)$ and $\H_1(\Graph)$. One passes from $\HAlg$ to the subspace $\LAlg$ via a similar argument.
\end{proof}
\end{corollary}

Before we prove Lemma~\ref{lem:equivariance}, let us unpack a few further consequences, which will help us relate the constructions of this section to those in \S\ref{s:galois_cohomology}. These are most naturally phrased in terms of an auxiliary filtration, defined in terms of the existing filtrations and the monodromy operator $N$.

\begin{definition}\label{def:coM}
For all rational vertices $u,v$, we define an increasing $N$-stable filtration $\coM_\bullet$ on $\gr^\M_\bullet\O(\pi_1^\Q(\rGraph;u,v))^\dual$ by $\M$-graded subspaces, by declaring that an element $\gamma\in\gr^\M_{-n}\O(\pi_1^\Q(\rGraph;u,v))^\dual$ lies in $\coM_{-k}$ just when
\[
N^j(\gamma)\in\W_{-\lceil\frac{n+k}2\rceil-j}\gr^\M_{-n-2j}\O(\pi_1^\Q(\rGraph;u,v))^\dual
\]
for all $j\geq0$.

We also let $\coM_\bullet$ denote the $N$-stable filtration of $\HAlg$ defined analogously. This is a filtration by $(\W,\M)$-bigraded subspaces, and $\gamma\in\gr^\W_{-m}\gr^\M_{-n}\HAlg$ lies in $\coM_{-k}$ just when
\[
N^{\lfloor m-\frac{n+k}2\rfloor+1}(\gamma)=0.
\]

Both of these filtrations are compatible with path-composition/multiplication and path-reversal/antipode, but not with comultiplication.
\end{definition}

The $\coM$-filtration is related to the other two filtrations as follows.

\begin{proposition}\label{prop:coM_basics}
We have $\coM_0\gr^\M_\bullet\O(\pi_1^\Q(\rGraph))^\dual=\gr^\M_\bullet\O(\pi_1^\Q(\rGraph))^\dual$. For all integers $k$ we have $\coM_{-k}\gr^\M_\bullet\O(\pi_1^\Q(\rGraph))^\dual\leq\M_{-k}\gr^\M_\bullet\O(\pi_1^\Q(\rGraph))^\dual$ and $\coM_{-k}\gr^\M_{-n}\O(\pi_1^\Q(\rGraph))^\dual\leq\W_{-\lceil\frac n2\rceil-k}\gr^\M_{-n}\O(\pi_1^\Q(\rGraph))^\dual$. The same is true with $\O(\pi_1^\Q(\rGraph))^\dual$ replaced with $\HAlg$.
\begin{proof}
In the interests of brevity, we give a proof for $\HAlg$ that also applies to the complete Hopf groupoid $\gr^\M_\bullet\O(\pi_1^\Q(\rGraph))^\dual$. The first assertion follows from the fact that $\W_{-\lceil\frac n2\rceil}\gr^\M_{-n}\HAlg=\gr^\M_{-n}\HAlg$ for all $n$. The final assertion is immediate from the definition.

For the second assertion, it suffices to prove that $\coM_{-k-1}\gr^\M_{-k}\HAlg=0$ for all $k$. Suppose that $\gamma\in\gr^\M_{-k}\HAlg$ lies in $\coM_{-k-1}$ for some $k$. We show by induction that $\gamma\in\W_{-k-j-1}$ for all $j$, so that $\gamma=0$ as desired. The base case $j=0$ follows from the definition of $\coM_\bullet$. For the inductive step, we have $N^j(\gamma)\in\W_{-k-j-1}\gr^\M_{-k-2j}\HAlg$. But the proof of Corollary~\ref{cor:graph_w-m} shows that the map $N^j\colon\gr^\M_{-k}\gr^\W_{-k-j}\HAlg\isoarrow\gr^\M_{-k-2j}\gr^\W_{-k-j}\HAlg$ is an isomorphism, and hence $\gamma\in\W_{-k-j-1}$ as desired.
\end{proof}
\end{proposition}

For our applications, it will suffice to know the following compatibility of the isomorphism $\Hiso$ with the $\coM$-filtration, which in particular implies that $\Hiso$ induces an $N$-equivariant isomorphism on $\coM$-graded pieces.

\begin{lemma}\label{lem:equivariance_coM}
The isomorphism $\Hiso\colon\gr^\M_\bullet\O(\pi_1^\Q(\rGraph))^\dual\isoarrow\HAlg$ is a $\coM$-filtered isomorphism, and $\ad_N\Hiso=(N\Hiso-\Hiso N)$ is $\coM$-filtered of degree $-2$.
\begin{proof}
That $\Hiso$ and its inverse are $\coM$-filtered maps follows from Lemma~\ref{lem:equivariance} and the identity
\[
N^j\Hiso(\gamma)-\Hiso N^j(\gamma) = \sum_{i=1}^j{j\choose i}\ad_N^i\Hiso\left(N^{j-i}(\gamma)\right).
\]
A similar argument establishes the final part.
\end{proof}
\end{lemma}

We now turn to the proof of Lemma~\ref{lem:equivariance}, beginning with a preliminary calculation which illustrates the kinds of calculations involved.

\begin{proposition}\label{prop:alog(delta)}
Let $\rGraph$ be a reduction graph. Then for any edge $e$ of $\rGraph$ we have
\[\alog(\delta_e)\equiv N(e^*)\text{ modulo $\W_{-2}\HAlg$.}\]
\begin{proof}
Write $\tau_e = \alog(\delta_e) - N(e^*) \in \gr^\M_{-2}\gr^\W_{-1}\HAlg$, so that we want to show that $\tau_e=0$ for every edge $e$. It suffices to prove:
\begin{enumerate}[label=\alph*),ref=\alph*]
	\item\label{proppart:alog_edge} $\tau_e+\tau_{e^{-1}}=0$ for all edges $e$;
	\item\label{proppart:alog_vertex} $\sum_{e\in\Edge{\Graph}\,,\,\source(e)=v}\tau_e=0$ for all vertices $v$; and
	\item\label{proppart:alog_loop} $\sum_{e\in\plEdge{\Graph}}\lambda_e\length(e)\tau_e\!=\!0$ for all homology classes $[\gamma]\!=\!\sum_{e\in\plEdge{\Graph}}\lambda_ee\in\H_1(\Graph)$.
\end{enumerate}
Indeed, the first two conditions say that $\sum_{e\in\plEdge{\Graph}}\tau_e\otimes e \in \gr^\W_{-1}\HAlg\otimes\H_1(\Graph)$ is a $\gr^\W_{-1}\HAlg$-valued homology class, and the third condition says that this lies in the kernel of the cycle pairing, and hence is zero.

\smallskip

Condition~\eqref{proppart:alog_edge} follows from the edge relations (Definition~\ref{def:pi1_reduction_graph}), which imply that $\alog(\delta_{e^{-1}})=-\Hiso(e)\alog(\delta_e)\Hiso(e)^{-1}\equiv-\alog(\delta_e)$ modulo $\W_{-2}\HAlg$. Condition~\eqref{proppart:alog_vertex} follows from the vertex relation (Remark~\ref{rmk:easier_vertex_relation}), which implies that
\[
\sum_{\substack{e\in\Edge{\Graph}\\\source(e)=v}}\tau_e=\sum_{\substack{e\in\Edge{\Graph}\\\source(e)=v}}\alog(\delta_e)=-\sum_{\substack{e\in\HEdge{\Graph}\\\source(e)=v}}\alog(\delta_e)-\alog(\delta_v)\in\W_{-2}\gr^\M_{-2}\HAlg.
\]

Finally, we note from Definition~\ref{def:inverse_iso_on_H_1} that we have an equality
\begin{equation}\label{eq:homology_classes}
[\gamma] = \sum_{e\in\plEdge{\Graph}}\lambda_eN(e^*) = \sum_{e\in\plEdge{\Graph}}\lambda_e\frac{\exp(\length(e)\ad_{e^*})-1}{\ad_{e^*}}(\alog(\delta_e))
\end{equation}
in $\gr^\M_{-2}\HAlg$. Condition~\eqref{proppart:alog_loop} follows by taking the image of this identity modulo~$\W_{-2}$.
\end{proof}
\end{proposition}

\begin{proof}[Proof of Lemma~\ref{lem:equivariance}]
We proceed by strong induction on $k$. One readily verifies that for composable elements $\gamma,\gamma'$ of $\gr^\M_\bullet\O(\pi_1^\Q(\rGraph))^\dual$ we have
\begin{equation}\label{eq:equivariance_composition}
\ad_N^k\Hiso(\gamma'\gamma) = \sum_{i+j=k}{k\choose i}\ad_N^i\Hiso(\gamma')\cdot\ad_N^j\Hiso(\gamma).
\end{equation}
Since by assumption $\ad_N^i\Hiso$ is $\W$-filtered of degree $-i-1$ for all $0<i<k$, it suffices to check that $\ad_N^k\Hiso$ is $\W$-filtered of degree $-k-1$ on generators, i.e.\ to verify the following:
\begin{enumerate}
	\item\label{lempart:equivariance_edges} $\ad_N^k\Hiso(e)\in\W_{-k-1}\HAlg$ for all rational edges $e$;
	\item\label{lempart:equivariance_loops} $\ad_N^k\Hiso(\gamma)\in\W_{-k-2}\HAlg$ for all loops $\gamma\in\pi_1^\Q(\Graph)$ in the underlying graph $\Graph$;
	\item\label{lempart:equivariance_deltas} $\ad_N^k\Hiso(\log(\delta_e))\in\W_{-k-2}\HAlg$ for all rational edges $e$;
	\item\label{lempart:equivariance_betas} $\ad_N^k\Hiso(\log(\beta_{v,i})),\ad_N^k\Hiso(\log(\beta_{v,i}'))\in\W_{-k-2}\HAlg$ for all vertices $v$ and all integers $1\leq i\leq g(v)$; and
	\item\label{lempart:equivariance_half-edge_deltas} $\ad_N^k\Hiso(\log(\delta_e))\in\W_{-k-3}\HAlg$ for all half-edges $e$.
\end{enumerate}
Of these, parts~\eqref{lempart:equivariance_betas} and~\eqref{lempart:equivariance_half-edge_deltas} are immediate, since the relevant quantities are $0$.

We begin by proving~\eqref{lempart:equivariance_edges}. For an edge $e$ and $s\in[0,\length(e)]\cap\Q$, we write $e_s$ for the rational edge from $\source(e)$ to the rational vertex a distance $s$ along $e$. It follows from Lemma~\ref{lem:C-K_on_edges} that the function $s\mapsto\ad_N^k\Hiso(e_s)$ is coordinatewise polynomial in $s$, and that its derivative is given by
\begin{equation}\label{eq:differentiate_adjoint_along_edge}
\scalebox{0.83}{
$\displaystyle{\dby{s}\ad_N^k\Hiso(e_s) = \sum_{i+j=k}{k\choose i}\Bigl(N^i(e^*)-iN^{i-1}\bigl(\exp(se^*)\alog(\delta_e)\exp(-se^*)\bigr)\Bigr)\cdot\ad_N^j\Hiso(e_s).}$
}
\end{equation}
(For instance, one can verify~\eqref{eq:differentiate_adjoint_along_edge} at $s=0$ using Lemma~\ref{lem:C-K_on_edges} and the description of $N$ in Proposition~\ref{prop:explicit_hopf_gpd} and then use~\eqref{eq:equivariance_composition}.)

Now our inductive hypothesis implies that $\ad_N^j\Hiso(e_s)\in\W_{-j}$ for all $0\leq j\leq k$. We also have $N^i(e^*)-iN^{i-1}\left(\exp(se^*)\alog(\delta_e)\exp(-se^*)\right)\in\W_{-i-1}$ for all $0\leq i\leq k$: for $i=0$ this is automatic; for $i=1$ this follows from Proposition~\ref{prop:alog(delta)}; and for $2\leq i\leq k$ this follows from our inductive assumption since $N^{i-1}\left(\exp(se^*)\alog(\delta_e)\exp(-se^*)\right)=\ad_N^{i-1}\Hiso(e_s\log(\delta_e)e_s^{-1})$. Putting this together with~\eqref{eq:differentiate_adjoint_along_edge} shows that $\dby{s}\ad_N^k\Hiso(e_s)\in\W_{-k-1}$, and hence we have $\ad_N^k\Hiso(e)\in\W_{-k-1}$ also. Thus, we have proven~\eqref{lempart:equivariance_edges}.

Next, we prove~\eqref{lempart:equivariance_loops}. For this, we note that for any edge $e$ we have $\ad_N^j\Hiso(e_s)\in\W_{-j-1}$ for all $0<j\leq k$, by our inductive hypothesis and part~\eqref{lempart:equivariance_edges}. It thus follows from~\eqref{eq:differentiate_adjoint_along_edge} that
\[
\dby{s}\ad_N^k\Hiso(e_s) \equiv N^k(e^*) - kN^{k-1}\left(\exp(se^*)\alog(\delta_e)\exp(-se^*)\right)
\]
modulo $\W_{-k-2}$. Integrating up and using Proposition~\ref{prop:loop_integral_along_edge}, we see that
\begin{equation}\label{eq:differential_of_adjoint_integrated}
\ad_N^k\Hiso(e) \equiv \length(e) N^k(e^*) - kN^{k-1}\left(\frac{\exp(\length(e)\ad_{e^*})-1}{\ad_{e^*}}\left(\alog(\delta_e)\right)\right)
\end{equation}
modulo $\W_{-k-2}$, for all rational edges $e$.

Now it follows from~\eqref{eq:equivariance_composition} that, modulo $\W_{-k-2}$, $\ad_N^k\Hiso$ is additive with respect to composition of paths in $\Graph$. Hence if $\gamma$ is a loop in $\Graph$ with associated homology class $[\gamma]=\sum_{e\in\plEdge{\Graph}}\lambda_ee$, we have
\[
\ad_N^k\Hiso(\gamma) \equiv \sum_{e\in\plEdge{\Graph}}\lambda_e\ad_N^k\Hiso(e) \equiv N^{k-1}([\gamma])-k N^{k-1}([\gamma]) = 0
\]
modulo $\W_{-k-2}$, using~\eqref{eq:differential_of_adjoint_integrated} and~\eqref{eq:homology_classes}. Thus, we have proven~\eqref{lempart:equivariance_loops}.

Finally, we prove~\eqref{lempart:equivariance_deltas}. We follow a similar strategy to Proposition~\ref{prop:alog(delta)}, showing that the quantities
\[
\tau_e := \ad_N^k\Hiso(\log(\delta_e)) = N^k\alog(\delta_e)
\]
in $\HAlg/\W_{-k-2}$ are all zero by verifying the equalities~\eqref{proppart:alog_edge}--\eqref{proppart:alog_loop}. Applying the operator $N$ to identity~\eqref{eq:differential_of_adjoint_integrated} shows that
\[
N\ad_N^k\Hiso(e) \equiv -k\length(e)N^k\alog(\delta_e)-k\frac{\length(e)^{k+1}}{k+1}\ad_{N(e^*)}^k\left(\alog(\delta_e)\right) \equiv -k\length(e)\tau_e
\]
modulo $\W_{-k-2}$, using Proposition~\ref{prop:alog(delta)}. Since $N\ad_N^k\Hiso$ is additive, modulo $\W_{-k-2}$, with respect to composition of paths, we obtain~\eqref{proppart:alog_edge} immediately, and~\eqref{proppart:alog_loop} follows from~\eqref{lempart:equivariance_loops}. Condition~\eqref{proppart:alog_vertex} follows from the vertex relations (cf.\ Remark \ref{rmk:easier_vertex_relation}), which imply that
\[
\sum_{\substack{e\in\Edge{\Gamma}\\\source(e)=v}}\tau_e = -N^k(\alog(\delta_v)) - \sum_{\substack{e\in\HEdge{\Gamma}\\\source(e)=v}}N^k(\alog(\delta_e)) = 0
\]
for all vertices $v$. Thus, we have proven~\eqref{lempart:equivariance_deltas}.
\end{proof}

\section{The non-abelian Kummer map for reduction graphs}
\label{s:n-a_kummer_for_graphs}

In this section, we finally relate our combinatorial constructions to the non-abelian Kummer map for curves, as defined in \S\ref{ss:nonab_Kummer}. To do so, we will first give an explicit combinatorial definition of a ``non-abelian Kummer map'' for a reduction graph, which we will show agrees with the corresponding map for curves in Theorem~\ref{thm:curve_kummer_is_graph_kummer}. This reduces the computation of the non-abelian Kummer map for a curve (as in Theorem~\ref{thm:description_of_M-graded}) to a purely combinatorial explicit calculation, which will be carried out in \S\ref{s:computation}.

\smallskip

To begin with, we define the space which will be the codomain of the non-abelian Kummer map, the space $\V$ appearing in Theorem~\ref{thm:description_of_kummer}.

\begin{definition}\label{def:V}
Let $\rGraph$ be a reduction graph, and let $\LAlg=\HAlg^\prim$ denote the bigraded Lie algebra of primitive elements of $\HAlg$. Thus $\LAlg$ is the bigraded pro-nilpotent Lie algebra generated by the same elements as $\HAlg$ in Construction~\ref{cons:HAlg}, subject to the same surface relation. We define the $\W$-graded subspace $\V\leq\gr^\M_{-2}\LAlg$ where $\gr^\W_{-k}\V$ consists of those elements such that $N^{k-1}(\gamma)=0$. In other words, $\V=\coM_{-2}\gr^\M_{-2}\LAlg$, where $\coM_\bullet$ is the filtration defined in Definition~\ref{def:coM}, restricted to~$\LAlg$.
\end{definition}

The definition of the non-abelian Kummer map for a reduction graph derives from the following easy observation.

\begin{proposition}
Let $\rGraph$ be a reduction graph $\gamma_{x,y}^\can\in\gr^\M_0\O(\rGraph;x,y)^\dual$ the canonical path between two rational vertices $x,y\in\QVert{\Graph}$. Then $\Hiso N(\gamma_{x,y}^\can)\in\V$, where $\Hiso\colon\gr^\M_{-2}\O(\pi_1^\Q(\rGraph;x,y))^\dual\isoarrow\gr^\M_{-2}\HAlg$ is the isomorphism from Theorem~\ref{thm:description_of_M-graded}.
\begin{proof}
There are two things to prove: that $\Hiso N(\gamma_{x,y}^\can)=\Hiso(\gamma_{y,x}^\can N(\gamma_{x,y}^\can))$ is a primitive element of $\HAlg$, and that it lies in $\coM_{-2}$. The former follows formally from the fact that $\gamma_{x,y}^\can$ is grouplike and $N$ is a coderivation. The latter follows from Lemma~\ref{lem:equivariance_coM} since $\Hiso N(\gamma_{x,y}^\can)=-(N\Hiso-\Hiso N)(\gamma_{x,y}^\can)\in\coM_{-2}\gr^\M_{-2}\HAlg$.
\end{proof}
\end{proposition}

\begin{definition}[Non-abelian Kummer map]\label{def:n-a_kummer_graph}\index{non-abelian Kummer map $\classify$}
Let $\rGraph$ be a reduction graph with a basepoint $b\in\QVert{\Graph}$. We define a map\[\classify\colon\QVert{\Graph}\rightarrow\V,\]called the \emph{non-abelian Kummer map}, by\[\classify(x):=\Hiso N(\gamma_{b,x}^\can).\]

For a positive integer $n$, we denote by $\classifyeq n\colon\QVert{\Graph}\rightarrow\gr^\W_{-n}\V$ the composite of $\classify$ with the projection $\V\twoheadrightarrow\gr^\W_{-n}\V$ coming from the $\W$-grading on $\V$, so that $\classify$ is the product of the maps $\classifyeq n$. We also denote by $\classifyleq n\colon\QVert{\Graph}\rightarrow\V/\W_{-n-1}=\prod_{r=1}^n\gr^\W_{-r}\V$ the composite of $\classify$ with the quotient map $\V\rightarrow\V/\W_{-n-1}$, so that $\classifyleq n$ is the product of the maps $\classifyeq r$ for $1\leq r\leq n$.
\end{definition}

\begin{remark}\label{rmk:n-a_kummer_indept_basepoint}
We are choosing to suppress the basepoint $b$ from the notation for the non-abelian Kummer map, since changing the basepoint only changes $\classify$ by an additive constant. This follows from the easily verified identity $\Hiso N(\gamma_{u,w}^\can)=\Hiso N(\gamma_{u,v}^\can)+\Hiso N(\gamma_{v,w}^\can)$, so that changing the basepoint from $b$ to $b'$ changes the non-abelian Kummer map by $\Hiso N(\gamma_{b',b}^\can)=-\classify(b')$.
\end{remark}

Before we embark on relating this graph-theoretically defined map $\classify$ to the usual non-abelian Kummer map for curves from \S\ref{ss:nonab_Kummer}, let us briefly record the following proposition, interpreting the non-abelian Kummer map as the map controlling the existence of $N$-invariant paths. This will not be directly used in the sequel, but provides some purely combinatorial motivation for the definition of $\classify$.

\begin{proposition}\label{prop:invt_paths_are_canonical}
Let $\rGraph$ be a reduction graph with a basepoint $b\in\QVert{\Graph}$, and let $n\in\N$. Then the map\[N\colon\gr^\M_0\O(\pi_1^\Q(\rGraph;x,y))^\dual/\W_{-n-1}\rightarrow\gr^\M_{-2}\O(\pi_1^\Q(\rGraph;x,y))^\dual/\W_{-n-1}\]has non-trivial kernel if and only if $\classifyleq n(x)=\classifyleq n(y)$, i.e.\ $\classifyeq r(x)=\classifyeq r(y)$ for all $1\leq r\leq n$. When this occurs, the kernel is one-dimensional, spanned by the canonical path $\gamma_{x,y}^\can$.
\begin{proof}
Since $\Hiso N(\gamma_{x,y}^\can)=\classifyleq n(y)-\classifyleq n(x)$, we just need to verify that the kernel of $N$ is always contained in the span of $\gamma_{x,y}^\can$. Suppose that $\gamma$ is an element of the kernel, and write $\Hiso(\gamma)=\gamma_0+\gamma_1$ with $\gamma_0\in\gr^\W_0\gr^\M_0\O(\pi_1^\Q(\rGraph;x,y))^\dual$ and $\gamma_1\in\prod_{r=1}^n\gr^\W_{-r}\gr^\M_0\O(\pi_1^\Q(\rGraph;x,y))^\dual$. It follows from Lemma~\ref{lem:equivariance_coM} that $N\Hiso(\gamma)=N(\gamma_1)\in\coM_{-2}\gr^\M_{-2}\HAlg$; this implies that $\gamma_1\in\coM_{-2}\gr^\M_0\HAlg$ by definition, and hence $\gamma_1=0$ by Proposition~\ref{prop:coM_basics}. Thus $\Hiso(\gamma)=\gamma_0\in\Q$, i.e.\ $\gamma$ is in the $\Q$-span of $\gamma_{x,y}^\can$ as desired.
\end{proof}
\end{proposition}

\subsection{Relation to non-abelian Kummer maps of curves}

Fix a smooth connected curve $X$ over a $p$-adic number field $K$, a basepoint $b\in X(K)$, and a prime $\ell$. Let $\rGraph$ denote the reduction graph of $X$, and choose as in Remark~\ref{rmk:pic-lef_for_groupoids} a $\W$- and $\M$-filtered and $N$-equivariant equivalence\[\vartheta\colon\pi_1^{\Q_\ell}(X_{\overline K})\isoarrow\pi_1^{\Q_\ell}(\rGraph)\]of $\Q_\ell$-pro-unipotent fundamental groupoids lying over the reduction map $X(\overline K)\rightarrow\QVert{\Graph}$, for some prime $\ell\neq p$. We write $U=\pi_1^{\Q_\ell}(X_{\overline K},b)$.

The remainder of this section will be devoted to proving the compatibility between the non-abelian Kummer maps of $X$ and $\rGraph$, proving the following result.

\begin{theorem}\label{thm:curve_kummer_is_graph_kummer}
There is an injection $\iota\colon\varinjlim_L\H^1(G_L,U))\hookrightarrow\V_{\Q_\ell}:=\prod_{n>0}(\Q_\ell\otimes\gr^\W_{-n}\V)$ fitting into a commuting square
\begin{center}
\begin{tikzcd}
X(\overline K) \arrow{r}{\classify}\arrow{d}{\red} & \varinjlim_L\H^1(G_L,U) \arrow[hook]{d}{\iota} \\
\QVert{\Graph} \arrow{r}{\classify} & \V_{\Q_\ell}
\end{tikzcd}
\end{center}
where the leftmost vertical arrow is the reduction map and the two horizontal maps are the non-abelian Kummer maps defined in \S\ref{ss:nonab_Kummer} and Definition~\ref{def:n-a_kummer_graph}.

The map $\iota$ is $\Q_\ell$-linear and $\W$-graded for graded vector space structure on the set $\varinjlim_L\H^1(G_L,U(\Q_\ell))$ described in Remark~\ref{rmk:cohomology_grading}, so in particular factors through injections $\varinjlim_L\H^1(G_L,U/\W_{-n-1})\hookrightarrow(\V/\W_{-n-1})_{\Q_\ell}$ for every $n$.
\end{theorem}

As a first step towards the proof of this theorem, let us construct the map $\iota\colon\varinjlim_L\H^1(G_L,U(\Q_\ell))\hookrightarrow\V_{\Q_\ell}$. A large part of this has already been achieved in Theorem~\ref{thm:cohomology_is_vector_space}, which provides an embedding
\[
\H^1(G_L,U) \isoarrow \Hom(\Q_\ell(1),\Lie(U)^\can)^{G_L} \hookrightarrow \Lie(U)^\can,
\]
where we identify $\Q_\ell(1)\nciso\Q_\ell$ using our choice of inertia generator $\sigma$.
We now identify the right-hand side with $\V_{\Q_\ell}$.

\begin{proposition}\label{prop:V_is_canonical}
Under the $\W$-filtered and $\M$-graded isomorphisms\[\gr^\M_\bullet\Lie(U)\xisoarrow{\vartheta}\gr^\M_\bullet\Lie\left(\pi_1^{\Q_\ell}(\rGraph,\red(b))\right)\xisoarrow{\Hiso}\LAlg_{\Q_\ell}\]from Theorem~\ref{thm:graph_comparison} and Theorem~\ref{thm:description_of_M-graded}, the subspace $\Lie(U)^\can\leq\gr^\M_{-2}\Lie(U)$ corresponds to the subspace $\V_{\Q_\ell}\leq\gr^\M_{-2}\LAlg_{\Q_\ell}$. The identification $\Lie(U)^\can\isoarrow\V_{\Q_\ell}$ is $\W$-graded with respect to the grading from Remark~\ref{rmk:cohomology_grading}.
\begin{proof}
If we let $\coM_\bullet$ denote the filtration on $\gr^\M_\bullet\Lie(U)$ defined as in Definition~\ref{def:coM}, then the definition of $\Lie(U)^\can$ asserts that $\Lie(U)^\can=\coM_{-2}\gr^\M_{-2}\Lie(U)$, and $\V:=\coM_{-2}\LAlg$ also by definition. The $N$-equivariant isomorphism $\vartheta$ is automatically a $\coM$-filtered isomorphism, as is $\Hiso$ by Lemma~\ref{lem:equivariance_coM}. The first part follows.

For the second part, we consider the square
\begin{center}
\begin{tikzcd}
\Lie(U)^\can \arrow{r}{\Hiso\vartheta}[swap]{\hiso}\arrow[hook]{d} & \V_{\Q_\ell} \arrow[hook]{d} \\
\prod_{k>1}\gr^\W_{-k}\gr^\M_{2-2k}\Lie(U) \arrow{r}{\Hiso\vartheta}[swap]{\hiso} & \prod_{k>1}\gr^\W_{-k}\gr^\M_{2-2k}\LAlg_{\Q_\ell},
\end{tikzcd}
\end{center}
where the vertical maps are given by (an appropriate projection of) $\prod_{k>1}N^{k-2}$. It follows from Lemma~\ref{lem:equivariance_coM} that the difference between the two diagonal composite maps lands in $\prod_{k>1}\coM_{-4}\gr^\W_{-k}\gr^\M_{2-2k}\LAlg_{\Q_\ell}$, which is zero by Proposition~\ref{prop:coM_basics}, and hence this square commutes. The bottom, right-hand and left-hand maps are all $\W$-graded (the latter by definition in Remark~\ref{rmk:cohomology_grading}), and hence so too is the top map.
\end{proof}
\end{proposition}

Combined with Theorem~\ref{thm:cohomology_is_vector_space} (and identifying $\Q_\ell(1)\nciso\Q_\ell$ using our choice of $\sigma\in I_K$), this produces the desired embedding $\iota\colon\varinjlim_L\H^1(G_L,U)\hookrightarrow\V_{\Q_\ell}$ as the composite
\[
\varinjlim_L\H^1(G_L,U)\hookrightarrow\Lie(U)^\can\isoarrow\V_{\Q_\ell}.
\]

The other half of the proof of Theorem~\ref{thm:curve_kummer_is_graph_kummer} is provided by the following result.

\begin{proposition}\label{prop:canonical_path_is_canonical_path}
The canonical path $\gamma_{\red(b),\red(x)}^\can\in\gr^\M_0\pi_1^{\Q_\ell}(\rGraph;\red(b),\red(x))$ as defined in Definition~\ref{def:can_paths} corresponds to the canonical element of $\gr^\M_0\pi_1^{\Q_\ell}(X_{\overline K};b,x)$ from Remark~\ref{rmk:canonical_elt} under the isomorphism
\[
\vartheta\colon\pi_1^{\Q_\ell}(X_{\overline K};b,x) \isoarrow \pi_1^{\Q_\ell}(\rGraph;\red(b),\red(x))
\]
from Theorem~\ref{thm:graph_comparison} (cf.\ Remark~\ref{rmk:pic-lef_for_groupoids}), for any $x\in X(\overline K)$.
\begin{proof}
The canonical element $\gamma_{b,x}^\can\in\gr^\M_0\pi_1^{\Q_\ell}(X_{\overline K};b,x)$ is characterised by the fact that $\log\left((\gamma_{b,x}^\can)^{-1}\exp(N)(\gamma_{b,x}^\can)\right)\in\Lie(U)^\can$. On the other hand, Proposition~\ref{prop:V_is_canonical} and an easy calculation shows that the canonical path satisfies
\[
\Hiso\log\left((\gamma_{\red(b),\red(x)}^\can)^{-1}\exp(N)\left(\gamma_{\red(b),\red(x)}^\can\right)\right) = \Hiso N(\gamma_{\red(b),\red(x)}^\can) \in \V
\]
We are then done by Proposition~\ref{prop:V_is_canonical}.
\end{proof}
\end{proposition}

\begin{proof}[Proof of Theorem~\ref{thm:curve_kummer_is_graph_kummer}]
For any $x\in X(\overline K)$, we have
\begin{align*}
\iota(\classify(x)) &= \Hiso(\vartheta((\gamma_{b,x}^\can)^{-1}\exp(N)(\gamma_{b,x}^\can))) \\
 &= \Hiso((\gamma_{\red(b),\red(x)}^\can)^{-1}N(\gamma_{\red(b),\red(x)}^\can)) = \classify(\red(x))
\end{align*}
in $\V_{\Q_\ell}$, using the calculation in the proof of Proposition~\ref{prop:canonical_path_is_canonical_path}.
\end{proof}
\section{Height pairings and harmonic analysis on graphs}
\label{s:harmonic_analysis}

To complete the proof of Theorem \ref{thm:description_of_kummer}, it remains for us to show that the non-abelian Kummer map associated to a reduction graph $\rGraph$ is computed by a height pairing against the measures on $\rGraph$ to be given in Construction~\ref{cons:measures}. For us, this height pairing will be a $\Q$-valued positive-semidefinite symmetric pairing on a certain space $\bMeas^0(\rGraph)$ of \emph{piecewise polynomial measures with log poles on $\HEdge{\rGraph}$ of total mass $0$}, generalising the more well-known height pairing on the space $\Div^0(\Graph)$ of divisors of total mass $0$. The construction of this height pairing, which we review in this section, comes from harmonic analysis on graphs, for which we follow \cite{baker-faber}\footnote{Our presentation in this section differs from that in \cite{baker-faber} in several ways. Firstly, our graphs are always \emph{rationally} metrised so that we obtain a $\Q$-valued theory, and are permitted multiple edges and loops. Secondly and more importantly, the graphs which we consider need not be compact, due to the presence of half-edges. In the interests of simple exposition, we will simply cite the corresponding analogous results from \cite{baker-faber} in this section, leaving it to the reader to make the necessary adjustments to the proofs.}.

\subsection{The height pairing on divisors}\label{ss:on_divisors}

Before constructing the height pairing in full generality, let us begin by recalling the construction of the height pairing on divisors on a rationally metrised graph $\Graph$. Recall that a \emph{divisor} on $\Graph$ is just a formal $\Q$-linear combination $\sum_{v\in\Vert{\Graph}}\lambda_v\cdot v$ of vertices of $\Graph$. The space of divisors is denoted $\Div(\Graph)=\Q\cdot\Vert{\Graph}$, and the subspace $\Div^0(\Graph)$ of \emph{divisors of total mass $0$} is defined to be the kernel of the sum-of-coordinates map $\Div(\Graph)\rightarrow\Q$.

There is a linear map $\dLapl\colon\Div(\Graph)\rightarrow\Div(\Graph)$ determined by the \emph{Laplacian matrix}\index{Laplacian\dots!\dots matrix $\dLapl$} $\dLapl$, whose entries are given by\[\dLapl_{uv}=\begin{cases}\sum_{\source(e)=u\neq\target(e)}\frac1{\length(e)}&\text{if $u=v$,}\\-\sum_{\source(e)=u,\target(e)=v}\frac1{\length(e)}&\text{if $u\neq v$.}\end{cases}\]The Laplacian is a symmetric matrix, has one-dimensional kernel spanned by $\sum_{v\in\Vert{\Graph}}v$, and has image exactly the space $\Div^0(\Graph)$ \cite[Corollary 2]{baker-faber}. The \emph{height pairing} $\langle\mathbf u,\mathbf v\rangle$ of two elements $\mathbf u,\mathbf v\in\Div^0(\Graph)$ is then defined by\[\langle\mathbf u,\mathbf v\rangle:=\transp{\mathbf u}\dLapl^{-1}\mathbf v\](where $\dLapl^{-1}\mathbf v$ denotes any preimage of $\mathbf v$ -- the value of $\langle\mathbf u,\mathbf v\rangle$ is independent of this choice). This is a positive-definite symmetric pairing on $\Div^0(\Graph)$, which moreover satisfies the following additional inequality, coming from the theory of electrical circuits \cite[\S4]{baker-faber}.

\begin{proposition}[{\cite[Exercise~9]{baker-faber}}, {\cite[Lemma~2.17]{baker-faber:extra}}]\label{prop:bounds_on_potentials}
Let $u,v,x$ be vertices of a rationally metrised graph $\Graph$. Then we have\[0\leq\langle x-u,v-u\rangle\leq\langle v-u,v-u\rangle.\]We have equality $0=\langle x-u,v-u\rangle$ (resp.\ $\langle x-u,v-u\rangle=\langle v-u,v-u\rangle$) if and only if $x=u$, $u=v$, or $x$ and $v$ lie in different connected components of $\Graph\setminus\{u\}$ (resp.\ $x=v$, $u=v$, or $x$ and $u$ lie in different connected components of $\Graph\setminus\{v\}$).
\end{proposition}

Finally, there is a close link between the height pairing on $\Div^0(\Graph)$ and the cycle pairing on $\H_1(\Graph)$. Indeed, there is a short exact sequence\[0\rightarrow\H_1(\Graph)\rightarrow\Q\cdot\mnEdge{\Graph}\overset\partial\rightarrow\Div^0(\Graph)\rightarrow0\]where $\partial(e):=\target(e)-\source(e)$. The natural pairing on $\Q\cdot\mnEdge{\Graph}$ (for which the edges $e$ form an orthogonal basis with $\langle e,e\rangle=\length(e)$, as in \S\ref{s:cheng-katz}) restricts to the cycle pairing on $\H_1(\Graph)$, and also induces the height pairing on $\Div^0(\Graph)$ as follows.

\begin{lemma}\label{lem:height_vs_cycle}
The map $\partial\colon\Q\cdot\mnEdge{\Graph}\rightarrow\Div^0(\Graph)$ is the projection onto an orthogonal direct summand, i.e.\ its adjoint $\partial^*\colon\Div^0(\Graph)\rightarrow\Q\cdot\mnEdge{\Graph}$ is a section of $\partial$, and hence induces an orthogonal direct sum decomposition\[\Q\cdot\mnEdge{\Graph}=\H_1(\Graph)\oplus\Div^0(\Graph).\]In this, $\H_1(\Graph)$, $\Div^0(\Graph)$ and $\Q\cdot\mnEdge{\Graph}$ are endowed respectively with the cycle pairing, height pairing and natural pairing from \S\ref{s:cheng-katz}.
\begin{proof}
The adjoint is given by
\begin{equation*}\label{eq:adjoint_of_bdary}
\partial^*(\mu) = \sum_{e\in\plEdge{\Graph}}\frac{\langle\partial(e),\mu\rangle}{\length(e)}e.\tag{\dag}
\end{equation*}
To verify that it is a section of $\partial$, we write $\mu=\sum_v\lambda_v\cdot v=\Lapl\Phi$, so that
\begin{align*}
\partial\partial^*(\mu) &= \sum_{e\in\plEdge{\Graph}}\frac{\langle\partial(e),\mu\rangle}{\length(e)}\partial(e) = -\sum_{e\in\Edge{\Graph}}\frac{\langle\partial(e),\mu\rangle}{\length(e)}\source(e) \\
 &= -\sum_{v\in\Vert{\Graph}}\left(\sum_{\source(e)=v}\frac{\Phi(\target(e))-\Phi(\source(e))}{\length(e)}\right)\cdot v=\sum_{v\in\Vert{\Graph}}\lambda_v\cdot v=\mu
\end{align*}
as desired.
\end{proof}
\end{lemma}

As a consequence of this lemma, we obtain an explicit formula for the isomorphism $N\colon\H^1(\Graph)\isoarrow\H_1(\Graph)$. The constants $\lambda_{e,e'}$ here will appear again in Construction~\ref{cons:measures}.

\begin{corollary}\label{cor:computing_N(e*)}
For every edge $e$ of $\Graph$ we have\[N(e^*)=\sum_{e'\in\plEdge{\Graph}}\lambda_{e,e'}\cdot e'\in\Q\cdot\mnEdge{\Graph}\]where\[\lambda_{e,e'}:=\begin{cases}\frac1{\length(e)}-\frac{\langle\partial(e),\partial(e)\rangle}{\length(e)^2}&\text{if $e'=e$,}\\-\frac{\langle\partial(e),\partial(e')\rangle}{\length(e)\length(e')}&\text{if $e'\neq e^{\pm1}$.}\end{cases}\]
\begin{proof}
By definition, $N(e^*)$ is the orthogonal projection to $\H_1(\Graph)$ of the element $\frac1{\length(e)}e\in\Q\cdot\mnEdge{\Graph}$. The orthogonal projection of this element to the complementary subspace $\Div^0(\Graph)$ is $\sum_{e'\in\plEdge{\Graph}}\frac{\langle\partial(e),\partial(e')\rangle}{\length(e)\length(e')}e'$, giving the desired equality.
\end{proof}
\end{corollary}

\subsection{The height pairing on measures}
\label{ss:height_pairing_divisors}

In order to describe the non-abelian Kummer map in terms of a height pairing, we need to generalise the above pairing on divisors in two ways. Firstly, we need to extend the domain of the pairing from divisors to \emph{measures}, in essence allowing the points in a divisor to be ``spread out'' along edges of $\Graph$. Secondly, we need to take account of the half-edges of $\Graph$ by incorporating boundary conditions at $\HEdge{\Graph}$. Thus, the domain of the height pairing will be the following space.

\begin{definition}\label{def:measures}\index{measure, $\bMeas(\Graph)$}
Let $\Graph$ be a rationally metrised graph. A \emph{piecewise polynomial measure on $\Graph$ with log poles at $\HEdge{\Graph}$}\footnote{We will often simply refer to these as \emph{measures} in the interests of brevity.} is a formal sum\[\mu=\sum_{e\in\Edge{\Graph}}g_e(s_e)\cdot|\d s_e|+\sum_{v\in\Vert{\Graph}}\lambda_v\cdot v+\sum_{e\in\HEdge{\Graph}}\lambda_e\cdot e\]with $g_e\in\Q[s_e]$ and $\lambda_v,\lambda_e\in\Q$, such that $g_{e^{-1}}(s)=g_e(\length(e)-s)$ for all edges $e$. Here the $s_e$ are formal variables indexed by edges $e$: the element $|\d s_e|$ is thought of as arc length along edge $e$, and the element $v$ is thought of as a delta distribution located at $v$.

We denote by $\bMeas(\Graph)$ the $\Q$-vector space of all such measures, and define $\bMeas^0(\Graph)$ to be the codimension $1$ subspace of \emph{measures of total mass $0$}, i.e.\ the subspace consisting of elements $\mu\in\bMeas(\Graph)$ such that\[\sum_{e\in\plEdge{\Graph}}\int_0^{\length(e)}\!\!\!g_e(s_e)\d s_e+\sum_{v\in\Vert{\Graph}}\lambda_v+\sum_{e\in\HEdge{\Graph}}\lambda_e=0,\]
where the integrals here denote usual line integrals of the polynomials $g_e(s_e)$ with respect to the parameter $s_e$.
\end{definition}

As with the height pairing on divisors, the height pairing on measures will be defined via a certain Laplacian operator. However, unlike the Laplacian matrix whose domain and codomain are equal, the polynomial Laplacian we will use is best described as a map from a space of \emph{functions} to our space $\bMeas(\Graph)$ of measures.

\begin{definition}[{cf.\ \cite[Definition 5]{baker-faber}}]\label{def:poly_laplacian}\index{Laplacian\dots!\dots operator $\Lapl$}
Let $\Graph$ be a reduction graph. For a function $f\colon\QVert{\Graph}\rightarrow\Q$, we will write $f_e\colon[0,\length(e)]\cap\Q\rightarrow\Q$ (resp.\ $f_e\colon[0,\infty)\cap\Q\rightarrow\Q$) for the restriction of $f$ to an edge $e$ (resp.\ half-edge $e$). We say that $f$ is \emph{piecewise polynomial with log poles on $\HEdge{\Graph}$} just when $f_e$ is a polynomial function for every edge $e$ and $f_e$ is an affine function (polynomial function of degree $\leq1$) for every half-edge $e$. Equivalently\footnote{Strictly speaking, this alternative description is only valid when $\Graph$ is not the graph with a single vertex and no edges or half-edges.}, a piecewise polynomial function $f$ with log poles on $\HEdge{\Graph}$ is a formal linear combination $f=\sum_ef_e(s_e)\in\bigoplus_{e\in\Edge{\Graph}\cup\HEdge{\Graph}}\Q[s_e]$ such that:
\begin{itemize}
    \item for all edges or half-edges $e_0$, $e_1$ with $\source(e_0)=\source(e_1)$, we have $f_{e_0}(0)=f_{e_1}(0)$;
    \item for every edge $e$, we have $f_{e^{-1}}(s)=f_e(\length(e)-s)$; and
    \item for every half-edge $e$, $\deg(f_e)\leq1$.
\end{itemize}

We denote the space of all piecewise polynomial functions on $\Graph$ with log poles on $\HEdge{\Graph}$ by $\Omlog(\Graph)$, and define the \emph{Laplacian operator}\[\Lapl\colon\Omlog(\Graph)\rightarrow\bMeas(\Graph)\]by\[\Lapl(f):=-\sum_{e\in\Edge{\Graph}}f_e''(s_e)\cdot|\d s_e|-\sum_{v\in\Vert{\Graph}}\left(\sum_{\source(e)=v}f_e'(0)\right)\cdot v+\sum_{e\in\HEdge{\Graph}}f_e'\cdot e.\]
\end{definition}

\begin{remark}\label{rmk:compare_laplacians}
The Laplacian operator $\Lapl$ as defined in Definition~\ref{def:poly_laplacian} can be viewed as a generalisation of the Laplacian matrix $\dLapl$ in \S\ref{ss:height_pairing_divisors}. There is an embedding $F\colon\Div(\Graph)\hookrightarrow\Omlog(\Graph)$, taking a divisor $\sum_{v}\lambda_v\cdot v$ to the unique function $f\colon\QVert{\Graph}\rightarrow\Q$ which is constant on half-edges and which linearly interpolates between the values $\lambda_v$ at each vertex $v$. With respect to the evident inclusion $\Div(\Graph)\hookrightarrow\bMeas(\Graph)$, the Laplacian operator $\Lapl\colon\Omlog(\Graph)\rightarrow\bMeas(\Graph)$ restricts to the map $\dLapl\colon\Div(\Graph)\rightarrow\Div(\Graph)$ induced by the Laplacian matrix \cite[Theorem 4]{baker-faber}.
\end{remark}

\begin{remark}\label{rmk:stabilised_laplacians}
If $\Graph'$ is a subdivision of $\Graph$, then it is clear that $\Omlog(\Graph)\subseteq\Omlog(\Graph')$, and we obtain an embedding $\Meas(\Graph)\hookrightarrow\Meas(\Graph')$ by specifying that $g_e(s_e)\cdot|\d s_e|\mapsto g_e(s_{e_1})\cdot|\d s_{e_1}|+g_e(s_{e_2}+\length(e_1))\cdot|\d s_{e_2}|$ in the notation of Definition~\ref{def:rational_vertices}. It is easily checked that the Laplacian operator $\Lapl$ is compatible with these inclusions, as are the integration and height pairings to be introduced shortly. When proving results using harmonic analysis, we will thus permit ourselves to tacitly replace $\Graph$ with a subdivision, for instance allowing divisors to be linear combinations of \emph{rational} vertices. In all cases, the precise choice of subdivision of $\Graph$ in which the argument takes place is immaterial, so we will suppress if from our notation.
\end{remark}

As for the Laplacian matrix, the Laplacian operator is almost invertible.

\begin{lemma}\label{lem:laplacian_nearly_iso}
Let $\Graph$ be a rationally metrised graph. Then the Laplacian operator $\Lapl$ has one-dimensional kernel spanned by the constant function $1$, and has image exactly $\bMeas^0(\Graph)$.
\begin{proof}
It is obvious that $\Lapl(1)=0$. That the kernel is spanned by $1$ is \cite[Theorem 3]{baker-faber}. That the image is contained in $\bMeas^0(\Graph)$ follows from \cite[Corollary~1]{baker-faber}.

Since it will be useful in practical applications to be able to explicitly invert the Laplacian operator $\Lapl$, we will take a little time to explain how to produce polynomial functions $f$ such that $\Lapl(f)$ is equal to a given $\mu\in\bMeas^0(\Graph)$. Writing $\mu$ as in Definition~\ref{def:measures}, we take for each edge $e$ a double antiderivative $h_e(s_e)=\iint g_e(s_e)\d s_e\d s_e$, normalised so that $h_e(0)=h_e(\length(e))=0$. If we let\[h=\sum_{e\in\Edge{\Graph}}h_e(s_e)-\sum_{e\in\HEdge{\Graph}}\lambda_e\cdot s_e,\]then $h\in\Omlog(\Graph)$ and we have that $\mu+\Lapl(h)\in\Div^0(\Graph)$. Hence by Remark~\ref{rmk:compare_laplacians}, a preimage of $\mu$ is given by $\dLapl^{-1}(\mu+\Lapl(h))-h$, where $\dLapl$ is the Laplacian matrix and $\dLapl^{-1}(\mu+\Lapl(h))$ is interpreted as a piecewise affine function on $\QVert{\Graph}$ as in Remark~\ref{rmk:compare_laplacians}.
\end{proof}
\end{lemma}

Using the Laplacian operator, we may now construct the desired height pairing\[\langle\cdot,\cdot\rangle\colon\bMeas^0(\Graph)\otimes\bMeas^0(\Graph)\rightarrow\Q\]analogously to the height pairing on divisors.

\begin{definition}[Height pairing on measures]\label{def:height_pairing}
Let $\Graph$ be a rationally metrised graph. We define an integration pairing
\begin{align*}
\int\colon\Omlog(\Graph)\otimes\bMeas(\Graph)&\rightarrow\Q \\
\int\colon f\otimes\mu&\mapsto\int f\d\mu
\end{align*}
where
\[\int f\d\mu:=\sum_{e\in\plEdge{\Graph}}\int_0^{\length(e)}\!\!\!f_e(s_e)g_e(s_e)\d s_e+\sum_{v\in\Vert{\Graph}}\lambda_vf(v)+\sum_{e\in\HEdge{\Graph}}\lambda_ef(\source(e))\]in the notation of Definitions~\ref{def:measures} and~\ref{def:poly_laplacian}. Thus $\bMeas^0(\Graph)$ is the kernel of the map $\mu\mapsto\int1\d\mu$.

We define the \emph{height pairing} $\langle\cdot,\cdot\rangle\colon\bMeas^0(\Graph)\otimes\bMeas^0(\Graph)\rightarrow\Q$ to be the pairing given by\[\langle\mu,\nu\rangle:=\int\invLapl(\mu)\d\nu\](where $\invLapl(\mu)\in\Omlog(\Graph)$ denotes any preimage of $\mu$; the value of $\langle\mu,\nu\rangle$ is independent of this choice). This is a positive-semidefinite\footnote{The kernel of the height pairing is the span of the elements $e-\source(e)$ for $e$ a half-edge of $\Graph$. Hence the height pairing is positive-definite if and only if $\Graph$ has no half-edges.} symmetric pairing on $\bMeas^0(\Graph)$ \cite[Theorem 1]{baker-faber}, extending the height pairing on $\Div^0(\Graph)$.
\end{definition}

Aside from positive-semidefiniteness and symmetry of the height pairing, the only other result we will need is that the height pairing already suffices to invert the Laplacian operator $\Lapl$. This is essentially formal.

\begin{lemma}\label{lem:height_pairing_x-b}
Let $\Graph$ be a rationally metrised graph and $\mu\in\bMeas^0(\Graph)$. For any fixed $b\in\QVert{\Graph}$, the function
\begin{align*}
f\colon\QVert{\Graph}&\rightarrow\Q \\
f\colon x&\mapsto\langle x-b,\mu\rangle
\end{align*}
is piecewise polynomial with log poles on $\HEdge{\Graph}$, and $\Lapl(f)=\mu$.
\begin{proof}
By symmetry of the height pairing, we know\[f(x)=\int\invLapl(\mu)\cdot\d(x-b)=\invLapl(\mu)(x)-\invLapl(\mu)(b).\]This is piecewise polynomial in $x$ with log poles on $\HEdge{\Graph}$, and $\Lapl(f)=\mu$.
\end{proof}
\end{lemma}

\section{Computing the non-abelian Kummer map}
\label{s:computation}

In this section, we will carry out the promised computation of local non-abelian Kummer maps of curves, by computing the non-abelian Kummer map of a reduction graph from Definition~\ref{def:n-a_kummer_graph}. Using the theory outlined in the previous section, we will express this map in terms of height pairings against certain measures. These are the measures appearing in Theorem~\ref{thm:description_of_kummer}.

\begin{construction}\label{cons:measures}
Let $\rGraph$ be a reduction graph, and let $\V$ denote the space defined in Definition~\ref{def:V}. We define a series $(\mu_n)_{n>0}$ of $\gr^\W_{-n}\V$-valued measures on $\Graph$ of total mass $0$, i.e.\ elements of $\gr^\W_{-n}\V\otimes\bMeas^0(\Graph)$, as follows. We set $\mu_1=0$ and 
\[
\mu_2:=-\sum_{e\in\Edge{\rGraph}}\ad_{e^*}(N(e^*))\cdot|\d s_e|+\sum_{v\in\Vert{\rGraph}}\log(\delta_v)\cdot v+\sum_{e\in\HEdge{\rGraph}}\log(\delta_e)\cdot e,
\]
where $\log(\delta_v):=\sum_{i=1}^{g(v)}[\log(\beta_{v,i}'),\log(\beta_{v,i})]$. For $n\geq3$, we proceed recursively, setting 
\[
\scalebox{0.93}{$\displaystyle{
\mu_n:=-\!\!\sum_{e\in\Edge{\rGraph}}\!\left(2\ad_{e^*}\classify_{n-1,e}'-\ad_{e^*}^2\classify_{n-2,e}+\sum_{e'\in\Edge{\rGraph}}\lambda_{e,e'}\ad_{e^*}\ad_{(e')^*}\int_0^{\length(e')}\!\!\!\classify_{n-2,e}\d s_e\right)\!\cdot\!|\d s_e|,
}$}
\]
where $\classify_{r,e}$ denotes the restriction to edge $e$ of the function $x\mapsto\langle x-b,\mu_r\rangle$ for some fixed basepoint $b\in\QVert{\Graph}$, $\classify_{r,e}'$ denotes its derivative with respect to arc-length, and
\[
\lambda_{e,e'}=\begin{cases}\frac1{\length(e)}-\frac{\langle\partial(e),\partial(e)\rangle}{\length(e)^2}&\text{if $e'=e$}\\-\frac{\langle\partial(e),\partial(e')\rangle}{\length(e)\length(e')}&\text{if $e'\neq e^{\pm1}$,}\end{cases}
\]
as in Corollary~\ref{cor:computing_N(e*)}.

\noindent(It is not immediately clear that these measures are of total mass $0$ or that they are independent of the basepoint $b$: this will be proved during this section.)
\end{construction}

\begin{theorem}\label{thm:description_of_graph_kummer}
Let $\rGraph$ be a reduction graph, and let $(\mu_n)_{n>0}$ be the sequence of $\gr^\W_{-n}\V$-valued measures on (the underlying graph of) $\rGraph$ constructed in Construction~\ref{cons:measures}. For any basepoint $b\in\QVert{\Graph}$, the components of the non-abelian Kummer map $\classify\colon\QVert{\Graph}\rightarrow\V$ (Definition~\ref{def:n-a_kummer_graph}) are given by\[\classifyeq n(x)=\langle x-b,\mu_n\rangle.\]
\end{theorem}

Combined with Theorem~\ref{thm:curve_kummer_is_graph_kummer}, this completes the proof of Theorem~\ref{thm:description_of_kummer}.

\begin{remark}\label{rmk:j_1}
It is easy to see that $\gr^\W_{-1}\V=0$, so that Theorem~\ref{thm:description_of_graph_kummer} only really has content for $n\geq2$.
\end{remark}

Fix a reduction graph $\rGraph$ and a basepoint $b\in\QVert{\rGraph}$, writing $\classify\colon\QVert{\rGraph}\rightarrow\V$ for the non-abelian Kummer map (Definition~\ref{def:n-a_kummer_graph}) and $(\mu_n)_{n>0}$ for the series of $\gr^\W_{-n}\V$-valued measures on $\rGraph$ constructed in Construction~\ref{cons:measures}. The remainder of this section will be devoted to proving that each $\classifyeq n$ is a piecewise polynomial function with log poles on $\HEdge{\rGraph}$, satisfying $\Lapl(\classifyeq n)=\mu_n$ for all $n>0$. This in particular proves Theorem~\ref{thm:description_of_graph_kummer}, since $\classifyeq n(x)$ and $\langle x-b,\mu_n\rangle$ differ by a constant by Lemmas~\ref{lem:height_pairing_x-b} and~\ref{lem:laplacian_nearly_iso}, and both vanish at $x=b$. Note that this also proves that the $\mu_n$ have total mass $0$ and are independent of the basepoint $b$ (using Remark~\ref{rmk:n-a_kummer_indept_basepoint}).

\subsection{The non-abelian Kummer map along edges}

We will ultimately deduce the equalities $\Lapl(\classifyeq n)=\mu_n$ from an explicit formula describing the non-abelian Kummer map $\classify$ along edges and half-edges of $\rGraph$. This is a straightforward consequence of Corollary~\ref{cor:can_path_along_edge}, describing the canonical path along edges and half-edges.

\begin{proposition}\label{prop:j_along_edges}
For any edge or half-edge $e$ of $\rGraph$, we have\[\classify_e(s)=\exp(se^*)\left\lbrack\classify(\source(e))+s\Hiso\log(\delta_e)\right\rbrack\exp(-se^*)+\exp(se^*)N_b(\exp(-se^*))\]where $N_b=\Hiso_b\circ N\circ\Hiso_b^{-1}$ with $\Hiso_b\colon\O(\pi_1^\Q(\rGraph,b))^\dual\isoarrow\HAlg$ the isomorphism from Theorem~\ref{thm:description_of_M-graded}.
\begin{proof}
Corollary~\ref{cor:can_path_along_edge} assures us that the canonical path from $b$ to the rational vertex $v_{e,s}$ a distance $s$ along $e$ is given by $e_s\gamma_{b,\source(e)}^\can\Hiso^{-1}(\exp(-se^*))$. Applying $\Hiso N$ and using that $\Hiso(e_s)=\exp(se^*)$ by Lemma~\ref{lem:C-K_on_edges}, we obtain the desired formula.
\end{proof}
\end{proposition}

\begin{corollary}\label{cor:j_piecewise_polynomial}
For each $n>0$, the map $\classifyeq n\colon\QVert{\rGraph}\rightarrow\gr^\W_{-n}\V$ is piecewise polynomial of degree $\leq n$ with log poles on $\HEdge{\rGraph}$. Its restriction to any bridge (edge $e$ such that $\rGraph\setminus\{e^{\pm1}\}$ is disconnected) is linear.
\begin{proof}
The class of functions $\Q\rightarrow\HAlg$ whose component in $\gr^\W_{-n}\HAlg$ is a polynomial function of degree $\leq n$ is closed under pointwise addition and multiplication. It follows from Proposition~\ref{prop:j_along_edges} that each $s\mapsto\classify_e(s)$ lies in this class. In the particular case that $e$ is a half-edge or a bridge, we have $e^*=0$ and we see that $\classify_e(s)=\classify(\source(e))+s\Hiso\log(\delta_e)$ is linear in $s$.
\end{proof}
\end{corollary}

Note that Corollary~\ref{cor:j_piecewise_polynomial}, implies that the Zariski closure of the image of $\classify\colon\QVert{\rGraph}\rightarrow\V$ has dimension $\leq1$; combined with Theorem~\ref{thm:curve_kummer_is_graph_kummer} we obtain the second part of Proposition~\ref{prop:kummer_bounds} (except the equality condition).

As a direct consequence of the formula in Proposition~\ref{prop:j_along_edges}, we obtain the following identities for $\Lapl(\classify)$ in terms of $N_b$.

\begin{proposition}\label{prop:diff_eqn_for_j}
The map $\classify\colon\QVert{\rGraph}\rightarrow\V$ satisfies the following properties:
\begin{itemize}
	\item for all half-edges $e$ we have $\classify_e'=\Hiso\log(\delta_e)$;
	\item for all vertices $v$ we have\[\sum_{\source(e)=v}\classify_e'(0)=-\log(\delta_v):=-\sum_{i=1}^{g(v)}[\log(\beta_{v,i}'),\log(\beta_{v,i})];\text{ and}\]
	\item for all edges (and half-edges) $e$ we have\[\classify_e''-2\ad_{e^*}\classify_e'+\ad_{e^*}^2\classify_e=\ad_{e^*}N_b(e^*).\]
\end{itemize}
\begin{proof}
Differentiating the identity in Proposition~\ref{prop:j_along_edges} gives
\begin{equation}\label{eq:once-differentiated_j}
\classify_e'-\ad_{e^*}\classify_e=\exp(se^*)\Hiso\log(\delta_e)\exp(-se^*)-N_b(e^*).\tag{$\ast$}
\end{equation}
The first identity is now obvious since $e^*=0$ for half-edges $e$. For the second identity, we know that\[\sum_{\source(e)=v}\classify_e'(0)=\sum_{\source(e)=v}\ad_{e^*}\classify(v)+\sum_{\source(e)=v}\Hiso\log(\delta_e)+\sum_{\source(e)=v}N_b(e^*).\]The first and last sums on the right-hand side vanish since $\sum_{\source(e)=v}e^*=0$ (Remark~\ref{rmk:e*e_is_identity}), and the middle sum is equal to $-\log(\delta_v)$ by the vertex relations from Definition~\ref{def:pi1_reduction_graph}. The third identity comes from differentiating the above expression for $\classify_e'-\ad_{e^*}\classify_e$.
\end{proof}
\end{proposition}

To complete our description of $\Lapl(\classify)$, we want to eliminate the dependence on $N_b$ in Proposition~\ref{prop:diff_eqn_for_j}. We do this by expressing each $N_b(e^*)$ back in terms of $N(e^*)$ and $\classify$.

\begin{proposition}\label{prop:integral_for_N_b}
For every edge $e$ we have\[N_b(e^*)=N(e^*)+\sum_{e'\in\plEdge{\rGraph}}\lambda_{e,e'}\ad_{(e')^*}\int_0^{\length(e')}\!\!\!\classify_{e'}\d s_{e'},\]where\[\lambda_{e,e'}:=\begin{cases}\frac1{\length(e)}-\frac{\langle\partial(e),\partial(e)\rangle}{\length(e)^2}&\text{if $e'=e$,}\\-\frac{\langle\partial(e),\partial(e')\rangle}{\length(e)\length(e')}&\text{if $e'\neq e^{\pm1}$}\end{cases}\]as in Construction~\ref{cons:measures} and Corollary~\ref{cor:computing_N(e*)}.
\begin{proof}
In $\gr^\M_{-2}\HAlg$ we have the following equalities:
\begin{align*}
\sum_{e'\in\plEdge{\rGraph}}\lambda_{e,e'}\int_0^{\length(e')}\!\!\!\classify_e'\d s_{e'} &= 0; \\
\sum_{e'\in\plEdge{\rGraph}}\lambda_{e,e'}\int_0^{\length(e')}\!\!\!\exp(s(e')^*)\Hiso\log(\delta_{e'})\exp(-s(e')^*)\d s_{e'} &= N(e^*); \\
\sum_{e'\in\plEdge{\rGraph}}\lambda_{e,e'}\int_0^{\length(e')}\!\!\!N_b((e')^*)\d s_{e'} &= N_b(e^*).
\end{align*}
The first of these identities follows from the fact that $N(e^*)=\sum_{e'\in\plEdge{\rGraph}}\lambda_{e,e'}e'$ lies in $\H_1(\rGraph)$, while the third is a consequence of the fact that $N(e^*)-\frac1{\length(e)}e$ lies in the orthogonal complement of $\H_1(\rGraph)$ in $\Q\cdot\mnEdge{\rGraph}$ (see Corollary~\ref{cor:computing_N(e*)}). The second identity follows from Proposition~\ref{prop:loop_integral_along_edge} and Definition~\ref{def:inverse_iso_on_H_1}.

Integrating the identity (\ref{eq:once-differentiated_j}) from the proof of Proposition~\ref{prop:diff_eqn_for_j} along edges $e'$ of~$\rGraph$ and taking a $\Q$-linear combination weighted by the $\lambda_{e,e'}$, we find that\[-\sum_{e'\in\plEdge{\rGraph}}\lambda_{e,e'}\int_0^{\length(e')}\!\!\!\ad_{(e')^*}\classify_{e'}\d s_{e'}=N(e^*)-N_b(e^*),\]which rearranges to the desired expression.
\end{proof}
\end{proposition}

We now prove Theorem~\ref{thm:description_of_graph_kummer}.
\begin{proof}[Proof of Theorem~\ref{thm:description_of_graph_kummer}]
Combining Propositions~\ref{prop:diff_eqn_for_j} and~\ref{prop:integral_for_N_b} shows that $\Lapl(\classify)$ is given by\[\Lapl(\classify)=\sum_{e\in\Edge{\rGraph}}g_e(s_e)\cdot|\d s_e|+\sum_{v\in\Vert{\rGraph}}\log(\delta_v)+\sum_{e\in\HEdge{\Graph}}\log(\delta_e),\]where\[g_e=-2\ad_{e^*}\classify_e'+\ad_{e^*}^2\classify_e-\ad_{e^*}N(e^*)-\sum_{e'\in\plEdge{\rGraph}}\lambda_{e,e'}\ad_{e^*}\ad_{(e')^*}\int_0^{\length(e')}\!\!\!\classify_{e'}\d s_{e'}.\]The result follows by taking $\W$-homogenous components, since all the terms $e^*$, $N(e^*)$, $\log(\delta_v)$ and $\log(\delta_e)$ appearing are $\W$-homogenous of degrees $-1$, $-1$, $-2$ and $-2$ respectively.
\end{proof}

We will also use the following consequence of these calculations, refining Corollary~\ref{cor:j_piecewise_polynomial} by determining the leading coefficients of the coefficients $\classifyeq n$ of $\classify$.

\begin{proposition}\label{prop:leading_coeffs}
For each $n>0$, the restriction of the function $\classifyeq n\colon\QVert{\rGraph}\rightarrow\gr^\W_{-n}\V$ to each edge $e$ of $\rGraph$ is a polynomial in arc-length $s_e$ of the form\[\classifyeqe ne(s_e)=\frac{n-1}{n!}\ad_{e^*}^{n-1}(N(e^*))\cdot s_e^n+O(s_e^{n-1}).\]
\begin{proof}
Taking $n$th graded pieces of the second-order differential equation in Proposition~\ref{prop:diff_eqn_for_j} and differentiating provides the identity\[\classifyeqe ne^{(n)}-2\ad_{e^*}\classifyeqe{n-1}e^{(n-1)}+\ad_{e^*}^2\classifyeqe{n-2}e^{(n-2)}=\begin{cases}\ad_{e^*}(N(e^*))&\text{if $n=2$,}\\0&\text{if $n>2$,}\end{cases}\]where we are using Proposition~\ref{prop:integral_for_N_b} to deduce that the $(-1)$th $\W$-graded piece of $N_b(e^*)$ is $N(e^*)$. Since $\classifyeq0=0$ trivially and $\classifyeq1=0$ by Remark~\ref{rmk:j_1}, we find by an easy induction that $\classifyeqe ne^{(n)}=(n-1)\ad_{e^*}^{n-1}(N(e^*))$, as desired.
\end{proof}
\end{proposition}

\subsection{A preliminary application to quadratic Chabauty}

Let $X/K$ be a hyperbolic curve with $\Q_\ell$-pro-unipotent \'etale fundamental group $U(X)$, and let $U$ be a quotient of $U(X)/\W_{-3}$ satisfying~\WM. From the explicit description of $\classify_2$ in Corollary \ref{cor:j_piecewise_polynomial}, we obtain the following property of the map
\[
\classify_U\colon X(K)\rightarrow\H^1(G_K,U)
\]
which is used in \cite{BD}. Let $q_e$ denote the function on an edge $e$
\[
x\mapsto \frac{1}{2}x(x-l(e))
\]
where we identify $e$ with $[0,l(e)]$ via $s_e$. Given a $\Q_\ell$-vector space $V$ and $D=\sum D_v\cdot v\in\Div^0(\Graph)\otimes V$, let $F(D)$ denote the unique piecewise linear function $\QVert{\Graph}\rightarrow V$ which is linear on each edge and takes value $D_v$ at $v$ as in Remark \ref{rmk:compare_laplacians}.

\begin{corollary}\label{not_even_cor}
We have $\classify_2(x)=g(x)-g(b)$, where
\begin{align*}
g & =\sum _{e\in\Edge{\Graph}}[e^*,\log(\delta_e)]\cdot q_e+\sum_{e\in\HEdge{\Graph}}\log(\delta_e)\cdot s_e \\ 
& \:+F\biggl(\dLapl ^{-1}\Bigl(\sum_{v\in\Vert{\Graph}}\bigl(\log(\delta_v)-\frac12\sum_{\substack{e\in\Edge{\Graph} \\ \source(e)=v}}\length(e)[e^*,\log(\delta_e)]+\sum_{\substack{e\in\HEdge{\Graph} \\ \source(e)=v}}\log(\delta_e)\bigr)\cdot v\Bigr)\biggr)
\end{align*}
where $\dLapl^{-1}$ is the inverse Laplacian matrix.
\end{corollary}
\begin{proof} 
In $\gr_{-2}^\W\V$, we have $N(e^*)=\log(\delta_e)$. Hence this follows from Theorem \ref{thm:curve_kummer_is_graph_kummer} and Theorem \ref{thm:description_of_graph_kummer}.
\end{proof}

In \cite{BD}, the following result is used to determine rational points of the curve
\[
X:y^2 = x^6+ax^4+ax^2+1,
\]
for $a$ such that the elliptic curve 
\[
E:y^2 = x^3+ax^2+ax+1
\]
has rank $2$. The isogeny $\Jac(X)\sim E\times E $ induces an isomorphism
\[
\bigwedgesquare\H^1_\et(X_{\overline K},\Q_\ell)\iso\Sym^2\H^1_\et(E_{\overline K},\Q_\ell)\oplus\Q_\ell(-1)^{\oplus3}.
\]
which defines a quotient $U_{\Sym}$ of $U_2(X)$ which is an extension of $\H^1_\et(X_{\overline K},\Q_\ell)^\dual $ by $\Sym^2V_E:=\Sym^2\H^1_\et(E_{\overline K},\Q_\ell)^\dual$.

\begin{lemma}\label{lemma:used_in_BD}
Suppose $X/K$ is of the form 
\[
y^2 = x^6+ax^4+ax^2+1,
\]
that $X$ has split semi-stable reduction, and that the special fibre of a stable model is isomorphic to two elliptic curves meeting at one point. Then there is linear map $l\colon\Q\rightarrow\H^1(G_K,\Sym^2V_E)=\H^1(G_K,U_{\Sym})$ such that $\classify_{U_{\Sym}}$ is the composite of an affine embedding of the reduction graph of $X$ into $\Q$ with $l$.
\end{lemma}
\begin{proof}
The reduction graph of $X$ consists of two vertices of genus $1$ connected by a single edge $e$. That $\classify_{U_{\Sym}}$ is an affine function on $e$ is given by Corollary~\ref{cor:j_piecewise_polynomial}.
\end{proof}
\section{Operations on reduction graphs}
\label{s:graph_ops}

Now that we have proved Theorem~\ref{thm:description_of_kummer}, we turn our attention to its principal application: the relative Oda--Tamagawa criterion (Theorem~\ref{thm:groupoid_oda}). We will see that this follows immediately from a certain combinatorial analogue for stable reduction graphs (Theorem~\ref{thm:combinatorial_injectivity}). In order to facilitate the proof of this result, we will in this section study three simplifying operations (Lemmas~\ref{lem:half-edge_reduction},~\ref{lem:resistance_reduction} and~\ref{lem:block_reduction}) on reduction graphs which will allow us to reduce computations for general graphs to simpler subquotients. We will also introduce the notion of \emph{maximal cut sets} as in \cite{AMO}, which will play a major role in our analysis of the weight $-2$ part of the Kummer map in \S\ref{ss:partial_injectivity_weight_-2}.

\subsection{Elimination of half-edges}
\label{ss:half-edge_removal}

As an illustration of the type of simplifying operations we will be using, consider a reduction graph $\rGraph$ with a half-edge $e_0$, and write $\rGraph'=\rGraph\setminus\{e_0\}$ for the same reduction graph with $e_0$ removed. We let\[\rho\colon\QVert{\Graph}\rightarrow\QVert{\Graph'}\]denote the map sending all of $e_0$ to $\source(e_0)\in\Vert{\Graph'}$ and defined in the obvious way outside $e_0$. We also denote by $\rho_*\colon\HAlg(\rGraph)\rightarrow\HAlg(\rGraph')$ the bigraded map sending $\log(\delta_{e_0})$ to $0$ and acting in the obvious manner on all other generators of $\HAlg(\rGraph)$.

For us, the important property of the map $\rho$ is that it is compatible with the non-abelian Kummer map for graphs.

\begin{lemma}[Elimination of half-edges]\label{lem:half-edge_reduction}
The non-abelian Kummer maps for $\rGraph$ and $\rGraph'$ (at basepoints $b\in\QVert{\Graph}$ and $\rho(b)\in\QVert{\Graph'}$) fit into a commuting square
\begin{center}
\begin{tikzcd}
\QVert{\Graph} \arrow{r}{\classify}\arrow{d}{\rho} & \V(\rGraph) \arrow{d}{\rho_*} \\
\QVert{\Graph'} \arrow{r}{\classify} & \V(\rGraph')
\end{tikzcd}
\end{center}
where the right-hand vertical map is the restriction of $\rho_*\colon\HAlg(\rGraph)\rightarrow\HAlg(\rGraph')$.
\end{lemma}

The proof of this result (and all others like it in this section) follows from a corresponding compatibility result with the description of the $\M$-graded fundamental groupoid in Theorem~\ref{thm:description_of_M-graded}.

\begin{proposition}\label{prop:half-edge_compatibility}
The map $\rho\colon\QVert{\Graph}\rightarrow\QVert{\Graph'}$ extends to a morphism\[\rho_*\colon\pi_1^\Q(\rGraph)\rightarrow\pi_1^\Q(\rGraph')\]of $\Q$-pro-unipotent groupoids\footnote{In order to define these groupoids in Definition~\ref{def:pi1_reduction_graph}, one strictly needs to choose orderings on the sets of edges incident to vertices, and the pushforward maps on fundamental groupoids only make sense if these orderings are chosen compatibly in a certain sense. In the interests of brevity, we will omit mention of these orderings in this section, leaving it to the reader to order the products in the correct manner.} such that the induced map\[\rho_*\colon\O(\pi_1^\Q(\rGraph))^\dual\rightarrow\O(\pi_1^\Q(\rGraph'))^\dual\]of complete Hopf groupoids is $N$-equivariant and $\W$- and $\M$-filtered. The induced map $\gr^\M_\bullet\rho_*\colon\gr^\M_\bullet\O(\pi_1^\Q(\rGraph))^\dual\rightarrow\gr^\M_\bullet\O(\pi_1^\Q(\rGraph'))^\dual$ fits into a commuting square
\begin{equation}\label{eq:half-edge_compatibility}
\begin{tikzcd}
\gr^\M_\bullet\O(\pi_1^\Q(\rGraph))^\dual \arrow{r}{\Hiso}[swap]{\hiso}\arrow{d}{\gr^\M_\bullet\rho_*} & \HAlg(\rGraph) \arrow{d}{\rho_*} \\
\gr^\M_\bullet\O(\pi_1^\Q(\rGraph'))^\dual \arrow{r}{\Hiso}[swap]{\hiso} & \HAlg(\rGraph')
\end{tikzcd}
\end{equation}
where the horizontal maps are the isomorphisms from Theorem~\ref{thm:description_of_M-graded}.
\end{proposition}

Before we prove this proposition, we briefly explain how this implies Lemma~\ref{lem:half-edge_reduction}.

\begin{proof}[Proof of Lemma~\ref{lem:half-edge_reduction}]
Commutativity of \eqref{eq:half-edge_compatibility} in $\M$-degree $0$ implies in particular that $\gr^\M_\bullet\rho_*(\gamma_{b,x}^\can)=\gamma_{\rho(b),\rho(x)}^\can$ for all rational vertices $x$. Hence by \eqref{eq:half-edge_compatibility} and $N$-equivariance we have\[\rho_*(\classify(x))=\rho_*\Hiso N(\gamma_{b,x}^\can)=\Hiso N\left(\gr^\M_\bullet\rho_*(\gamma_{b,x}^\can)\right)=\Hiso N(\gamma_{\rho(b),\rho(x)}^\can)=\classify(\rho(x))\]as desired.
\end{proof}

The proof of Proposition~\ref{prop:half-edge_compatibility} proceeds in three steps. None of these steps is overly complicated, but we will nonetheless take care to set them out clearly as this proof will be a microcosm of more complicated proofs to come.

\subsubsection*{Step $1$: Construction of $\rho_*$}

The pushforward map $\rho_*\colon\pi_1^\Q(\rGraph)\rightarrow\pi_1^\Q(\rGraph')$ is defined to send $\delta_{e_0}$ to $1$ (the identity based at $\source(e_0)$), and to act on the other generators of $\pi_1^\Q(\rGraph)$ from Definition~\ref{def:pi1_reduction_graph} in the obvious manner. One readily verifies that this map is well-defined and the induced map on complete Hopf groupoids is $N$-equivariant and $\W$- and $\M$-filtered.

\subsubsection*{Step $2$: Commutativity in $\M$-degree $0$}

Commutativity of \eqref{eq:half-edge_compatibility} in $\M$-degree $0$ is equivalent to the assertion that\[\int_{\rho_*(\gamma)}\underline\omega=\int_\gamma\underline\omega\]for all paths $\gamma$ in $\rGraph$ and all $\underline\omega\in\Shuf\H_1(\Graph')=\Shuf\H_1(\Graph)$. By the properties of the higher cycle pairing, this needs only be checked when $\gamma=e$ is an edge of (some subdivision of) $\rGraph$ and $\underline\omega=\omega_1\dots\omega_n$  is a pure tensor of length $n$. This is obvious when $e$ is an edge of $\Graph'$ already, and when $e$ is contained in $e_0$, both sides are equal to $1$ if $n=0$ and $0$ otherwise.

\subsubsection*{Step $3$: Commutativity in $\M$-degrees $<0$}

From commutativity of \eqref{eq:half-edge_compatibility} in $\M$-degree $0$, we see that $\gr^\W_\bullet\rho_*\colon\gr^\M_\bullet\O(\pi_1^\Q(\rGraph))^\dual\rightarrow\gr^\M_\bullet\O(\pi_1^\Q(\rGraph'))^\dual$ takes canonical paths in $\Graph$ to canonical paths in $\Graph'$. Hence it is compatible with the trivialisations of these complete Hopf groupoids constructed in \S\ref{ss:trivialisation}, so the objects (resp.\ morphisms) in the above square can be thought of as complete Hopf algebras (resp.\ morphisms thereof). Thus commutativity can be checked on generators of $\HAlg(\rGraph)$, identified with elements of the complete Hopf algebra $\gr^\M_\bullet\O(\pi_1^\Q(\rGraph))^\dual$ via the map $\Hiso^{-1}$ constructed in \S\ref{s:M-trivialisation}. On $\H^1(\Graph)$ (in $\M$-degree $0$) this has already been checked, and for all other generators this is obvious.

This completes the proof of Proposition~\ref{prop:half-edge_compatibility}.\qed

\subsection{Resistance reduction}
\label{ss:resistance_reduction}

The second type of reduction we will use in our arguments involves contracting down certain subgraphs to single edges. This is inspired by the theory of electrical networks, and corresponds to the operation of replacing a part of an electrial network with a single resistor with resistance equal to the total resistance across the subnetwork.

Thus consider a reduction graph $\rGraph$ containing a subgraph $\Cpt$ (of its underlying graph $\Graph$) and two distinct vertices $w_0,w_1\in\Vert{\Cpt}$ such for every vertex $w$ of $\Cpt$ other than $w_0$ and $w_1$, every edge or half-edge of $\rGraph$ incident to $w$ lies in $\Cpt$. We write $\rGraph'$ for the reduction graph formed from $\rGraph$ by replacing $\Cpt$ with a single edge $e_\Cpt$ from $w_0$ to $w_1$ (and its inverse) of length\[\length(e_\Cpt):=\langle w_1-w_0,w_1-w_0\rangle_\Cpt,\]where $\langle\cdot,\cdot\rangle_\Cpt$ denotes the height pairing on $\Cpt$. We declare the genus of a vertex $v\in\Vert{\Graph'}$ to be the same as the genus of the corresponding vertex of $\rGraph$.
\begin{figure}[!h]
\begin{center}
\begin{tikzpicture}
\draw[thick,dashed,color=darkgray,fill=lightgray,opacity=0.4] (0,1) circle (1);
\coordinate (w0) at (0,0);
\draw[fill=black] (w0) circle (2pt);
\coordinate (w1) at (1,1);
\draw[fill=black] (w1) circle (2pt);
\coordinate (o) at (0,1);
\draw[fill=black] (o) circle (2pt);
\coordinate (u) at (0,1.5);
\draw[fill=black] (u) circle (2pt);
\coordinate (v) at (-0.5,1);
\draw[fill=black] (v) circle (2pt);
\draw[thick] (u) -- (o) -- (w0) -- (v) -- (u) -- (w1) -- (o) -- (v);
\draw[thick] (w1) -- (0.5,1.5);
\draw[thick,dashed] (0.5,1.5) -- (0,2);
\coordinate (a) at (1,0);
\draw[fill=black] (a) circle (2pt);
\coordinate (b) at (2,0);
\draw[fill=black] (b) circle (2pt);
\coordinate (c) at (2,1);
\draw[fill=black] (c) circle (2pt);
\coordinate (d) at (1.5,0.5);
\draw[fill=black] (d) circle (2pt);
\draw[thick] (w0) -- (a) -- (b) -- (c) -- (d) -- (a) -- (w1) -- (c);
\draw[thick] (b) to[bend left] (d);
\draw[thick] (b) to[bend right] (d);
\draw[thick] (w1) -- (1.5,1.5);
\draw[thick,dashed] (1.5,1.5) -- (2,2);
\draw[thick,dashed,color=darkgray,fill=lightgray,opacity=0.4] (5,1) circle (1);
\coordinate (w0') at (5,0);
\draw[fill=black] (w0') circle (2pt);
\coordinate (w1') at (6,1);
\draw[fill=black] (w1') circle (2pt);
\coordinate (a') at (6,0);
\draw[fill=black] (a') circle (2pt);
\coordinate (b') at (7,0);
\draw[fill=black] (b') circle (2pt);
\coordinate (c') at (7,1);
\draw[fill=black] (c') circle (2pt);
\coordinate (d') at (6.5,0.5);
\draw[fill=black] (d') circle (2pt);
\draw[thick] (w0') -- (a') -- (b') -- (c') -- (d') -- (a') -- (w1') -- (c');
\draw[thick] (b') to[bend left] (d');
\draw[thick] (b') to[bend right] (d');
\draw[thick] (w1') -- (6.5,1.5);
\draw[thick,dashed] (6.5,1.5) -- (7,2);
\draw[thick] (w0') to[bend left] (w1');
\end{tikzpicture}
\end{center}
\label{fig:resistance_reduction}
\caption{On the left, a graph $\Graph$ with a suitable subgraph $\Cpt$ indicated by the light grey dashed disc -- the two vertices on its circumference are $w_0$ and $w_1$. On the right, the graph $\Graph'$ formed by replacing $\Cpt$ with a single edge.}
\end{figure}
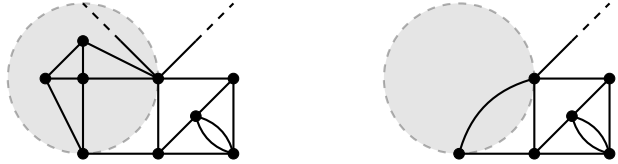

There is a map\[\rho\colon\QVert{\Graph}\rightarrow\QVert{\Graph'},\]defined in the obvious way outside $\Cpt$, and sending a rational vertex $v\in\Cpt$ to the rational vertex a distance $\langle v-w_0,w_1-w_0\rangle_\Cpt$ along $e_\Cpt$. Proposition~\ref{prop:bounds_on_potentials} ensures that $\langle w-w_0,w_1-w_0\rangle_\Cpt\in[0,\length(e_\Cpt)]$, so that this map is well-defined.

After suitable subdivision of $\rGraph$ and $\rGraph'$, the map $\rho$ is even induced by a morphism $\rho\colon\Graph\rightarrow\Graph'$ of underlying graphs (where our graph morphisms are permitted to contract edges and half-edges to single points). In particular, $\rho$ induces pushforward maps
\begin{align*}
\rho_*\colon\H_1(\Graph)&\rightarrow\H_1(\Graph'), \\
\rho_*\colon\H^1(\Graph)&\rightarrow\H^1(\Graph'),
\end{align*}
which are compatible under the isomorphisms $N\colon\H^1(\Graph)\isoarrow\H_1(\Graph)$ and $N\colon\H^1(\Graph')\isoarrow\H_1(\Graph')$ induced by the cycle pairing (see \S\ref{s:cheng-katz}). We will see later in Remark~\ref{rmk:resistance_reduction_projection} that these maps extend to a bigraded morphism\[\rho_*\colon\HAlg(\rGraph)\rightarrow\HAlg(\rGraph')\](sending generators of $\HAlg(\rGraph)$ of the form $\log(\beta_{v,i})$, $\log(\beta_{v,i}')$ or $\log(\delta_e)$ to either the corresponding generator of $\HAlg(\rGraph')$ or zero as appropriate). These pushforward maps allow us to state our desired compatibility result between the non-abelian Kummer maps associated to $\rGraph$ and $\rGraph'$, exactly as in Lemma~\ref{lem:half-edge_reduction}.

\begin{lemma}[Resistance reduction]\label{lem:resistance_reduction}
The non-abelian Kummer maps for $\rGraph$ and $\rGraph'$ fit into a commuting square
\begin{center}
\begin{tikzcd}
\QVert{\Graph} \arrow{r}{\classify}\arrow{d}{\rho} & \V(\rGraph) \arrow{d}{\rho_*} \\
\QVert{\Graph'} \arrow{r}{\classify} & \V(\rGraph')
\end{tikzcd}
\end{center}
where the right-hand vertical map is the restriction of $\rho_*\colon\HAlg(\rGraph)\rightarrow\HAlg(\rGraph')$.
\end{lemma}

This will follow from the following compatibility result with the description of the $\M$-graded fundamental groupoid in Theorem~\ref{thm:description_of_M-graded}. The proof of Lemma~\ref{lem:resistance_reduction} from this Proposition~\ref{prop:resistance_compatibility} is exactly the same as the proof of Lemma~\ref{lem:half-edge_reduction} from Proposition~\ref{prop:half-edge_compatibility}.

\begin{proposition}\label{prop:resistance_compatibility}
The map $\rho\colon\QVert{\Graph}\rightarrow\QVert{\Graph'}$ extends to a morphism\[\rho_*\colon\pi_1^\Q(\rGraph)\rightarrow\pi_1^\Q(\rGraph')\]of $\Q$-pro-unipotent groupoids such that the induced map\[\rho_*\colon\O(\pi_1^\Q(\rGraph))^\dual\rightarrow\O(\pi_1^\Q(\rGraph'))^\dual\]of complete Hopf groupoids is $N$-equivariant and $\W$- and $\M$-filtered. The induced map $\gr^\M_\bullet\rho_*\colon\gr^\M_\bullet\O(\pi_1^\Q(\rGraph))^\dual\rightarrow\gr^\M_\bullet\O(\pi_1^\Q(\rGraph'))^\dual$ fits into a commuting square
\begin{equation}\label{eq:resistance_compatibility}
\begin{tikzcd}
\gr^\M_\bullet\O(\pi_1^\Q(\rGraph))^\dual \arrow{r}{\Hiso}[swap]{\hiso}\arrow{d}{\gr^\M_\bullet\rho_*} & \HAlg(\rGraph) \arrow{d}{\rho_*} \\
\gr^\M_\bullet\O(\pi_1^\Q(\rGraph'))^\dual \arrow{r}{\Hiso}[swap]{\hiso} & \HAlg(\rGraph')
\end{tikzcd}
\end{equation}
where the horizontal maps are the isomorphisms from Theorem~\ref{thm:description_of_M-graded}.
\end{proposition}

Our proof of Proposition~\ref{prop:resistance_compatibility} follows the same structure as the proof of Proposition~\ref{prop:half-edge_compatibility}.

\subsubsection*{Step $1$: Construction of $\rho_*$}

To begin with, we explain how to construct the pushforward\[\rho_*\colon\pi_1^\Q(\rGraph)\rightarrow\pi_1^\Q(\rGraph')\]on $\Q$-pro-unipotent fundamental groupoids by describing the image of the generators from Definition~\ref{def:pi1_reduction_graph}. Generators coming from edges $e$ are sent to the generator coming from $\rho(e)$, and generators of the form $\beta_{v,i}$, $\beta_{v,i}'$ or $\delta_e$ for vertices $v$ or half-edges $e$ are sent to the corresponding generator of $\pi_1^\Q(\rGraph')$. Finally, for edges $e$ of $\rGraph$ we define\[\rho_*(\delta_e):=\delta_{\rho(e)}^{\length(\rho(e))/\length(e)},\]recalling that $\length(\rho(e))=\langle\partial(e),w_1-w_0\rangle_\Cpt$ when $e\in\Edge{\Cpt}$, and $\length(\rho(e))=\length(e)$ otherwise.

That this map is well-defined follows straightforwardly from the following identity.

\begin{proposition}\label{prop:quotient_edge_lengths}
For any edge $e'\in\Edge{\Graph'}$ we have\[\frac1{\length(e')}=\sum_{\rho(e)=e'}\frac1{\length(e)}.\]
\begin{proof}
If $e'$ is not the image of an edge in $\Cpt$ this is obvious. Otherwise, suppose without loss of generality that $e'$ is oriented from $w_0$ to $w_1$ in the image of $\Cpt$. Removing from $\Cpt$ all edges mapping to $e'$ disconnects it into two components\footnote{These two components are in fact connected, but the proof doesn't require this.} We let $\Phi_{e'}\in\Omlog(\Cpt)$ denote the function which is equal to $0$ (resp.\ $1$) on the component containing $w_0$ (resp.\ $0$), and linearly interpolating between these values on the edges mapping to $e'$. We see that $\Lapl(\Phi_{e'})=\sum_{\rho(e)=e'}\frac1{\length(e)}\partial(e)$ so that\[1=\Phi_{e'}(w_1)-\Phi_{e'}(w_0)=\left\langle\sum_{\rho(e)=e'}\frac1{\length(e)}\partial(e),w_1-w_0\right\rangle=\sum_{\rho(e)=e'}\frac{\length(e')}{\length(e)},\]as desired.
\end{proof}
\end{proposition}

The induced map $\rho_*\colon\O(\pi_1^\Q(\rGraph))^\dual\rightarrow\O(\pi_1^\Q(\rGraph'))^\dual$ is clearly $\W$- and $\M$-filtered, and $N$-equivariance can be verified on the standard generators of $\O(\pi_1^\Q(\rGraph))^\dual$. This is trivial for all generators except edges $e$, and for these generators $N$-equivariance follows from the definition of $\rho_*$.

\subsubsection*{Step $2$: Commutativity in $\M$-degree $0$}

The next step in our proof of Proposition~\ref{prop:resistance_compatibility} is to prove the commutativity of \eqref{eq:resistance_compatibility} in $\M$-degree $0$. As in the proof of Proposition~\ref{prop:half-edge_compatibility}, this amounts to an assertion that the adjoint $\rho^*\colon\H_1(\Graph')\rightarrow\H_1(\Graph)$ with respect to the cycle pairing on $\H_1$ is also the adjoint with respect to the higher cycle pairing on $\O(\pi_1^\Q(\Graph))^\dual$. This will follow from the following compatibility between various pushforward maps.

\begin{proposition}\label{prop:harmonic_quotient}
The pushforward map $\rho_*\colon\Q\cdot\mnEdge{\Graph}\rightarrow\Q\cdot\mnEdge{\Graph'}$ is the direct sum of the pushforward maps
\begin{align*}
\rho_*\colon\H_1(\Graph) &\rightarrow \H_1(\Graph') \\
\rho_*\colon\Div^0(\Graph) &\rightarrow \Div^0(\Graph')
\end{align*}
with respect to the orthogonal decomposition in Lemma~\ref{lem:height_vs_cycle}.
\begin{proof}
The pushforward maps fit into a morphism
\begin{center}
\begin{tikzcd}
0 \arrow{r} & \H_1(\Graph) \arrow{r}\arrow{d}{\rho_*} & \Q\cdot\mnEdge{\Graph} \arrow{r}{\partial}\arrow{d}{\rho_*} & \Div^0(\Graph) \arrow{r}\arrow{d}{\rho_*} & 0 \\
0 \arrow{r} & \H_1(\Graph') \arrow{r} & \Q\cdot\mnEdge{\Graph'} \arrow{r}{\partial} & \Div^0(\Graph) \arrow{r} & 0
\end{tikzcd}
\end{center}
of exact sequences, so it suffices to prove that the adjoint $\rho^*\colon\Q\cdot\mnEdge{\Graph'}\rightarrow\Q\cdot\mnEdge{\Graph}$ of $\rho_*$ takes $\H_1(\Graph')$ into $\H_1(\Graph)$.

Now this adjoint map is given by\[\rho^*(e')=\length(e')\sum_{\rho(e)=e'}\frac1{\length(e)}e,\]from which it is clear that homology classes in $\Graph'$ not meeting the image of $\Cpt$ are taken to homology classes in $\Graph$. To deal with the remaining classes, it suffices to prove that\[\partial\rho^*(e_\Cpt)=w_1-w_0.\]

To show this, let us write also $\rho\in\Omlog(\Cpt)$ for the function $v\mapsto\langle v-w_0,w_1-w_0\rangle_\Cpt$, so that we have\[\Lapl(\rho)=\sum_{e\in\plEdge{\Cpt}}\frac{\langle\partial(e),w_1-w_0\rangle_\Cpt}{\length(e)}\partial(e)=\partial\rho^*(e_\Cpt)\]since $\langle\partial(e),w_1-w_0\rangle_\Cpt=\epsilon_e\length(\rho(e))$, with $\epsilon_e=\sgn(\rho(\target(e))-\rho(\source(e)))\in\{\pm1,0\}$. But Lemma~\ref{lem:height_pairing_x-b} shows that $\Lapl(\rho)=w_1-w_0$, which is what we wanted to show.
\end{proof}
\end{proposition}

\begin{remark}\label{rmk:resistance_reduction_projection}
Combining Proposition~\ref{prop:quotient_edge_lengths} and the formula for $\rho^*$ in the proof of Proposition~\ref{prop:harmonic_quotient} shows that the three pushforward maps $\rho_*$ in Proposition~\ref{prop:harmonic_quotient} are orthogonal projections onto direct summands, that is, satisfy $\rho^*\rho_*=1$ with their adjoints $\rho^*$. Since the pushforward $\rho_*\colon\H^1(\Graph)\rightarrow\H^1(\Graph')$ is by definition the dual to $\rho^*\colon\H_1(\Graph')\rightarrow\H_1(\Graph)$, we see that the combined pushforward map\[\rho_*\colon\H^1(\Graph)\otimes\H_1(\Graph)\rightarrow\H^1(\Graph')\otimes\H_1(\Graph')\]takes the element of $\H^1(\Graph)\otimes\H_1(\Graph)=\End(\H_1(\Graph))$ corresponding to the identity to the corresponding element of $\H^1(\Graph')\otimes\H_1(\Graph')$. It follows that the pushforward map $\oHAlg(\rGraph)\rightarrow\oHAlg(\rGraph')$ induced by the pushforward on $\H^1$ and $\H_1$ preserves the surface relation from Construction~\ref{cons:HAlg}, and hence factors through a map $\rho_*\colon\HAlg(\rGraph)\rightarrow\HAlg(\rGraph')$ as we claimed earlier.
\end{remark}

With the compatibility result from Proposition~\ref{prop:harmonic_quotient} it is now straightforward to prove commutativity of \eqref{eq:resistance_compatibility} in $\M$-degree $0$, equivalently that we have the identity\[\int_{\rho_*(\gamma)}\underline\omega'=\int_\gamma\rho^*(\underline\omega')\]for all paths $\gamma$ in $\Graph$ and all $\underline\omega\in\Shuf\H_1(\Graph')$, where $\rho^*\colon\H_1(\Graph')\rightarrow\H_1(\Graph)$ is the adjoint to $\rho_*$ (i.e.\ the dual of $\rho_*\colon\H^1(\Graph)\rightarrow\H^1(\Graph')$). Indeed, this needs only be checked when $\gamma=e$ is an edge of $\rGraph$ and $\underline\omega=\omega_1\dots\omega_n\in\H_1(\Graph')^{\otimes n}$ is a pure tensor of length $n$. But in this case the desired identity is\[\frac1{n!}\prod_i\int_{\rho_*(e)}\omega_i=\frac1{n!}\prod_i\int_e\rho^*(\omega_i),\]which holds by Proposition~\ref{prop:harmonic_quotient}.

\subsubsection*{Step $3$: Commutativity in $\M$-degrees $<0$}

It remains to show that \eqref{eq:resistance_compatibility} commutes in all $\M$-degrees. As in the proof of Proposition~\ref{prop:half-edge_compatibility}, the map\[\gr^\M_\bullet\rho_*\colon\gr^\M_\bullet\O(\pi_1^\Q(\rGraph))^\dual\rightarrow\gr^\M_\bullet\O(\pi_1^\Q(\rGraph'))^\dual\]can be viewed as a morphism of complete Hopf algebras (i.e.\ is compatible with the trivialisations of both sides provided by canonical paths as in \S\ref{ss:trivialisation}), and so we may check commutativity on generators of $\HAlg(\rGraph)$.

The generators $\H^1(\Graph)$ in $\M$-degree $0$ have already been dealt with, and the generators arising from genus of vertices and from half-edges are trivial to check. For generators arising from $\H_1(\Graph)$ in $\M$-degree $-2$, we first note that $\rho_*(e^*)=\frac{\length(\rho(e))}{\length(e)}\rho(e)^*$ -- this follows for instance from Proposition~\ref{prop:harmonic_quotient} since the homology class corresponding to $e^*$ is the projection to $\H_1(\Graph)$ of $\frac1{\length(e)}e\in\Q\cdot\mnEdge{\Graph}$ (cf.\ Corollary~\ref{cor:computing_N(e*)}). Thus if $\gamma=\sum_{e\in\plEdge{\Graph}}\lambda_e\cdot e\in\H_1(\Graph)$, then it follows from Definition~\ref{def:inverse_iso_on_H_1} that we have
\begin{align*}
\Hiso\circ\gr^\M_\bullet\rho_*\circ\Hiso^{-1}(\gamma) &= \sum_{e\in\plEdge{\Graph}}\lambda_e\frac{\exp(\length(e)\ad_{\rho_*(e^*)})-1}{\ad_{\rho_*(e^*)}}\cdot\Hiso\log(\rho_*(\delta_e)) \\
 &= \sum_{e\in\plEdge{\Graph}}\lambda_e\frac{\exp(\length(\rho(e))\ad_{\rho(e)^*})-1}{\ad_{\rho(e)^*}}\cdot\Hiso\log(\delta_{\rho(e)}) \\
 &= \rho_*(\gamma).
\end{align*}
This completes the proof of Proposition~\ref{prop:resistance_compatibility}.\qed

\subsection{Decomposition into $2$-connected components}
\label{ss:blocks}

The final simplification operation we will consider consists of breaking up a reduction graph $\rGraph$ into its $2$-connected components via an adaptation of the classical \emph{block--cutvertex decomposition} \cite[page 74]{bollobas} that takes account of the genus of vertices. This decomposition expresses a general graph as a union of certain subgraphs called \emph{blocks} by gluing them at certain vertices called \emph{cutvertices}.

\begin{definition}[Cutvertices and $2$-connected components]\label{def:2-conn_cpt}\index{$2$-connected component}
A (connected) graph $\Cpt$ is called \emph{$2$-connected} just when $\H_1(\Cpt)\neq0$ and the graph $\Cpt\setminus\{v\}$ formed by removing\footnote{To be precise, $\Cpt\setminus\{v\}$ denotes the graph formed from $\Cpt$ by removing $v$ from $\Vert{\Cpt}$, adding into $\Vert{\Cpt}$ one new vertex for every edge or half-edge $e$ with $\source(e)=v$, and adjusting the source function $\source\colon\Edge{\Cpt}\sqcup\HEdge{\Cpt}\rightarrow\Vert{\Cpt}$ accordingly.} a vertex $v$ is connected for all vertices $v$. A \emph{$2$-connected component} $\Cpt$ of $\Graph$ is a maximal $2$-connected subgraph. A \emph{block} $\Cpt$ of $\rGraph$ is a subgraph of $\Graph$ which is either:
\begin{itemize}
	\item a $2$-connected component of $\Graph$;
	\item a bridge of $\Graph$ (edge $e$ such that $\Graph\setminus\{e^{\pm1}\}$ is disconnected) together with its endpoints;
	\item a half-edge of $\Graph$ together with its endpoint; or
	\item a single vertex $\{v\}$ of $\rGraph$ with $g(v)>0$.
\end{itemize}

A vertex $v$ of $\rGraph$ is called a \emph{cutvertex} just when $g(v)>0$ or $\Graph\setminus\{v\}$ is disconnected.
\end{definition}

\begin{remark}\label{rmk:cutvertices_and_subdivision}\leavevmode
\begin{enumerate}[label=\alph*),ref=\alph*]
	\item The $2$-connected components of $\Graph$ are all compact (have no half-edges) and bridgeless.
	\item Any loop in $\Graph$ (edge $e$ such that $\source(e)=\target(e)$) is a $2$-connected component; all other $2$-connected components are loopless.
	\item Subdivision of edges and half-edges of $\rGraph$ doesn't change its $2$-connected components (except by subdividing their edges), and subdividing a bridge (resp.\ half-edge) splits that bridge into two bridges (resp.\ a half-edge and a bridge).
\end{enumerate}
\end{remark}

The notion of blocks and cutvertices of $\rGraph$ leads to the following structure theorem, saying that $\rGraph$ is built out of blocks, glued together at cutvertices in a tree-like structure.

\begin{proposition}[Block--cutvertex decomposition {\cite[page 74]{bollobas}}]\label{prop:bc_decomp}\index{block--cutvertex decomposition}
$\rGraph$ is the union of its blocks, any two different blocks intersect in $\leq1$ point, and these intersection points are exactly the cutvertices. If $\bc(\rGraph)$ denotes the bipartite graph with one vertex-class the set of blocks, the other vertex-class the set of cutvertices, and an edge connecting each cutvertex to each of the blocks containing it, then $\bc(\rGraph)$ is a tree.
\begin{figure}[!h]
\begin{center}
\begin{tikzpicture}
\draw[thick,dashed,color=darkgray,fill=lightgray,opacity=0.4] (0,0.75) circle (0.75);
\draw[thick,dashed,color=darkgray,fill=lightgray,opacity=0.4] (0,-0.75) circle (0.75);
\draw[thick,dashed,color=darkgray,fill=lightgray,opacity=0.4] (2,0) circle (0.5);
\draw[thick,dashed,color=darkgray,fill=lightgray,opacity=0.4] (4.25,0) circle (0.75);
\coordinate (a) at (0,0);
\draw[fill=black] (a) circle (2pt);
\coordinate (b) at (0.875,0);
\draw[fill=black] (b) circle (2pt);
\coordinate (c) at (1.5,0);
\draw[fill=black] (c) circle (2pt);
\coordinate (d) at (2,0.5);
\draw[fill=black] (d) circle (2pt);
\coordinate (e) at (2.5,0);
\draw[fill=black] (e) circle (2pt);
\coordinate (f) at (3.5,0);
\draw[fill=black] (f) circle (2pt);
\coordinate (g) at (4.375,0.5);
\draw[fill=black] (g) circle (2pt);
\coordinate (h) at (4.375,-0.5);
\draw[fill=black] (h) circle (2pt);
\coordinate (i) at (3,1);
\draw[thick] (a) -- (b) -- (c) -- (e) -- (f) -- (g) -- (h) -- (f);
\draw[thick] (c) -- (d) -- (e);
\draw[thick] (d) -- (i);
\draw[thick] (0,0.5) circle (0.5);
\draw[thick] (0,-0.5) circle (0.5);
\draw[thick] (g) to[bend left] (h);
\draw[thick] (e) -- (2.75,-0.5);
\draw[thick,dashed] (2.75,-0.5) -- (3.25,-1.5);
\draw[thick,fill=lightgray] (i) circle (6pt);
\node at (i) {$1$};
\end{tikzpicture}
\end{center}
\label{fig:2-connected_cpts}
\caption{A (semistable) reduction graph $\rGraph$ with its four $2$-connected components indicated by the light grey dashed discs. Together with its four bridges, single half-edge and single vertex of genus $1$, these comprise the blocks of $\rGraph$. The six cutvertices of $\Graph$ are the five vertices on the perimeter of these discs, plus the single vertex of degree $2$ and genus $0$. The block--cutvertex decomposition of this graph clearly has the structure of a tree.}
\end{figure}
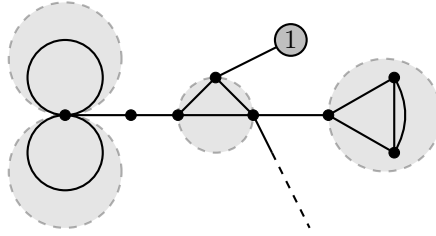
\end{proposition}
The block--cutvertex decomposition of $\rGraph$ provides an orthogonal decomposition\[\H_1(\Graph)=\bigoplus_{\Cpt}\H_1(\Cpt)\]of the homology of $\Graph$ into the homology of its blocks (or just its $2$-connected components). We write $\iota_*\colon\H_1(\Cpt)\hookrightarrow\H_1(\Graph)$ for the direct summand inclusion, and $\iota_*\colon\H^1(\Cpt)\hookrightarrow\H^1(\Graph)$ for its adjoint with respect to the cycle pairing (the dual map to the product projection $\iota^*\colon\H_1(\Graph)\rightarrow\H_1(\Cpt)$).

We want to extend this to a decomposition of the bigraded complete Hopf algebra $\HAlg(\rGraph)$. For this we first need to define the structure of a reduction graph $\rCpt$ on each block $\Cpt$ of $\rGraph$, which is done as follows. We firstly add to $\Cpt$ a new half-edge $e_{\Cpt,v}$ with source $v$ for each cutvertex $v$ of $\rGraph$ contained in $\Cpt$. We then endow this graph with the obvious metric inherited from $\rGraph$. Finally, if $\Cpt=\{v\}$ consists of a single vertex of genus $g(v)>0$, then we assign it genus $g(v)$ again in $\rCpt$; otherwise we assign every vertex of $\Cpt$ genus $0$.

There is then a pushforward map\[\iota_*\colon\HAlg(\rCpt)\rightarrow\HAlg(\rGraph),\]defined via the pushforward maps $\iota_*$ on $\H_1(\Cpt)$ and $\H^1(\Cpt)$, and defined to send generators $\log(\beta_{v,i})$, $\log(\beta_{v,i}')$ and $\log(\delta_e)$ to the corresponding generators of $\HAlg(\rGraph)$ when $\Cpt$ is a vertex $v$ of genus $>0$ or a half-edge $e$ respectively. In order to define $\iota_*$ on generators of the form $\log(\delta_{e_{\Cpt,v}})$, we adopt the following notation.

\begin{notation}\label{notn:surface_relation_of_blocks}
For each block $\Cpt$ of $\rGraph$, we define an element $\sigma_\Cpt\in\gr^\M_{-2}\gr^\W_{-2}\HAlg(\rGraph)$ by
\[
\sigma_\Cpt:=
\begin{cases}
\sum_i[\xi'_i,\xi_i] & \text{if $\Cpt$ is a $2$-connected component,} \\
0 & \text{if $\Cpt$ is a bridge,} \\
\log(\delta_e) & \text{if $\Cpt$ is a half-edge $e$,} \\
\sum_i[\log(\beta_{v,i}'),\log(\beta_{v,i})] & \text{if $\Cpt$ is a vertex $\{v\}$ with $g(v)>0$,}
\end{cases}
\]
where in the first point, $(\xi_i)$ is a basis of $\H^1(\Cpt)$ with dual basis $(\xi_i')$ of $\H_1(\Cpt)$. Note that $\sum_{\Cpt}\sigma_\Cpt=0$ by the surface relation in $\rGraph$.
\end{notation}

We now complete the definition of $\iota_*\colon\HAlg(\rCpt)\rightarrow\HAlg(\rGraph)$ by specifying\[\iota_*(\log(\delta_{e_{\Cpt,v}})):=\sum_{\Cpt'\sim v\sim\Cpt}\sigma_{\Cpt'}\]where the sum is taken over all components $\Cpt'$ lying in the same component of $\bc(\rGraph)\setminus\{\Cpt\}$ as $v$. This does indeed define a pushforward $\HAlg(\rCpt)\rightarrow\HAlg(\rGraph)$: the image of the surface relation in $\rCpt$ is $\sigma_\Cpt+\sum_{v\in\Cpt}\iota_*(\log(\delta_{e_{\Cpt,v}}))=0$ since $\bc(\rGraph)$ is a tree.

These pushforward maps give a suitable decomposition of $\HAlg(\rGraph)$, extending the decomposition of $\H_1(\Graph)$ above.

\begin{lemma}\label{lem:decomposition_of_HAlg}
$\HAlg(\rGraph)$, as a $\W$- and $\M$-graded complete Hopf algebra, is the quotient of the completed free product $\CFree_\Cpt\HAlg(\rCpt)$ by the ideal generated by the bihomogenous primitive elements\[\sum_{\Cpt\ni v}\log(\delta_{e_{\Cpt,v}})\]for each cutvertex $v$.
\begin{proof}
Write $\HAlg^\bc(\rGraph)$ for the quotient of $\CFree_\Cpt\HAlg(\rCpt)$ by the ideal generated by the specified elements. The map $\CFree_\Cpt\HAlg(\rCpt)\rightarrow\HAlg(\rGraph)$ induced by the maps $\iota_*\colon\HAlg(\rCpt)\rightarrow\HAlg(\rGraph)$ above takes the elements $\sum_{\Cpt\ni v}\log(\delta_{e_{\Cpt,v}})$ are taken to $(\deg_{\bc(\rGraph)}(v)-1)\cdot\sum_{\Cpt}\sigma_\Cpt=0$, and hence induces a map $\iota_*\colon\HAlg^\bc(\rGraph)\rightarrow\HAlg(\rGraph)$.

We show that $\iota_*$ is an isomorphism by constructing its inverse $\iota^*\colon\HAlg(\rGraph)\rightarrow\HAlg^\bc(\rGraph)$. This is the map defined in the obvious manner on generators of $\HAlg(\rGraph)$; for instance, $\gr^\M_0\gr^\W_{-1}\HAlg(\rGraph)=\H^1(\Graph)=\bigoplus_{\Cpt}\H^1(\Cpt)$ is mapped into $\HAlg^\bc(\rGraph)$ via the maps $\H^1(\Cpt)\hookrightarrow\HAlg(\rCpt)\rightarrow\CFree_\Cpt\HAlg(\rCpt)\rightarrow\HAlg^\bc(\rGraph)$ for each block/$2$-connected component $\Cpt$. This map is well-defined, since the image of the surface relation $\sigma=\sum_\Cpt\sigma_\Cpt$ is\[\iota^*(\sigma)=\sum_{\Cpt}\iota^*(\sigma_\Cpt)=-\sum_\Cpt\sum_{v\in\Cpt}\log(\delta_{e_{\Cpt,v}})=-\sum_v\sum_{\Cpt\ni v}\log(\delta_{e_{\Cpt,v}})=0\]since the surface relations for each $\HAlg(\rCpt)$ ensure that $\iota^*(\sigma_\Cpt)+\sum_{v\in\Cpt}\log(\delta_{e_{\Cpt,v}})$ for every block $\Cpt$.

It is clear from the construction that $\iota_*\iota^*$ is the identity on $\HAlg(\rGraph)$, and that to verify that $\iota^*\iota_*$ is the identity on $\HAlg^\bc(\rGraph)$, it suffices to verify that\[\log(\delta_{e_{\Cpt,v}})=\sum_{\Cpt'\sim v\sim\Cpt}\iota^*\sigma_{\Cpt'}\]for every cutvertex--block incidence $v\in\Cpt$. Since $\bc(\rGraph)$ is a tree, it suffices to prove this identity after summing over all cutvertices $v$ contained in a fixed $\Cpt$, and after summing over all blocks $\Cpt$ containing a fixed $v$. For the former, we have\[\sum_{v\in\Cpt}\sum_{\Cpt'\sim v\sim\Cpt}\iota^*\sigma_{\Cpt'}=\sum_{\Cpt'\neq\Cpt}\iota^*\sigma_{\Cpt'}=-\iota^*\sigma_\Cpt=\sum_{v\in\Cpt}\log(\delta_{e_{\Cpt,v}}),\]since $\sum_{\Cpt'}\sigma_{\Cpt'}=0$ by the surface relation in $\rGraph$. For the latter, we have\[\sum_{\Cpt\ni v}\sum_{\Cpt'\sim v\sim\Cpt}\iota^*\sigma_{\Cpt'}=(\deg_{\bc(\rGraph)}(v)-1)\cdot\sum_{\Cpt'}\iota^*\sigma_{\Cpt'}=0=\sum_{\Cpt\ni v}\log(\delta_{e_{\Cpt,v}}).\]
\end{proof}
\end{lemma}

\begin{corollary}\label{cor:block_injectivity}
If $\rGraph$ is semistable, then the maps $\iota_*\colon\HAlg(\rCpt)\hookrightarrow\HAlg(\rGraph)$ are injective for all blocks $\Cpt$.
\begin{proof}
This is easiest to see on the level of pro-nilpotent Lie algebras. It follows from Lemma~\ref{lem:decomposition_of_HAlg} that, when $\rGraph$ is semistable, $\LAlg(\rGraph)$ is produced from $\LAlg(\rCpt)$ via operations of the following form:
\begin{itemize}
	\item from a pro-nilpotent Lie algebra $\LAlg_0$, forming the pro-nilpotent Lie algebra $\LAlg_0\langle\xi_1\rangle$ freely generated by $\LAlg_0$ and an element $\xi_1$; and
	\item from a pro-nilpotent Lie algebra $\LAlg_0$ and an element $\xi_0\in\LAlg_0$, forming the pro-nilpotent Lie algebra $\LAlg_0\langle\xi_1,\xi_2\:|\:[\xi_1,\xi_2]=\xi_0\rangle$ generated by $\LAlg_0$ and two elements $\xi_1$, $\xi_2$ subject to the relation $[\xi_1,\xi_2]=\xi_0$.
\end{itemize}
It suffices to prove that the maps $\LAlg_0\rightarrow\LAlg_0\langle\xi_1\rangle$ and $\LAlg_0\rightarrow\LAlg_0\langle\xi_1,\xi_2\:|\:[\xi_1,\xi_2]=\xi_0\rangle$ arising from these two operations are injective. The former is even split, and for the latter, it suffices to deal with the case that $\LAlg_0$ is finite dimensional (so nilpotent). 

Choose a faithful representation $\rho\colon\LAlg_0\hookrightarrow\n_n$ where $\n_n$ is the Lie algebra of strictly upper triangular $n\times n$ matrices, and further embed $\n_n\hookrightarrow\n_{2n}$ as the Lie subalgebra consisting of matrices with only $0$s in their odd-numbered columns and rows. If we write $\rho(\xi_1)\in\n_{2n}$ for the matrix with $1$s on the superdiagonal and $0$s elsewhere, then we easily see that there is a $\rho(\xi_2)\in\n_{2n}$ with $[\rho(\xi_1),\rho(\xi_2)]=\rho(\xi_0)$. These choices of elements $\rho(\xi_1)$, $\rho(\xi_2)$ thus provide a lifting of the map $\rho\colon\LAlg_0\hookrightarrow\n_{2n}$ through $\LAlg_0\langle\xi_1,\xi_2\:|\:[\xi_1,\xi_2]=\xi_0\rangle$, and hence the map $\LAlg_0\rightarrow\LAlg_0\langle\xi_1,\xi_2\:|\:[\xi_1,\xi_2]=\xi_0\rangle$ is injective.
\end{proof}
\end{corollary}

Equipped with the map $\iota_*\colon\HAlg(\rCpt)\rightarrow\HAlg(\rGraph)$, we are now finally able to state the desired compatibility between the non-abelian Kummer maps associated to $\rGraph$ and its blocks.

\begin{lemma}[Restriction to blocks]\label{lem:block_reduction}
For every block $\Cpt$ in $\rGraph$, the non-abelian Kummer maps for $\rCpt$ and $\rGraph$ fit into a commuting square
\begin{center}
\begin{tikzcd}
\QVert{\Cpt} \arrow{r}{\classify}\arrow{d}{\iota} & \V(\rCpt) \arrow{d}{\iota_*} \\
\QVert{\Graph} \arrow{r}{\classify} & \V(\rGraph).
\end{tikzcd}
\end{center}
\end{lemma}

As usual, this will follow from the following compatibility result with the description of the $\M$-graded fundamental groupoid in Theorem~\ref{thm:description_of_M-graded}. The proof of Lemma~\ref{lem:block_reduction} from this Proposition~\ref{prop:block_compatibility} is exactly the same as that of Lemma~\ref{lem:half-edge_reduction} from Proposition~\ref{prop:half-edge_compatibility}.

\begin{proposition}\label{prop:block_compatibility}
The map $\iota\colon\QVert{\Cpt}\rightarrow\QVert{\Graph}$ extends to a morphism\[\iota_*\colon\pi_1^\Q(\rCpt)|_{\QVert{\Cpt}}\rightarrow\pi_1^\Q(\rGraph)\]of $\Q$-pro-unipotent groupoids such that the induced map\[\iota_*\colon\O(\pi_1^\Q(\rCpt))^\dual|_{\QVert{\Cpt}}\rightarrow\O(\pi_1^\Q(\rGraph))^\dual\]of complete Hopf groupoids is $N$-equivariant and $\W$- and $\M$-filtered. The induced map $\gr^\M_\bullet\iota_*\colon\gr^\M_\bullet\O(\pi_1^\Q(\rCpt))^\dual|_{\QVert{\Cpt}}\rightarrow\gr^\M_\bullet\O(\pi_1^\Q(\rGraph))^\dual$ fits into a commuting square
\begin{equation}\label{eq:block_compatibility}
\begin{tikzcd}
\gr^\M_\bullet\O(\pi_1^\Q(\rCpt))^\dual|_{\QVert{\Cpt}} \arrow{r}{\Hiso}[swap]{\hiso}\arrow{d}{\gr^\M_\bullet\iota_*} & \HAlg(\rCpt) \arrow{d}{\iota_*} \\
\gr^\M_\bullet\O(\pi_1^\Q(\rGraph))^\dual \arrow{r}{\Hiso}[swap]{\hiso} & \HAlg(\rGraph)
\end{tikzcd}
\end{equation}
where the horizontal maps are the isomorphisms from Theorem~\ref{thm:description_of_M-graded}.
\end{proposition}

\subsubsection*{Step $1$: Constructions of $\iota_*$}

To begin the proof of Proposition~\ref{prop:block_compatibility}, we explain how to construct the pushforward map
\[
\iota_*\colon\pi_1^\Q(\rCpt)|_{\QVert{\Cpt}}\rightarrow\pi_1^\Q(\rGraph).
\]
This is defined in the obvious way on all the generators of $\pi_1^\Q(\rCpt)$ from Definition~\ref{def:pi1_reduction_graph} save the generators $\delta_{e_{\Cpt,v}}$ associated to the half-edges introduced into $\Cpt$ for each cutvertex $v\in\Cpt$. For these generators, we define
\[
\iota_*(\delta_{e_{\Cpt,v}})^{-1}:=
\begin{cases}
\prod_{\source(e)=v,e\in\Edge{\Cpt}}\delta_e & \text{if $\Cpt$ a $2$-connected component or bridge,} \\
\delta_e & \text{if $\Cpt$ is a half-edge $e$,} \\
\prod_i[\beta_{v,i}',\beta_{v,i}] & \text{if $\Cpt=\{v\}$ with $g(v)>0$.}
\end{cases}
\]

It is easy to check that this produces a well-defined morphism of $\Q$-pro-unipotent groupoids $\iota_*\colon\pi_1^\Q(\rCpt)|_{\QVert{\Cpt}} \rightarrow\pi_1^\Q(\rGraph)$, and that the induced map $\iota_*\colon\O(\pi_1^\Q(\rCpt))^\dual|_{\QVert{\Cpt}} \rightarrow\O(\pi_1^\Q(\rGraph))^\dual$ is $N$-equivariant and $\M$-filtered. It is also $\W$-filtered, as follows from the topological description in \S\ref{ss:graphs_of_groups_topology}. Alternatively, this follows automatically once we have proved commutativity of the square in Proposition~\ref{prop:block_compatibility}: the only thing we need check is that the elements $\iota_*(\log(\delta_{e_{\Cpt,v}}))$ lie in $\W_{-2}\O(\pi_1^\Q(\rGraph))^\dual$, but commutativity of the square shows that $\gr^\M_\bullet\iota_*(\log(\delta_{e_{\Cpt,v}}))\in\W_{-2}\gr^\M_{-2}\O(\pi_1^\Q(\rGraph))^\dual$, so that $\iota_*(\log(\delta_{e_{\Cpt,v}}))\in\W_{-2}+\M_{-3}=\W_{-2}$.

\subsubsection*{Step $2$: Commutativity in $\M$-degree $0$}

Next, we show commutativity of \eqref{eq:block_compatibility} in $\M$-degree $0$. As usual, this amounts to proving the identity\[\int_{\iota_*(\gamma)}\underline\omega=\int_\gamma\iota^*(\underline\omega)\]for all paths $\gamma$ in $\Cpt$ and all $\underline\omega\in\Shuf\H_1(\Graph)$, where $\iota^*\colon\H_1(\Graph)\rightarrow\H_1(\Cpt)$ is the adjoint to the pushforward on homology (the product projection from the direct sum decomposition $\H_1(\Graph)=\bigoplus_{\Cpt'}\H_1(\Cpt')$). It suffices to prove this when $\gamma=e$ is an edge of $\Cpt$ and $\underline\omega=\omega_1\dots\omega_n$ is a pure tensor of length $n$ with each $\omega_i\in\H_1(\Cpt_i)$ for some block $i$. But in this case, both sides of the desired identity are equal to $\frac1{n!}\prod_i\int_\gamma\omega_i$ if all $\Cpt_i$ are equal to $\Cpt$, and $0$ otherwise.

\subsubsection*{Step $3$: Commutativity in $\M$-degrees $<0$}

It remains to show that \eqref{eq:block_compatibility} commutes in all $\M$-degrees. As in the proofs of Propositions~\ref{prop:half-edge_compatibility} and~\ref{prop:block_compatibility}, the map\[\gr^\M_\bullet\iota_*\colon\gr^\M_\bullet\O(\pi_1^\Q(\rCpt))^\dual|_{\QVert{\Cpt}}\rightarrow\gr^\M_\bullet\O(\pi_1^\Q(\rGraph))^\dual\]can be viewed as a morphism of complete Hopf algebras (i.e.\ is compatible with the trivialisations of both sides provided by canonical paths as in \S\ref{ss:trivialisation}), and so we may check commutativity on generators of $\HAlg(\rCpt)$.

The generators $\H^1(\Cpt)$ in $\M$-degree $0$ have already been dealt with, and the generators arising from genus of vertices, from half-edges \emph{of $\rGraph$} and from $\H_1(\Cpt)$ are easy to check. For generators of the form $\log(\delta_{e_{\Cpt,v}})$, we want to verify that\[\Hiso\gr^\M_\bullet\iota_*(\log(\delta_{e_{\Cpt,v}}))=\iota_*\Hiso\log(\delta_{e_{\Cpt,v}}),\]for every block--cutvertex incidence $\Cpt\ni v$.

As in the proof of Lemma~\ref{lem:decomposition_of_HAlg}, it suffices to prove this identity after summing over all cutvertices $v$ contained in a fixed $\Cpt$, and after summing over all blocks $\Cpt$ containing a fixed $v$. According to the calculations in the proof of Lemma~\ref{lem:decomposition_of_HAlg}, we want to prove the identities
\begin{align*}
\sum_{v\in\Cpt}\Hiso\gr^\M_\bullet\iota_*\log(\delta_{e_{\Cpt,v}}) &= -\sigma_\Cpt, \\
\sum_{\Cpt\ni v}\Hiso\gr^\M_\bullet\iota_*\log(\delta_{e_{\Cpt,v}}) &= 0
\end{align*}
for every block $\Cpt$ and cutvertex $v$ respectively.

The first identity is clear when $\Cpt$ is a half-edge or higher-genus vertex of $\rGraph$; when $\Cpt$ is a $2$-connected component or bridge, it is verified by the calculation
\begin{align*}
\sum_{\substack{v\in\Cpt\\\text{$v$ cutvertex}}}\Hiso\gr^\M_\bullet\iota_*(\log(\delta_{e_{\Cpt,v}})) &= -\sum_{\substack{v\in\Cpt\\\text{$v$ cutvertex}}}\sum_{\substack{\source(e)=v\\e\in\Edge{\Cpt}}}\Hiso\log(\delta_e)\\
 &= -\sum_{u\in\Cpt}\sum_{\substack{\source(e)=u\\e\in\Edge{\Cpt}}}\Hiso\log(\delta_e)\\
 &= -\sum_i[\xi_i',\xi_i] = -\sigma_\Cpt
\end{align*}
where $(\xi_i)$ is a basis of $\H^1(\Cpt)$ with dual basis $(\xi_i')$ of $\H_1(\Cpt)$, as in Notation~\ref{notn:surface_relation_of_blocks}. The second line follows from the first by the vertex relations $\sum_{\source(e)=u}\log(\delta_e)=0$ in $\gr^\M_{-2}\O(\pi_1^\Q(\rGraph))^\dual$ at non-cutvertices $u$ in $\Cpt$, and the third follows from the second by the same argument as in Proposition~\ref{prop:sigma_0}.

The second identity follows from the vertex identity at vertex $v$. This completes the proof of Proposition~\ref{prop:block_compatibility}\qed

\subsection{Cut pairs and maximal cut systems}\index{maximal cut system}

In order to complete our study of the non-abelian Kummer map of a graph in weight $-2$, we will need to probe slightly further into the $2$-connected structure of our reduction graph $\rGraph$, exploiting \emph{pairs} of edges which together disconnect $\rGraph$. The notion of interest is already introduced in \cite{AMO}.

\begin{definition}[Maximal cut systems]\label{def:max_cuts}
A \emph{cut pair} (in $\rGraph$) is a pair of (different) non-bridge edges $e_0$, $e_1$ such that $\rGraph\setminus\{e_0^{\pm1},e_1^{\pm1}\}$ is disconnected, with $\source(e_0)$ and $\source(e_1)$ lying in different components. The relation on the non-bridge edges of $\rGraph$ where $e_0\sim e_1$ if and only if $(e_0,e_1)$ is a cut pair or $e_0=e_1$ is an equivalence relation \cite[Lemma 1.1]{AMO}, and we refer to its equivalence classes as \emph{maximal cut systems}\footnote{Our terminology here differs slightly from that of \cite[Definition 1.9]{AMO}, in that we permit maximal cut systems to have size $1$.}.

If $\e$ is a maximal cut system, then $\e^{-1}=\{e^{-1}\:|\:e\in\e\}$ is a (different) maximal cut system, and we write $\plMaxCut{\Graph}$ for the set of maximal cut systems up to identifying $\e\sim\e^{-1}$.
\end{definition}

\begin{remark}\label{rmk:cut_pairs_and_cohomology}
Non-bridge edges $e_0$, $e_1$ form a cut pair if and only if $e_0^*=e_1^*\in\H^1(\Graph)$. This is analogous to the obvious fact that $e^*=0$ if and only if $e$ is a bridge.
\end{remark}

There is a refinement of the block--cutvertex decomposition of a reduction graph $\rGraph$ (Proposition~\ref{prop:bc_decomp}) taking account of a single maximal cut system. In order to state this concisely, we adopt the following terminology.

\begin{definition}[Relative $2$-connected components]\label{def:rel_2-conn_cpt}
Fix a maximal cut system $\e$ in $\rGraph$. A \emph{$2$-connected component $\Cpt$ relative to $\e$} is a maximal $2$-connected subgraph of $\Graph\setminus\{\e^{\pm1}\}$. A \emph{block $\Cpt$ relative to $\e$} is a subgraph of $\Graph$ which is either:
\begin{itemize}
	\item an edge of $\rGraph$ in $\e$;
	\item a $2$-connected component relative to $\e$;
	\item a bridge of $\rGraph$;
	\item a half-edge of $\rGraph$ together with its endpoints; or
	\item a single vertex $\{v\}$ with $g(v)>0$ in $\rGraph$.
\end{itemize}

A vertex $v$ of $\rGraph$ is called a \emph{relative cutvertex} just when $v$ is an endpoint of an edge in $\e$, $g(v)>0$ or $\rGraph\setminus\{v\}$ is disconnected.
\end{definition}

\begin{proposition}[Relative block--cutvertex decomposition]\label{prop:rel_bc_decomp}
For any maximal cutsystem $\e$, $\rGraph$ is the union of its blocks relative to $\e$, any two different relative blocks intersect in $\leq1$ point, and these intersection points are exactly the relative cutvertices. If $\bc_\e(\rGraph)$ denotes the bipartite graph with one vertex-class the set of relative blocks, the other vertex-class the set of relative cutvertices, and an edge connecting each relative cutvertex to each of the relative blocks containing it, then $\bc_\e(\rGraph)$ is a graph with a unique cycle. If $\rGraph$ is already $2$-connected with $g(v)=0$ for all vertices $v$, then $\bc_\e(\rGraph)$ is a cycle.
\end{proposition}

\begin{remark}\label{rmk:max_cuts_in_2-conn_cpts}
It follows from the (relative) block--cutvertex decomposition that if $\Cpt$ is a $2$-connected component of $\rGraph$ (resp.\ $2$-connected component relative to some maximal cut system $\e$), then the maximal cut systems in $\Cpt$ are unions of maximal cut systems in $\rGraph$. (In other words, if $(e_0,e_1)$ is a cut pair in $\rGraph$ with $e_0\in\Edge{\Cpt}$, then $e_1\in\Edge{\Cpt}$ also and $(e_0,e_1)$ is a cut pair in $\Cpt$.)
\end{remark}

The final result we will need regarding maximal cut systems is the following result giving an explicit description of the ``exterior intersection pairing'' on $\rGraph$.

\begin{lemma}\label{lem:exterior_intersection_max_cut}
The image of the exterior intersection pairing
\begin{align*}
\H_1(\Graph)\otimes\H_1(\Graph) &\rightarrow \Q\cdot\Edge{\Graph}^+ \\
\gamma_0\otimes\gamma_1 &\mapsto \sum_{e\in\Edge{\Graph}}e^*(\gamma_0)e^*(\gamma_1)\cdot e
\end{align*}
is exactly the span of the formal sums of maximal cut systems.
\begin{proof}
It is clear that the image is contained in the span of the maximal cut systems. It remains to show that the formal sum of any maximal cut system $\e$ is in the image of the exterior intersection pairing. For the sake of simplicity, let us assume that $\rGraph=\Graph$ is $2$-connected with $g(v)=0$ for all vertices $v$ (the argument is easily adapted in the general case).

By the relative block--cutvertex decomposition (Proposition~\ref{prop:rel_bc_decomp}) we may enumerate $\e$ as $e_0,e_1,\dots,e_{m-1}$ such that $\target(e_i)$ and $\source(e_{i+1})$ lie in the same connected component of $\Graph\setminus\{\e^{\pm1}\}$ for every $i$ (with indices taken mod $m$). The components of $\Graph\setminus\{e^{\pm1}\}$ are unions of $2$-connected components of $\Graph$ relative to $\e$, and hence are $2$-edge-connected \cite[Chapter III.2]{bollobas}. By Menger's Theorem \cite[Theorem III.5(ii)]{bollobas}, for each $i$ we may choose edgewise-disjoint paths $\gamma_i$, $\gamma_i'$ in $\Graph\setminus\{\e^{\pm1}\}$ from $\target(e_i)$ to $\source(e_{i+1})$. Then $\gamma=\gamma_{m-1}e_{m-1}\dots\gamma_1e_1\gamma_0e_0$ and $\gamma'=\gamma'_{m-1}e_{m-1}\dots\gamma'_1e_1\gamma'_0e_0$ are loops in $\Graph$ whose exterior intersection is $\sum_ie_i$, as desired.
\end{proof}
\end{lemma}

\section{Proof of the relative Oda--Tamagawa criterion}
\label{s:injectivity}

We finally turn our attention to the relative Oda--Tamagawa criterion (Theorem~\ref{thm:groupoid_oda}) asserting that the fibres of the non-abelian Kummer map $\classify\colon X(K)\rightarrow\H^1(G_K,\pi_1^{\Q_\ell}(X_{\overline K},b))$ associated to a hyperbolic curve $X/K$ are the same as the fibres of the reduction map (Definition~\ref{def:reduction_map}). In light of Theorem~\ref{thm:curve_kummer_is_graph_kummer}, this assertion (over all finite extensions $L/K$) is equivalent to the assertion that the non-abelian Kummer map $\classify\colon\QVert{\Graph}\rightarrow\V$ associated to the reduction graph of $X$ is injective. We will prove a more general version of this which precisely captures the fibres of the first $n$ weight-graded components of $\classify$ for every~$n$. This in particular also establishes the more refined version asserted in the introduction.

\begin{theorem}[Injectivity Theorem]\label{thm:combinatorial_injectivity}
Let $\rGraph$ be a stable reduction graph. If $x,y\in\QVert{\Graph}$ are two distinct rational vertices, then:
\begin{itemize}
	\item (constancy in weight $-1$) $\classifyleq1(x)=\classifyleq1(y)$;
	\item (near-injectivity in weight $-2$) $\classifyleq2(x)=\classifyleq2(y)$ if and only if $x$ and $y$ lie in the same $2$-connected component $\Cpt$ of $\rGraph$ and there is an isometric involution $\iota$ of $\Cpt$ fixing the cutvertices of $\rGraph$ in $\Cpt$ pointwise, interchanging $x$ and $y$, and such that the quotient $\Cpt/\iota$ is a tree; and
	\item (injectivity in weight $-3$) $\classifyleq3(x)\neq\classifyleq3(y)$.
\end{itemize}
Here, $\classifyleq n\colon\QVert{\Graph}\rightarrow\prod_{r=1}^n\gr^\W_{-r}\V$ denotes the non-abelian Kummer map in weight~$\geq -n$.
\end{theorem}

\begin{remark}\label{rmk:injectivity_over_R}
The methods we use in the proof of Theorem~\ref{thm:combinatorial_injectivity} are essentially $\R$-linear in nature, and as such yield a slightly stronger result than stated above. Since the non-abelian Kummer map $\classify\colon\QVert{\Graph}\rightarrow\V$ associated to a stable reduction graph restricts to a componentwise polynomial function on edges, it extends uniquely to a continuous function $\Graph\rightarrow\V_\R$ on the underlying topological space $\Graph$ of $\rGraph$. This $\R$-linear non-abelian Kummer map satisfies the following properties:
\begin{itemize}
	\item $\classifyleq1$ is constant;
	\item $\classifyleq2$ factors through a closed immersion $\rGraph/\iota\hookrightarrow\gr^\W_{-2}\V_\R$, where $\rGraph/\iota$ denotes the quotient of $\rGraph$ by all isometric involutions of its $2$-connected components satisfying the properties in Theorem~\ref{thm:combinatorial_injectivity}; and
	\item $\classifyleq n$ is a closed immersion for all $n\geq3$.
\end{itemize}
Indeed, an easy $\R$-linearisation of our arguments in this section establishes that the relevant maps above are injective; they are then automatically closed immersions since they are linear on half-edges by Corollary~\ref{cor:j_piecewise_polynomial}.

We believe, though our argument here does not prove, that the following even stronger result should be true. If we let $\A(\Graph)$ denote the algebraisation of $\Graph$, i.e.\ the affine curve over $\Q$ given by taking a copy of $\A^1_\Q$ for each edge or half-edge and making the obvious identifications, then the non-abelian Kummer map extends to a morphism $\classify\colon\A(\Graph)\rightarrow\V$ of $\Q$-schemes, where $\V$ is viewed as an infinite-dimensional affine space in the usual way. We believe that this map should be a closed immersion, as should its finite-level analogues $\classifyleq n$ for all $n\geq3$.
\end{remark}

Recall that we write $\classifyeq r$ for the components of the non-abelian Kummer map $\classifyleq n\colon\QVert{\Graph}\rightarrow\prod_{r=1}^n\gr^\W_{-r}\V$. Since $\classifyeq1=0$ by Remark~\ref{rmk:j_1}, the content of Theorem~\ref{thm:combinatorial_injectivity} is that the maps $\classifyeq2$ and $\classifyeq3$ are jointly injective, and gives a precise characterisation of the failure of $\classifyeq2$ to be injective.

\subsection{An integral condition for triviality in weight $-2$}

The bulk of the work in the proof of Theorem~\ref{thm:combinatorial_injectivity} is contained in the final part, giving a precise description of the fibres of the non-abelian Kummer map in weight $-2$. The centrepiece of our proof is the following reformulation of the condition $\classifyeq2(x)=\classifyeq2(y)$ in the language of harmonic analysis.

\begin{lemma}\label{lem:harmonic_condition}
Let $\rGraph$ be a reduction graph and $x,y\in\QVert{\Graph}$ distinct rational vertices. Choose $\Phi\in\Omlog(\Graph)$ such that $\Lapl(\Phi)=y-x$ (cf.\ Remark~\ref{rmk:stabilised_laplacians}). Then $\classifyeq2(x)=\classifyeq2(y)$ if and only if there is a constant $\kappa$ such that\[\frac1{\sum_{e\in\e}\length(e)}\sum_{e\in\e}\int\Phi|\d s_e|=\kappa=\Phi(v)\]for all maximal cut systems $\e$ in $\rGraph$ and all vertices $v$ which are either the endpoint of a half-edge or satisfy $g(v)>0$. (The function $\Phi$ is well-defined up to constants; with an appropriate normalisation we may assume that $\Phi=0$.)
\begin{proof}
Let $\Phi\in\Omlog(\rGraph)$ be any function such that $\Lapl(\Phi)=y-x$ (uniquely defined up to additive constants). It follows from Theorem~\ref{thm:description_of_graph_kummer} that $\classifyeq2(x)=\classifyeq2(y)$ if and only\[\langle y-x,\mu_2\rangle:=\int\Phi\d\mu_2=0.\]

We now write the $\gr^\W_{-2}\V$-valued measure $\mu_2$ as\[\mu_2=\sum_{\e\in\plMaxCut{\Graph}}\length(\e)[N(\e^*),\e^*]\cdot\mu_\e+\sum_{\substack{v\in\Vert{\Graph}\\g(v)>0}}\log(\delta_v)\cdot v+\sum_{e\in\HEdge{\Graph}}\log(\delta_e)\cdot e,\]where for a maximal cut system $\e$ we write $\e^*:=e^*\in\H^1(\Graph)$ for any $e\in\e$ (this is independent of $e$ by Remark~\ref{rmk:cut_pairs_and_cohomology}), $\length(\e):=\sum_{e\in\e}\length(e)$ for the total length of $\e$, and\[\mu_\e:=\frac1{\length(\e)}\sum_{e\in\e}|\d s_e|\]for the uniform distribution on $\e$ of total mass $1$.

We claim that the coefficients $[N(\e^*),\e^*]$, $\log(\delta_v)$ and $\log(\delta_e)$ are linearly independent in $\gr^\W_{-2}\V$ up to the single relation\[\sum_{\e\in\plMaxCut{\Graph}}\length(\e)[N(\e^*),\e^*]+\sum_{\substack{v\in\Vert{\Graph}\\g(v)>0}}\log(\delta_v)+\sum_{e\in\HEdge{\Graph}}\log(\delta_e)=0.\]Indeed, Lemma~\ref{lem:exterior_intersection_max_cut} implies that the elements $\e^*\otimes\e^*$ (for $\e\in\plMaxCut{\Graph}$) are linearly independent in $\H^1(\Graph)\otimes\H^1(\Graph)$, so that the elements $[N(\e^*),\e^*]$, $\log(\delta_v)$ and $\log(\delta_e)$ are linearly independent in $\gr^\M_{-2}\gr^\W_{-2}\oHAlg$. Since $\gr^\W_{-2}\V$ is a quotient of a subspace of $\gr^\M_{-2}\gr^\W_{-2}\oHAlg$ by a one-dimensional subspace, it follows that there is at most one non-trivial relation between these elements in $\gr^\W_{-2}\V$. Since $\mu_2$ has total mass $0$, we have the non-trivial relation $\int1\d\mu_2=0$ asserted above.

Thus $\int\Phi\d\mu_2=0$ if and only if there is a constant $\kappa$ such that $\int\Phi\d\mu_\e=\kappa$ and $\Phi(v)=\kappa=\Phi(\source(e))$ for all maximal cut systems $\e$, vertices $v$ with $g(v)>0$, and half-edges $e$. If such a $\kappa$ exists, we may normalise $\Phi$ so that $\kappa=0$, and this condition with $\kappa=0$ is equivalent to the given one.
\end{proof}
\end{lemma}

For the reader's convenience, we collect a few facts about the function $\Phi$ appearing in Lemma~\ref{lem:harmonic_condition}. For details, see Remark~\ref{rmk:compare_laplacians}, Lemma~\ref{lem:laplacian_nearly_iso}, Remark~\ref{rmk:stabilised_laplacians}, Proposition~\ref{prop:bounds_on_potentials} and Definition~\ref{def:poly_laplacian} in \S\ref{s:harmonic_analysis}.

\begin{remark}[Properties of $\Phi$]\label{rmk:Phi}
The function $\Phi$ from Lemma~\ref{lem:harmonic_condition} is given by $\Phi(v)=\langle y-x,v-x\rangle+\text{const}$. It is a piecewise affine function on $\rGraph$, constant on half-edges not containing $x$ or $y$ in their interior, and independent of subdivision of $\rGraph$. It satisfies the inequality $\Phi(x)\leq\Phi(v)\leq\Phi(y)$, and for any rational vertex $v$ satisfies the harmonicity condition\[\sum_{\source(e)=v}\Phi_e'(0)=\begin{cases}1&\text{if $v=x$,}\\-1&\text{if $v=y$,}\\0&\text{otherwise,}\end{cases}\]where $\Phi_e'(0)$ denotes the outgoing derivative of $\Phi$ with respect to arc length along an edge or half-edge $e$.
\end{remark}

\subsection{Reduction to $2$-connected graphs}

As a first application of our translation of the vanishing of non-abelian Kummer maps into the language of harmonic analysis, we prove that we can reduce the problem to the case of $2$-connected graphs.

\begin{lemma}\label{lem:reduction_to_2-conn_cpts}
Suppose that $\rGraph$ is a semistable reduction graph and $x,y$ two distinct rational vertices of $\rGraph$. Fix any integer $n\geq2$.
\begin{enumerate}
	\item If $\classifyleq n(x)=\classifyleq n(y)$, then $x$ and $y$ lie in the same $2$-connected component $\Cpt$ of $\rGraph$. (In particular, neither $x$ nor $y$ lie in the interior of a bridge or half-edge.)
	\item Suppose that $x$ and $y$ lie in the same $2$-connected component $\Cpt$ of $\rGraph$. If $\rCpt$ denotes the (semistable) reduction graph corresponding to $\Cpt$ as in \S\ref{ss:blocks}, then $\classifyleq n(x)=\classifyleq n(y)$ if and only if $\classifyleq n^{\rCpt}(x)=\classifyleq n^{\rCpt}(y)$, where $\classify^{\rCpt}\colon\QVert{\Cpt}\rightarrow\V(\rCpt)$ denotes the non-abelian Kummer map associated to $\rCpt$.
\end{enumerate}
\begin{proof}
The second part is an immediate consequence of Lemma~\ref{lem:block_reduction} and Corollary~\ref{cor:block_injectivity}. For the first part, assume for contradiction that $\classifyeq2(x)=\classifyeq2(y)$ but that $x$ and $y$ do not lie in the same $2$-connected component of $\rGraph$. We let $\Phi\colon\QVert{\Graph}\rightarrow\Q$ denote the function from Lemma~\ref{lem:harmonic_condition}, normalised so that $\kappa=0$.

Subdividing $\rGraph$ if necessary, we may assume $x$ and $y$ are vertices of $\rGraph$ not lying in the same block of $\rGraph$. Hence there is a cutvertex $u$ separating $x$ from $y$, i.e.\ $x$ and $y$ lie in different connected components of $\Graph\setminus\{u\}$. Suppose without loss of generality that $\Phi(u)\geq0$. We have $\Phi(v)=\langle y-u,v-u\rangle-\langle x-u,v-u\rangle+\Phi(u)$, and so by Proposition~\ref{prop:bounds_on_potentials} we have $\Phi(v)>0$ for $v$ in the same connected component of $\Graph\setminus\{u\}$ as $y$. We thus find from Lemma~\ref{lem:harmonic_condition} that the component of $\Graph\setminus\{u\}$ containing $y$ contains no non-bridge edges, half-edges or vertices of genus $>0$. This contradicts semistability of $\rGraph$.
\end{proof}
\end{lemma}

\subsection{Partial injectivity in weight $-2$}
\label{ss:partial_injectivity_weight_-2}

In light of Lemma~\ref{lem:reduction_to_2-conn_cpts}, in order to prove the second point of Theorem~\ref{thm:combinatorial_injectivity}, it suffices to consider only reduction graphs $\rGraph=\Graph$ satisfying the following property.
\begin{equation}\tag{$\ast$}\label{condn:2-conn}
\strut\parbox{\dimexpr\linewidth-4.25cm}{$\Graph$ is stable and its block--cutvertex decomposition consists of a single $2$-connected component and some number (possibly zero) of half-edges with distinct endpoints.}\strut
\end{equation}
Such a graph necessarily has $g(v)=0$ for all vertices $v$.

In working such graphs, we will routinely use the following facts about their isometric involutions, most of which are a straightforward consequence of the functoriality of our constructions under isomorphisms.

\begin{proposition}[Properties of involutions]\label{prop:iota}
Let $\Graph$ be a graph satisfying condition \condn{}, and let $\iota$ be an isometric involution of $\Graph$ such that the quotient $\Graph/\iota$ is a tree. Pick any $\iota$-fixed $b\in\QVert{\Graph}$ (such a point exists). Then:
\begin{itemize}
	\item the maximal cut systems in $\Graph$ have size $\leq2$ -- the maximal cut system containing an edge $e$ is $\{e,\iota(e)^{-1}\}$;
	\item the natural action $\iota_*$ of $\iota$ on $\gr^\W_{-n}\HAlg$ (resp.\ $\gr^\W_{-n}\LAlg$, resp.\ $\gr^\W_{-n}\V$) is multiplication by $(-1)^n$;
	\item if $\classify$ denotes the non-abelian Kummer map based at $b$, then $\classify\circ\iota=\iota_*\circ\classify$;
	\item for any $x\in\QVert{\Graph}$, the function $\Phi(z):=\langle z-b,\iota(x)-x\rangle$ satisfies:
	\begin{itemize}
		\item $\Phi(\iota(z))=-\Phi(z)$;
		\item $\Phi(z)=0$ if and only if $z$ is $\iota$-fixed; and
		\item $\Phi$ is not constant on any edge of $\Graph$.
	\end{itemize}
\end{itemize}
\end{proposition}

One direction of the desired implication -- that the existence of a suitable involution implies $\classifyeq2(x)=\classifyeq2(y)$ -- is immediate from these observations.

\begin{corollary}
Let $\Graph$ be a reduction graph satisfying \condn{} and $\iota$ an isometric involution of $\Graph$ fixing the half-edges of $\Graph$ pointwise such that $\Graph/\iota$ is a tree. Then we have\[\classifyeq2(\iota(x))=\classifyeq2(x)\]for all $x\in\QVert{\Graph}$.
\begin{proof}
By Remark~\ref{rmk:n-a_kummer_indept_basepoint} we may assume that the implicit basepoint $b$ in the definition of the non-abelian Kummer map $\classify$ is fixed by $\iota$. By Proposition~\ref{prop:iota} the induced action of $\iota$ on $\gr^W_{-2}\V$ is trivial, and hence $\classifyeq2(\iota(x))=\classifyeq2(x)$ as desired.
\end{proof}
\end{corollary}

Proving the converse direction is considerably more fiddly, so we begin with a special case.

\begin{proposition}\label{prop:depth_2_bananas}
Let $\Graph$ be a reduction graph satisfying condition \condn{} with two distinct rational vertices $x,y$ satisfying $\classifyeq2(x)=\classifyeq2(y)$. Suppose moreover that every maximal cut system in $\Graph$ has size $1$. Then either:
\begin{itemize}
    \item $\Graph$ is a $k$-banana graph\footnote{Recall that the \emph{$k$-banana graph} is the graph with two vertices connected by $k$ parallel edges.} for some $k\geq3$, has no half-edges, and $x$ and $y$ are mapped to one another under the banana involution; or
    \item $\Graph$ has a single vertex, edge and half-edge, and $x$ and $y$ are mapped to one another under the unique involution of $\Graph$ inverting its edge.
\end{itemize}
\begin{proof}
Take a function $\Phi$ as in Lemma~\ref{lem:harmonic_condition}, normalised so that $\kappa=0$. Recall from Remark~\ref{rmk:Phi} that $\Phi(x)<0<\Phi(y)$.

Consider the graph $\Graph_0$ formed by removing from $\Graph$ the $\leq2$ edges containing $x$ or $y$ in their interior; this graph is connected (since $\Graph$ is bridgeless and has no cut pairs). Since $\Phi$ is linear on edges of $\Graph_0$ with average value $0$ on each edge, there are two cases: either $\Phi=0$ on $\Graph_0$; or $\Graph_0$ is bipartite, taking values $\pm\alpha\neq0$ on the two vertex-classes.

Let us say that a vertex $v$ is \emph{$y$-like} if it is either equal to $y$ or the endpoint of an edge $e$ containing $y$ in its interior such that $x$ does not lie in the interval $[v,y]\subseteq e$. Note that vertices $v$ with $\Phi(v)>0$ are necessarily $y$-like, by harmonicity. We define \emph{$x$-like} similarly.

Now we claim that $\Graph$ has exactly one $y$-like vertex $v$ (resp.\ $x$-like vertex $u$). Indeed, suppose that $v$ and $v'$ are different $y$-like vertices, so that they are the endpoints of an edge $e$ containing $y$ but not $x$ in its interior. In order that the average value of $\Phi$ on $e$ be $0$, we must have $\Phi(v)=\Phi(v')<0$. Hence by harmonicity, $v$ and $v'$ must also be $x$-like, and so $\Phi(v)=\Phi(v')>0$ by the same argument. This gives a contradiction.

If now $u=v$, then $\Graph$ must contain at least one loop $e$ based at $u=v$ (with $x$ or $y$ or both in its interior). It then follows from condition \condn{} that $\Graph$ consists exactly of this loop together with a single half-edge with source $u=v$, and that both $x$ and $y$ lie in the interior of $e$.

If instead $u\neq v$, then by harmonicity the values $\pm\alpha$ of $\Phi$ attained at vertices must be non-zero. It follows that $u$ and $v$ are the only vertices of $\Graph$, and hence that $\Graph_0$ is a banana graph. By $2$-connectedness, we see that $\Graph$ itself is a banana graph (it has no half-edges since $\Phi$ doesn't vanish at vertices) and hence has $\geq3$ edges by stability. Note that $x$ and $y$ are both contained in a single edge $e$ of $\Graph$, and that either $\{x,y\}=\{u,v\}$ or both $x$ and $y$ lie in the interior of $e$.

It remains to show in these two cases that $x$ and $y$ are laid out symmetrically on $e$ -- we may assume that $\source(e)=u$, and that both $x$ and $y$ lie in the interior of $e$, with $x$ before $y$.  Note that in both cases we have $\Phi(u)+\Phi(v)=0$ and $\Phi'_e(0)+\Phi'_{e^{-1}}(0)=0$ by harmonicity (see Remark~\ref{rmk:Phi}). If $x$ and $y$ lie a distance $s_x$ and $\length(e)-s_y$ along $e$, then we compute\[\int_e\Phi|\d s_e|=\frac{2\Phi(u)+\length(e)\Phi_e'(0)}2(s_x-s_y).\]Since $\Phi(u)\leq0$ and $\Phi'_e(0)<0$, we thus have $s_x=s_y$, so that $x$ and $y$ are laid out symmetrically on $e$, as desired.
\end{proof}
\end{proposition}

\begin{proposition}\label{prop:partial_inj_for_2-conn_graphs}
Let $\Graph$ be a reduction graph satisfying condition \condn{} with two distinct rational vertices $x,y$ such that $\classifyeq2(x)=\classifyeq2(y)$. Then there is an isometric involution $\iota$ of $\Graph$ taking $x$ to $y$, fixing every half-edge of $\Graph$, and such that the quotient $\Graph/\iota$ is a tree.
\begin{proof}
We proceed by strong induction based on the \emph{Euler characteristic}\[\chi(\Graph)=2-2g(\Graph)-\#\HEdge{\Graph}\]of $\Graph$; in other words, we will prove the existence of the desired involution $\iota$ on $\Graph$ under the assumption that the result holds for all reduction graphs $\Graph'$ satisfying condition \condn{} and with $\chi(\Graph')>\chi(\Graph)$.

Suppose firstly that $\Graph$ has a half-edge $e$. If $\Graph$ consists of a single loop together with $e$ then we are done by~\ref{prop:depth_2_bananas}; otherwise we write $\Graph'$ for the stabilisation of the graph formed by contracting $e$ to a point as in \S\ref{ss:half-edge_removal}. We write $\rho\colon\QVert{\Graph}\rightarrow\QVert{\Graph'}$ for the contraction map, so that $\rho(x)=x$ and $\rho(y)=y$ since neither $x$ nor $y$ lie in the interior of $e$ by Lemma~\ref{lem:reduction_to_2-conn_cpts}. We see that $\chi(\Graph')>\chi(\Graph)$ and $\Graph'$ satisfies condition \condn{}, so that by inductive assumption there is an isometric involution $\iota$ of $\Graph'$ taking $x$ to $y$, fixing every half-edge of $\Graph'$, and such that the quotient $\Graph'/\iota$ is a tree.

Let $\Phi'\colon\QVert{\Graph'}\rightarrow\Q$ and $\Phi\colon\QVert{\Graph}\rightarrow\Q$ be the functions in Lemma~\ref{lem:harmonic_condition} for $\Graph'$ and $\Graph$ respectively, normalised so that $\kappa=0$ and similarly for $\Graph'$; we claim that $\Phi=\Phi'\circ\rho$. Indeed, we see that $\Lapl(\Phi'\circ\rho)=y-x$, so that $\Phi$ and $\Phi'\circ\rho$ differ by a constant which the normalisation conditions imply is~$0$.

In particular, we deduce that $\Phi'$ vanishes at the endpoint of $e$, so that this endpoint is fixed by $\iota$ (see Proposition~\ref{prop:iota}), and hence $\iota$ extends to an involution of $\Graph$ fixing $e$ pointwise. This satisfies all the desired properties.

\smallskip

It remains to consider the case that $\Graph$ has no half-edges, so $2$-connected. If all maximal cut systems in $\Graph$ have size $1$, then we are done by Proposition~\ref{prop:depth_2_bananas}, so we may suppose that $\Graph$ has a maximal cut system $\e$ of size $\geq2$. There are two possibilities:
\begin{enumerate}
	\item there is a $2$-connected component $\Cpt$ of $\Graph$ relative to $\e$ containing neither $x$ nor $y$ in its interior; or
	\item there are exactly two $2$-connected components of $\Graph$ relative to $\e$, and $x$ and $y$ lie in these two components.
\end{enumerate}

In fact, the second case cannot occur. To see this, let us write $\Cpt$ for the $2$-connected component of $\Graph$ relative to $\e$ which contains $x$ but not $y$; the relative block--cutvertex decomposition (Proposition~\ref{prop:rel_bc_decomp}) ensures that $\Cpt$ meets the rest of $\Graph$ in exactly two vertices. We write $\Graph'$ for the stabilisation of the graph $\Graph/\Cpt$ formed by contracting $\Cpt$ to an edge as in \S\ref{ss:resistance_reduction} and $\rho\colon\QVert{\Graph}\rightarrow\QVert{\Graph'}$ for the contraction map. We see easily that $\chi(\Graph')>\chi(\Graph)$ and that $\Graph'$ satisfies condition \condn{}. Since $\rho(x)\neq y$ in $\Graph'$, there is an isometric involution $\iota$ of $\Graph'$ interchanging $\rho(x)$ and $y$ such that $\Graph'/\iota$ is a tree.

However, we see from the relative block--cutvertex decomposition that the edge $e$ of $\Graph'$ containing $\rho(\Cpt)$ is a maximal cut system of size $1$ (its complement is the relative $2$-connected component of $\Graph$ containing $y$), and hence is inverted by $\iota$ (see Proposition~\ref{prop:iota}). This is a contradiction since $e$ contains $\rho(x)$ but not $y$ in its interior.

\smallskip

It remains to deal with the first case. Here we write $\Cpt$ for a $2$-connected component of $\Graph$ relative to $\e$ which contains neither $x$ nor $y$ in its interior, writing $\Graph'$ for the stabilisation of $\Graph/\Cpt$ and $\rho\colon\QVert{\Graph}\rightarrow\QVert{\Graph'}$ as before. Once again, we deduce that $\Graph'$ possesses an isometric isomorphism $\iota$ interchanging $x$ and $y$ such that $\Graph'/\iota$ is a tree.

Let $\Phi'\colon\QVert{\Graph'}\rightarrow\Q$ and $\Phi\colon\QVert{\Graph}\rightarrow\Q$ be the functions in Lemma~\ref{lem:harmonic_condition} for $\Graph'$ and $\Graph$ respectively, normalised so that $\kappa=0$ and similarly for $\Graph'$; we claim that $\Phi=\Phi'\circ\rho$. To see this, note that up to a constant, $\Phi'\circ\rho$ is the function $v\mapsto\langle\rho_*(v-x),y-x\rangle_{\Graph'}=\langle v-x,\rho^*(y-x)\rangle_{\Graph}$, where $\rho^*\colon\Div^0(\Graph')\rightarrow\Div^0(\Graph)$ is the adjoint to $\rho_*\colon\Div^0(\Graph)\rightarrow\Div^0(\Graph')$ under the height pairing. It follows from Proposition~\ref{prop:harmonic_quotient} that this map fits into a commuting square
\begin{center}
\begin{tikzcd}
\Q\cdot\mnEdge{\Graph'} \arrow{r}{\rho^*}\arrow{d}{\partial} & \Q\cdot\mnEdge{\Graph} \arrow{d}{\partial} \\
\Div^0(\Graph') \arrow{r}{\rho^*} & \Div^0(\Graph).
\end{tikzcd}
\end{center}
If we choose a path $\gamma$ from $x$ to $y$ in $\Graph'$ not meeting the interior of $e_\Cpt=\rho(\Cpt)$ (and view $\gamma$ as an element of $\Q\cdot\mnEdge{\Graph'}$) then by chasing $\gamma$ around the above square we see $\rho^*(y-x)=\rho^*\partial(\gamma)=\partial\rho^*(\gamma)=\partial(\gamma)=y-x$. Hence $\Phi'\circ\rho$ and $\Phi$ differ by a constant which is easily checked to be~$0$.

Now consider the restriction $\Phi_\Cpt$ of $\Phi=\Phi'\circ\rho$ to $\Cpt$. This is non-constant by Proposition~\ref{prop:iota}, and is easily seen to satisfy $\Lapl(\Phi_\Cpt)=\lambda(w_1-w_0)$ for some $\lambda\in\Q^\times$, where $w_0$ and $w_1$ are the two boundary vertices of $\Cpt$. Moreover, from Remark~\ref{rmk:max_cuts_in_2-conn_cpts} we see that\[\sum_{e\in\e}\int\Phi_\Cpt|\d s_e|=0\]for all maximal cut systems $\e$ in $\Cpt$. Since $\chi(\Cpt)>\chi(\Graph)$ and the stabilisation of $\Cpt$ satisfies condition \condn{}, we deduce from Lemma~\ref{lem:harmonic_condition} (for the function $\lambda^{-1}\Phi$) and our inductive hypothesis that there is also an isometric involution $\iota$ of $\Cpt$ interchanging $w_0$ and $w_1$ such that $\Cpt/\iota$ is a tree\footnote{Technically, this argument fails in the single case that $\Cpt$ is a $2$-banana graph, i.e.\ a loop with two vertices on it, when the stabilisation does not exist. In this case, the existence of the desired involution $\iota$ is obvious.}.

It remains to show that the isometric involutions $\iota$ of $\Graph'$ and $\Cpt$ can be glued to produce an involution of $\Graph$ with the desired properties. The only thing to prove is that the involution $\iota$ of $\Graph'$ interchanges $w_0$ and $w_1$. But the involution on $\Cpt$ ensures that $\Phi(w_0)=-\Phi(w_1)$, so that the midpoint of the segment $\rho(\Cpt)$ is fixed by $\iota$ by Proposition~\ref{prop:iota}. It follows that $w_0$ and $w_1$ are interchanged by $\iota$ on $\Graph'$, as desired.
\end{proof}
\end{proposition}

\subsection{Injectivity in weight $-3$}

To complete the proof of the Injectivity Theorem~\ref{thm:combinatorial_injectivity} we finally address the third point: the injectivity of the non-abelian Kummer map associated to a stable reduction graph $\rGraph$ in weights $\geq-3$. With the strong restrictions on the fibres of $\classifyeq2$ already proved, we only need consider a very limited set of reduction graphs $\rGraph$ -- using the reduction operations in \S\ref{s:graph_ops} we can even reduce this to two particular examples. These examples we deal with by hand, exploiting their natural symmetries to simplify the argument.

\begin{lemma}
Let $\rGraph$ be a stable reduction graph and $x,y$ two distinct rational vertices of $\rGraph$. Then $\classifyleq3(x)\neq\classifyleq3(y)$.
\begin{proof}
Suppose for contradiction that $\classifyleq3(x)=\classifyleq3(y)$. By Lemma~\ref{lem:reduction_to_2-conn_cpts}, $x$ and $y$ lie in the same $2$-connected component of $\rGraph$, and we may assume that $\rGraph=\Graph$ satisfies \condn{}. Since $\classifyeq2(x)=\classifyeq2(y)$, we obtain by Proposition~\ref{prop:partial_inj_for_2-conn_graphs} an isometric involution $\iota$ of $\Graph$ fixing its half-edges pointwise, interchanging $x$ and $y$, and such that the quotient $\Graph/\iota$ is a tree. By a repeated application of the reduction operations in \S\ref{ss:half-edge_removal} and \S\ref{ss:resistance_reduction} (and stabilisation) we obtain a graph $\Graph'$ along with a map\[\rho\colon\QVert{\Graph}\rightarrow\QVert{\Graph'}\]where $\Graph'$ is either a $3$-banana graph with no half-edges or is a single loop with a single half-edge, $\rho(x)\neq\rho(y)$ are exchanged by the evident involution of $\Graph'$, and $\classifyleq3(\rho(x))=\classifyleq3(\rho(y))$ by Lemmas~\ref{lem:half-edge_reduction} and~\ref{lem:resistance_reduction}. It thus suffices to derive a contradiction in the case that $\Graph$ is either of these two graphs and $\iota$ its evident involution, exchanging $x$ and $y$.

\smallskip

We begin with the case that $\Graph$ is a loop $e$ with a single half-edge (so a single vertex $b$, which we take as the basepoint in the definition of the non-abelian Kummer map). Since $\iota_*$ acts as $-1$ on $\gr^\W_{-3}\V$ (see Proposition~\ref{prop:iota}), we have that $\classifyeq3\circ\iota=-\classifyeq3$ as $\Q$-valued functions on $\QVert{\Graph}$. Since by assumption $\classifyeq3(x)=\classifyeq3(y)$, we have see that $\classifyeq3(x)=\classifyeq3(y)=\classifyeq3(b)=0$.

Now pick any functional $\psi\colon\gr^\W_{-3}\HAlg\rightarrow\Q$ such that $\psi(\ad_{e^*}^2(N(e^*)))\neq0$. It follows from Proposition~\ref{prop:leading_coeffs} that the function\[\psi\circ\classifyeqe3e\colon s\mapsto\psi\circ\classifyeq3(v_{e,s})\]is a cubic function in $s$ with $s^3$ coefficient equal to $\frac13\psi(\ad_{e^*}^2(N(e^*)))\neq0$. However, this function vanishes for four values of $s$ (one corresponding to each of $x$ and $y$, and two corresponding to $b$), which is impossible.

\smallskip

It remains to treat the case that $\Graph$ is a $3$-banana graph with no half-edges. We may label its edges as $e_0,e_1,e_2$ such that $\source(e_0)=\source(e_1)=\source(e_2)$ and $x,y$ lie on $e_0$. Choosing a basepoint $b$ at the midpoint of $e_0$, we see just as before that the map $\classifyeq3$ satisfies $\classifyeq3\circ\iota=-\classifyeq3$, and hence by assumption $\classifyeq3(x)=\classifyeq3(y)=0$.

Now $\LAlg$ is the free bigraded Lie algebra generated by $e_0^*$ and $e_1^*$ in bidegree $(-1,0)$ (note that $e_2^*=-e_0^*-e_1^*$) and $N(e_0^*)$ and $N(e_1^*)$ in bidegree $(-1,-2)$, subject to the sole relation\[(\length_0+\length_2)[e_0^*,N(e_0^*)]+(\length_1+\length_2)[e_1^*,N(e_1^*)]+\length_2[e_0^*,N(e_1^*)]+\length_2[e_1^*,N(e_0^*)]=0\]in bidegree $(-2,-2)$, where we write $\length_i:=\length(e_i)$ for short. We now choose a functional $\psi\colon\gr^\W_{-3}\gr^\M_{-2}\LAlg\rightarrow\Q$ such that\[\psi(\ad_{e_0^*}^2(N(e_0^*)))=\psi(\ad_{e_1^*}^2(N(e_1^*)))=0\neq\psi(\ad_{e_2^*}^2(N(e_2^*)));\]for instance we may take the map given by\[\psi(\ad_{e_i^*}\ad_{e_j^*}(N(e_k^*)))=\begin{cases}\length_1+\length_2&\text{if $(i,j,k)=(0,1,0)$,}\\-\length_2&\text{if $(i,j,k)=(0,1,1)$,}\\0&\text{else.}\end{cases}\]

Now the function $\psi\circ\classifyeq3\colon\QVert{\Graph}\rightarrow\Q$ restricts to a cubic function on each edge, with $s^3$ coefficient equal to $\frac13\psi(\ad_{e_i^*}^2(N(e_i^*)))$ on edge $e_i$ by Proposition~\ref{prop:leading_coeffs}. In particular, $\psi\circ\classifyeq3$ is actually quadratic on edges $e_0$ and $e_1$. Since $\psi\circ\classifyeq3\circ\iota=-\psi\circ\classifyeq3$, it is actually linear on these edges. Since $\psi\circ\classifyeq3$ vanishes at $x$ and $y$, it vanishes uniformly on $e_0$, and hence also on $e_1$.

Hence the function $\psi\circ\classifyeq3$ vanishes at both endpoints of $e_2$. Since it is harmonic, it in fact vanishes to order $2$ at these endpoints, which is impossible since it is a non-zero cubic function on $e_2$. This concludes the proof.
\end{proof}
\end{lemma} 
\section{The anabelian reconstruction theorem}\label{s:reconstruction}

We now turn to the proof of the anabelian construction theorem (Theorem~\ref{thm:anabelian_reconstruction}). As promised in Remark~\ref{rmk:mono-anabelian}, we will in fact give an explicit recipe for reconstructing the stable reduction graph $\rGraph$ of a hyperbolic curve $X/K$ from its $\Q_\ell$-pro-unipotent \'etale fundamental groupoid $\pi_1^{\Q_\ell}(X_{\overline K})$, whose object-set is $X(\overline K)$. This will be accomplished in the following sequence of three propositions, explaining how to recover the underlying graph, metric and genus function on $\rGraph$ respectively. The first of these is simply the observation that we can read off the underlying graph from the image of the reduction map via a join-the-dots process.

\begin{proposition}[Reconstruction of the graph]\label{prop:reconstruct_graph}
Let\[\classify\colon\ob(\pi_1^{\Q_\ell}(X_{\overline K}))\rightarrow\varinjlim_L\H^1(G_L,\pi_1^{\Q_\ell}(X_{\overline K},b))\]denote the non-abelian Kummer map defined in \S\ref{ss:nonab_Kummer}, at some basepoint $b\in\ob(\pi_1^{\Q_\ell}(X_{\overline K}))$. (This definition uses only the $G_K$-equivariant groupoid structure of $\pi_1^{\Q_\ell}(X_{\overline K})$.)

Choose an embedding $\Q_\ell\hookrightarrow\C$. Then the topological closure of the image of $\classify$ in $\varinjlim_L\H^1(G_L,U)(\C)$ for the $\C$-topology is the underlying topological space of the reduction graph $\rGraph$, and the map $\classify\colon\ob(\pi_1^{\Q_\ell}(X_{\overline K}))\rightarrow\overline{\im(\classify)}$ is the reduction map.
\begin{proof}
The non-abelian Kummer map fits into a commuting square
\begin{center}
\begin{tikzcd}
X(\overline K) \arrow{r}{\classify}\arrow[hook]{d}{\red} & \varinjlim_L\H^1(G_L,U)(\C)\arrow{d}{\iota} \\
\Graph \arrow{r}{\classify} & \V_\C
\end{tikzcd}
\end{center}
where the left-hand vertical arrow has dense image, the bottom map is a closed immersion by Remark~\ref{rmk:injectivity_over_R} and the right-hand vertical map is a linear map of pro-finite-dimensional $\C$-vector spaces. This establishes the proposition.
\end{proof}
\end{proposition}

\begin{remark}
In Remark~\ref{rmk:injectivity_over_R} we stated a conjecture that the non-abelian Kummer map should induce a closed immersion $\classify\colon\A(\Graph)\hookrightarrow\V$ of the algebraisation of the graph. This would imply a counterpart to Proposition~\ref{prop:reconstruct_graph} for the Zariski topology: the Zariski-closure of the image of $\classify\colon X(\overline K)\rightarrow\varinjlim_L\H^1(G_L,U)$ should be $\A(\Graph)_{\Q_\ell}$.
\end{remark}

Note that Proposition~\ref{prop:reconstruct_graph} reconstructs the graph structure on $\Graph=\overline{\im(\classify)}$: the vertices are exactly the points not having any neighbourhood homeomorphic to $(-\epsilon,\epsilon)$, the edges and the half-edges are the connected components of the complement of the vertices, and the endpoint(s) of an edge or half-edge are the points of its boundary. The metric on edges can be recovered as follows.

\begin{proposition}[Reconstruction of the metric]
Let $e$ be an edge of $\Graph$ (regarded as an open subset of $\overline{\im(\classify)}$). Then
\[
\length(e) = \lim_{L/K}\frac{\#\left(\classify\left(\ob(\pi_1^{\Q_\ell}(X_{\overline K}))^{G_L}\right)\cap e\right)}{e(L/K)}
\]
where the limit is taken over finite extensions $L/K$.
\begin{proof}
Suppose that $L/K$ is an extension over which $X$ acquires split semistable reduction, so that $e$ corresponds to a chain of $e(L/K)\cdot\length(e)-1$ copies of $\P^1_{k_L}$ in the special fibre of a split semistable model, where $k_L$ is the residue field of $L$. Each of these copies of $\P^1_{k_L}$ contains the reduction of an $L$-point of $X$, and hence $\#\left(\classify(X(L))\cap e\right)=e(L/K)\cdot\length(e)-1$. This rearranges to the desired equality.
\end{proof}
\end{proposition}

It remains to recover the genus function at (rational) vertices of $\Graph$. For this we use the $\M$-filtration and monodromy operator $N$ on $\O(\pi_1^{\Q_\ell}(X_{\overline K}))^\dual$, which are determined by the Galois action as in \S\ref{ss:gal_coh_field}.

\begin{proposition}[Reconstruction of the genus function]
Let $v$ be a vertex of $\rGraph$ which is the reduction of some $x\in X(\overline K)$. Then
\[
g(v) = \frac12\dim_{\Q_\ell}\left(\left(\gr^\M_{-1}\O(\pi_1^{\Q_\ell}(X_{\overline K},x))^\dual\right)^{N=0}\right).
\]
\begin{proof}
We will prove that $V:=\left(\gr^\M_{-1}\O(\pi_1^\Q(\rGraph,v))^\dual\right)^{N=0}$ is spanned by $\log(\beta_{v,i})$ and $\log(\beta_{v,i}')$ for $1\leq i\leq g(v)$. That these elements are linearly independent and $N$-invariant is obvious, so we need only show that $V$ is contained in the span of these elements. We proceed in three steps.

Firstly, we consider the isomorphism $\Hiso\colon\gr^\M_\bullet\O(\pi_1^\Q(\rGraph,v))^\dual\isoarrow\HAlg$ from Theorem~\ref{thm:description_of_M-graded}, and claim that the image of $V$ is contained in $\gr^\W_{-1}\gr^\M_{-1}\HAlg$. To see this, consider any $\gamma\in V$ and write $\Hiso(\gamma)=\gamma_1+\gamma_2$ with $\gamma_1\in\gr^\W_{-1}\gr^\M_{-1}\HAlg$ and $\gamma_e\in\W_{-2}\gr^\M_{-1}\HAlg$. Since $\gamma$ is $N$-invariant, it follows from Lemma~\ref{lem:equivariance_coM} that $N\Hiso(\gamma)=N(\gamma_2)\in\coM_{-2}\gr^\M_{-3}\HAlg$, where $\coM_\bullet$ is the filtration on $\HAlg$ defined in Definition~\ref{def:coM}. By definition, this already implies that $\gamma_2\in\coM_{-2}\gr^\M_{-1}\HAlg$, and hence $\gamma_2=0$ by Proposition~\ref{prop:coM_basics}. Hence $\Hiso(\gamma)=\gamma_1\in\gr^\W_{-1}\gr^\M_{-1}\HAlg$ as claimed.

Next, for a vertex $u\in\Vert{\Graph}$ we write $V_u\leq\gr^\W_{-1}\gr^\M_{-1}\HAlg$ for the span of the elements $\alog(\beta_{u,i})$ and $\alog(\beta_{u,i}')$ for $1\leq i\leq g(u)$, and claim that the image of $V$ inside $\gr^\W_{-1}\gr^\M_{-1}\HAlg=\bigoplus_{u\in\Vert{\Graph}}V_u$ splits as a direct sum of some of the $V_u$. To see this, let $\MCG_u$ denote the group of automorphisms of the free group on $2g(u)$ generators $\beta_{u,i}$ and $\beta_{u,i}'$ ($1\leq i\leq g(u)$) which fix the element $\prod_{i=1}^{g(u)}[\beta_{u,i}',\beta_{u,i}]$. There is a natural representation $\rho_u\colon\MCG_u\rightarrow\GL_{2g(u)}(\Z)$ given by the action on the abelianisation. The image of this representation is $\Sp_{2g(u)}(\Z)$: we may identify $\MCG_u$ as the mapping class group of an oriented genus $g(u)$ surface with one puncture and one marked point by the Dehn--Nielsen--Baer Theorem \cite[Theorem 8.8 \& below]{primer}, and apply \cite[Theorem 6.4]{primer}.

Now there is an $N$-equivariant action of $\prod_u\MCG_u$ on $\O(\pi_1^\Q(\rGraph))^\dual$ acting in the obvious way on generators of the form $\beta_{u,i}$ and $\beta_{u,i}'$ from Definition~\ref{def:pi1_reduction_graph}, and acting trivially on edges $e$ and generators of the form $\delta_e$. The induced $N$-equivariant action on $\gr^\M_\bullet\O(\pi_1^\Q(\rGraph))^\dual$ then factors through $\prod_u\Sp_{2g(u)}(\Z)$ such that the isomorphism $\Hiso\colon\gr^\M_\bullet\O(\pi_1^\Q(\rGraph))^\dual\isoarrow\HAlg$ is $\prod_u\Sp_{2g(u)}(\Z)$-equivariant for the natural action on the right-hand side (acting on generators $\alog(\beta_{u,i})$ and $\alog(\beta_{u,i}')$ in the obvious manner, and fixing all other generators). Thus $\Hiso V$, being a $\prod_u\Sp_{2g(u)}(\Z)$-subrepresentation of the tautological representation $\bigoplus_uV_u$, splits as a direct sum of some of the components $V_u$.

Finally, we show that the image of $V$ does not contain $V_u$ for $u\neq v$ of genus $>0$, i.e.\ that $V$ does not contain all the elements $\gamma_{u,v}^\can\log(\beta_{u,i})\gamma_{v,u}^\can$ and $\gamma_{u,v}^\can\log(\beta_{u,i}')\gamma_{v,u}^\can$ for some such $u$. Indeed, if it did, then we would have
\[
0=\Hiso N\left(\gamma_{u,v}^\can\log(\beta_{u,i})\gamma_{v,u}^\can\right)=[\alog(\beta_{u,i}),\Hiso N(\gamma_{v,u}^\can)]\text{ in $\gr^\M_{-3}\HAlg$}
\]for all $1\leq i\leq g(u)$, and similarly for $\alog(\beta_{u,i}')$. But since the common centraliser of the elements $\alog(\beta_{u,i})$ and $\alog(\beta_{u,i}')$ in $\HAlg$ is just $\Q$, this forces $0=\Hiso N(\gamma_{v,u}^\can)=\classify(u)-\classify(v)$, contradicting the Injectivity Theorem~\ref{thm:combinatorial_injectivity}. This completes the proof.
\end{proof}
\end{proposition}
\section{Applications to the Chabauty--Kim method}\label{s:examples}
In this section we give a survey of applications of Theorem \ref{thm:description_of_graph_kummer}. We first give a general discussion of applications of Theorem \ref{thm:description_of_graph_kummer} to quadratic Chabauty algorithms. In the case of $X$ a projective curve minus a point, this leads to a relation to Zhang's canonical measure on the reduction graph of a pointed curve. In the case of a general affine curve, we explain how the explicit description of the image of $X(K)$ in $\H^1 (G_K ,U_n )$ leads to non-trivial refinements of the sets $X(\Z_p )_{S,n}$ of weakly global points defined in \cite{BDCKW}. For projective curves, it also gives improvements on the bounds on $\#X(\Q _\ell )_2 $ given in \cite{BD}.
\subsection{Application to the quadratic Chabauty method}\label{ss:application_QC}
In this subsection we assume that $X$ is either projective or projective minus a point. In either case, we have $\H ^1 (G_K ,U_1 )=0$. The setting of quadratic Chabauty, in the sense of \cite{BalakrishnanDogra1}, is that one has a Galois-equivariant surjection 
$
\gr _{\Cent} ^2 U_2  \twoheadrightarrow \Q _{\ell }(1).
$
Since $\gr _{\Cent} ^2 U_2 $ is central in $U_2$, this defines by pushout what we shall refer to as a $\Q _{\ell }(1)$-\textit{quotient} of~$U_2 $
\[
1\to \Q _{\ell }(1)\to U_Z \to U_1 \to 1.
\]
Since $\H ^1 (G_K ,U_Z )\simeq \H^1 (G_K ,\Q _{\ell }(1))$, the non-abelian Kummer map defines a map
\[
j_Z :X(\Q _{\ell })\to \H^1 (G_K ,\Q _{\ell }(1)).
\]
Via an isomorphism 
\begin{equation}\label{eqn:normalisation_character}
\H ^1 (G_K ,\Q _{\ell }(1))\simeq \Q _{\ell },
\end{equation}
we obtain a $\Q _{\ell }$-valued map. In this subsection we describe how Theorem \ref{thm:description_of_graph_kummer} allows us to explicitly compute $j_Z$ for a specific choice of \eqref{eqn:normalisation_character}. One motivation for such an explicit computation is for algorithms for determining rational points on curves. More precisely, by \cite[Theorem 1.2]{BalakrishnanDogra1}, to obtain an algorithm to compute the obstruction set $X(\Q _{\ell })_{U_Z }$, one needs an algorithm to compute the local $\ell $-adic heights $h_v (A_Z (x))$ for all $\ell $, where $A_Z (x)$ is a certain filtered Galois representation with graded pieces $\Q _p ,U_1 ,\Q _p (1)$, and $h_v$ is the local height associated to an idele class character $\chi =(\chi _v )_v $.  Algorithms for computing the contribution at $\ell $ are given in \cite{Balakrishnanetc}. As we explain below, we may recover the contributions at other primes from a computation of $j_Z$.

We briefly recall some properties of the passage from Selmer schemes to mixed extensions (see e.g. \cite[4.3.1]{BalakrishnanDogra1}). Given representations $V_0 ,V_1 ,V_2 $ of a profinite group $G$, a mixed extension with graded pieces $V_0 ,V_1 ,V_2 $ is a filtered $G$-representation
\[
M=M_0 \supset M_1 \supset \ldots \supset M_3 =0
\]
together with $G$-equivariant isomorphisms 
\[
\psi _i :V_i \stackrel{\simeq }{\longrightarrow }M_i /M_{i+1}
\]
There is a natural notion of a morphism of mixed extensions, and the set of isomorphism classes of mixed extensions is naturally in bijection with $\H^1 (G,U(V_0 ,V_1 ,V_2 ))$, where $U(V_0 ,V_1 ,V_2 )$ is the group of isomorphisms of $V_0 \oplus V_1 \oplus V_2 $ which act unipotently with respect to the filtration
\[
V_0 \oplus V_1 \oplus V_2  \supset V_1 \oplus V_2 \supset V_2 .
\]
More generally, for any mixed extension $M$ with graded pieces $V_0 ,V_1 ,V_2 $, if $\Aut (M)$ denotes the group of vector space automorphisms of $M$ which are unipotent with respect to the its filtration, then there is a bijection between the set of isomorphism classes of mixed extensions and $\H^1 (G,\Aut (M))$, given by sending a mixed extension $M'$ to the $G$-equivariant $\Aut (M)$-torsor of unipotent vector space isomorphisms between $M$ and $M'$.

The universal enveloping algebra of $\Lie (U_Z )$ has a quotient $A_Z (b)$ which is a mixed extension with graded pieces $\Q _\ell ,U_1 ,\Q _{\ell }(1)$. The action of $U_Z$ on $A_Z (b)$ by left multiplication defines a morphism
\[
\H^1 (G_K ,U_Z )\to \H^1 (G_K ,U(\Q _{\ell },U_1 ,\Q _{\ell }(1)))
\]
sending a $U_Z$-torsor $c$ to the twist of $A_Z (b)$ by the torsor $c$. This map also admits a description as the composite
\[
\H^1 (G_K ,U_Z )\to \H^1 (G_K ,\Aut (A_Z (b)))\isoarrow  \H^1 (G_K ,U(\Q _{\ell },U_1 ,\Q _{\ell }(1)))
\]
where the first map is pushforward by the homomorphism $U_Z \to \Aut (A_Z (b))$ and the second
 is induced by the $G$-equivariant $(\Aut (A_Z (b)),U(\Q _{\ell },U_1 ,\Q _{\ell }(1)))$-bitorsor of unipotent isomorphisms between $A_Z (b)$ and $\Q _{\ell }\oplus U_1 \oplus \Q _{\ell }(1)$. The \textit{local height} of a mixed extension $M$ with graded pieces $\Q _p ,U_1 ,\Q _p (1)$,  with respect to a character $\chi \in \H^1 (G_K ,\Q _{\ell })$, is $\chi \cup [M]$, where $[M]$ is a cohomology class in $ \H ^1 (G_K ,\Q _{\ell }(1))$, and $\H ^2 (G_K ,\Q _{\ell }(1))$ is identified with $\Q_\ell$ via local class field theory. By \cite[4.3.1]{BalakrishnanDogra1}, the cohomology class $[M]$ may be defined to be the image of the class of $M$ in $\H ^1 (G_K ,U(\Q _{\ell },U_1 ,\Q _{\ell }(1))$ in $\H^1 (G_K ,\Q _{\ell }(1))$ under the isomorphism
\begin{equation}\label{eqn:endless_waffle}
\H ^1 (G_K ,\Q _{\ell }(1))\simeq \H ^1 (G_K ,U(\Q _{\ell },U_1 ,\Q _{\ell }(1))
\end{equation}
induced by the short exact sequence
\[
1\to \Q _{\ell }(1)\to U(\Q _{\ell },U_1 ,\Q _{\ell }(1))\to U_1 \oplus \Hom (U_1 ,\Q _{\ell }(1))\to 1
\]
and the fact that $U_1 \simeq \Hom (U_1 ,\Q _{\ell }(1))$ have vanishing cohomology.

\begin{lemma}
For all $x\in X(\Q _p )$,
\[
h_p (A_Z (x))=\chi _p \cup j_Z (x),
\]
where the cup product is from viewing $\chi _p $ as an element of $\H ^1 (G_{\Q _p },\Q _{\ell }(1))$, and the identification of $\H ^2 (G_{\Q _p },\Q _{\ell }(1))$ with $\Q _\ell $ is from local class field theory.
\end{lemma}
\begin{proof}
By the above, the map $x\mapsto h_p (A_Z (x))$ can be described as the image of $j_{U_Z}(x)$ under the composite map
\[
\H^1 (G_K ,U_Z )\to \H^1 (G_K ,U(\Q _{\ell },U_1 ,\Q _{\ell }(1))) \stackrel{\simeq }{\longrightarrow } \H^1 (G_K ,\Q _{\ell }(1))\stackrel{\cdot \cup \chi _p }{\longrightarrow } \Q _{\ell }.
\]
The first map is the pushforward induced by an inclusion $U_Z \hookrightarrow U(\Q _{\ell },U_1 ,\Q _{\ell }(1))$ which restricts to the identity on the subgroups $\Q _{\ell }(1)$. Putting these observations together, we deduce that $h_p (A_Z (x))=\chi _p \cup j_Z (x)$.
\end{proof}

Recall from \S\ref{s:M-trivialisation} we have an isomorphism
\[
\Lie(U_2 )^\can = \Lie (\gr _C ^2 U_2  )^{\can }\simeq \gr ^\W _{-2}\V_{\Q _{\ell }}.
\]
The quotient $U_Z$ satisfies \WM and the inclusion $\Q_\ell(1)\hookrightarrow U_Z$ induces an isomorphism $\Q_\ell(1)\simeq\Lie(U_Z)^\can$. Hence by \S\ref{ss:nonab_Kummer} we have a map
\[
\classify_Z\colon X(\overline K) \rightarrow \Hom(\Q_\ell(1),\Lie(U_Z)^\can)=\Q_\ell
\]
which is a coordinate of the non-abelian Kummer map in depth $2$.

By Theorem~\ref{thm:graph_comparison}, the surjection $Z$ corresponds to a surjection
\[
Z_*\colon\gr^\W_{-2}\LAlg_{\Q_\ell}\twoheadrightarrow\Q_\ell,
\]
and by Theorem~\ref{thm:description_of_kummer} the map $\classify_Z$ is given by height pairing against the measure $\mu_Z=Z_*(\mu_2)$. Our aim is to describe this measure, and hence the map $\classify_Z$, explicitly  for certain $\Q_\ell(1)$-quotients $U_Z$.

\subsubsection{$\Q _{\ell }(1)$-quotients from trace zero endomorphisms of the Jacobian}
We first recall how $\Q _\ell (1)$-quotients of $U_2 $ arise from endomorphisms of the Jacobian. For the remainder of this subsection we assume $X$ is projective.
Given an endomorphism $F$ of $U_1 =\H_1^\et(X_{\overline K},\Q_\ell)$, let $c( F)$ denote the corresponding homomorphism 
$
\Q _{\ell }(-1)\to \H^1 _{\et }(X_{\overline{K}},\Q _{\ell })^{\otimes 2}
$
defined as the composite
\begin{equation}\label{eq:chern_cup}
\Q _{\ell }(-1)\to \H^1 _{\et }(X_{\overline{K}},\Q _{\ell }) \otimes  \H ^1_\et(X_{\overline{K}},\Q _{\ell })^* (-1)\to  \H^1 _{\et }(X_{\overline{K}},\Q _{\ell })^{\otimes 2},
\end{equation}
where the first map is $\Q _{\ell }(-1)\otimes $ the endomorphism $F$, and the second is $\H^1 _{\et }(X_{\overline{K}},\Q _{\ell })\otimes $ Poincar\'e duality. Under Poincar\'e duality, the cup-product corresponds to the trace.
Hence given an $F$ of trace zero, we obtain a homomorphism $\gr _{\Cent} ^2 U_2 \to \Q _{\ell } (1)$ defined as the composite
\begin{align*}
& \gr ^2 _C U=\ker\left(\bigwedgesquare\H^1_\et(X_{\overline K},\Q_\ell)\stackrel\cup\longrightarrow\Q_\ell(-1)\right)^\dual  \\
& \hspace{2cm}\hookrightarrow 
\ker \left( \H ^1 _{\et }(X_{\overline{K}},\Q _{\ell })^{\otimes 2}\stackrel\cup\longrightarrow\Q_\ell(-1)\right)^\dual
\stackrel{c(F)^\dual }{\longrightarrow }\Q _{\ell }(1)
\end{align*}
(here and in what follows, for a vector space $V$ we take the map $\bigwedgesquare V\to V^{\otimes 2}$ to be the map $v\wedge w \mapsto \frac{1}{2}(v\otimes w-w\otimes v)$).
This morphism is zero if and only if $F$ is in the $(-1)$-eigenspace of the Rosati involution (see \cite[11.5.3]{BirkenhakeLange}). If $F$ is non-zero and fixed by the Rosati involution, 
 then $F$ induces a nonzero morphism
$
\gr ^2 _C U^{\can }\to \Q _{\ell }(1),
$
which we also denote $c(F)^* $,
and hence defines a $\Q_\ell(1)$-quotient of $U_2$. To ease notation we will denote $\mu _{c(F)^* }$ simply as $\mu _F$.

\subsubsection{$\Q _{\ell }$-quotients of $\gr ^{\W }_{-2}\V $}
We now explain how to determine the corresponding quotient of $\LAlg_{\Q _{\ell }}$, the corresponding surjection $\gr ^{\W }_{-2}\V _{\Q _{\ell }}\to \Q _{\ell }$, and hence the $\Q _{\ell }$-valued measure $\mu _F$.
Poincar\'e duality induces a map
\[
\gr _{\bullet}^{\M }\bigwedge \nolimits^2 \H^1 _{\et }(X_{\overline{K}},\Q _{\ell }) \to \Q_{\ell }(-1), 
\]
and the kernel is dual to $\gr _\Cent ^2 U_2 $.
Hence the kernel of `graded Poincar\'e duality'
\[
\gr _{\bullet}^{\M }\bigwedgesquare \H^1 _{\et }(X_{\overline{K}},\Q _{\ell }) \to \Q_{\ell }(-1), 
\]
is dual, via Theorem \ref{thm:description_of_M-graded}, to $\gr _\Cent ^2 \LAlg$. 
Recall that we have a distinguished choice of generator of the kernel of 
\[
\bigwedgesquare \LAlg^{\ab }\to \gr _\Cent ^2 \LAlg,
\]
given by
\begin{equation}\label{eqn:relation}
\sigma =\sum_{i=1}^\gGamma \xi_i' \wedge \xi_i +\sum_v\sum_{i=1}^{g(v)}\alog(\beta_{v,i}')\wedge \alog(\beta_{v,i})
\end{equation}
(see Construction \ref{cons:HAlg}). Viewing \eqref{eqn:relation} as an element of $(\mathcal{L}^{\ab })^{\otimes 2}$, we obtain an isomorphism
\[
(\mathcal{L}^{\ab})^* \stackrel{\sim }{\longrightarrow }\mathcal{L}^{\ab },
\]
given by sending $v \in (\mathcal{L}^{\ab })^*$ to the functional sending $w \in (\mathcal{L}^{\ab })^*$ to
\begin{align*}
2v \otimes w (\sigma )= & \sum_{i=1}^\gGamma v (\xi_i' ) w^* ( \xi_i )-w (\xi_i' ) v ( \xi_i ) \\
& +\sum_v\sum_{i=1}^{g(v)}v (\alog(\beta_{v,i}')) w (\alog(\beta_{v,i}))-w (\alog(\beta_{v,i}')) v (\alog(\beta_{v,i})).
\end{align*}
Hence, $\beta _{v,i}\mapsto (\beta _{v,i}')^* $ and $\beta _{v,i}'\mapsto -\beta _{v,i}^* $. By Remark \ref{rmk:e*e_is_identity}, $e^*$ is sent to the projection of $e$ onto $\H _1 (\Gamma )$, and the orthogonal projection of $e$ onto $\H _1 (\Gamma )$ is sent to $-e^*$. Tensoring with $\mathcal{L}^{\ab }$, we obtain isomorphisms
\begin{equation}\label{eqn:gr2_gr0}
\gr _{-2}^{\M }(\LAlg^{\ab \otimes 2 }_{\Q _{\ell }})\stackrel{\sim }{\longrightarrow }\gr _0 ^{\M }\left( \LAlg^{\ab }_{\Q _{\ell } }\otimes (\LAlg^{\ab }_{\Q _{\ell }})^* \right).
\end{equation}
(here and in what follows, for a vector space $V$ we make the usual identification of $V \otimes V^* $ with the endomorphism algebra of $V$).
The isomorphism \eqref{eqn:gr2_gr0} induces an injection
\[
\gr _{-2}^{\M }(\gr ^2 _{\Cent }\LAlg )\hookrightarrow \left(\gr _0 ^{\M }\left(\mathrm{End} (\LAlg )^{\ab }\right)^{\Tr =0}\right)^* .
\]
Dualising, an $\M$-graded endomorphism of $\LAlg^\ab_{\Q_\ell}$ of trace zero hence defines a morphism
$
\gr _\Cent ^2 \LAlg_{\Q _{\ell }}\to \Q _{\ell }$,
inducing a commutative diagram whose vertical maps are isomorphisms
\[
\begin{tikzpicture}
\matrix (m) [matrix of math nodes, row sep=3em,
column sep=3em, text height=1.5ex, text depth=0.25ex]
{\gr _\bullet ^{\M }\gr ^2 _\Cent U & \gr _\bullet ^{\M} \End (U^{\ab })^{\Tr =0} & \Q _{\ell }(1) \\
 \gr ^2 _\Cent \LAlg_{\Q_\ell} & \End (\LAlg  ^{\ab }_{\Q _{\ell }})^{\Tr =0} & \Q _{\ell } \\};
\path[->]
(m-1-1) edge[auto] node[auto]{} (m-2-1)
edge[auto] node[auto] {  } (m-1-2)
(m-1-2) edge[auto] node[auto] {} (m-2-2)
edge[auto] node[auto] {  } (m-1-3)
(m-2-1) edge[auto] node[auto] {} (m-2-2)
(m-1-3) edge[auto] node[auto] {} (m-2-3)
(m-2-2) edge[auto] node[auto] {} (m-2-3);
\end{tikzpicture}
\]
We obtain the following explicit formula for the $\Q _{\ell }(1)$-quotients associated to trace zero endomorphisms.
\begin{lemma}\label{lemma:gr_Qell1}
Given $F\in \gr _0 ^{\M }(\LAlg^{\ab }_{\Q _{\ell }}) \otimes (\LAlg^{\ab }_{\Q _{\ell }})^* $, the induced morphism $\bigwedgesquare H(v)\to \Q _{\ell }$ sends $\sum_{i=1}^{g(v)}\alog(\beta_{v,i}')\wedge\alog(\beta_{v,i})$ to $\frac{1}{2}\Tr (F|H(v))$. The induced morphisms $\End (\H ^1 (\Gamma ))\to \Q _{\ell }$ and $\End (\H _1 (\Gamma ))\to \Q _{\ell }$ are given by the dual of the morphisms
$
\Q _{\ell }\to \End (\H ^1 (\Gamma ,\Q _{\ell }))
$
and $\Q _{\ell }\to \End (\H _1 (\Gamma ,\Q _{\ell }))$
sending $1$ to $F|_{\H^1 (\Gamma ,\Q _{\ell })}$ and $F|_{\H_1 (\Gamma ,\Q _{\ell })}$ respectively.
\end{lemma}

We now describe the local height explicitly in terms of the endomorphism. We denote by
$\pi :\Q \cdot \Edge{\Gamma }\to \H _1 (\Gamma )$ the orthogonal projection.

\begin{corollary}\label{cor:explicit_formula}
Given an endomorphism $F$ of $\Jac(X)$ which is fixed by the Rosati involution, with corresponding endomorphism $\gr^\M_\bullet F$ on $\LAlg^{\ab }_{\Q _{\ell }}$, the corresponding measure $\mu _F$ is given by
\[
\sum _{e\in\Edge{\Graph} }\frac{1}{\length (e)} \langle e,F(\pi (e)) \rangle \cdot |\d s_e |+\frac{1}{2}\sum _{v\in\Vert{\Graph}} \Tr (F|H(v))\cdot v.
\]
\end{corollary}
\begin{proof}
By Theorem \ref{thm:description_of_graph_kummer} the measure $\mu _F $ is the image of 
\[
\mu_2:=-\sum_{e\in\Edge{\rGraph}}\ad_{e^*}(N(e^*))\cdot|\d s_e|+\sum_{v\in\Vert{\rGraph}}\log(\delta_v)\cdot v+\sum_{e\in\HEdge{\rGraph}}\log(\delta_e)\cdot e
\]
under the surjection $\gr^\W_{-2}\V _{\Q _{\ell }}\otimes\bMeas^0(\Graph)\to \bMeas^0(\Graph)_{\Q _{\ell }}$ defined by $F$. We use the description of this surjection give in Lemma \ref{lemma:gr_Qell1}. Note that the dual of a morphism
$
F:\Q _{\ell }\to \mathrm{End} (\H _1 (\Gamma )_{\Q _{\ell }})
$
is given by the morphism 
\[
\mathrm{End} (\H _1 (\Gamma )_{\Q _{\ell }})\to \Q _{\ell }
\]
sending $G$ to $\Tr (F\circ G)$. 

In $\End (\H _1 (\Gamma ))$, $\ad _{e^* }(N(e^* ))$ is sent to $\frac{1}{2}N(e^* )\otimes \pi (e)^* $.  In $\End (\H ^1 (\Gamma ))$, $\ad _{e^* }(N(e^* ))$ is sent to $\frac{1}{2}e^* \otimes N(e^* )$. In $\End (H(v))$, $\log (\delta _e )$ is sent to $\frac{1}{2}\sum_{i=1}^{g(v)}\alog(\beta_{v,i})^* \otimes \alog(\beta_{v,i})+\alog(\beta_{v,i}')^* \otimes \alog(\beta_{v,i}')$, which is $\frac{1}{2}\cdot 1_{H(v)}$, where $1_{H(v)}$ denotes the idempotent projection onto $H(v)$.

Note that the image of $N(e^* ) \otimes e^* $ in $\End (\H_1 (\Gamma ))$ is, by definition, $\frac{1}{\length (e)}\cdot 1_{\pi (e)}$, where $1_{\pi (e)}$ denotes the idempotent projection onto the subspace spanned by $\pi (e)$. We have $\Tr (F\circ 1_{\pi (e)})=\langle \pi (e), F(\pi (e)) \rangle $, and (since $\H _1 (\Gamma )$ is an orthogonal direct summand)  $\langle \pi (e), F(\pi (e)) \rangle  =\langle e, F(\pi (e)) \rangle $. The image of $e^* \otimes N(e^* )$ in $\End (\H ^1 (\Gamma ))$ is given by the dual of $\frac{1}{\length (e)}\cdot 1_{\pi (e)}$, and since the image of $F$ in $(\mathcal{L}^{\ab }_{\Q _{\ell }})^{\otimes 2}$ lands in $\bigwedgesquare \mathcal{L}^{\ab }_{\Q _{\ell }}$, we have $\Tr (1_{\pi (e)}\circ F)= \Tr (1_{\pi (e)^* } \circ F)$.
\end{proof}

\begin{remark}
By \cite[\S 5.2 and Theorem 6.3]{BalakrishnanDogra1}, if $Z$ comes from a correspondence $C\subset X\times X$, then the function $\classify_C\colon X(\Q _p )\to \Q _{\ell }$ may be identified with the following function. For $b,z\in X(K)$, let $D_C (b,z)\in \Div ^0 (X)(\overline{K})$ denote the divisor $(i_{\Delta }^* -i_1 ^* -i_2 ^* )(C)$, where $i_1 ,i_2 ,i_{\Delta }:X\hookrightarrow X\times X$ are the closed immersions given by $\{ b\} \times X,X\times \{z\}$ and the diagonal respectively. Then, for all $z$ disjoint from the support of $i_{\Delta }^* (C)$, $\classify_C$ is equal to
\[
z\mapsto h_p (z-b,D_C (b,z)),
\]
up to a scalar. Hence we obtain a formula for this height pairing in terms of harmonic analysis on graphs.
\end{remark}

\subsection{Examples}\label{ss:examples}
\subsubsection{The case of a tree}
If $\Graph $ is a tree, then $\LAlg^{\ab }_{\Q_\ell} \simeq \bigoplus _v H(v)$, and we deduce that
\[
\mu _F =\frac{1}{2}\sum \Tr (F|H(v))\cdot v.
\]
In particular, this covers the case of $X_{\mathrm{ns} }^+ (p)$ when $p>3$ is congruent to 3 modulo 4, by Edixhoven--Parent \cite[Theorem 3.5]{EdixhovenParent}. We deduce that if $F$ is an endomorphism which acts with trace zero on the \'etale cohomology of each irreducible component of the special fibre of $X_{\mathrm{ns} }^+ (p)$, then $\mu _F =0$, and the function $j_F$ is trivial.

\subsubsection{The case of a banana}
Suppose $\rGraph $ is a banana graph (i.e. $\Vert{\Gamma}=\{ v,w\}$ and $\partial _0 (e)\neq \partial _1 (e)$ for all edges $e$) with edges of possibly varying lengths and both vertices having genus zero. For example, these graphs arise as the the reduction graphs of $X_0 (p)$ at primes above $p$, by Deligne--Rapoport \cite{DeligneRapoport}. Then one can calculate the projection $\pi (e)$ explicitly, giving the formula
\[
\mu _F =\sum _{e\in E(\Gamma ) }\frac{1}{l(e)}\langle e,F(e-\sum _{e' \in \Edge{\Gamma } }\frac{1}{\nu \cdot \length (e')}e')\rangle \cdot |\d s_e| ,
\]
where $\nu :=\sum _{e' \in \Edge{\Gamma }}\length (e')^{-1} $.

\subsubsection{Banana graphs of higher genus}
Suppose we have $n+2$ vertices $v_1, v_2 ,h_1 ,\ldots ,h_n$, with each $h_i$ connected to $v_1 $ and $v_2 $ by edges of length $\ell$, and suppose that now the $v_i$ may have higher genus. To compute the pseudo-inverse of the Laplacian, we can replace the underlying graph with a banana graph $\Gamma '$ with $n$ edges, giving
\[
\mu _F =\frac{1}{2}\sum \Tr (F|H(v))\cdot v+\sum _{e\in E(\Gamma ' ) }\frac{1}{l(e)}\langle e,F(e-\sum _{e' \in \Edge{\Gamma } }\frac{1}{\nu \cdot \length (e')}e')\rangle \cdot |\d s_e| .
\]
This describes $j_F$ on $X_{\mathrm{ns}}^+ (p)$ when $p\cong 1 mod 4$, by \cite[Theorem 3.5]{EdixhovenParent}.

To use these formulae for explicit computations on the modular curves above, one needs a way of computing the action of an endomorphism on the homology of the dual graph. In the case of modular curves one usually takes the endomorphism $F$ to be a linear combination of Hecke operators. Algorithms for computing the action of $F$ on the homology of the dual graph are then well-known (see e.g. \cite[\S 3]{edixhoven1989stable}).

\subsection{Relation to the canonical measure when $X$ is projective minus a point}\label{ss:proj_minus_pt}
Suppose $X$ is non-projective, and $\overline{X}-X=\{ x_0 \}$ for $x_0\in\overline X(K)$. Then the reduction graph  $\rGraph$ is a reduction graph with exactly one half-edge $e_0 $, located at the vertex $v_0 =\red(x_0)$. The discussion of $\Q _{\ell }(1)$-quotients in this case is very similar to section \ref{ss:application_QC}, except that because the fundamental group is free, there is no `trace zero' condition for the endomorphisms. We have an isomorphism
\begin{equation}\label{eqn:iso_punctured}
\gr ^2 _\Cent U\simeq \bigwedgesquare U_1  \subset  \H _1 ^{\et }(X_{\overline{K}},\Q _{\ell })^{\otimes 2}.
\end{equation}
We have a surjection
\[
\bigwedgesquare U_1 \to \Q _{\ell }(1)
\]
defined as the composite
\[
\bigwedgesquare U_1 \hookrightarrow \H _1 ^{\et }(X_{\overline{K}},\Q _{\ell })\otimes \H ^1 _{\et } (X_{\overline{K}},\Q _{\ell })(1)\to \Q _{\ell }(1)
\]
where the first map is the composite of \eqref{eqn:iso_punctured} with Poincar\'e duality and the second is $\Q _{\ell }(1)\otimes $ the trace (this is just the construction from the previous subsection, with the endomorphism taken to be the identity).
Since every surjection $\gr ^2 _\Cent U\to \Q _{\ell }(1)$ is a linear combination of this surjection and one of the ones from trace zero endomorphisms discussed above, in this subsection we focus on discussing this case, and the relation with the canonical measure defined by Zhang and Chinburg--Rumely.

Arguing as in the proof of Corollary \ref{cor:explicit_formula}, we see that the image of $\mu _2 $ is given by
\begin{equation}\label{eq:mu_Z_primitive}
\mu_Z:=\sum _{e\in E(\Gamma )}\frac{1}{\length (e)}\langle e,\pi (e) \rangle \cdot |\d s_e| +\sum _v g(v) \cdot v -g(X)\cdot e_0 .
\end{equation}

We now recall the definition of the canonical measure, following Zhang \cite{zhang}.
For $u,v \in \Vert \Graph $, we let $r(u,v)$ denote the resistance between $u$ and $v$, viewing $\Graph $ as an electrical circuit with edges $\Edge{\Graph}^+$, where an edge $e$ has resistance $\length (e)$ (we forget about half-edges in this definition).

For an edge $e\in \Edge{\Graph}$, we define $R_e$ to be the resistance $r_{\Graph -e^{\pm 1}}(\source (e),\target (e))$ in the metrised graph $\Graph -e^{\pm 1}$ obtained by deleting the edge $e$ and its inverse (if $\Graph-e^{\pm1}$ is disconnected, we set $R_e=\infty$). Formally, we can define this, as in \cite[\S 3]{zhang} or \cite{CR}, as follows: let $g_p (q,\cdot)$ be the unique function in $\Omega _{\log }^0 (\Graph )$ satisfying
\begin{itemize}
\item $\nabla ^2 g_p (q,\cdot) =q-p $;
\item $g_p (q,p)=0$.
\end{itemize}
Then $r(p,q):=g_p (q,q)$. Note from this definition that $g_z(x,y)=\langle x-z,y-z\rangle$ with $\langle\cdot,\cdot\rangle$ the height pairing on divisors, so that $r(x,y)=\langle x-y,x-y\rangle$.

We have the identity
\begin{equation}\label{eqn:zhang_resist}
\frac1{\langle\partial(e),\partial(e)\rangle}=\frac{1}{r(\source (e),\target (e))}=\frac{1}{R_e} +\frac{1}{\length (e)}.
\end{equation}
(see \cite[Proposition 3.3]{zhang}).

In \cite[Theorem 2.11]{CR}, the \textit{canonical measure} $\mu _{\can }$ of a metrised graph is defined. It is proved that $\mu _{\can }$ is a measure of total mass $1$ and satisfies the identity
\begin{equation}\label{eq:canonical_identity}
\mu _{\can }=\sum _{v\in\Vert{\Graph}}\left(1-\frac{1}{2}\val (v)\right)\cdot v +\sum _{e\in \Edge{\Graph} }\frac{|\d s_e| }{l(e)+R_e },
\end{equation}
so that the discrete part of $\mu_\can$ is $-\frac12K_\Graph$ with $K_\Graph=\sum_v(\val(v)-2)\cdot v$ the canonical divisor of the graph $\Graph$ \cite[\S2.1]{zhang}. Combining ~\eqref{eq:mu_Z_primitive} and~\eqref{eq:canonical_identity} we deduce the following relation between the canonical measure and $\mu _Z$.

\begin{lemma}
The measure $\mu _Z$ from \eqref{eq:mu_Z_primitive} satisfies
\[
\mu_Z=\mu_\can + \frac12K_{\rGraph} - g(X)\cdot e_0
\]
where $K_{\rGraph}:=\sum_v(2g(v)+\val(v)-2)\cdot v$ denotes the canonical divisor of the reduction graph $\rGraph$ (cf.\ \cite[\S2.1]{zhang}).
\end{lemma}

\subsection{A refinement of the Chabauty--Kim obstruction for $S$-integral points}
The piecewise polynomial structure of the non-abelian Kummer map has the following curious consequence. Let $\overline{X}/\Q$ be the smooth compactification of a smooth hyperbolic curve $X/\Q $, and define $D:=\overline{X}-X$. Let $\overline{\mathcal{X}}$ be the minimal regular model of $\overline{X}$, and $\mathcal{X}=\overline{\mathcal{X}}-\mathcal{D}$, where $\mathcal{D}$ is a reduced horizontal divisor with fibre $D$ over $\Q $. Let $T_0 $ be the set of primes of bad reduction for $X$ (i.e. the primes for which $\overline{\mathcal{X}}$ has bad reduction or $\mathcal{D}$ is not \'etale), let $S$ be a finite set of primes, let $\ell $ be a prime not in $T_0 \cup S$ and let $T=T_0 \cup S \cup \{ \ell\}$. Let $G_{\Q ,T}$ denote the maximal quotient of $G_\Q$ unramified outside $T$.

The Chabauty--Kim method aims to locate $\mathcal{X}(\Z_S )$ inside $\prod _{p\in S}X(\Q _p )\times $ \\$\prod _{p\in T-S}\mathcal{X}(\Z_p )$ using the commutative diagram
\begin{center}
\begin{tikzcd}
\mathcal X(\Z_S) \arrow{r}{\classify}\arrow[hook]{d} & \H^1(G_{\Q,T},U_n) \arrow{d}{\prod\loc_p} \\
\prod_{p\in S}X(\Q_p)\times\prod_{p\in T-S}\mathcal X(\Z_p) \arrow{r}{\prod\classify_{n,p}} & \prod_{p\in T}\H^1(G_{\Q_p},U_n).
\end{tikzcd}
\end{center}
One can use this diagram to define a subset $\mathcal{X}(\Z_{\ell } )_{S,n}$ of $\mathcal{X}(\Z_{\ell })$ as follows. For $p$ in $T_0 -S$, the map $j_{n,p}$ has finite image. One defines $\mathcal{X}(\Z_{\ell } )_{S,n}$ to be the subset of $\mathcal{X}(\Z_{\ell } )$ consisting of points $x_{\ell } $ which extend to a point $(x_p )_{p\in T-S}\in\prod_{p\in T-S}\mathcal X(\Z_p)$ whose image in $\prod _{p\in T-S}\H^1 (G_{\Q _p },U_n )$ lies in the image of $\H ^1 (G_{\Q ,T},U_n )$. By Kim \cite[\S 3]{selmer_varieties}, $\mathcal{X}(\Z_{\ell } )_{S,n}$ will be finite whenever 
\begin{equation}\label{bad}
\sum _{i=1}^n \dim \H ^1 _{f,S}(G_{\Q ,T},\gr ^i _\Cent U)< \sum _{i=1}^n \dim \H ^1 _f (G_{\Q _{\ell } },\gr ^i _\Cent U),
\end{equation}
where $\gr ^i _\Cent U$ denotes the graded pieces of the $\Q_\ell$-pro-unipotent \'etale fundamental group with respect to the descending central series.

By Theorem \ref{thm:description_of_graph_kummer}, we know that, for all $p\in S$, the Zariski closure $j_{n,p}(X(\Q _p ))^{\mathrm{Zar}}$ of $j_{n,p}(X(\Q _p ))$ in $\H ^1 (G_{\Q _p },U_n )$ is at most one dimensional (and exactly one dimensional for all $n\geq 2$ if $X$ has a $\Q_p$-rational cusp). Hence, we define $\mathcal{X}(\Z_{\ell } )_{S,n}^{\min }$ to be the subset of $\ell $-adic points $x_{\ell } $ such that $\classify_{n,\ell}(x_\ell)$ extends to a tuple $(c_p)_{p\in T}$ which lies in the Zariski closure of the image of $\H^1(G_{\Q,T},U_n)$ and each $c_p$ lies in the Zariski closure of the image of $\classify_{n,p}$ (i.e.\ $\classify_{n,p}(X(\Q_p))^\Zar$ for $p\in S$ and $\classify_{n,p}(\mathcal X(\Z_p))^\Zar$ for $p\in T-S$). 
\begin{lemma}
$\mathcal{X}(\Z_{\ell })_{S,n}^{\min }$ is finite whenever
\begin{equation}\label{good}
\#S + \sum _{i=1}^n \dim \H ^1 _{f}(G_{\Q ,T},\gr ^i _\Cent U)< \sum _{i=1}^n \dim \H ^1 _f (G_{\Q _{\ell } },\gr ^i _\Cent U).
\end{equation}
\end{lemma}
\begin{proof}
Let $\Sel_{S,U_n}^{\min}(\mathcal X) \subset \H^1 (G_{\Q ,T},U_n )$ denote the fibre of $\prod_{p\in T}\im(\classify_{n,p})^\Zar$ under the map $\prod _{p\in T}\loc _p $. Then it will be enough to show that the dimension of $\Sel_{S,U_n}^{\min}(\mathcal X)$ is less than
\[
\dim \H^1 _f (G_{\Q _{\ell }},U_n )= \sum _{i=1}^n \dim \H ^1 _f (G_{\Q _{\ell } },\gr ^i _\Cent U)
\]
whenever \eqref{good} is satisfied. The $\Q_\ell$-scheme $\prod _{p\in T-\{\ell \}}\im(\classify_{n,p})^\Zar$ is finite type of dimension $\# S$. Let $(c_p )\in 
\prod _{p\in T-\{\ell \}}\im(\classify_{n,p})^{\Zar }(F)$ be an $F$-point in the image of $\Sel (U_n )^{\min }$ for some finite extension $F|\Q_\ell$. 
Enlarging $F$ if necessary, let $c\in \Sel_{S,U_n}^{\min}(\mathcal X)(F)$ be a point in the fibre of $(c_p)$, for some finite extension $F|\Q _{\ell }$.
Arguing exactly as in \cite[Lemma 3.1]{dogra2019unlikely},  we see that the fibre of $(c_p)$ in $\Sel_{S,U_n}^{\min}(\mathcal X)$ is isomorphic, over $F$, to
$
\H^1 _f (G_{\Q ,T},(U_n )_F ^c )
$
where $(U_n )_F ^c$ is the twist of $U_n$, base changed to $F$, by $c$. For all $2\leq i\leq n$, $\H ^1 _f (G_{\Q ,T},(U_i )_F ^c )$ is an $\H^1 _f (G_{\Q ,T},\gr _\Cent ^i (U)_F ^c )$-torsor over $\H ^1 _f (G_{\Q ,T},(U_{i-1} )_F ^c )$, hence
\[
\dim_F \H^1 _f (G_{\Q ,T},(U_n )_F ^c)\leq \sum _{i=1}^n \dim_F \H ^1 _f (G_{\Q ,T},\gr _\Cent ^i (U)_F ^c ).
\]
Since the action of $U_n$ on itself by conjugation is trivial on $\gr ^i _{\Cent }$, we have 
\[
\H^1 _f (G_{\Q _{\ell }},\gr ^i _{\Cent }(U)_F ^c )\iso \H^1 _f (G_{\Q _{\ell }},\gr ^i _{\Cent }U)_F .
\] Hence the dimension of $\Sel_{S,U_n}^{\min}(\mathcal X)$ is bounded by $\#S + \sum _{i=1}^n \dim \H ^1 _{f}(G_{\Q ,T},\gr ^i _\Cent U)$, as required.
\end{proof}

As illustrated below, there are several simple examples where \eqref{good} is satisfied but \eqref{bad} is not.
In the examples below, there were already alternative approaches to proving algebraicity of the maps $\classify_{n,p}$, however they are the simplest non-trivial examples of the phenomenon.

\subsubsection{Example: $\mathbb{P}^1 -\{ 0,1,\infty \}$}\label{sss:P1_finite}
In this case, $j_{n,p}\simeq j_{1,p}$ for all $p\in S$ and $n\geq 1$. If we identify the map $j_{1,p}$ with the map
\begin{align*}
\Q _p ^\times -\{ 1\} & \to \mathbb{A}^2 _{\Q _{\ell }}, \\
x & \mapsto (v_p (x),v_p(1-x)),
\end{align*}
then the Zariski closure of $j_{n,p}(X(\Q _p ))$ is simply 
\begin{equation}\label{eqn:zar_closure_jnp}
\{0 \}\times \mathbb{A}^1 _{\Q _{\ell } }\cup \mathbb{A}^1 _{\Q _{\ell } } \times \{0\} \cup \Delta _{\mathbb{A}^1 _{\Q _{\ell } }},
\end{equation}
(of course here the fact that $j_{n,p}$ is piecewise polynomial is elementary).
For example, when $\# S=n=1$, then $\mathcal{X}(\Z_{\ell })_{S,n}$ will equal $\mathcal{X}(\Z_{\ell })$, whereas $\mathcal{X}(\Z_{\ell })_{S,n}^{\min}$ will be finite, equal to the set of $x\in \mathcal{X}(\Z_{\ell })$ satisfying
\[
\log _{\ell }(x)\log _{\ell }(1-x)(\log _{\ell }(x)-\log _{\ell}(1-x))=0.
\]
(in this case, it even happens that $\mathcal{X}(\Z_{\ell })_{S,1}^{\mathrm{min}}$ is the same for all $S$ of size one, but this phenomenon will not hold in general).

When $\# S=2$, $\mathcal{X}(\Z_{\ell } )_{S,n}$ would not be expected to be
finite when $n\leq 4$, whereas $\mathcal{X}(\Z_{\ell } )_{S,n}^{\mathrm{min}}$ is finite for $n\geq 2$: let $S=\{p,q\}$ denote the two distinct primes. As we have not discussed tangential basepoints thus far, we will take the basepoint to be a $\mathbb{Z}[\frac{1}{pq}]$ point of $\mathcal{X}$.

We let $U=U_2$ be the maximal $2$-step unipotent quotient of the $\Q_\ell$-pro-unipotent \'etale fundamental group of $X_{\overline\Q}$ for $\ell\notin\{p,q\}$, which is a central extension of $\Q_\ell(1)^{\oplus2}$ by $\Q_\ell(2)$. It follows from \cite{DCW} that the Selmer scheme $\Sel_{S,U}(\mathcal X)$ defined in section \ref{ss:application_CK} is isomorphic to $\A^4_{\Q_\ell}$ 
, with the non-abelian Kummer map $\classify\colon\mathcal X(\Z_S)\rightarrow\Sel_{S,U}(\mathcal X)$ being given by
\[u\mapsto\left(v_p(u/b),v_p\Bigl(\frac{1-u}{1-b}\Bigr),v_q(u/b),v_q\Bigl(\frac{1-u}{1-b}\Bigr)\right).\]
However, the local Selmer scheme $\H^1_f(G_{\Q_\ell},U)\simeq\A^3_{\Q_\ell}$ is three-dimensional and so the usual Chabauty inequality \eqref{bad} fails. 

\subsubsection{Example: An elliptic curve minus the origin}\label{sss:E_finite}
The next simplest case is that of $S$-integral points on an elliptic curve, i.e. $S$-integral points of $\mathcal{X}:=\mathcal{E}-O$, where $\mathcal{E}$ is the N\'eron model of an elliptic curve $E$ and $O$ is the identity section. Here, when $S$ is empty and $E$ has rank one, finiteness of $\mathcal{X}(\mathbb{Z}_{\ell })_{S,2}$ is proved in \cite{Kim10}, by verifying \eqref{bad} for $n=2$. When $S$ is non-empty, refined Selmer schemes allow one to generalise this to larger $S$ and $n$.

\begin{corollary}
Let $(E,O)/\Q $ be an elliptic curve of analytic rank one. Let $\mathcal{X}$ denote the N\'eron model minus the identity section. Let $S=\{p\}$.
\begin{enumerate}
\item For all $p$, \eqref{good} is satisfied and hence $\mathcal{X}(\mathbb{Z}_{\ell })_{S,3}$ is finite.
\item Inequality \eqref{bad} is satisified if and only if $p$ is not a prime of split multiplicative reduction.
\end{enumerate}
\end{corollary}
\begin{proof}
Let $V_{\ell }E$ denote the $\Q _{\ell }$ Tate module of $E$. The graded pieces of $U_3$ with respect to the central series filtration are $V_{\ell }E$, $\Q _{\ell }(1)=\wedge ^2 V_{\ell }E$ and $V_{\ell }(E)(1)$.
By Kato's theorem \cite[Theorem 18.4]{kato:secret}, we have 
\[
\dim \H ^1 _f (G_{\Q ,T},V_{\ell }(E)(1))=1.
\]
If $E$ has analytic rank 1, then by Gross--Zagier--Kolyvagin \cite{Kolyvagin90}, we have \[
\rk \H ^1 _f (G_{\Q ,T},V_{\ell }(E))=1.
\] On the other hand, we have
\begin{align*}
\dim \H ^1 _f (G_{\Q _{\ell }},U_3 ) & =\dim D_{\dR}(U_3 )/\Fil ^0 \\
& = \dim D_{\dR}(V_{\ell }(E))/\Fil ^0+\dim D_{\dR}(\Q _{\ell }(1))/\Fil ^0 \\ & +\dim D_{\dR}(V_{\ell }(E)(1))/\Fil ^0 \\
& = 4.\\
\end{align*}
Finally, note that 
\[
\H^1 _{f,S} (G_{\Q ,T},\Q _{\ell }(1))\simeq \mathbb{Z}_S ^\times \otimes \Q _{\ell },
\]
while by Poitou--Tate duality,
\[
\H^1 _{f,S}(G_{\Q ,T}V_{\ell }E(1))/\H^1 _f (G_{\Q ,T},V_{\ell }E(1))\simeq \H^1 (G_{\Q _p},V_{\ell }(1)),
\]
which has dimension one if $p$ is a prime of split multiplicative reduction, and zero otherwise.
\end{proof}

\subsection{Applications to the effective Chabauty--Kim method}
The final application is to \textit{bounding} the number of rational points on curves.
In \cite{BDeff} it was shown that
By \cite[Theorem 1.1]{BDeff}, we have
\begin{equation}\label{eqn:effective_zero}
\# X(\Q _{\ell } )_ 2 \leq \kappa _{\ell }\left( \prod _{p\in T_0 }n_p \right) \# X(\mathbb{F}_{\ell } )(16g^3+15g^2-16g+10),
\end{equation}
where $n_p = \# j_U (\Q _p )$. Hence the corollary says that we can improve the constants $n_p $.

We say a smooth projective curve over $\Q $ has irredeemably bad reduction at a prime $p$ if it does not have potential good reduction at $p$.
\begin{corollary}\label{cor:effective}
Let $X$ be smooth projective of genus $g$ over $\Q $. Let $T_0$ be the set of primes of irredeemably bad reduction. Let $T_1 $ be the set of irredeemably bad reduction at which the dual graph of the special fibre of a stable model is a tree. Let $T_2 =T_0 \backslash T_1 $. For $p\in T_0$, let $e_p$ denote the number of edges of the dual graph of the special fibre of a stable model of $X$ at $p$ (i.e. over a finite extension of $\Q _p$), and let $i_p$ denote the number of irreducible components of a regular semi-stable model over a choice of finite extension of $\Q _p $. Let $J:=\Jac (X)$, $r:=\rk J(\Q )$, $\rho (J):=\rk \NS (J)(\Q )$. For $i=1,2$, let $S_i $ be a subset of $T_i $, and let $S_0 =S_1 \cup S_2 $. 
Suppose that $r=g$ and 
\[
\#S_1 +2\# S_2 <\rho (J)-1,
\]
or $X\times X$ satisfies the Bloch--Kato conjecture \cite[Conjecture 5.3]{blochkato} and 
\[
\# S_1 +2\# S_2 < g^2 -g.
\]
Then
\[
\# X(\Q _{\ell })_2 <\kappa_{\ell } \left(\prod _{p\in S_0 }e_p \right) \left(\prod _{p\in T_0 \backslash S_0 }i_p \right)\# X(\mathbb{F}_{\ell } )(16g^3+15g^2-16g+10),
\]
where $\kappa _{\ell} :=1+\frac{\ell -1}{\ell -2}\frac{1}{\log (\ell )}$.
\end{corollary}
\begin{remark}
Note that $e_p$ can be bounded in terms of the genus: since each vertex $v$ satisfies $\val (v)+2g(v)\geq 3$, and $g=1+\sum _v \frac{1}{2}\val (v)+g(v)-1$, we have
\[
e_p \leq 3(g-1).
\]
Hence when $S_0 =T_0$, the bound obtained is purely in terms of $\ell $, $g$ and the number of primes of irredeemably bad reduction for $X$.
\end{remark}

\begin{proof}[Proof of Corollary \ref{cor:effective}]
We only describe the parts of the arguments which are different from \cite{BDeff}. 
For $\alpha =(\alpha _p )_{p\in T_0 }\in \prod _{p\in T_0 }\H^1 (G_{\Q _p },U)$, we define $\Sel (U)_{\alpha }\subset \H ^1 _{f,T_0 }(G_{\Q ,T},U)$ to be the fibre of $(\alpha _p )$ under $\prod _{p\in T_0 }\H ^1 (G_{\Q _p },U)$. We define $X(\Q _{\ell } )_{\alpha }$ to be the pre-image of $\Sel (U)_{\alpha }$ in $X(\Q _{\ell } )_{\alpha }$. For an $\ell$-adic $G_{\Q ,T}$ representation $V$, we define $\H^1 _s (G_{\Q ,T},V):=\H^1 (G_{\Q ,T},V)/\H^1 _f (G_{\Q ,T},V)$. If $D=\sum n_i x_i $ is an effective divisor with $x_i$ pairwise distinct, let $D[1]$ denote the divisor $D+|D|=\sum (n_i +1)x_i $.

It will be enough to prove that, for all $z_0 \in X(\F _{\ell })$, 
\[
\# X(\Q _{\ell } )_ 2 \cap ] z_0 [ \leq \kappa _{\ell }\left(\prod _{p\in S_0 }e_p \right) \left(\prod _{p\in T_0 \backslash S_0 }i_p \right) (16g^3+15g^2-16g+10),
\]
where $]z_0[ \subset X(\Q _{\ell })$ denotes the set of points reducing to $z_0 $ modulo $\ell $. We fix $z_0 \in X(\F _{\ell })$. Let $D$ be a non-zero divisor with support disjoint from $z_0 $ modulo $\ell $, such that $\omega _1 ,\ldots ,\omega _{2g}$ are differentials in $\H^0 (X,\Omega (D))$ forming a basis of $\H^1 _{\mathrm{dR}}(X)$, let $Y:=X-|D|$. 

If the $\ell$-adic closure of $\Jac (X)(\Q )$ in $\Jac (X)(\Q _{\ell })$ is not finite index, then the bound in the statement of the corollary follows froms Coleman's bound \cite{coleman}. Hence we assume that it is finite index. In this case \cite[Proposition 4.1]{BD} implies that, for all $\varphi \in \H^1 _s (G_{\Q ,T},(\gr ^2 _\Cent U)^* (1))$,
\[
X(\Q _\ell )_{\alpha } \subset \left\{ x\in X(\Q _{\ell }):h_{\ell ,\varphi }(x)+\sum _{v\in T-\{\ell \}}h_{v,\varphi }(\alpha _v )=h_{\varphi }(S\circ T (j_{1,\ell}(x)))\right\} .
\]
The only facts we need to recall about $h_{\ell ,\varphi },h_{v,\varphi }$, $h_{\varphi }$, $S$ and $T$ are
\begin{enumerate}
\item By \cite[Proposition 6.4]{BD}, for any choice of $\varphi $ there are constants $a_{ij},a_i  $ and $h\in \H^0 (X,\mathcal{O}(D[1]))$ such that for all 
\begin{equation}\label{eqn:local_height_formula}
h_{\ell ,\varphi }(x)-h_{\varphi }(S\circ T (j_{1,\ell}(x)))=\sum a_{ij}\int ^z _b \omega _i \omega _j +\sum a_i \int ^z _b \omega _i +\int ^z _b \xi +h(z) . 
\end{equation}
\item For $v\in T-\{ \ell \}$, the function $h_{v,\varphi }:\H^1 (G_{\Q _v },\gr ^2 _{\Cent }(U))\to \Q _{\ell }$ is linear in $\varphi $.
\end{enumerate}

For $p\in S_0$, let $\Graph _p $ denote the dual graph of the special fibre of a stable model of $X$ at $p$.
By \cite{BDeff}, for $z_0 \in X(\Q _p )$ we may choose $\omega _i ,\xi ,h$ as above, so that none of the points of $X-Z$ reduce to $z_0$ modulo $p$, and for any $\lambda \in \Q _p $, the number of zeroes of 
\[
\sum a_{ij}\int ^z _b \omega _i \omega _j +\sum a_i \int ^z _b \omega _i +\int ^z _b \xi +h(z) +\lambda ,
\]
for $z$ reducing to $z_0 $ modulo $p$, is at most $\kappa_{\ell }(16g^3+15g^2-16g+10)$. Hence it is enough to prove that for each tuple of edges $(e_p )\in \prod _{p\in S_0 }\Graph _p $, there is a quotient $U$ of $U_2 $, and a non-zero $\varphi \in \H^1 _ s (G_{\Q ,T},(\gr ^2 _\Cent U)^* (1))$ such that $h_{\varphi ,p}(x_p )=0$ for all $x_p \in X(\Q _p )$ whose reduction lies on the edge $e_p$. Under either of the two assumptions of the theorem, we know that there is a quotient $U$ such that $\dim \H^1 _s (G_{\Q ,T_0 \cup \{\ell \}},(\gr ^2 _\Cent U)^* (1))$ has dimension greater than $\# S_1 +2\# S_2 $ (see \cite[\S 3]{BalakrishnanDogra1} and \cite[\S 2]{BD} respectively). By Theorem \ref{thm:description_of_graph_kummer}, for any choice of $\varphi $, $h_{\varphi ,p}$ is linear along the edge $e_p$ when $p$ is in $T_1$ and quadratic when $p$ is in $T_2$. Since $h_{\varphi ,p}$ is linear in $\varphi $, we deduce that there is a non-zero choice of $\varphi $ such that $h_{\varphi ,p}(x_p )$ is identically zero on $e_p$ for all $p\in S_0$.
\end{proof}
\appendix
\section{Groupoids, Hopf groupoids and Lie algebras}
\label{appx:hopf_gpds}

In this appendix, we give background on some of the main linear algebraic objects appearing in this paper. A standard tool in the study of affine group-schemes $\pi$ is their commutative Hopf algebras $\O(\pi)$ \cite{milne}, or equivalently their complete cocommutative Hopf algebras $\O(\pi)^\dual$ \cite[\S3.2.6]{javier}. When $\pi$ is replaced by an affine groupoid-scheme, there is a multi-object generalisation of this tool, known as the \emph{complete (cocommutative) Hopf groupoid} $\O(\pi)^\dual$ associated to $\pi$, which is built out of the natural operations on the objects $\O(\pi(x,y))^\dual$. It is through these linear algebraic objects that we will analyse the affine groupoid-schemes appearing in this paper, along with all the extra structures (Galois actions, filtrations, etc.) that they carry. Indeed, certain structures are better behaved on the side of complete Hopf groupoids, for instance the definition of a filtration is more natural, and there is a self-evident notion of the associated graded of a filtration.

Our aim in this appendix, then, is to outline the basic theory of complete Hopf groupoids and their relation to affine groupoid-schemes, with particular attention to the extra structures which these objects might carry. In order to work uniformly with all possible extra structures and combinations thereof, we will follow the approach of \cite{homotopy_theory_operads} and set up the theory in categorical language, into which one can substitute, for example, the category of bifiltered vector spaces with a bifiltered endomorphism of bidegree $(0,-2)$ to obtain the desired results. Note that the results from \cite{homotopy_theory_operads} are not directly applicable, since we will generally work in categories in which limits, rather than colimits, are well-behaved. Nonetheless, we borrow heavily from the approach in \cite{homotopy_theory_operads}, mimicking the same strategies in several key proofs.

\subsection{Groupoids and Hopf groupoids}

\begin{definition}\label{def:gpoid}\index{groupoid}
Let $\Cat$ be a category with finite products. A \emph{groupoid} $\pi$ in $\Cat$ consists of a set $\ob(\pi)$ and an object $\pi(x,y)\in\ob(\Cat)$ for every $x,y\in\ob(\pi)$, endowed with the following structure maps:
\begin{itemize}
	\item \emph{identities} $\eta\colon \ast\rightarrow\pi(x):=\pi(x,x)$ for every $x\in\ob(\Cat)$;
	\item \emph{composition maps} $\mu\colon\pi(y,z)\times\pi(x,y)\rightarrow\pi(x,z)$ for all $x,y,z\in\ob(\Cat)$; and
	\item \emph{inversion maps} $S\colon\pi(x,y)\rightarrow\pi(y,x)$ for all $x,y\in\ob(\Cat)$
\end{itemize}
which make the following diagrams commute:
\begin{center}
\begin{tikzcd}
\pi(y,z)\times\pi(x,y)\times\pi(w,x) \arrow{r}{\mu\times1}\arrow{d}[swap]{1\times\mu} & \pi(x,z)\times\pi(w,x) \arrow{d}{\mu} \\
\pi(y,z)\times\pi(w,y) \arrow{r}{\mu} & \pi(w,z)
\end{tikzcd}
\vspace{0.4cm}
\begin{tikzcd}
\pi(x,y) \arrow{r}{1\times\eta}\arrow[equal]{rd} & \pi(x,y)\times\pi(x) \arrow{d}{\mu} & \pi(y)\times\pi(x,y) \arrow{d}[swap]{\mu} & \pi(x,y) \arrow{l}[swap]{\eta\times1}\arrow[equal]{ld} \\
 & \pi(x,y) & \pi(x,y) &
\end{tikzcd}
\vspace{0.4cm}
\begin{tikzcd}
{} & \pi(x,y)\times\pi(x,y) \arrow{r}{1\times S} & \pi(x,y)\times\pi(y,x) \arrow{r}{\mu} & \pi(x) \\
\pi(x,y) \arrow{ru}{\Delta}\arrow{rd}[swap]{\Delta}\arrow{rr}{\varepsilon} & & \ast \arrow{ru}[swap]{\eta}\arrow{rd}{\eta} & \\
 & \pi(x,y)\times\pi(x,y) \arrow{r}[swap]{S\times 1} & \pi(y,x)\times\pi(x,y) \arrow{r}[swap]{\mu} & \pi(y),
\end{tikzcd}
\end{center}
where $\Delta$ and $\varepsilon$ denote the diagonal map and the unique map to the terminal object, respectively. A groupoid in $\Cat$ with a single object is a \emph{group} in~$\Cat$.

There is an evident notion of a \emph{morphism} $f\colon\pi'\rightarrow\pi$ of groupoids in $\Cat$, namely a function $f\colon\ob(\pi')\rightarrow\ob(\pi)$ together with morphisms $f\colon\pi'(x,y)\rightarrow\pi(f(x),f(y))$ which are compatible with all the structure maps. We will write $\GPD(\Cat)$ for the (1-)category of groupoids in $\Cat$. A product-preserving functor $F\colon\Dat\rightarrow\Cat$ between categories with finite products also induces a functor $F\colon\GPD(\Dat)\rightarrow\GPD(\Cat)$.
\end{definition}

\begin{example}
\leavevmode
\begin{itemize}
	\item Groupoids in $\SET$ are groupoids in the usual sense; we write simply $\GPD$ for the category of groupoids in $\SET$. Generally, for any locally small category $\Cat$ with finite products, the Yoneda functor $\Cat\hookrightarrow[\Cat^\op,\SET]$ gives a fully faithful embedding $\GPD(\Cat)\hookrightarrow[\Cat^\op,\GPD]$, meaning that groupoids in $\Cat$ satisfy all the usual identities satisfied by groupoids in $\SET$.
	\item If $\F$ is a field of characteristic $0$, groupoids in the category of affine $\F$-schemes will be referred to as \emph{affine groupoid-schemes}. By a \emph{pro-unipotent groupoid}\index{groupoid!\dots pro-unipotent}, we will mean an affine groupoid-scheme $\pi$ for which each $\pi(x)$ is a pro-unipotent group over $\F$. We will write $\GPD_\F$ and $\GPD_{\F,\uni}$ for the category of affine groupoid-schemes and the full subcategory of pro-unipotent groupoids, respectively.
\end{itemize}
\end{example}

The notion of a \emph{Hopf groupoid} makes sense in any symmetric monoidal category, and is simultaneously a generalisation and a special case of Definition~\ref{def:gpoid}.

\begin{definition}
Let $\TCat=(\Cat,\otimes,\1)$ be a symmetric monoidal category. A \emph{(cocommutative) coalgebra} in $\TCat$ is an object $\acoAlg$ of $\TCat$ endowed with a comultiplication map $\Delta\colon\acoAlg\rightarrow\acoAlg\otimes\acoAlg$ and a counit map $\varepsilon\colon \acoAlg\rightarrow\1$ such that $\Delta$ is coassociative, cocommutative and $\varepsilon$ is a counit for $\Delta$.

The coalgebras in $\TCat$ form a category $\coALG(\TCat)$. This category has finite products: the binary products are given by the natural coalgebra structure on the monoidal product $\acoAlg\otimes\acoAlg'$ of two coalgebras, and the terminal object is $\1$ with its natural coalgebra structure. For a coalgebra $\acoAlg$, the structure morphisms $\Delta\colon\acoAlg\rightarrow\acoAlg\otimes\acoAlg$ and $\varepsilon\colon\acoAlg\rightarrow\1$ are morphisms of coalgebras, namely the diagonal map and unique map to the terminal object in $\coALG(\TCat)$ respectively.
\end{definition}

\begin{definition}[{\cite[\S9]{homotopy_theory_operads}}]\label{def:hopf_gpoid}\index{Hopf groupoid}
Let $\TCat$ be a symmetric monoidal category. A \emph{(cocommutative) Hopf groupoid} in $\TCat$ (also called a \emph{Hopf category} in \cite{hopf_cats}) is a groupoid in $\coALG(\TCat)$; the category of Hopf groupoids in $\TCat$ is denoted $\HGPD(\TCat)$. A \emph{Hopf algebra} in $\TCat$ is a Hopf groupoid with one object; the category of Hopf algebras in $\TCat$ is denoted $\HALG(\TCat)$. A symmetric monoidal functor $F\colon\TDat\rightarrow\TCat$ induces also a functor $F\colon\HGPD(\TDat)\rightarrow\HGPD(\TCat)$ and similarly for $\HALG$.
\end{definition}

\begin{example}\label{ex:some_gpds}
\leavevmode
\begin{itemize}
	\item Suppose that $\Cat$ is a category with finite products, and let $\carTCat$ denote $\Cat$ with its Cartesian monoidal structure. Then a coalgebra in $\carTCat$ is the same as an object of $\Cat$ (with its comultiplication and counit given by the diagonal map and unique map to the terminal object, respectively). Hence $\HGPD(\carTCat)=\GPD(\Cat)$.
	\item Suppose that $\TCat=\VEC_\F$ is the category of vector spaces over a field $\F$, endowed with its usual tensor product. Then a Hopf algebra in $\VEC_\F$ is a Hopf algebra in the usual sense. We write $\HGPD_\F=\HGPD(\VEC_\F)$ and $\HALG_\F=\HALG(\VEC_\F)$ for the categories of Hopf groupoids and Hopf algebras over $\F$, respectively.
	
	\smallskip
	
	\noindent There is a functor $\F\cdot\colon\SET\rightarrow\VEC_\F$ sending a set $S$ to the vector space $\F S$ on basis $S$. This induces a functor $\F\cdot\colon\GPD\rightarrow\HGPD_\F$ sending a groupoid $\pi$ to its \emph{groupoid algebra} $\F\pi$. This is a multi-object generalisation of the functor sending a group $\pi$ to its group algebra $\F\pi$.
	\item Taking $\TCat=\fVEC_\F$, resp.\ $\gr\VEC_\F$, to be the category of filtered, resp.\ graded\footnote{In this paper, the symmetric monoidal structure on graded vector spaces will always be that of weight-graded vector spaces; i.e.\ we take the symmetric structure where no signs intervene.}, vector spaces over a field $\F$, we obtain the categories $\fHGPD_\F=\HGPD(\fVEC_\F)$ and $\grHGPD_\F=\HGPD(\grVEC_\F)$ of \emph{filtered, resp.\ graded, Hopf groupoids over $\F$}. There is a functor $\gr_\bullet\colon\fHGPD_\F\rightarrow\grHGPD_\F$ given by taking the associated graded.
\end{itemize}
\end{example}

Note that if $\pi$ is an affine groupoid-scheme over a field $\F$, the duals $\O(\pi(x,y))^\dual$ of the affine rings of the $\pi(x,y)$ do \emph{not} in general form a Hopf groupoid in $\VEC_\F$, since the comultiplication may not take values in $\O(\pi(x,y))^\dual\otimes\O(\pi(x,y))^\dual$ \cite[Example~3.61]{javier}. Instead, this is our first example of a \emph{complete Hopf groupoid}, i.e.\ a Hopf groupoid in a category of suitably completed vector spaces. There are two natural candidates for such a category of vector spaces: the category of \emph{complete filtered vector spaces} as in \cite[Appendix~A]{quillen1969rational} and \cite[\S7.3]{homotopy_theory_operads}; or the category of \emph{pro-finite dimensional vector spaces} as in \cite[\S3.2.6]{javier} and \cite[\S2.1]{betts2017motivic}. We will work here with the latter category, which has better abstract properties for our applications.

\begin{definition}\index{pro-finite dimensional vector space}
Let $\F$ be a field. The category $\CVEC_\F$ of \emph{pro-finite dimensional vector spaces (over $\F$)} is the category of pro-objects in the category of finite dimensional vector spaces. Explicitly, the objects of $\CVEC_\F$ are formal inverse limits $\varprojlim_iV_i$ of finite dimensional vector spaces indexed by a cofiltered poset\footnote{More generally, one can allow $\mathcal I$ to be a cofiltered category. This defines the same category $\CVEC_\F$ up to equivalence.} $\mathcal I$. The morphisms are given by
\[
\Hom_{\CVEC_\F}(\varprojlim\nolimits_jW_j,\varprojlim\nolimits_iV_i) := \varprojlim\nolimits_i\varinjlim\nolimits_j\Hom_{\VEC_\F}(W_j,V_i).
\]
The category $\CVEC_\F$ carries a symmetric \emph{completed tensor product} $\hatotimes$, defined by
\[
\Bigl(\varprojlim\nolimits_iV_i\Bigr)\hatotimes\Bigl(\varprojlim\nolimits_jW_j\Bigr) := \varprojlim\nolimits_{i,j}\bigl(V_i\otimes W_j\bigr),
\]
with unit the one-dimensional vector space $\1=\F$.

\smallskip

A \emph{filtration} $\M_\bullet$ on a pro-finite dimensional vector space $V$ is an increasing sequence
\[
\dots\leq\M_{-1}V\leq\M_0V\leq\M_1V\leq\dots
\]
of pro-finite dimensional subspaces; $\M_\bullet$ is called \emph{separated} just when $\bigcap_k\M_{-k}V=0$. The collection of filtered pro-finite dimensional vector spaces forms a symmetric monoidal category $(\fCVEC_\F,\hatotimes,\1)$ by making the usual definitions, and the objects with separated filtrations form a symmetric monoidal subcategory $(\fsCVEC_\F,\hatotimes,\1)$.

\smallskip

A \emph{grading} $\gr^\M_\bullet$ on a pro-finite dimensional vector space $V$ is a product decomposition
\[
V = \prod_{k\in\Z}\gr^\M_{-k}V.
\]
Again, the collection of graded pro-finite dimensional vector spaces assembles into a symmetric monoidal category $(\grCVEC_\F,\hatotimes,\1)$ by making the usual definitions.
\end{definition}

\begin{remark}
The symmetric monoidal category $\CVEC_\F$ is canonically dual to the symmetric monoidal category $\VEC_\F$ of all vector spaces. The duality takes a pro-finite dimensional vector space $V=\varprojlim_i(V_i)$ to its \emph{decompleted dual} $V^\reddual:=\varinjlim_i(V_i^\dual)$ and takes a vector space $W$ to its dual $W^\dual=\varprojlim_i(W_i^\dual)$, where $W_i$ runs through the finite dimensional subspaces of $W$. This description makes obvious many nice properties of $\CVEC_\F$: it is abelian and satisfies Grothendieck's axioms AB3--4 and AB3*--5*.
\end{remark}

\begin{example}\label{ex:affine_vs_complete_hopf}
The category $\CHGPD_\F$ of \emph{complete Hopf groupoids}\index{Hopf groupoid!\dots complete} over a field $\F$ is defined to be the category $\HGPD(\CVEC_\F)$ of Hopf groupoids in pro-finite dimensional vector spaces (cf.\ \cite[Definition~3.60]{javier}).

There is a (covariant) equivalence of categories $\AFF_\F\isoarrow\coALG(\CVEC_\F)$ between the category of affine schemes and the category of complete coalgebras (coalgebras in $\CVEC_\F$), taking an affine scheme $\pi$ to its dual affine ring $\O(\pi)^\dual$ and taking a complete coalgebra $\acoAlg$ to the affine scheme $\Spec(\acoAlg^\reddual)$ representing its functor of \emph{grouplike elements}\index{grouplike element}
\[
\Lambda \mapsto \acoAlg^\gplike(\Lambda) := \{\gamma\in\varprojlim\nolimits_i\left(\Lambda\otimes\acoAlg_i\right)\::\:\Delta(\gamma)=\gamma\hatotimes\gamma\text{ and }\varepsilon(\gamma)=1\},
\]
where $\acoAlg=\varprojlim_i(\acoAlg_i)$ and the inner inverse limit is taken in $\SET$. It follows purely formally that there is an equivalence of categories
\[
\GPD_\F \isoarrow \CHGPD_\F
\]
between the categories of affine groupoid-schemes and complete Hopf groupoids.
\end{example}

\subsection{Hopf groupoids in a tensor category}

We now examine in more detail the algebraic structure of Hopf groupoids when the base category $\TCat$ is suitably linear-algebraic in nature. We have in mind chiefly the categories $\VEC_\F$, $\fVEC_\F$, $\grVEC_\F$, $\CVEC_\F$, $\fCVEC_\F$, $\grCVEC_\F$ above, so in particular we are interested in settings where the underlying category of $\TCat$ may not be abelian. We thus restrict attention to categories satisfying the following slight weakening of the axioms of an abelian category.

\begin{definition}
For us, a \emph{tensor category} is a symmetric monoidal category $\TCat$ satisfying the following (self-dual) conditions:
\begin{itemize}
	\item the underlying category $\Cat$ is an additive category with kernels and cokernels (hence all finite limits and colimits);
	\item epimorphisms and strict epimorphisms in $\Cat$ are stable under base change, i.e.\ in a pullback square
	\begin{center}
	\begin{tikzcd}
	A' \arrow{r}{f'}\arrow{d} & B' \arrow{d} \\
	A \arrow{r}{f} & B,
	\end{tikzcd}
	\end{center}
	if $f$ is an epimorphism or strict epimorphism (i.e.\ a cokernel), then so is $f'$;
	\item monomorphisms and strict monomorphisms in $\Cat$ are stable under co-base change (in the dual sense to the above); and
	\item the monoidal product $\otimes$ preserves kernels and cokernels separately in each variable (hence is biadditive).
\end{itemize}
The first three axioms assert that $\Cat$ is \emph{almost abelian and integral} in the terminology of \cite[\S1]{rump}; the former property is more usually referred to as being \emph{quasi-abelian} \cite[Definition~1.1.3]{schneiders}. Note that in general monomorphisms and strict monomorphisms are stable under base change, and epimorphisms and strict epimorphisms are stable under co-base change \cite[Proposition~4]{rump}.
\end{definition}

\begin{remark}
The axioms of a tensor category $\TCat$ ensure that it possesses a well-behaved notion of subobjects. Specifically, by a \emph{subobject} of an object $B$, we shall mean an isomorphism class of strict monomorphisms $A\rightarrowtail B$. The collection of subobjects of $B$ is partially ordered by \emph{containment}, where $A$ contains $A'$ if and only if the inclusion $A'\rightarrowtail B$ factors (uniquely) through $A\rightarrowtail B$. If $A$ is a subobject of $B$, then the subobjects of $A$ are exactly the subobjects of $B$ contained in $A$ \cite[Propositions~1.1.7~\&~1.1.8]{schneiders}.

There are two important examples of subobjects, namely kernels of maps $B\rightarrow C$ and images of maps $f\colon C\rightarrow B$. Here, the \emph{image} of $f$ is the kernel of $\coker(f)$, or equivalently the unique subobject of $B$ such that $f$ factors through an epimorphism $C\rightarrow\im(f)$ \cite[Corollary~1]{rump}. Subobjects are preserved by any limits that exist in $\TCat$, as well as tensor products.

The poset of subobjects of $B$ is a lattice, with the meet of two subobjects $A,A'$ being their pullback $A\cap A':=A\times_BA'$ and their join being $A+A':=\im\left(A\oplus A'\rightarrow B\right)$. The poset of subobjects is both covariant and contravariant functorial: given a morphism $f\colon B\rightarrow C$ and subobjects $A\leq B$ and $D\leq C$, one can form the \emph{image} $f(A)=\im\left(A\rightarrow C\right)$ and \emph{preimage} $f^{-1}(D)=D\times_CB$. These operations form a Galois connection, i.e.\ $A\leq f^{-1}(D)$ if and only if $f(A)\leq D$; in particular $f(-)$ preserves joins and $f^{-1}(-)$ preserves intersections.

There is a dual notion of a \emph{quotient} of an object $B$, namely an isomorphism class of strict epimorphisms $B\twoheadrightarrow C$, which is partially ordered by domination. There is an order-reversing bijection between subobjects and quotients of $B$, sending a subobject $A$ to $B/A=\coker\left(A\rightarrowtail B\right)$ \cite[Remark~1.1.2]{schneiders}.
\end{remark}

\begin{definition}
Let $\aHAlg$ be a Hopf groupoid in a tensor category $\TCat$. An \emph{ideal} $\aHIdeal\unlhd\aHAlg$ consists of a collection of subobjects $\aHIdeal(x,y)\leq\aHAlg(x,y)$ in $\TCat$ such that, for every $x,y,z$, the images of the composition maps $\aHIdeal(y,z)\otimes\aHAlg(x,y)\rightarrow\aHAlg(x,z)$ and $\aHAlg(y,z)\otimes\aHIdeal(x,y)\rightarrow\aHAlg(x,z)$ are contained in $\aHIdeal(x,z)$. An ideal $\aHIdeal\unlhd\aHAlg$ is called a \emph{Hopf ideal} just when the image of each $\Delta\colon\aHIdeal(x,y) \rightarrow \aHAlg(x,y)\otimes\aHAlg(x,y)$ is contained in $\aHIdeal(x,y)\otimes\aHAlg(x,y)+\aHAlg(x,y)\otimes\aHIdeal(x,y)$. If $\aHIdeal$ is a Hopf ideal of $\aHAlg$, there is a natural Hopf groupoid structure on the quotients $(\aHAlg/\aHIdeal)(x,y):=\aHAlg(x,y)/\aHIdeal(x,y)$.

If $\aHIdeal_0$ and $\aHIdeal_1$ are two ideals of $\aHAlg$, their \emph{product} $\aHIdeal_0\aHIdeal_1$ is the ideal defined by
\begin{align*}
\aHIdeal_0\aHIdeal_1(x,y) &:= \im\left(\aHIdeal_0(x,y)\otimes\aHIdeal_1(x)\rightarrow\aHAlg(x,y)\right) \\
 &\;= \im\left(\aHIdeal_0(y)\otimes\aHAlg(x,y)\otimes\aHIdeal_1(x)\rightarrow\aHAlg(x,y)\right) \\
 &\;= \im\left(\aHIdeal_0(y)\otimes\aHIdeal_1(x,y)\rightarrow\aHAlg(x,y)\right).
\end{align*}
This operation defines an associative product on ideals whose unit is the unit ideal~$\aHAlg$. The product of two Hopf ideals is not in general a Hopf ideal.
\end{definition}

\begin{example}\label{ex:aug_ideal}
Let $f\colon\aHAlg\rightarrow\aHAlg'$ be a morphism of Hopf groupoids in a tensor category $\TCat$. Its \emph{kernel} $\ker(f)$ is, by definition, the collection of subobjects
\[
\ker(f)(x,y) := \ker\left(f\colon\aHAlg(x,y)\rightarrow\aHAlg'(f(x),f(y))\right).
\]
This is a Hopf ideal of $\aHAlg$.

In particular, taking $\aHAlg'=\1$ to be the terminal Hopf groupoid, for every Hopf groupoid $\aHAlg$ in $\TCat$, the collection of subspaces
\[
\aHIdeal(x,y) := \ker\left(\varepsilon\colon\aHAlg(x,y)\rightarrow\1\right)
\]
forms a Hopf ideal of $\aHAlg$, called the \emph{augmentation ideal}. In the particular case of a Hopf algebra, the counit $\varepsilon\colon\aHAlg\rightarrow\1$ is split by the unit map $\eta\colon\1\rightarrow\aHAlg$, and hence we have a direct sum decomposition $\aHAlg=\1\oplus\aHIdeal$.
\end{example}

\subsubsection{Pro-nilpotent Hopf algebras}

The powers of the augmentation ideal in a Hopf groupoid allow us to isolate a particular class of Hopf groupoids which are very closely related to pro-unipotent groupoids. In order to make sense of the definition, we need to further restrict attention to tensor categories $\CTCat$ in which limits are well-behaved, for example any of the categories $\CVEC_\F$, $\fCVEC_\F$, $\grCVEC_\F$ above.

\begin{definition}
A tensor category $\CTCat$ is called \emph{complete} just when:
\begin{itemize}
	\item $\CTCat$ has products (hence all limits);
	\item cofiltered limits in $\CTCat$ preserve cokernels; and
	\item the monoidal product $\hatotimes$ preserves products separately in each variable.
\end{itemize}
It follows from the existence of products that one can form the intersection of any set of subobjects of an object in $\CTCat$, and that nested intersections are preserved by taking images under morphisms.
\end{definition}

\begin{definition}\index{Hopf groupoid!\dots pro-nilpotent}
Let $\aHAlg$ be a Hopf groupoid in a complete tensor category $\CTCat$, and let $\aHIdeal^k$ denote the powers of the augmentation ideal from Example~\ref{ex:aug_ideal}. We say that $\aHAlg$ is \emph{pro-nilpotent} just when
\[
\bigcap_{k\geq0}\aHIdeal^k(x,y) = 0
\]
for all $x,y$, or equivalently when the natural map
\[
\aHAlg(x,y) \longrightarrow \varprojlim\nolimits_k\left(\aHAlg(x,y)/\aHIdeal^k(x,y)\right)
\]
is an isomorphism for all $x,y$. We write $\HGPD(\CTCat)_\nil$ for the full subcategory of pro-nilpotent Hopf groupoids.
\end{definition}

\begin{proposition}\label{prop:check_pro-nilpotence_locally}
Let $\aHAlg$ be a Hopf groupoid in a complete tensor category $\CTCat$. Then $\aHAlg$ is pro-nilpotent if and only if all of the Hopf algebras $\aHAlg(x)$ are pro-nilpotent.
\begin{proof}
If $\aHIdeal^k$ denotes the powers of the augmentation ideal of $\aHAlg$, then $\aHIdeal^k(x)$ are the powers of the augmentation ideal of the Hopf algebra $\aHAlg(x)$ for all $x$. Hence $\aHAlg$ being pro-nilpotent implies that each $\aHAlg(x)$ is pro-nilpotent. In the other direction, each $\aHIdeal^k(x,y)$ is the image of the multiplication map $\aHIdeal^k(x)\hatotimes\aHAlg(x,y)\rightarrow\aHAlg(x,y)$, and hence each $\aHAlg(x)$ being pro-nilpotent implies that $\aHAlg$ is pro-nilpotent, since images preserve nested intersections of subobjects. 
\end{proof}
\end{proposition}

For us, the relevance of the pro-nilpotence condition is that it corresponds to pro-unipotence on the side of affine groupoid-schemes.

\begin{lemma}\label{lem:unipotent_vs_nilpotent_hopf}
Let $\F$ be a field of characteristic $0$. Under the equivalence
\[
\GPD_\F \isoarrow \CHGPD_\F
\]
between affine groupoid-schemes and complete Hopf groupoids from Example~\ref{ex:affine_vs_complete_hopf}, the pro-unipotent groupoids correspond to the pro-nilpotent Hopf groupoids.
\begin{proof}
By Proposition~\ref{prop:check_pro-nilpotence_locally}, it suffices to check that under the equivalence between affine group-schemes and complete Hopf algebras, the pro-unipotent groups correspond to the pro-nilpotent Hopf algebras. This is provided by \cite[Proposition~3.94]{javier}.
\end{proof}
\end{lemma}

\subsection{Pro-nilpotent Lie algebras}

By specialising Lemma~\ref{lem:unipotent_vs_nilpotent_hopf}, we recover the classical equivalence of categories between pro-unipotent groups and pro-nilpotent Hopf algebras. As is well-known, this correspondence extends to a third, even simpler, type of linear algebraic object, namely pro-nilpotent Lie algebras.

\begin{definition}\index{Lie algebra}
A \emph{Lie algebra} in a $\Q$-linear tensor category $\TCat$ is an object $\aLAlg$ endowed with a Lie bracket
\[
[\cdot,\cdot]\colon\aLAlg\otimes\aLAlg\rightarrow\aLAlg
\]
which is antisymmetric and satisfies the Jacobi identity. With respect to the evident notion of morphism of Lie algebras, these form a category $\LIE(\TCat)$.

The \emph{descending central series} of a Lie algebra is the sequence of subobjects $\Cent^k\aLAlg\leq\aLAlg$ given by the images of the iterated commutator map $\aLAlg^{\otimes k}\rightarrow\aLAlg$.

We say that a Lie algebra $\aLAlg$ in a $\Q$-linear complete tensor category $\CTCat$ is \emph{pro-nilpotent}\index{Lie algebra!\dots pro-nilpotent} just when $\bigcap_{k\geq1}\Cent^k\aLAlg=0$, or equivalently just when the natural map $\aLAlg\rightarrow\varprojlim_k\left(\aLAlg/\Cent^k\aLAlg\right)$ is an isomorphism. We denote by $\LIE(\CTCat)_\nil$ the full subcategory of pro-nilpotent Lie algebras.
\end{definition}

\begin{example}
One can show that the category $\CLIE_{\F,\nil}:=\LIE(\CVEC_\F)_\nil$ of pro-nilpotent Lie algebras over a characteristic $0$ field $\F$ is equivalent to the category of pro-objects in the category of finite dimensional nilpotent Lie algebras. The same is not true without the assumption of pro-nilpotence: there are Lie algebras in $\CVEC_\F$ which are not inverse limits of finite dimensional Lie algebras (dualise \cite[p9]{michaelis}).
\end{example}

\begin{example}
Let $\TCat$ be a $\Q$-linear tensor category and $\aHAlg$ a Hopf algebra in $\TCat$. Endowed with the algebra commutator bracket, $\aHAlg$ has the structure of a Lie algebra. Its Lie subalgebra $\aHAlg^\prim$ of \emph{primitive elements}\index{primitive element} is, by definition, the kernel of the map
\[
\Delta-\eta\otimes1-1\otimes\eta\colon\aHAlg\rightarrow\aHAlg^{\otimes2}.
\]
If $\aHAlg$ is a pro-nilpotent Hopf algebra in a $\Q$-linear complete tensor category $\CTCat$, then $\aHAlg^\prim$ is a pro-nilpotent Lie algebra.
\end{example}

The main theorem of this section, the completed Milnor--Moore Theorem~\ref{thm:cmm}, asserts that the functor $\aHAlg\mapsto\aHAlg^\prim$ is an equivalence of categories. The inverse functor is given by a suitably completed analogue of the usual universal enveloping algebra functor.

\begin{definition}\label{def:CUEnv}\index{completed universal enveloping algebra}
Let $\CTCat$ be a $\Q$-linear complete tensor category. The \emph{complete tensor algebra} $\CCTensor V$ on an object $V$ of $\CTCat$ is the pro-nilpotent Hopf algebra
\[
\CCTensor V := \prod_{k\geq0}V^{\hatotimes k}
\]
whose multiplication is induced from the tensor product and whose comultiplication is induced from the map $\eta\hatotimes1+1\hatotimes\eta\colon V\rightarrow \1\hatotimes V\oplus V\hatotimes\1$.

If $\aLAlg$ is a pro-nilpotent Lie algebra in $\CTCat$, its \emph{completed universal enveloping algebra} $\CUEnv(\aLAlg)$ is, by definition, the quotient of $\CCTensor\aLAlg$ by the ideal generated by the image of the map
\begin{align*}
\aLAlg^{\hatotimes2} &\rightarrow \CCTensor\aLAlg \\
x\hatotimes y &\mapsto [x,y]-xy-yx.
\end{align*}
This image is a Hopf ideal of $\CCTensor\aLAlg$, and hence $\CUEnv(\aLAlg)$ inherits the structure of a pro-nilpotent Hopf algebra.

The completed tensor algebra functor $\CCTensor\colon\Cat\rightarrow\HALG(\CTCat)_\nil$ and completed universal enveloping algebra functor $\CUEnv\colon\LIE(\CTCat)_\nil\rightarrow\HALG(\CTCat)_\nil$ are left adjoint to the primitive elements functor: any map $V\rightarrow\aHAlg^\prim$, resp.\ any Lie algebra map $\aLAlg\rightarrow\aHAlg^\prim$, extends uniquely to a Hopf algebra map $\CCTensor(V)\rightarrow\aHAlg$, resp.\ $\CUEnv(\aLAlg)\rightarrow\aHAlg$.
\end{definition}

\begin{example}
If $\aLAlg$ is a finite dimensional nilpotent Lie algebra over a characteristic $0$ field $\F$, then its completed universal enveloping algebra $\CUEnv(\aLAlg)$ is the completion of the usual enveloping algebra $\UEnv(\aLAlg)$ with respect to its augmentation ideal, as in \cite[Definition~3.73]{javier}.
\end{example}

\begin{theorem}[Completed Milnor--Moore]\label{thm:cmm}
Let $\CTCat$ be a $\Q$-linear complete tensor category and let $\aLAlg$ and $\aHAlg$ be a pro-nilpotent Lie algebra and a pro-nilpotent Hopf algebra in $\CTCat$, respectively. Then the natural maps
\[
\aLAlg \rightarrow \CUEnv(\aLAlg)^\prim \:\:\text{ and }\:\: \CUEnv(\aHAlg^\prim) \rightarrow \aHAlg,
\]
are isomorphisms. The functors $\CUEnv\colon \LIE(\CTCat)_\nil \leftrightarrows \HALG(\CTCat)_\nil : (-)^\prim$ form an adjoint equivalence of categories.
\end{theorem}

\begin{corollary}[to the proof]\label{cor:complete_tensor_functor}
Let $F\colon \CTDat \rightarrow \CTCat$ be a product-preserving symmetric monoidal functor between $\Q$-linear complete tensor categories, and let $\aLAlg$ and $\aHAlg$ be a pro-nilpotent Lie algebra and a pro-nilpotent Hopf algebra, respectively, in $\CTDat$. Then the natural maps
\[
\CUEnv(F(\aLAlg)) \rightarrow F(\CUEnv(\aLAlg)) \:\:\text{and}\:\: F(\aHAlg^\prim) \rightarrow F(\aHAlg)^\prim
\]
are isomorphisms.
\end{corollary}

\begin{example}\label{ex:graded_CUEnv}
Let $\F$ be a characteristic $0$ field and $\aLAlg$ a filtered pro-nilpotent Lie algebra over $\F$. Then there is a canonical isomorphism
\[
\CUEnv(\gr_\bullet\aLAlg) \isoarrow \gr_\bullet\CUEnv(\aLAlg)
\]
of graded pro-nilpotent Hopf algebras. Note that this doesn't immediately follow from the definition of $\CUEnv$, since the functor $\gr_\bullet$ doesn't preserve cokernels.
\end{example}

\subsubsection{Structure of pro-nilpotent Hopf algebras}

We now turn to the proof of the completed Milnor--Moore Theorem~\ref{thm:cmm}. The argument, which is the most complicated in this section, is essentially the same as the proof of \cite[Theorem~3.2.19]{homotopy_theory_operads}, using limits and our pro-nilpotence property in place of colimits and a conilpotence property.

For this section, we fix a $\Q$-linear complete tensor category $\CTCat$. For an object $V\in\CTCat$ we denote by $\CSym^\bullet(V)=\prod_{n\geq0}\CSym^n(V)$ the complete symmetric algebra on $V$, where $\CSym^n(V):=V^{\hatotimes n}/S_n$. This is a pro-nilpotent Hopf algebra in $\CTCat$, which is moreover commutative. A straightforward calculation similar to \cite[Proposition~7.2.14(a)]{homotopy_theory_operads} verifies that $\CSym^\bullet(V)^\prim=V$; this is the special case of Theorem~\ref{thm:cmm} for $V$ endowed with the zero Lie bracket.

We will repeatedly use the following construction. If $\aHAlg$ is a pro-nilpotent Hopf algebra and $\phi\colon V\rightarrow\aHAlg$ is a map in $\CTCat$ whose image is contained in the augmentation ideal, we get a symmetrised product map
\[
\s\colon\CSym^\bullet(V)\rightarrow\aHAlg
\]
characterised by the fact that it acts on $\CSym^n(V)$ by $\frac1{n!}\sum_{\sigma\in S_n}\mu^{(n)}\circ\phi^{\hatotimes n}\circ\sigma_*$ where $\mu^{(n)}$ denotes the $n$-fold multiplication in $\aHAlg$ and $\sigma_*$ denotes the natural action of $S_n$ on $V^{\hatotimes n}$. When the image of $V\rightarrow\aHAlg$ is contained in the primitive elements, the symmetrised product map is a morphism of coaugmented coalgebras with involution.

The symmetrised product map affords us very precise control of the structure of pro-nilpotent Hopf algebras. This constitutes the bulk of the work in the proof of Theorem~\ref{thm:cmm}.

\begin{lemma}\label{lem:structure_of_hopf}
Let $\aHAlg$ be a pro-nilpotent Hopf algebra in $\CTCat$. Then the symmetrised product map
\[
\s\colon\CSym^\bullet(\aHAlg^\prim) \rightarrow \aHAlg
\]
is an isomorphism of coaugmented coalgebras with involution.
\begin{proof}[Proof, following {\cite[Theorem~3.2.16]{homotopy_theory_operads}}]
The endomorphism algebra of $\aHAlg$ carries an associative \emph{convolution product} defined by $f\ast g:=\mu\circ(f\hatotimes g)\circ\Delta$. The endomorphism $e^0=\eta\circ\varepsilon$ is a unit for the convolution product, and the identity endomorphism decomposes as $1=e^0+\pi$ where $\pi\colon\aHAlg\rightarrow\aHIdeal$ is the projection onto the augmentation ideal. Pro-nilpotence of $\aHAlg$ ensures that we can make sense of certain infinite sums of endomorphisms: if $(f_k)_{k\geq0}$ are endomorphisms such that $\im(f_k)\leq\aHIdeal^k$ for all $k$, then $\sum_{k\geq0}f_k$ defines an endomorphism of $\aHAlg$, namely the inverse limit of the morphisms $\sum_{k=0}^nf_k\colon\aHAlg\rightarrow\aHAlg/\aHIdeal^{n+1}$.

We thus define an endomorphism $e^1$ by
\[
e^1=\log_*(1)=\sum_{k\geq1}\frac{(-1)^{k+1}}k\pi^{\ast k}=\sum_{k\geq1}\frac{(-1)^{k+1}}k\mu^{(k)}\circ\pi^{\hatotimes k}\circ\Delta^{(k)}
\]
where $\mu^{(k)}$ and $\Delta^{(k)}$ denote the $k$-fold multiplication and comultiplication maps respectively.

The endomorphisms $e^0,e^1$ satisfy the identities
\begin{align}
\Delta\circ e^0 &= (e^0\hatotimes e^0)\circ\Delta & \Delta\circ e^1 &= (e^1\hatotimes e^0+e^0\hatotimes e^1)\circ\Delta \nonumber \\
\varepsilon\circ e^0 &= \varepsilon & \varepsilon\circ e^1 &= 0 \label{eq:e^1_identities}\\
e^0\circ\eta &= \eta & e^1\circ\eta &= 0. \nonumber
\end{align}
The identities for $e^0$ are easy to verify. For those for $e^1$, we note that the structure morphisms $\Delta\colon\aHAlg\rightarrow\aHAlg\hatotimes\aHAlg$, $\varepsilon\colon\aHAlg\rightarrow\1$ and $\eta\colon\1\rightarrow\aHAlg$ are morphisms of Hopf algebras (since $\aHAlg$ is cocommutative), and hence commute with the action of $\log_*(1)$ on either side. An easy computation verifies that $\log_*(1_\1)=0$ and $\log_*(1_{\aHAlg\hatotimes\aHAlg})=\log_*(1_\aHAlg\hatotimes1_\aHAlg)=\log_*(1_\aHAlg)\hatotimes e^0+e^0\hatotimes\log_*(1_\aHAlg)$, which gives the desired identities.

The comultiplication identity implies in particular that $(\Delta-1\hatotimes\eta-\eta\hatotimes1)\circ e^1=0$, so that $e^1\aHAlg\leq\aHAlg^\prim$. We will show that this inclusion is an equality, and that the symmetrised product map
\[
\s\colon\CSym^\bullet(e^1\aHAlg) \rightarrow \aHAlg
\]
is an isomorphism. Indeed, it suffices to prove the latter claim, since then $\s$, being an isomorphism of coaugmented coalgebras, induces an isomorphism $e^1\aHAlg=\CSym^\bullet(e^1\aHAlg)^\prim\isoarrow\aHAlg^\prim$.

We construct a map
\[
\overline\s\colon\aHAlg \rightarrow \CSym^\bullet(e^1\aHAlg)
\]
by taking the product of the maps $\frac1{r!}(e^1)^{\hatotimes r}\circ\Delta^{(r)}\colon\aHAlg\rightarrow\CSym^r(e^1\aHAlg)$; it follows purely formally that $\overline\s$ is a morphism of coaugmented coalgebras. On the one hand, the composite $\s\overline\s$ is given by
\begin{align*}
\s\overline\s=\sum_{k\geq0}\frac1{(k!)^2}\sum_{\sigma\in S_k}\mu^{(k)}\circ\sigma_*\circ(e^1)^{\hatotimes k}\circ\Delta^{(k)} &= \sum_{k\geq0}\frac1{k!}\mu^{(k)}\circ(e^1)^{\hatotimes k}\circ\Delta^{(k)} \\
 &= \sum_{k\geq0}\frac1{k!}(e^1)^{\ast k} = \exp_*(e^1) = 1
\end{align*}
since $\Delta$ is cocommutative.

On the other hand, to show that $\overline\s\s=1$, it suffices to prove that it acts as the identity on each $\CSym^n(e^1\aHAlg)$. For  $n=0,1$, this follows from~\eqref{eq:e^1_identities}, since for instance
\begin{align*}
\overline\s\s\circ e^1 &= \sum_{k\geq0}\frac1{k!}(e^1)^{\hatotimes k}\circ\Delta^{(k)}\circ e^1 \\
 &= \sum_{k\geq0}\frac1{k!}(e^1)^{\hatotimes k}\circ\left(\sum_{i=0}^{k-1}(e^0)^{\hatotimes i}\hatotimes e^1\hatotimes(e^0)^{\hatotimes k-i-1}\right)\circ\Delta^{(k)} = e^1,
\end{align*}
where in the final equality we use the fact, again a consequence of~\eqref{eq:e^1_identities}, that $e^0$ and $e^1$ are orthogonal idempotents.

We now show by induction that $\overline\s\s$ acts as the identity on $\CSym^{\leq n}(e^1\aHAlg):=\prod_{r=0}^n\CSym^r(e^1\aHAlg)$ for all $n$. To do this, we note that the map $\Delta-1\hatotimes\eta-\eta\hatotimes1$ takes $\CSym^{\leq n}(e^1\aHAlg)$ into $\CSym^{\leq n-1}\hatotimes\CSym^{\leq n-1}$, so by our inductive assumption and the fact that $\overline\s\s$ is a morphism of coaugmented coalgebras, we have that
\[
(\Delta-1\hatotimes\eta-\eta\hatotimes1)\circ\overline\s\s=\Delta-1\hatotimes\eta-\eta\hatotimes1
\]
on $\CSym^{\leq n}(e^1\aHAlg)$. Since $\ker(\Delta-1\hatotimes\eta-\eta\hatotimes1)=\CSym^\bullet(e^1\aHAlg)^\prim=e^1\aHAlg$, this says that $\overline\s\s|_{\CSym^{\leq n}(e^1\aHAlg)}=1+\psi$ for some $\psi\colon\CSym^{\leq n}(e^1\aHAlg)\rightarrow e^1\aHAlg$. In particular, $\overline\s\s(\CSym^{\leq n}(e^1\aHAlg))\leq\CSym^{\leq n}(e^1\aHAlg)$. But then we have
\[
1+\psi = \overline\s\s|_{\CSym^{\leq n}(e^1\aHAlg)}=\left(\overline\s\s|_{\CSym^{\leq n}(e^1\aHAlg)}\right)^{\circ2} = 1+2\psi
\]
since $\s\overline\s=1$, and hence $\psi=0$. This completes the induction, and hence the proof that $\s$ and $\overline\s$ are mutually inverse.
\end{proof}
\end{lemma}

\begin{remark}
As the notation suggests, the endomorphisms $e^0,e^1$ are part of a complete orthogonal system of idempotents $(e^r)_{r\geq0}$, called the \emph{Eulerian idempotents} \cite[p250]{homotopy_theory_operads}. The corresponding product decomposition is exactly the natural one on the complete symmetric algebra.
\end{remark}

To complete the proof of Theorem~\ref{thm:cmm}, we describe the decomposition from Lemma~\ref{lem:structure_of_hopf} in the particular case of the completed universal enveloping algebra of a pro-nilpotent Lie algebra. This is an analogue of the Poincar\'e--Birkhoff--Witt Theorem.

\begin{lemma}[Completed Poincar\'e--Birkhoff--Witt]\label{lem:cpbw}
Let $\aLAlg$ be a pro-nilpotent Lie algebra in $\CTCat$. Then the symmetrised product map
\[
\s\colon\CSym^\bullet(\aLAlg) \rightarrow \CUEnv(\aLAlg)
\]
is an isomorphism of coaugmented coalgebras with involution. The map $\aLAlg\rightarrow\CUEnv(\aLAlg)$ induces an isomorphism $\aLAlg\isoarrow\CUEnv(\aLAlg)^\prim$.
\begin{proof}[Proof, following {\cite[Theorem~7.2.17]{homotopy_theory_operads}}]
The two statements in the lemma are equivalent by Lemma~\ref{lem:structure_of_hopf}.

For an object $V$ of $\CTCat$, we let $\CFLie(V)$ denote the free pro-nilpotent Lie algebra, i.e.\ the pro-nilpotent Lie algebra representing maps in $\CTCat$ from $V$ to pro-nilpotent Lie algebras. This can be constructed, for example, as a quotient of the product $\prod_w V^{\hatotimes|w|}$ indexed by words $w$ of the free magma on one generator. It follows from the universal properties (Definition~\ref{def:CUEnv}) that the completed universal enveloping algebra of $\CFLie(V)$ is the completed tensor algebra $\CCTensor(V)$. An argument similar to that of \cite[Propositions~7.2.8~\&~7.2.14(b)]{homotopy_theory_operads} shows that the map $\CFLie(V)\rightarrow\CUEnv(\CFLie(V))=\CTensor(V)$ admits a splitting in $\CTCat$ and sets up an isomorphism $\CFLie(V)\isoarrow\CTensor(V)^\prim$. The lemma is thus proved in this case.

We deduce the general case from the free case. Any pro-nilpotent Lie algebra $\aLAlg$ admits a presentation
\begin{center}
\begin{tikzcd}
\CFLie(\CFLie(\aLAlg)) \arrow[shift left=1.2ex]{r}\arrow[shift right=1.2ex]{r} & \CFLie(\aLAlg) \arrow{l}\arrow{r} & \aLAlg
\end{tikzcd}
\end{center}
as a reflexive coequaliser in $\LIE(\CTCat)_\nil$. This is also a reflexive coequaliser in $\Cat$. Since the completed symmetric algebra functor and completed universal enveloping algebra functors are left adjoints, they preserve coequalisers, and so we deduce a diagram
\begin{center}
\begin{tikzcd}
\CSym^\bullet(\CFLie^2(\aLAlg)) \arrow[shift left=1.2ex]{r}\arrow[shift right=1.2ex]{r}\arrow{d}{\s}[swap]{\viso} & \CSym^\bullet(\CFLie(\aLAlg)) \arrow{l}\arrow{r}\arrow{d}{\s}[swap]{\viso} & \CSym^\bullet(\aLAlg) \arrow{d}{\s} \\
\CCTensor(\CFLie(\aLAlg)) \arrow[shift left=1.2ex]{r}\arrow[shift right=1.2ex]{r} & \CCTensor(\aLAlg) \arrow{l}\arrow{r} & \CUEnv(\aLAlg)
\end{tikzcd}
\end{center}
in which both rows are reflexive coequalisers in $\HALG(\CTCat)_\nil$, hence in $\Cat$, and the obvious squares commute. Since the two left-hand vertical maps are isomorphisms by the free pro-nilpotent Lie algebra case, so too is the right-hand vertical map, as desired.
\end{proof}
\end{lemma}

\begin{remark}
There is an alternative formulation of the classical Poincar\'e--Birkhoff--Witt Theorem, which says that if $(x_i)_{i\in I}$ is an ordered basis of a Lie algebra $\aLAlg$ over a characteristic $0$ field $\F$, then a basis of the universal enveloping algebra $\UEnv(\aLAlg)$ is given by the products $x_{i_1}\dots x_{i_k}$ indexed by the set $\mathbf I$ of increasing sequences $i_1\leq\dots\leq i_k$ in $I$. The reader should be cautioned that the natural analogue of this theorem in the context of pro-nilpotent Lie algebras is untrue: given an isomorphism $\F^I\isoarrow\aLAlg$ of pro-finite-dimensional vector spaces with $\aLAlg$ a pro-nilpotent Lie algebra over a characteristic $0$ field $\F$, the induced morphism $\F^{\mathbf I}\rightarrow\CUEnv(\aLAlg)$ can fail to be an isomorphism.

For example, let us take $\aLAlg$ to be the three-dimensional Lie algebra with ordered basis $x,y,z$ and bracket defined by\[[x,y]=[y,z]=-x-z\:,\:[x,z]=0,\]so that $\aLAlg$ is the three-dimensional Heisenberg Lie algebra and hence is two-step nilpotent. It is easy to check that in the completed universal enveloping algebra $\CUEnv(\aLAlg)$ one has the identity $x\cdot\exp(y)+\exp(y)\cdot z=0$, which asserts an infinite linear dependence between the elements $xy^j$ and $y^jz$. This ensures that the map $\F^{\mathbf I}\rightarrow\CUEnv(\aLAlg)$ has non-trivial kernel.
\end{remark}

\begin{proof}[Proof of Theorem~\ref{thm:cmm}]
One of the two assertions is already contained in the statement of Lemma~\ref{lem:cpbw}. The other follows from Lemmas~\ref{lem:structure_of_hopf} and~\ref{lem:cpbw} by combining the isomorphisms $\CUEnv(\aHAlg^\prim)\iso\CSym^\bullet(\aHAlg^\prim)\iso\aHAlg$.
\end{proof}

\begin{proof}[Proof of Corollary~\ref{cor:complete_tensor_functor}]
We deal with the case of the completed universal enveloping algebra, the case of primitive elements being similar. The functor $F$, preserving products and direct summands, commutes with the construction of completed symmetric algebras. It is easy to see that the square
\begin{center}
\begin{tikzcd}
\CSym^\bullet(F\aLAlg) \arrow{r}{\s}[swap]{\hiso}\arrow{d}{\viso} & \CUEnv(F\aLAlg) \arrow{d} \\
F\CSym^\bullet(\aLAlg) \arrow{r}{F\s}[swap]{\hiso} & F\CUEnv(\aLAlg)
\end{tikzcd}
\end{center}
commutes, so that the right-hand vertical map is an isomorphism as desired.
\end{proof}

We will need one further corollary of Lemma~\ref{lem:structure_of_hopf} in what follows. Let $\aHAlg$ be a pro-nilpotent Hopf algebra, with corresponding pro-nilpotent Lie algebra $\aLAlg=\aHAlg^\prim$. For any coalgebra $\acoAlg$, we have natural bijections
\begin{equation}\label{eq:homs_into_hopf_algs}
\Hom_{\coALG(\CTCat)}(\acoAlg,\aHAlg) \iso \Hom_{\coALG(\CTCat)}(\acoAlg,\CSym^\bullet(\aLAlg)) \iso \Hom_{\Cat}(\acoAlg,\aLAlg)
\end{equation}
from Lemma~\ref{lem:structure_of_hopf} and the fact that $\CSym^\bullet$ is right adjoint to the forgetful functor $\coALG(\CTCat)\rightarrow\Cat$. The left-hand side of~\eqref{eq:homs_into_hopf_algs} carries a group structure (as $\aHAlg$ is a group object in $\coALG(\CTCat)$), and the right-hand side carries a $\Q$-Lie algebra structure, where the Lie bracket of $x,y\colon\acoAlg\rightarrow\aLAlg$ is given by the composite
\[
\acoAlg \overset\Delta\longrightarrow \acoAlg\hatotimes\acoAlg \overset{x\hatotimes y}\longrightarrow \aLAlg\hatotimes\aLAlg \overset{[\cdot,\cdot]}\longrightarrow \aLAlg.
\]
These two structures are related as follows.

\begin{proposition}\label{prop:super-general_bch}
Under the bijection~\eqref{eq:homs_into_hopf_algs} above, the group law on the set $\Hom_{\coALG(\CTCat)}(\acoAlg,\aHAlg)$ corresponds to the group law on $\Hom_{\Cat}(\acoAlg,\aLAlg)$ given by the Baker--Campbell--Hausdorff formula
\[
x\cdot y = x + y + \frac12[x,y] + \frac1{12}[x,[x,y]] - \frac1{12}[y,[x,y]] + \dots
\]
which, for maps $x,y\colon\acoAlg\rightarrow\aLAlg$ in $\CTCat$, converges to a well-defined map $x\cdot y\colon\acoAlg\rightarrow\aLAlg$ by pro-nilpotence.
\begin{proof}
The bijection $\Hom_{\Cat}(\acoAlg,\aLAlg)\iso\Hom_{\coALG(\CTCat)}(\acoAlg,\aHAlg)$ sends $x\colon\acoAlg\rightarrow\aLAlg\hookrightarrow\aHAlg$ to
\[
\exp_\ast(x) := \sum_{k\geq0}\frac1{k!}x^{\ast k} = \sum_{k\geq0}\frac1{k!}\mu^{(k)}\circ x^{\hatotimes k}\circ\Delta^{(k)}
\]
where $\ast$ denotes the convolution product on $\Hom_{\Cat}(\acoAlg,\aHAlg)$. Thus we have that $\exp_\ast(x\cdot y)=\exp_\ast(x)\ast\exp_*(y)$, and hence $x\cdot y$ is given by the Baker--Campbell--Hausdorff formula.
\end{proof}
\end{proposition}

\subsection{Filtered groupoids}\label{ss:filtered_gpoids}\index{groupoid!\dots filtered}

The pro-unipotent groupoids we study in this paper frequently carry extra structures, such as Galois actions, which are of interest to us. Via the correspondence from Lemma~\ref{lem:unipotent_vs_nilpotent_hopf}, such extra structures are reflected on the side of pro-nilpotent Hopf algebras. For example, if $G$ is an affine monoid-scheme, then the correspondence from Lemma~\ref{lem:unipotent_vs_nilpotent_hopf} lifts to an equivalence of categories between $G$-equivariant pro-unipotent groupoids and pro-nilpotent Hopf algebras in the category of pro-finite dimensional representations of $G$ (interpreted in the obvious sense). This is a formal consequence of the equivalence of categories between affine schemes and complete coalgebras from Example~\ref{ex:affine_vs_complete_hopf}.

There is one kind of extra structure, however, for which the correspondence between these structures on groupoids and on Hopf groupoids is not formal, namely filtrations. On the side of pro-nilpotent Hopf groupoids, the notion of filtration is more or less obvious, but the definition on the side of groups is slightly less natural.

\begin{definition}\label{def:filtered_gpd}
Let $\pi$ be a group in a category $\Cat$ with finite products. By a \emph{{non-positive} filtration} on $\pi$, we mean an increasing sequence
\[
\dots\leq\M_{-2}\pi\leq\M_{-1}\pi\leq\M_0\pi=\pi
\]
of normal subgroups (i.e.\ kernels of morphisms $\pi\rightarrow\pi'$) such that the commutator map $\M_{-i}\pi\times\M_{-j}\pi\rightarrow\pi$ factors through $\M_{-i-j}\pi$ for all $i,j\geq0$. A \emph{morphism} $f\colon\pi'\rightarrow\pi$ of filtered groups is a morphism of the underlying groups such that $f\colon\M_{-i}\pi'\rightarrow\pi$ factors through $\M_{-i}\pi$ for all $i\geq0$.

Let now $\pi$ be a groupoid in $\Cat$. By a \emph{non-positive filtration} on $\pi$, we mean a non-positive filtration on each of the groups $\pi(x)$ such that the conjugation map
\[
\pi(x,y)\times\M_{-i}\pi(x) \rightarrow \pi(y)
\]
factors through $\M_{-i}\pi(y)$ for every $x,y$ and all $i\geq0$. Equivalently, the isomorphism $\pi(x,y)\times\pi(x) \isoarrow \pi(y)\times\pi(x,y)$ induces isomorphisms
\[
\pi(x,y)\times\M_{-i}\pi(x) \isoarrow \M_{-i}\pi(y)\times\pi(x,y)
\]
for every $x,y$ and all $i\geq0$.

We denote by $\fleqGPD(\Cat)$ the category of non-positively filtered groupoids in~$\Cat$.
\end{definition}

\begin{example}\label{ex:normal_subalgebra}
Suppose that $\TCat$ is a tensor category, and that $\aHAlg$ is a Hopf algebra in $\TCat$ (so a group in $\coALG(\TCat)$). The normal Hopf subalgebras of $\aHAlg$ (normal subgroups in $\coALG(\TCat)$) are exactly the Hopf subalgebras of the form $\1\oplus\aHIdeal$ for $\aHIdeal$ a Hopf ideal of $\aHAlg$.
\end{example}

If $\TCat$ is a tensor category, we define $\fleqTCat$ to be the category of non-positively filtered objects in $\TCat$, i.e.\ the category whose objects are objects $V$ of $\TCat$ endowed with an increasing sequence
\[
\dots\leq\M_{-2}V\leq\M_{-1}V\leq\M_0V=V
\]
of subobjects, and whose morphisms are morphisms in $\TCat$ preserving these subobjects. The usual definition of tensor product  filtrations endows $\fleqTCat$ with the structure of a symmetric monoidal category.

\begin{proposition}
Let $\TCat$ be a tensor category. Then $\fleqTCat$ is also a tensor category, which is complete if $\TCat$ is.
\begin{proof}
It is easy to see that $\fleqTCat$ is additive and has kernels and cokernels, and that limits in $\fleqTCat$ are computed pointwise. A morphism $f\colon A\rightarrow B$ in $\fleqTCat$ is monic (resp.\ epic) in $\fleqTCat$ if and only if it is monic (resp.\ epic) in $\TCat$. The morphism $f$ is strictly monic in $\fleqTCat$ if and only if it is strictly monic in $\CTCat$ and each $\M_{-i}A\rightarrow\M_{-i}B$ is strictly monic (i.e.\ $A$ has the subspace filtration); it is strictly epic in $\fleqTCat$ if and only if it is strictly epic in $\TCat$ and each $\M_{-i}A\rightarrow\M_{-i}B$ is epic in $\TCat$ (i.e.\ $A$ has the image filtration).

It follows immediately that monomorphisms in $\fleqTCat$ are stable under co-base change and epimorphisms are stable under base change. To show that strict epimorphisms are stable under base change, consider a pullback square
\begin{center}
\begin{tikzcd}
A' \arrow{r}{f'}\arrow{d} & B' \arrow{d} \\
A \arrow[two heads]{r}{f} & B
\end{tikzcd}
\end{center}
in $\fleqTCat$ where $f$ is a strict epimorphism. Thus $f'$ is a strict epimorphism in $\TCat$, and from the pullback squares
\begin{center}
\begin{tikzcd}
\M_{-i}A' \arrow{r}{f'}\arrow{d} & \M_{-i}B' \arrow{d} \\
\M_{-i}A \arrow{r}{f} & \M_{-i}B
\end{tikzcd}
\end{center}
with $f$ epic we deduce also that $f'\colon\M_{-i}A'\rightarrow\M_{-i}B'$ is epic in $\TCat$. Hence $f'$ is strictly epic in $\fleqTCat$. Stability of monomorphisms under co-base change follows by running the same argument in the opposite category $(\fleqTCat)^\op=\fgeq((\TCat)^\op)$ (where $\fgeq$ denotes non-negatively filtered objects).

It is easy to check that $\otimes$ preserves strict epimorphisms, and hence that it preserves cokernels in each argument. That it preserves kernels in each argument follows by the same argument in the opposite category.

Completeness of $\fleqCTCat$ for $\CTCat$ a complete tensor category is also easy to check, since limits of filtered objects are computed pointwise and cofiltered limits in $\CTCat$ commute with images.
\end{proof}
\end{proposition}

\begin{remark}
There are many variants one can make of the above definition, for instance filtrations which are bounded above or below, or (in the case of a complete tensor category $\CTCat$) unbounded filtrations or separated filtrations. These all also define tensor categories and complete tensor categories respectively.
\end{remark}

We thus have two possible notions of a non-positive filtration on a pro-nilpotent Hopf groupoid in a complete tensor category $\CTCat$: either a filtration in the sense of Definition~\ref{def:filtered_gpd} or a lift to a Hopf groupoid in $\fleqCTCat$. These two notions are equivalent to one another, which for example allows us to interpret filtrations on pro-unipotent groupoids on the side of pro-nilpotent Hopf groupoids.

\begin{lemma}\label{lem:filtrations_are_filtrations}
Let $\CTCat$ be a $\Q$-linear complete tensor category. Then there is an equivalence of categories
\[
\fleqHGPD(\CTCat)_\nil \isoarrow \HGPD(\fleqCTCat)_\nil
\]
compatible with the forgetful functors to $\HGPD(\CTCat)_\nil$.
\end{lemma}

\begin{example}
Let $\F$ be a characteristic $0$ field. Then there is an equivalence of categories
\[
\fleqGPD_{\F,\uni} \isoarrow \HGPD(\fleqCVEC_\F)_\nil.
\]
\end{example}

\subsubsection{Filtrations on pro-nilpotent Hopf algebras and Lie algebras}\label{sss:filtrations_on_liegebras}

To begin with, we prove Lemma~\ref{lem:filtrations_are_filtrations} in the single-object case, constructing an equivalence
\begin{equation}\label{eq:one_filtration_is_filtration}
\fleqHALG(\CTCat)_\nil\isoarrow\HALG(\fleqCTCat)_\nil.
\end{equation}
By the completed Milnor--Moore Theorem~\ref{thm:cmm}, the right-hand side of~\eqref{eq:one_filtration_is_filtration} is equivalent to $\LIE(\fleqCTCat)_\nil$, while the left-hand side is equivalent to the category of pro-nilpotent Lie algebras $\aLAlg$ in $\CTCat$ endowed with a sequence
\[
\dots\leq\M_{-2}\aLAlg\leq\M_{-1}\aLAlg\leq\M_0\aLAlg=\aLAlg
\]
of Lie ideals satisfying a certain commutation relation. To construct the equivalence~\eqref{eq:one_filtration_is_filtration}, it suffices to check that this commutation relation is equivalent to the sequence of Lie ideals being compatible with the Lie bracket.

\begin{proposition}
Let $\aHAlg$ be a pro-nilpotent Hopf algebra in $\CTCat$ with corresponding pro-nilpotent Lie algebra $\aLAlg=\aHAlg^\prim$. Let
\[
\dots\leq\aHAlg_{-2}\leq\aHAlg_{-1}\leq\aHAlg_0=\aHAlg
\]
be an increasing sequence of normal subalgebras (in the sense of Example~\ref{ex:normal_subalgebra}), and write $\M_{-k}\aLAlg:=\aHAlg_{-k}\cap\aLAlg$ for the corresponding sequence of Lie ideals. Then the subalgebras define a filtration on $\aHAlg$ (in the sense of Definition~\ref{def:filtered_gpd}) if and only if the $\M_{-k}\aLAlg$ endow $\aLAlg$ with the structure of a pro-nilpotent Lie algebra in $\fleqCTCat$.
\begin{proof}
The ideals $\M_{-k}\aLAlg$ endow $\aLAlg$ with the structure of a Lie algebra in $\fleqCTCat$ if and only if they satisfy $[\M_{-i},\M_{-j}]\leq\M_{-i-j}$ for all $i,j$, while the subalgebras $\aHAlg_{-k}$ define a filtration on $\aHAlg$ if and only if the image of $\aHAlg_{-i}\hatotimes\aHAlg_{-j}$ under the group commutator map
\begin{equation}\label{eq:silly_commutator}
\aHAlg^{\hatotimes2} \xrightarrow{\Delta^{\hatotimes2}} \aHAlg^{\hatotimes4}\xrightarrow{1\hatotimes\tau\hatotimes1} \aHAlg^{\hatotimes4}\xrightarrow{1^{\hatotimes2}\hatotimes S^{\hatotimes2}} \aHAlg^{\hatotimes4}\xrightarrow{\mu^{(4)}} \aHAlg
\end{equation}
is contained in $\aHAlg_{-i-j}$ for all $i,j$, where $\tau$ denotes the interchange of factors map. One implication is straightforward: the group commutator map~\eqref{eq:silly_commutator} restricts to the Lie bracket $\aLAlg^{\hatotimes2}\rightarrow\aLAlg$, and hence if the $\aHAlg_{-k}$ define a filtration on $\aHAlg$, then the $\M_{-k}\aLAlg$ are compatible with the Lie bracket.

For the converse implication, we argue via the Yoneda embedding. The subalgebras $\aHAlg_{-k}$ define a filtration on $\aHAlg$ if and only if, for every coalgebra $\acoAlg$, the group commutator map~\eqref{eq:silly_commutator} takes $\Hom_{\coALG(\CTCat)}(\acoAlg,\aHAlg_{-i})\times\Hom_{\coALG(\CTCat)}(\acoAlg,\aHAlg_{-j})$ into $\Hom_{\coALG(\CTCat)}(\acoAlg,\aHAlg_{-i-j})$ for all $i,j$. By Proposition~\ref{prop:super-general_bch}, this is equivalent to $\Hom_{\Cat}(\acoAlg,\M_{-i}\aLAlg)\times\Hom_{\Cat}(\acoAlg,\M_{-j}\aLAlg)$ being taken into $\Hom_{\Cat}(\acoAlg,\M_{-i-j}\aLAlg)$ by the commutator for the group law on $\Hom_{\Cat}(\acoAlg,\aLAlg)$ defined by the Baker--Campbell--Hausdorff formula by Proposition~\ref{prop:super-general_bch}. But this commutator is given by
\[
x\cdot y\cdot x^{-1}\cdot y^{-1} = [x,y] + \frac12[x,[x,y]] + \frac12[y,[x,y]] + \dots
\]
so this is certainly implied by the $\M_{-k}\aLAlg$ being compatible with the Lie bracket.
\end{proof}
\end{proposition}

\subsubsection{Filtrations on Hopf groupoids}

Having proved Lemma~\ref{lem:filtrations_are_filtrations} in the single-object context, we now explain how to extend this to the multi-object context. This essentially amounts to relating filtrations on a Hopf groupoid $\aHAlg$ to filtrations on the Hopf algebras $\aHAlg(x)$.

\begin{proposition}\label{prop:when_do_filtrations_extend}
Let $\TCat$ be a tensor category and $\aHAlg$ a Hopf groupoid in $\TCat$. Suppose that we have endowed each Hopf algebra $\aHAlg(x)$ with the structure of an object in $\HALG(\fleqTCat)$. Then $\aHAlg$ lifts to an object of $\HGPD(\fleqTCat)$ if and only the isomorphisms
\begin{equation}\label{eq:hopf_compatibility}
\aHAlg(x,y)\otimes\aHAlg(x) \isoarrow \aHAlg(y)\otimes\aHAlg(x,y)
\end{equation}
induce isomorphisms
\[
\aHAlg(x,y)\otimes\M_{-k}\aHAlg(x) \isoarrow \M_{-k}\aHAlg(y)\otimes\aHAlg(x,y)
\]
for all $x,y$ and all $k$; in this case the lift is unique up to unique isomorphism. If $\aHAlg$ and $\aHAlg'$ are Hopf groupoids in $\fleqTCat$ and $f\colon\aHAlg'\rightarrow\aHAlg$ is a morphism of Hopf groupoids in $\TCat$, then $f$ is a morphism in $\HGPD(\fleqTCat)$ if and only if each morphism $f\colon\aHAlg'(x)\rightarrow\aHAlg(f(x))$ is a morphism in $\HALG(\fleqTCat)$.
\begin{proof}
If $\aHAlg$ is a Hopf groupoid in $\fleqTCat$, then in particular each $\M_{-k}\aHAlg$ is an ideal in $\aHAlg$, whence it follows that~\eqref{eq:hopf_compatibility} identifies the subobjects claimed, and that for every $x,y$, we have
\begin{equation}\tag{$\ast$}\label{eq:filtrations_determined}
\M_{-k}\aHAlg(x,y) = \im\left(\aHAlg(x,y)\otimes\M_{-k}\aHAlg(x)\rightarrow\aHAlg(x,y)\right).
\end{equation}
This establishes the ``only if'' part of the first claim, uniqueness of lifts, and both directions of the final claim.

For the remaining part, that the compatibility condition implies the existence of a lift, we check that~\eqref{eq:filtrations_determined} defines on $\aHAlg$ the structure of an object of $\HGPD(\fleqTCat)$, i.e.\ that the structure morphisms of $\aHAlg$ lie in $\fleqTCat$. Compatibility with units is automatic, and compatibility with composition is the assertion that the ideals $\M_{-k}\aHAlg$ satisfy $\M_{-i}\cdot\M_{-j}\leq\M_{-i-j}$ for all $i,j$, which is obvious. Compatibility with the coalgebra structure on $\aHAlg(x,y)$ follows from the fact that the composition map $\aHAlg(x,y)\otimes\aHAlg(x)\rightarrow\aHAlg(x,y)$ is a morphism of coalgebras. Compatibility with the antipode follows from the commuting square
\begin{center}
\begin{tikzcd}
\aHAlg(x,y)\otimes\aHAlg(x) \arrow{r}\arrow{d}{\viso} & \aHAlg(x,y) \arrow{d}{S}[swap]{\viso} \\
\aHAlg(y,x)\otimes\aHAlg(y) \arrow{r} & \aHAlg(y,x).
\end{tikzcd}
\end{center}
\vspace{-0.9cm}
\end{proof}
\end{proposition}

\begin{proof}[Proof of Lemma~\ref{lem:filtrations_are_filtrations}]
By the discussion in \S\ref{sss:filtrations_on_liegebras}, it suffices to verify that if $\aHAlg$ is a pro-nilpotent Hopf groupoid in $\CTCat$ such that each $\aHAlg(x)$ is endowed with the structure of a pro-nilpotent Hopf algebra in $\fleqCTCat$, then these structures are compatible in the sense of Proposition~\ref{prop:when_do_filtrations_extend} if and only if the corresponding normal subalgebras $\aHAlg_{-k}(x)$ are compatible in the sense of Definition~\ref{def:filtered_gpd}. For this, we simply note that the isomorphism
\[
\aHAlg(x,y)\hatotimes\aHAlg(x) \isoarrow \aHAlg(y)\hatotimes\aHAlg(x,y)
\]
is compatible with Hopf algebra structures on $\aHAlg(x)$ and $\aHAlg(y)$, and hence restricts to an isomorphism
\begin{equation}\label{eq:lie_compatibility}
\aHAlg(x,y)\hatotimes\aLAlg(x) \isoarrow \aLAlg(y)\hatotimes\aHAlg(x,y)
\end{equation}
where $\aLAlg(x)$ and $\aLAlg(y)$ denote the Lie algebras of $\aHAlg(x)$ and $\aHAlg(y)$ respectively. It follows that both of the compatibility conditions are equivalent to the condition that~\eqref{eq:lie_compatibility} identifies $\aHAlg(x,y)\hatotimes\M_{-k}\aLAlg(x)$ and $\M_{-k}\aLAlg(y)\hatotimes\aHAlg(x,y)$ for all $x,y$ and for all $k$.
\end{proof}

\subsection{Mal\u cev completion}\index{Mal\u cev completion}

Finally, to conclude this appendix, we turn our attention to the main algebraisation process by which we will produce the pro-unipotent groupoids in this paper: Mal\u cev completion. The properties of this operation are well-known in the single-object context (i.e.\ for pro-unipotent groups), and we will sketch here the multi-object generalisation of this theory, with particular attention to the behaviour of filtrations under the operation. For the duration of this section, fix a characteristic $0$ field $\F$.

\begin{definition}[cf.\ {\cite[Definition 3.95]{javier}} and {\cite[Chapter 9.2]{homotopy_theory_operads}}]
Let $\pi$ be a groupoid (in $\SET$). Its \emph{$\F$-linear Mal\u cev completion} (or \emph{$\F$-pro-unipotent completion}) $\pi^\F$ is the pro-unipotent groupoid on object-set $\ob(\pi)$ admitting a morphism $\pi\rightarrow\pi^\F(\F)$ of groupoids which is initial among maps from $\pi$ into the $\F$-points of pro-unipotent groupoids. More precisely, for every pro-unipotent groupoid $U$, every morphism $\pi\rightarrow U(\F)$ of groupoids extends uniquely to a morphism $\pi^\F\rightarrow U$ of pro-unipotent groupoids. The Mal\u cev completion exists and is functorial by abstract nonsense. It is easy to check that $\ob(\pi^\F)=\ob(\pi)$ and that $\pi^\F(x)$ is the Mal\u cev completion of the group $\pi(x)$ for all $x$.

A filtration $\M_\bullet$ on a groupoid $\pi$ induces a filtration on $\pi^\F$, whereby $\M_{-n}\pi^\F(x)$ is the Zariski-closure of the image of $\M_{-n}\pi(x)\rightarrow\pi^\F(x)(\F)$, or equivalently the kernel of $\pi^\F(x)\rightarrow(\pi/\M_{-n})^\F(x)$.
\end{definition}

\begin{remark}\label{rmk:malcev_of_nilpotent}
The theory of Mal\u cev completions is particularly well-behaved for finitely generated nilpotent groups $\pi$. Indeed, the Mal\u cev completion of such a group is always finite-dimensional, of dimension equal to the Hirsch length of $\pi$ \cite[Definition 4.4]{CMZ}, and the kernel of the Mal\u cev completion map $\pi\rightarrow\pi^\F(\F)$ is exactly the torsion subgroup of $\pi$. It follows from additivity of Hirsch length that Mal\u cev completion is exact on finitely generated nilpotent groups.
\end{remark}

In this paper, we use two explicit descriptions of the Mal\u cev completion. The first describes the dual affine ring of the Mal\u cev completion as a completed groupoid-algebra.

\begin{lemma}\label{lem:J-adic_and_malcev_completion}
Let $\pi$ be a groupoid for which each group $\pi(x)$ is finitely generated, and write $\aHIdeal$ for the augmentation ideal of the groupoid-algebra $\F\pi$ from Example~\ref{ex:some_gpds}. Each vector space $\F\pi(x,y)/\aHIdeal^{n+1}(x,y)$ is finite dimensional, so that $\Fhat\pi=\varprojlim\left(\F\pi/\aHIdeal^{n+1}\right)$ has the structure of a pro-nilpotent Hopf groupoid. Then there is a canonical isomorphism $\O(\pi^\F)^\dual\iso\Fhat\pi$ induced by the evident map $\pi\rightarrow\Fhat\pi^\gplike(\F)$. If $\pi$ carries a filtration, this isomorphism is compatible with the filtrations induced on either side.
\begin{proof}
This is an easy multi-object generalisation of \cite[Theorem 3.99]{javier}. In the filtered case, $\M_{-n}\O(\pi^\F)^\dual$ is the span of the elements $\log(\gamma_1)\dots\log(\gamma_k)$ for $\gamma_i\in\M_{-n_i}\pi^\F(\F)$ with $\sum_in_i\geq n$, while $\M_{-n}\Fhat\pi$ is the span of the same elements with $\gamma_i\in\M_{-n_i}\pi$. Since $\M_{-n_i}\pi$ is Zariski-dense in $\M_{-n_i}\pi^\F(\F)$, the spans of these elements are the same.
\end{proof}
\end{lemma}

The second explicit description we will use relates the Mal\u cev Lie algebra $\Lie(\pi^\F)$ of a finitely generated group $\pi$ to another Lie algebra constructed classically from $\pi$, namely its graded Lie algebra with respect to a connected filtration as in \cite{labute}.

\begin{proposition}\label{prop:graded_malcev_liegebra}
Let $\pi$ be a finitely generated group endowed with a filtration $\W_\bullet$ such that $\W_{-1}\pi=\pi$. Then the $\W$-graded Mal\u cev Lie algebra $\gr^\W_\bullet\Lie(\pi^\F)$ is canonically isomorphic to the $\W$-graded pro-nilpotent Lie algebra\[\gr^\W_\bullet\Lie_\F(\pi):=\prod_{n>0}\left(\F\otimes_\Z\frac{\W_{-n}\pi}{\W_{-n-1}\pi}\right)\]with Lie bracket induced by the commutator maps $\W_{-n}\pi\times\W_{-m}\pi\rightarrow\W_{-n-m}\pi$.
\begin{proof}
The desired isomorphism $\gr^\W_\bullet\Lie_\F(\pi)\isoarrow\gr^\W_\bullet\Lie(\pi^\F)$ is the map sending $\gamma\in\W_{-n}\pi$ to $\log(\gamma)\in\gr^\W_{-n}\Lie(\pi^\F)$. It is easy to see that this map is a morphism of Lie algebras; we will show it is an isomorphism in two steps.

Suppose first that the quotients $\pi/\W_{-n}$ are torsion-free for all $n$. Pick a basis $\gamma^{(n)}_1,\dots,\gamma^{(n)}_{k_n}$ of the free $\Z$-module $\frac{\W_{-n}\pi}{\W_{-n-1}\pi}$ for each $n$, and lift each $\gamma^{(n)}_i$ to an element of $\pi$. It follows that the elements $\gamma^{(m)}_i$ with $m\leq n$ comprise a Mal\u cev basis of $\pi/\W_{-n-1}$ for every $n$ \cite[Definition~4.5]{CMZ}. It is then a theorem of Jennings \cite[Theorem~6.7]{CMZ} that the elements $\log(\gamma^{(m)}_i)$ with $m\leq n$ comprise a basis of the Mal\u cev Lie algebra $\Lie\left((\pi/\W_{-n-1})^\F\right)$ for every $n$. In particular, the elements $\log(\gamma^{(n)}_i)$ comprise a basis of\[\gr^\W_{-n}\Lie(\pi^\F)=\ker\left(\Lie\left((\pi/\W_{-n-1})^\F\right)\rightarrow\Lie\left((\pi/\W_{-n})^\F\right)\right).\]This shows that the desired map is an isomorphism in every degree $-n$ as desired.

Now suppose that $\W_\bullet$ is general. We write $\W'_\bullet$ for the filtration induced on $\pi$ by the map $\pi\rightarrow\pi^\F(\F)$, so that $\W'_{-n}\pi$ is the kernel of the map $\pi\rightarrow\left(\pi/\W_{-n}\right)^\F$. It follows from Remark~\ref{rmk:malcev_of_nilpotent} that $\W'_{-n}\pi/\W_{-n}\pi$ is the torsion subgroup of $\pi/\W_{-n}\pi$, and hence that each $\pi/\W'_{-n}$ is torsion-free. We see from this that the natural maps $\gr^\W_\bullet\Lie(\pi^\F)\rightarrow\gr^{\W'}_\bullet\Lie(\pi^\F)$ and $\gr^\W_\bullet\Lie_\F(\pi)\rightarrow\gr^{\W'}_\bullet\Lie_\F(\pi)$ are isomorphisms. Hence we reduce to the first case.
\end{proof}
\end{proposition} 

\bibliographystyle{alpha}
\printindex

\bibliography{references}

\end{document}